\documentclass[10pt,a4paper,oneside,english]{amsart}
\usepackage[T1]{fontenc}
\usepackage[utf8]{inputenc}
\usepackage{babel}
\usepackage{mathtools}
\usepackage{amstext}
\usepackage{amsthm}
\usepackage{amssymb}
\usepackage{esint}
\PassOptionsToPackage{normalem}{ulem}
\usepackage{ulem}
\usepackage[unicode=true,pdfusetitle,
 bookmarks=true,bookmarksnumbered=false,bookmarksopen=false,
 breaklinks=false,pdfborder={0 0 1},backref=false,colorlinks=false]
 {hyperref}

\makeatletter

\pdfpageheight\paperheight
\pdfpagewidth\paperwidth

\numberwithin{equation}{section}
\theoremstyle{plain}
\newtheorem{thm}{\protect\theoremname}[section]
\theoremstyle{definition}
\newtheorem{defn}[thm]{\protect\definitionname}
\theoremstyle{remark}
\newtheorem{rem}[thm]{\protect\remarkname}
\theoremstyle{plain}
\newtheorem{lem}[thm]{\protect\lemmaname}
\theoremstyle{plain}
\newtheorem{prop}[thm]{\protect\propositionname}
\theoremstyle{plain}
\newtheorem{cor}[thm]{\protect\corollaryname}

\makeatother

\providecommand{\corollaryname}{Corollary}
\providecommand{\definitionname}{Definition}
\providecommand{\lemmaname}{Lemma}
\providecommand{\propositionname}{Proposition}
\providecommand{\remarkname}{Remark}
\providecommand{\theoremname}{Theorem}

\begin{document}
\title[Blow-up solutions to CSS]{Construction of blow-up manifolds to the equivariant self-dual Chern--Simons--Schrödinger
equation}
\author{\noindent Kihyun Kim}
\email{khyun@ihes.fr}
\address{IHES, 35 route de Chartres, 91440 Bures-sur-Yvette, France}
\author{\noindent Soonsik Kwon}
\email{soonsikk@kaist.edu}
\address{Department of Mathematical Sciences, Korea Advanced Institute of Science
and Technology, 291 Daehak-ro, Yuseong-gu, Daejeon 34141, Korea}
\begin{abstract}
We consider the self-dual Chern--Simons--Schrödinger equation (CSS)
under equivariance symmetry. Among others, (CSS) has a static solution
$Q$ and the pseudoconformal symmetry. We study the quantitative description
of pseudoconformal blow-up solutions $u$ such that 
\[
u(t,r)-\frac{e^{i\gamma_{\ast}}}{T-t}Q\Big(\frac{r}{T-t}\Big)\to u^{\ast}\quad\text{as }t\to T^{-}.
\]
When the equivariance index $m\geq1$, we construct a set of initial
data (under a codimension one condition) yielding pseudoconformal
blow-up solutions. Moreover, when $m\geq3$, we establish the codimension
one property and Lipschitz regularity of the initial data set, which
we call the \emph{blow-up manifold}.

This is a forward construction of blow-up solutions, as opposed to
authors' previous work \cite{KimKwon2019arXiv}, which is a backward
construction of blow-up solutions with prescribed asymptotic profiles.
In view of the instability result of \cite{KimKwon2019arXiv}, the
codimension one condition established in this paper is expected to
be optimal.

We perform the modulation analysis with a robust energy method developed
by Merle, Raphaël, Rodnianski, and others. One of our crucial inputs
is a remarkable \emph{conjugation identity}, which (with self-duality)
enables the method of supersymmetric conjugates as like Schrödinger
maps and wave maps. It suggests how we proceed to higher order derivatives
while keeping the Hamiltonian form and construct adapted function
spaces with their coercivity relations. More interestingly, it shows
a deep connection with the Schrödinger maps at the linearized level
and allows us to find a repulsivity structure for higher order derivatives.
\end{abstract}

\maketitle
\global\long\def\R{\mathbb{R}}%
\global\long\def\C{\mathbb{C}}%
\global\long\def\Z{\mathbb{Z}}%
\global\long\def\N{\mathbb{N}}%
\global\long\def\D{\mathbf{D}}%
\global\long\def\Im{\mathrm{Im}}%
\global\long\def\Re{\mathrm{Re}}%
\global\long\def\tilde#1{\widetilde{#1}}%
\global\long\def\hat#1{\widehat{#1}}%

\tableofcontents{}

\section{Introduction}

In this paper, we study a quantitative description of the blow-up
dynamics for the equivariant self-dual Chern--Simons--Schrödinger
equation. More precisely, we construct a codimension one manifold
of initial data yielding the pseudoconformal blow-up solutions. In
contrast to the backward construction in the authors' previous work
\cite{KimKwon2019arXiv}, here we take an initial value problem point
of view and investigate the forward construction of the blow-up dynamics.

\subsection{Self-dual Chern--Simons--Schrödinger equation}

We consider the \emph{self-dual Chern}--\emph{Simons}--\emph{Schrödinger}
equation
\begin{equation}
\left\{ \begin{aligned}\D_{t}\phi & =i\D_{j}\D_{j}\phi+i|\phi|^{2}\phi,\\
F_{01} & =-\Im(\overline{\phi}\D_{2}\phi),\\
F_{02} & =\Im(\overline{\phi}\D_{1}\phi),\\
F_{12} & =-\tfrac{1}{2}|\phi|^{2},
\end{aligned}
\right.\label{eq:CovariantCSS}
\end{equation}
where $\phi:\R^{1+2}\to\C$ is a scalar field, $\D_{\alpha}\coloneqq\partial_{\alpha}+iA_{\alpha}$
for $\alpha\in\{0,1,2\}$ are the covariant derivatives associated
with the real-valued $1$-form $A\coloneqq A_{0}dt+A_{1}dx_{1}+A_{2}dx_{2}$,
and $F_{\alpha\beta}\coloneqq\partial_{\alpha}A_{\beta}-\partial_{\beta}A_{\alpha}$
are the curvature components. Repeated index $j$ means that we sum
over $j\in\{1,2\}$.

The Chern--Simons theory is a gauge theory on a three-dimensional
domain (space-time) and can be used to describe planar physical phenomena,
i.e. the dynamics of particles restricted to a plane. Several examples
are the Quantum Hall effect and high temperature superconductivity.
This is a sharp contrast to the Yang Mills or Maxwell theory formulated
on the four-dimensional space-time. In the 1990s, Chern--Simons models
were introduced to investigate vortex solutions in the planar quantum
electromagnetics. Jackiw and Pi introduced in \cite{JackiwPi1990PRL,JackiwPi1990PRD,JackiwPi1991PRD,JackiwPi1992Progr.Theoret.}
the Chern--Simons--Schrödinger equations as non-relativistic quantum
models that describe the dynamics of a large number of interacting
charged particles in the plane with electromagnetic gauge field.

\subsubsection*{Coulomb gauge}

\eqref{eq:CovariantCSS} is \emph{gauge-invariant}; any solution $(\phi,A)$
to \eqref{eq:CovariantCSS} and a function $\chi:\R^{1+2}\to\R$ give
rise to a gauge-equivalent solution $(e^{i\chi}\phi,A-d\chi)$ to
\eqref{eq:CovariantCSS}. In order to consider the Cauchy problem
of \eqref{eq:CovariantCSS}, we need to fix a gauge. In other words,
we choose the representatives of gauge-equivalent classes. In this
paper, we fix the gauge by imposing the \emph{Coulomb gauge condition}:
\[
\partial_{1}A_{1}+\partial_{2}A_{2}=0.
\]
A brief computation yields 
\begin{align*}
A_{0}=A_{0}[\phi] & =\epsilon_{jk}\Delta^{-1}\partial_{j}\Im(\overline{\phi}\D_{k}\phi),\\
A_{j}=A_{j}[\phi] & =\tfrac{1}{2}\epsilon_{jk}\Delta^{-1}\partial_{k}|\phi|^{2},
\end{align*}
where $\epsilon_{jk}$ is the anti-symmetric tensor with $\epsilon_{12}=1$.
Now \eqref{eq:CovariantCSS} can be reduced into a single evolution
equation of $\phi$ by substituting the above formulae. As recognized
in \cite{JackiwPi1990PRD}, \eqref{eq:CovariantCSS} under the Coulomb
gauge has a Hamiltonian structure. We note that the gauge potentials
$A_{\alpha}$ are of long-range as the convolution kernel of $\Delta^{-1}\partial_{j}$
has a nontrivial tail of size $O(|x|^{-1})$.

\subsubsection*{Equivariance}

Under the Coulomb gauge condition, we focus on \emph{equivariant}
solutions. This means that we work on solutions $\phi$ of the form
\[
\phi(t,x)=e^{im\theta}u(t,r),
\]
where $(r,\theta)$ denotes the usual polar coordinate of $\R^{2}$,
and $m\in\Z$ is called the \emph{equivariance index}. Under equivariance,
it is also convenient to write the connection $A$ in the polar coordinates,
i.e. $A=A_{0}dt+A_{r}dr+A_{\theta}d\theta$. It turns out that $A_{0}$,
$A_{r}$, and $A_{\theta}$ are spherically symmetric and their radial
parts are given by 
\[
A_{r}=0,\quad A_{\theta}=A_{\theta}[u]=-\frac{1}{2}\int_{0}^{r}|u|^{2}r'dr',\quad A_{0}=A_{0}[u]=-\int_{r}^{\infty}(m+A_{\theta})|u|^{2}\frac{dr'}{r'}.
\]
As a consequence, we can write \eqref{eq:CovariantCSS} in terms of
$u$: 
\begin{equation}
i\partial_{t}u+(\partial_{rr}+\frac{1}{r}\partial_{r})u-\Big(\frac{m+A_{\theta}}{r}\Big)^{2}u-A_{0}u+|u|^{2}u=0.\tag{CSS}\label{eq:CSS}
\end{equation}
This is the main equation that we consider.

\subsubsection*{Symmetries and conservation laws}

\eqref{eq:CSS} has various symmetries and associated conservation
laws. From the time translation ($u(t,r)\mapsto u(t+t_{0},r)$ for
$t_{0}\in\R$) and phase rotation symmetry ($u(t,r)\mapsto e^{i\gamma}u(t,r)$
for $\gamma\in\R$), the \emph{energy} and \emph{charge} are conserved:
\begin{align*}
E[u] & \coloneqq\int\frac{1}{2}|\partial_{r}u|^{2}+\frac{1}{2}\Big(\frac{m+A_{\theta}}{r}\Big)^{2}|u|^{2}-\frac{1}{4}|u|^{4}, & \text{(Energy)}\\
M[u] & \coloneqq\int|u|^{2}, & \text{(Charge)}
\end{align*}
where we denoted $\int f\coloneqq2\pi\int_{0}^{\infty}f(r)rdr$. There
is also the time reversal symmetry $u(t,r)\mapsto\overline{u}(-t,r)$.\footnote{In fact, there is the time reversal symmetry for \eqref{eq:CovariantCSS}
without symmetry reductions. It is \emph{not} simply given by conjugating
the scalar field $\phi$; it reads $\phi(t,x_{1},x_{2})\mapsto\overline{\phi}(-t,x_{1},-x_{2})$,
$A_{\alpha}(t,x_{1},x_{2})\mapsto A_{\alpha}(-t,x_{1},-x_{2})$ for
$\alpha\in\{0,2\}$, and $A_{1}(t,x_{1},x_{2})\mapsto-A_{1}(-t,x_{1},-x_{2})$.
In particular, the time reversal symmetry preserves the equivariance
index, and hence \eqref{eq:CSS} still enjoys the time reversal symmetry.}

Of particular importance are the \emph{$L^{2}$-scaling} \emph{invariance}
\[
u(t,r)\mapsto\frac{1}{\lambda}u\Big(\frac{t}{\lambda^{2}},\frac{r}{\lambda}\Big),\qquad\lambda\in\R_{+},
\]
and the \emph{pseudoconformal symmetry}
\[
u(t,r)\mapsto\frac{1}{t}e^{i\frac{|x|^{2}}{4t}}u\Big(-\frac{1}{t},\frac{r}{t}\Big).
\]
The associated algebraic identities are the \emph{virial identities}
\[
\left\{ \begin{aligned}\partial_{t}\Big(\int r^{2}|u|^{2}\Big) & =4\int\Im(\overline{u}\,r\partial_{r}u),\\
\partial_{t}\int\Im(\overline{u}\,r\partial_{r}u) & =4E[u].
\end{aligned}
\right.
\]

Let us finally mention the \emph{Hamiltonian structure} of \eqref{eq:CSS}:
\begin{equation}
\partial_{t}u=-i\frac{\delta E}{\delta u},\label{eq:HamiltonianStructure}
\end{equation}
where $\frac{\delta}{\delta u}$ is the Fréchet derivative under the
real inner product $(f,g)_{r}\coloneqq\int\Re(\overline{f}g)$.

\subsubsection*{Self-duality}

One of the special features of \eqref{eq:CSS} is the \emph{self-duality}.
We introduce the \emph{Bogomol'nyi operator} 
\[
\tilde{\D}_{+}\coloneqq\D_{1}+i\D_{2}.
\]
Within Coulomb gauge (so that $A_{\alpha}=A_{\alpha}[\phi]$), $\tilde{\mathbf{D}}_{+}$
also depends on $\phi$ (i.e., $\tilde{\mathbf{D}}_{+}=\tilde{\mathbf{D}}_{+}^{(\phi)}$).
However, we will suppress this $\phi$-dependence if there is no confusion.
We define $\D_{+}=\mathbf{D}_{+}^{(u)}$ by the radial part of $\tilde{\D}_{+}$:
\begin{equation}
\begin{aligned}\tilde{\D}_{+}[u(r)e^{im\theta}] & =[\D_{+}^{(u)}u](r)e^{i(m+1)\theta}.\\{}
[\D_{+}^{(u)}u](r) & =\partial_{r}u-\tfrac{1}{r}(m+A_{\theta}[u])u.
\end{aligned}
\label{eq:Def-D+}
\end{equation}
We note that $\tilde{\D}_{+}$ shifts the equivariance index by $1$.

Now we obtain the \emph{self-dual expression} of the energy: 
\begin{equation}
E[u]=\frac{1}{2}\int|\D_{+}^{(u)}u|^{2}.\label{eq:SelfDualEnergy}
\end{equation}
Then the Hamiltonian structure \eqref{eq:HamiltonianStructure} yields
the \emph{self-dual form} of \eqref{eq:CSS}: 
\begin{equation}
\partial_{t}u=-iL_{u}^{\ast}\D_{+}^{(u)}u,\label{eq:SelfDualEquation}
\end{equation}
where $L_{u}^{\ast}$ is the adjoint of the linearized operator $L_{u}$
of the expression $\D_{+}^{(u)}u$. We note that this self-dual expression
is tied to the fact that the coupling constant of the nonlinearity
(i.e. the coefficient of $|u|^{2}u$) is equal to $1$. As a consequence,
the energy is always nonnegative. We will later see that this leads
to the self-dual factorization of the linearized operator at the static
solution.

\subsubsection*{Static solution $Q$}

Let us fix $m$ to be a nonnegative integer. It is of great physical
interest to study time independent solutions to \eqref{eq:CSS}, so
called \emph{vortex} (or, \emph{static}) solutions. To find such solutions,
we need to solve time-independent \eqref{eq:CSS} (setting $i\partial_{t}u=0$).
However, the self-dual factorization \eqref{eq:SelfDualEquation}
says that it suffices to solve the first-order \emph{Bogomol'nyi equation}
\begin{equation}
\D_{+}^{(Q)}Q=0.\label{eq:Bogomolnyi-Eq}
\end{equation}
Note that $Q$ must have zero energy. In fact, a solution (with enough
regularity and decay) has zero energy if and only if it is static;
see \cite{HuhSeok2013JMP} and also the discussion in \cite[Section 1.3]{KimKwon2019arXiv}.

One can find a solution $Q$ to \eqref{eq:Bogomolnyi-Eq} following
Jackiw--Pi \cite{JackiwPi1990PRD}. In fact it reduces to the Liouville
equation, whose solutions are known. Thus \eqref{eq:Bogomolnyi-Eq}
has the following explicit solution (when $m\geq0$) 
\[
Q(r)=\sqrt{8}(m+1)\frac{r^{m}}{1+r^{2(m+1)}}.
\]
Moreover, the solution $Q$ to \eqref{eq:Bogomolnyi-Eq} is unique
up to scaling and phase rotation symmetries; see \cite{ChouWan1994PacificJMath,ByeonHuhSeok2012JFA,ByeonHuhSeok2016JDE}
and also the discussion in \cite[Section 1.3]{KimKwon2019arXiv}.

\subsection{Known results}

The equation \eqref{eq:CSS} under equivariance is known to be well-posed
in the scaling critical space $L^{2}(\R^{2})$. Indeed, Liu--Smith
\cite[Section 2]{LiuSmith2016} showed the small data global well-posedness
and large data local well-posedness of \eqref{eq:CSS}.

We mention some previous works on the covariant Chern--Simons--Schrödinger
equation without symmetry (this allows a general coupling constant
$g\in\mathbb{R}$ multiplied to the cubic nonlinearity $i|\phi|^{2}\phi$
in the RHS of \eqref{eq:CovariantCSS}, where $g=1$ corresponds to
the self-dual case). The local well-posedness of \eqref{eq:CovariantCSS}
in the critical space $L^{2}$ (under whatever gauges) remains open.
Under the Coulomb gauge, the local well-posedness in $H^{1}$ is proved
by Lim \cite{Lim2018JDE}, after the preceding works \cite{BergeDeBouardSaut1995Nonlinearity,Huh2013Abstr.Appl.Anal}.
If one changes to the heat gauge, Liu--Smith--Tataru \cite{LiuSmithTataru2014IMRN}
were able to lower the regularity assumption and prove the small data
local well-posedness in $H^{0+}$.

There are also works in global-in-time behaviors. Bergé--de Bouard--Saut
\cite{BergeDeBouardSaut1995Nonlinearity} applied Glassey's convexity
argument \cite{Glassey1977JMP} to the virial identities and obtained
a sufficient condition for finite time blow-up. These authors also
performed in \cite{BergeDeBouardSaut1995PRL} a formal computation
of the log-log blow-up for negative energy solutions. We remark that
these two works only apply to the negative energy solutions, which
are available only when $g>1$. On the other hand, Oh--Pusateri \cite{OhPusateri2015}
showed the global existence and scattering for small data in weighted
Sobolev spaces.

We come back to the equivariant self-dual case \eqref{eq:CSS}. Having
local theory at hands, it is natural to study global-in-time behavior
of large solutions. One of the central objects in the global-in-time
analysis is the \emph{ground state}, a standing wave solution with
minimal charge, in our context. It is believed to be the smallest
(in some sense) example exhibiting nonlinear behavior of the equation.
In case of \eqref{eq:CSS}, the static solution $Q$ plays a role
of the ground state. Liu--Smith \cite{LiuSmith2016} proved the following
\emph{threshold theorem}: any solution having charge less than that
of $Q$ is global and scatters. We also remark that the result of
\cite{LiuSmith2016} applies to the non-self-dual case. In particular,
the global well-posedness and scattering for any solutions is proved
in the defocusing regime (where $E[u]\geq0$ and $E[u]=0$ if and
only if $u=0$) on which there are no ground states.

The next question is the dynamics at and above the threshold. At the
threshold charge $M[u]=M[Q]$, there are two fundamental examples
of nonscattering solutions. The first one is the static solution $Q$
and the other one is the explicit pseudoconformal blow-up solution
obtained by applying the pseudoconformal transform to $Q$ \cite{JackiwPi1990PRD,Huh2009Nonlinearity}:
\begin{equation}
S(t,r)\coloneqq\frac{1}{|t|}Q\Big(\frac{r}{|t|}\Big)e^{-i\frac{r^{2}}{4|t|}},\qquad t<0.\label{eq:Def-S(t,r)}
\end{equation}

Recently, the authors \cite{KimKwon2019arXiv} studied pseudoconformal
blow-up solutions with prescribed asymptotic profiles. Here, by a
\emph{pseudoconformal blow-up solution}, we mean a solution $u$ to
\eqref{eq:CSS} which decomposes as 
\[
u(t,r)\approx S(t,r)+z(t,r),
\]
where $z(t,r)$ is regular at and near the blow-up time. Let $m\geq1$.
Fix a small smooth $m$-equivariant function $z^{\ast}(r)$, called
an \emph{asymptotic profile}, satisfying an additional degeneracy
condition $\partial_{r}^{(m)}z^{\ast}(0)=0$. Then, there exists a
pseudoconformal blow-up solution $u$ to \eqref{eq:CSS} on $(-\infty,0)$
such that $z(0,r)=z^{\ast}(r)$. Moreover, they exhibited an instability
mechanism, the \emph{rotational instability}, (see Section \ref{sec:ModifiedProfile}
for more details) of these pseudoconformal blow-up solutions. This
result is a backward construction of the blow-up solutions. In the
(NLS) context
\begin{equation}
i\partial_{t}\psi+\Delta\psi+|\psi|^{2}\psi=0,\qquad\psi:I\times\R^{2}\to\C,\tag{NLS}\label{eq:NLS}
\end{equation}
they are referred to as Bourgain--Wang solutions \cite{BourgainWang1997}.
The instability of Bourgain--Wang solutions in \eqref{eq:NLS} is
settled by Merle--Rapha\"el--Szeftel \cite{MerleRaphaelSzeftel2013AJM}.

\subsection{\label{subsec:Main-results}Main results}

In this paper, we study the forward construction and codimension one
property of the pseudoconformal blow-up solutions. Compared to the
backward construction of blow-up solutions with prescribed asymptotic
profiles, here we investigate a quantitative description of the dynamics
starting near $Q$ and aim to characterize the initial data set yielding
the pseudoconformal blow-up.

The first part of our main results is the construction of an initial
data set (under a codimension one condition) resulting in pseudoconformal
blow-up. The second part is to establish the codimension one property
and Lipschitz regularity of the initial data set (the blow-up manifold).

To state our result, we need to introduce some technicalities. To
describe the sense of the codimension one manifold, we introduce the
relevant data sets and their coordinates.

We first describe the $H_{m}^{3}$-initial data set $\mathcal{O}_{init}$
and the coordinates $(\lambda_{0},\gamma_{0},b_{0},\eta_{0},\epsilon_{0})$.
We fix a codimension four linear subspace $\mathcal{Z}^{\perp}$ of
$H_{m}^{3}$ for $\epsilon_{0}$. Codimension four conditions are
given by four orthogonality conditions \eqref{eq:ortho-cond}. Here,
$H_{m}^{k}$ denotes the usual Sobolev space $H^{k}(\R^{2})$ restricted
on $m$-equivariant functions; see Section \ref{subsec:Adapted-function-spaces}
for details. For some $b^{\ast}>0$, we define 
\[
\tilde{\mathcal{U}}_{init}\coloneqq\{(\lambda_{0},\gamma_{0},b_{0},\epsilon_{0})\in\R_{+}\times\R/2\pi\Z\times\R\times\mathcal{Z}^{\perp}:b_{0}\in(0,b^{\ast}),\ \|\epsilon_{0}\|_{H_{m}^{3}}<b_{0}^{3}\}.
\]
Next we introduce one more coordinate, $\eta_{0}$, to describe the
full open set of initial data. For some large universal constant $K>1$,
we define the set of coordinates 
\[
\mathcal{U}_{init}\coloneqq\{(\lambda_{0},\gamma_{0},b_{0},\eta_{0},\epsilon_{0}):(\lambda_{0},\gamma_{0},b_{0},\epsilon_{0})\in\tilde{\mathcal{U}}_{init},\ \eta_{0}\in(-\tfrac{K}{2}b_{0}^{3/2},\tfrac{K}{2}b_{0}^{3/2})\}.
\]
The initial data set $\mathcal{O}_{init}$ is defined by the set of
images: 
\begin{equation}
\mathcal{O}_{init}\coloneqq\{\frac{e^{i\gamma_{0}}}{\lambda_{0}}[P(\cdot;b_{0},\eta_{0})+\epsilon_{0}]\Big(\frac{r}{\lambda_{0}}\Big):(\lambda_{0},\gamma_{0},b_{0},\eta_{0},\epsilon_{0})\in\mathcal{U}_{init}\}\subseteq H_{m}^{3},\label{eq:Def-O-init}
\end{equation}
where $P(y;b_{0},\eta_{0})$ is the modified profile defined in \eqref{eq:def-mod-profile}.
Note that $\mathcal{O}_{init}$ is invariant under $L^{2}$-scalings
and phase rotations. It will be shown (see Lemma \ref{lem:decomp})
that if $b^{\ast}$ is sufficiently small, then the set $\mathcal{O}_{init}$
is open in the $H_{m}^{3}$-topology and one can view $(\lambda_{0},\gamma_{0},b_{0},\eta_{0},\epsilon_{0})$
as coordinates on $\mathcal{O}_{init}$. Note that the static solution
$Q$ does not belong to $\mathcal{O}_{init}$ (due to $Q=P(\cdot;0,0)$),
but it lies in the boundary of $\mathcal{O}_{init}$.

Starting from some initial data $u_{0}$ in $\mathcal{O}_{init}$,
we will construct (in a codimension one sense) a so called $H_{m}^{3}$-\emph{trapped
solution} $u$, which has the decomposition on its maximal forward
lifespan 
\begin{equation}
u(t,r)=\frac{e^{i\gamma(t)}}{\lambda(t)}[P(\cdot;b(t),\eta(t))+\epsilon(t,\cdot)]\Big(\frac{r}{\lambda(t)}\Big)\label{eq:decomposition-intro}
\end{equation}
and satisfies $\epsilon\in\mathcal{Z}^{\perp}$, 
\begin{equation}
b\in(0,b^{\ast}),\ \eta<Kb^{\frac{3}{2}},\ \|\epsilon\|_{L^{2}}<K(b^{\ast})^{\frac{1}{4}},\ \|\epsilon_{1}\|_{L^{2}}<Kb,\ \|\epsilon_{3}\|_{L^{2}}<Kb^{\frac{5}{2}}.\label{eq:def-H3-trapped}
\end{equation}
Here, $\epsilon_{1}=L_{Q}\epsilon$ and $\epsilon_{3}=A_{Q}^{\ast}A_{Q}L_{Q}\epsilon$
are the adapted derivatives of $\epsilon$ defined in Section \ref{subsec:Conjugation-identities}.
It will be shown that $H_{m}^{3}$-trapped solutions blow up in finite
time with the pseudoconformal blow-up rate.

We can now state the construction part of our main results.
\begin{thm}[Existence of blow-up solutions]
\label{thm:Existence}Let $m\geq1$. There exist constants $K>1$
and $b^{\ast}>0$ with the following properties.\footnote{In fact, we should add ``there exists $M>1$'' in the statement. In
the definition \eqref{eq:ortho-cond} of $\mathcal{Z}^{\perp}$, one
has to introduce the large parameter $M>1$.}
\begin{itemize}
\item (Openness and coordinates) The set $\mathcal{O}_{init}$ is open in
$H_{m}^{3}$ and $(\lambda_{0},\gamma_{0},b_{0},\eta_{0},\epsilon_{0})\in\tilde{\mathcal{U}}_{init}$
serves as coordinates for $\mathcal{O}_{init}$ in the sense of Lemma
\ref{lem:decomp}.
\item (Existence of trapped solutions in codimension one sense) Let $(\lambda_{0},\gamma_{0},b_{0},\epsilon_{0})\in\tilde{\mathcal{U}}_{init}$.
Then, there exists $\eta_{0}\in(-\frac{K}{2}b_{0}^{3/2},\frac{K}{2}b_{0}^{3/2})$
such that the solution $u(t,r)$ starting from the initial data 
\begin{equation}
u_{0}(r)=\frac{e^{i\gamma_{0}}}{\lambda_{0}}[P(\cdot;b_{0},\eta_{0})+\epsilon_{0}]\Big(\frac{r}{\lambda_{0}}\Big)\in\mathcal{O}_{init}\label{eq:Def-u_0}
\end{equation}
is a $H_{m}^{3}$-trapped solution.
\item (Pseudoconformal blow-up) The $H_{m}^{3}$-trapped solution $u$ satisfies:
\begin{enumerate}
\item (Finite-time blow-up) $u$ blows up in finite time $T=T(u_{0})\in(0,\infty)$.
\item (Pseudoconformal blow-up) There exist $\ell=\ell(u_{0})\in(0,\infty)$,
$\gamma^{\ast}=\gamma^{\ast}(u_{0})\in\R$, and $u^{\ast}\in L^{2}$
such that 
\[
u(t,r)-\frac{e^{i\gamma^{\ast}}}{\ell(T-t)}Q\Big(\frac{r}{\ell(T-t)}\Big)\to u^{\ast}\quad\text{in }L^{2}
\]
as $t\uparrow T$.
\item (Further regularity of the radiation)  $u^{\ast}$ has further regularity
\[
u^{\ast}\in H_{m}^{1}.
\]
\end{enumerate}
\end{itemize}
\end{thm}

We note that when $m\in\{1,2\}$, the profile $Q(y)e^{-ib\frac{y^{2}}{4}}$
for the explicit pseudoconformal blow-up solution $S(t,r)$ \emph{does
not} belong to $\mathcal{O}_{init}$, due to $Q(r)e^{-ib\frac{r^{2}}{4}}\notin H_{m}^{3}$.
Thus the trapped solutions constructed in Theorem \ref{thm:Existence}
are different from the blow-up solutions constructed in \cite{KimKwon2019arXiv}.

According to the instability result of \cite{KimKwon2019arXiv}, the
above trapped solutions are believed to be non-generic. In the next
theorem, we show that the set of initial data yielding trapped solutions
forms a codimension one manifold. Indeed, we obtain the \emph{uniqueness}
of $\eta_{0}$ to complement Theorem \ref{thm:Existence}. Moreover,
we establish the Lipschitz continuity of $\eta_{0}=\eta_{0}(b_{0},\epsilon_{0})$,
and hence the Lipschitz regularity of the blow-up manifold.

To achieve this, \emph{we further restrict the initial data set and
equivariance index: }
\[
u_{0}\in H_{m}^{5}\quad\text{and}\quad m\geq3.
\]
The restricted initial data set $\mathcal{O}_{init}^{5}$ is described
as follows. Let 
\begin{align*}
\tilde{\mathcal{U}}_{init}^{5} & \coloneqq\{(\lambda_{0},\gamma_{0},b_{0},\epsilon_{0})\in\R_{+}\times\R/2\pi\Z\times\R\times(\mathcal{Z}^{\perp}\cap H_{m}^{5}):b_{0}\in(0,b^{\ast}),\ \|\epsilon_{0}\|_{H_{m}^{5}}<b_{0}^{5}\},\\
\mathcal{U}_{init}^{5} & \coloneqq\{(\lambda_{0},\gamma_{0},b_{0},\eta_{0},\epsilon_{0}):(\lambda_{0},\gamma_{0},b_{0},\epsilon_{0})\in\tilde{\mathcal{U}}_{init}^{5},\ \eta_{0}\in(-\tfrac{K}{2}b_{0}^{3/2},\tfrac{K}{2}b_{0}^{3/2})\}.
\end{align*}
Next, the initial data set $\mathcal{O}_{init}^{5}$ is given by the
set of images: 
\begin{equation}
\mathcal{O}_{init}^{5}\coloneqq\{\frac{e^{i\gamma_{0}}}{\lambda_{0}}[P(\cdot;b_{0},\eta_{0})+\epsilon_{0}]\Big(\frac{r}{\lambda_{0}}\Big):(\lambda_{0},\gamma_{0},b_{0},\eta_{0},\epsilon_{0})\in\mathcal{U}_{init}^{5}\}\subseteq H_{m}^{5}.\label{eq:Def-O-init-5}
\end{equation}
Starting from some initial data $u_{0}$ in $\mathcal{O}_{init}^{5}$,
we construct $H_{m}^{5}$\emph{-trapped solutions} $u$, which enjoys
further smallness in the $\dot{H}_{m}^{5}$-norm: $u$ has the decomposition
\eqref{eq:decomposition-intro} on its maximal forward lifespan and
satisfies $\epsilon\in\mathcal{Z}^{\perp}\cap H_{m}^{5}$, 
\begin{equation}
b\in(0,b^{\ast}),\ \eta<Kb^{\frac{3}{2}},\ \|\epsilon\|_{L^{2}}<K(b^{\ast})^{\frac{1}{4}},\ \|\epsilon_{1}\|_{L^{2}}<Kb,\ \|\epsilon_{3}\|_{L^{2}}<Kb^{3},\ \|\epsilon_{5}\|_{L^{2}}<Kb^{\frac{9}{2}}.\label{eq:def-H5-trapped}
\end{equation}
Here, $\epsilon_{5}=A_{Q}^{\ast}A_{Q}A_{Q}^{\ast}A_{Q}L_{Q}\epsilon$
is an adapted derivative of $\epsilon$ defined in Section \ref{subsec:Conjugation-identities}.

We mean by a manifold, a subset of a Banach space, which can be locally
represented by a graph on a linear subspace. The regularity of the
manifold is described by the regularity of the graphs. Here we use
the following definition of a locally Lipschitz codimension one manifold
of a Banach space.
\begin{defn}[Locally Lipschitz codimension one manifold]
\label{def:localLipschitz}Let $\mathcal{M}$ be a subset of a Banach
space $X$. We say that $\mathcal{M}$ is a \emph{locally Lipschitz
codimension one manifold} of $X$ if it satisfies the following properties:
let $p\in\mathcal{M}$ be arbitrary.
\begin{itemize}
\item There exist codimension one closed subspace $X_{s}$ and one-dimensional
subspace $X_{u}$ of $X$ such that $X=X_{s}\oplus X_{u}$.
\item There exist open neighborhood $\mathcal{O}_{p,s}$ of $p$ in $p+X_{s}$,
open neighborhood $\mathcal{O}_{p}$ of $p$ in $X$, and a Lipschitz
map $f:\mathcal{O}_{p,s}\to X_{u}$ such that $[\mathrm{id}_{\mathcal{O}_{p,s}}\oplus f](\mathcal{O}_{p,s})=\mathcal{M}\cap\mathcal{O}_{p}$.
\end{itemize}
\end{defn}

We note that this property is invariant under $C^{1}$-diffeomorphisms.

There are several works in dispersive equations on establishing regularities
of the manifolds of objects exhibiting certain dynamics. To name a
few, Collot \cite{Collot2018MemAMS} constructed a Lipschitz manifold
of blow-up solutions to the energy-supercritical NLW. There is an
extensive literature on the study of similar manifolds, for example,
manifolds of global solutions that scatter to solitary waves, or blow-up
solutions. We refer to \cite{KriegerSchlag2006JAMS,Beceanu2012CPAM,BejenaruKriegerTataru2013AnalPDE,KriegerNakanishiSchlag2015MathAnn,DonningerKriegerSzeftelWong2016Duke,MartelMerleNakanishiRaphael2016CMP,BurzioKrieger2022MAMS}
and references therein. The most relevant one to this paper is the
work of Collot \cite{Collot2018MemAMS}.

Our next result is on the uniqueness of $\eta_{0}$ and the Lipschitz
regularity of the blow-up manifold.
\begin{thm}[Blow-up manifold when $m\geq3$]
\label{thm:BlowupManifold}Let $m\geq3$. There exist constants $K>0$
and $b^{\ast}>0$ with the following properties.
\begin{itemize}
\item (Existence of $H_{m}^{5}$-trapped solutions) The statements of Theorem
\ref{thm:Existence} hold when we replace $\tilde{\mathcal{U}}_{init},\mathcal{O}_{init},H_{m}^{3}$
by $\tilde{\mathcal{U}}_{init}^{5},\mathcal{O}_{init}^{5},H_{m}^{5}$.
Moreover, the radiation $u^{\ast}$ has further regularity $u^{\ast}\in H_{m}^{3}$.
\item (Uniqueness of $\eta_{0}$) Given $(\lambda_{0},\gamma_{0},b_{0},\epsilon_{0})\in\tilde{\mathcal{U}}_{init}^{5}$,
$\eta_{0}$ is unique in $(-\frac{K}{2}b_{0}^{3/2},\frac{K}{2}b_{0}^{3/2})$
such that the solution $u(t,r)$ starting from the initial data $u_{0}(r)$
in \eqref{eq:Def-u_0} is a $H_{m}^{5}$-trapped solution.
\item (Lipschitz blow-up manifold) Let $\mathcal{M}$ be the set of initial
data in $\mathcal{O}_{init}^{5}$ yielding $H_{m}^{5}$-trapped solutions.
Then, $\mathcal{M}$ is a locally Lipschitz codimension one manifold
in $H_{m}^{5}$.
\end{itemize}
\end{thm}

\emph{Comments on Theorems \ref{thm:Existence} and \ref{thm:BlowupManifold}.}

1. \emph{Codimension one condition}. One of the crucial observations
in authors' previous work \cite{KimKwon2019arXiv} is the finding
of \emph{unstable} modulation parameter $\eta$, which leads to the
rotational instability. This parameter $\eta$ is the source of codimension
one condition.

2. \emph{Comparison with the results of \cite{KimKwon2019arXiv}}.
In the authors' previous work \cite{KimKwon2019arXiv}, we constructed
and studied an instability mechanism of pseudoconformal blow-up solutions
with prescribed asymptotic profiles via the backward construction.
In particular, the authors exhibited an unstable direction (or, constructed
an one-parameter family of solutions) that prevents the pseudoconformal
blow-up. On the contrary, this work uses the forward construction.
Here we constructed a set of \emph{initial data} on that admits blow-up
solutions and showed that it is a locally Lipschitz codimension one
manifold. In view of \cite{KimKwon2019arXiv}, this work sheds light
on optimal conditional stability (codimension one) of pseudoconformal
blow-up solutions.

In \cite{KimKwon2019arXiv}, the authors imposed one extra degeneracy
condition on the asymptotic profiles. However, it is unclear how this
degeneracy condition on the asymptotic profiles is related to the
set of initial data.

3. \emph{The $m=0$ case}. Our proof breaks down in many places when
$m=0$. For example, $Q$ decays slower ($Q\sim\langle y\rangle^{-2}$).
Moreover, the Hardy controls of the adapted norms at $\dot{H}^{1}$
and $\dot{H}^{3}$ and repulsivity property of $A_{Q}A_{Q}^{\ast}$
are weakened near the spatial infinity. However, we believe that the
essential difficulty for the $m=0$ case is the construction of the
modified profile. In fact, the size of the radiation term (or, the
error from the modified profile) becomes critical and seems to require
further corrections on the modulation equations. Thus the current
profile seems not working. The $m=0$ case will be treated in a forthcoming
work. Experiences of energy-critical problems such as wave maps, Schrödinger
maps, or heat flows, suggest further refinement of the modified profiles
and logarithmic corrections to the modulation equations, exploiting
the resonance of the linearized operator and the tail computations
\cite{RaphaelRodnianski2012Publ.Math.,MerleRaphaelRodnianski2013InventMath,RaphaelSchweyer2013CPAM}.

4. \emph{Further regularity}. Although we stated $u^{\ast}\in H_{m}^{1}$
(or $u^{\ast}\in H_{m}^{3}$) in our main theorems, by inspecting
their proofs, for any $k\in\mathbb{N}$ one can prove $u^{\ast}\in H_{m}^{2k-1}$
if $m\geq2k-1$ and $\mathcal{O}_{init}$ is restricted to $H^{2k+1}$-functions.

5. \emph{Comparison with Schrödinger maps }(SM)\emph{ and nonlinear
Schrödinger equation \eqref{eq:NLS}}. \eqref{eq:CSS} shares many
features with energy-critical Schrödinger maps (SM) and \eqref{eq:NLS},
as they are critical Schrödinger-type equations. Earlier works on
those equations have become a nice guide for us to study \eqref{eq:CSS}.

\eqref{eq:CSS} and \eqref{eq:NLS} share many similarties such as
symmetries and conservation laws. Among others, the pseudoconformal
symmetry is the most crucial and motivates us to study pseudoconformal
blow-up solutions. However, they have essential differences between
the linearized operators, yielding completely different instability
mechanisms of pseudoconformal blow-up solutions as drawn in \cite{MerleRaphaelSzeftel2013AJM}
and \cite{KimKwon2019arXiv}. A construction of pseudoconformal blow-up
solutions for one-dimension quintic ($L^{2}$-critical) NLS in codimension
one sense is given in \cite{KriegerSchlag2009JEMS}.

One drastic difference between \eqref{eq:CSS} and \eqref{eq:NLS}
is that \eqref{eq:NLS} allows a \emph{stable} blow-up regime, namely,
the celebrated log-log blow-up studied by Merle and Raphaël \cite{MerleRaphael2005AnnMath,MerleRaphael2003GAFA,MerleRaphael2004InventMath,Raphael2005MathAnn,MerleRaphael2006JAMS,MerleRaphael2005CMP}.
As the formal analysis in \cite{BergeDeBouardSaut1995PRL} suggests,
it is believed that the focusing non self-dual CSS is similar to \eqref{eq:NLS}.

\eqref{eq:CSS} is also similar to the energy-critical Schrödinger
maps (SM). First of all, both of them are self-dual. The linearized
operators have self-dual factorizations and the method of supersymmetric
conjugates \cite{RodnianskiSterbenz2010Ann.Math.,RaphaelRodnianski2012Publ.Math.,MerleRaphaelRodnianski2013InventMath}
can be performed. A remarkable observation in this paper is that the
linearized operator of \eqref{eq:CSS} and that of (SM) are the same
after we go by one higher adapted derivative, due to the \emph{conjugation
identities} \eqref{eq:ConjugationIdentity}. See Remark \ref{rem:ConnectionSMWM}
for more details.

Furthermore, the modulation equations detected in the $1$-equivariant
(SM) \cite{MerleRaphaelRodnianski2013InventMath} (without log corrections)
are the same as those of \eqref{eq:CSS}. Therefore the codimension
one blow-up with rotational instability are expected in both cases.
As seen in \eqref{eq:Def-D+}, taking the first adapted derivative
shifts the equivariance index by $1$. Thus the $m=0$ case of \eqref{eq:CSS}
seems to share many features with the $1$-equivariant (SM) including
the logarithmic corrections to the modulation equations. In a forthcoming
work, the conjugation identities (Proposition \ref{prop:ConjugationIdentity})
will also be a key algebraic tool to attack the $m=0$ case.

However, (SM) with equivariance index $k\geq3$ has no blow-up near
the harmonic map, but the asymptotic stability is known \cite{GustafsonKangTsai2007CPAM,GustafsonKangTsai2008Duke,GustafsonNakanishiTsai2010CMP}.
This is in contrast to \eqref{eq:CSS}, where the pseudoconformal
blow-up occurs for all $m\geq1$.

6. \emph{Assumption $m\geq3$ for Theorem \ref{thm:BlowupManifold}}.
When establishing the Lipschitz regularity of the blow-up manifold,
we need to work with $H_{m}^{5}$-trapped solutions (more precisely,
the solutions with good $H_{m}^{5}$ a priori bounds). When $m\in\{1,2\}$,
current profiles and modulation parameters are insufficient to construct
$H_{m}^{5}$-trapped solutions. Due to slow decay of $Q$, there appear
exotic modes (of two real dimensions) from the coercivity relations.
We also note that when $m=2$, these exotic modes have worst decay
and there appear additional complications due to logarithmic weakening
of the Hardy controls near the spatial infinity. See Remark \ref{rem:Positivity-AAA-small-m}
for more details. Establishing the Lipschitz dependence of blow-up
solutions for $m\in\{1,2\}$ remains as an interesting open problem.

7. \emph{Rotational instability}. Our main theorems construct a codimension
one manifold of initial data yielding the pseudoconformal blow-up.
A natural question is to ask dynamics of solutions starting from initial
data that are close to, but do not belong to the constructed blow-up
manifold. In \cite{KimKwon2019arXiv}, the authors exhibited an instability
mechanism (the \emph{rotational instability}) for solutions that are
perturbed in a certain direction ($\rho$) from the pseudoconformal
blow-up solutions. In fact, this direction $\rho$ turns out to be
transversal to our blow-up manifold. We conjecture that the rotational
instability is universal for solutions near the blow-up manifold:
if the initial data $u_{0}$ lies in $\mathcal{O}_{init}\setminus\mathcal{M}$,
then they concentrate up to some small scale $\sim|\eta|>0$, then
stop concentrating but take an abrupt spatial rotation by the angle
$\mathrm{sgn}(\eta)(\frac{m+1}{m})\pi$ on the time scale $\sim|\eta|$,
and then expand out like $\overline{S(-t,r)}$ for $t\gtrsim|\eta|$
as observed in \cite{KimKwon2019arXiv}.

According to Theorem \ref{thm:BlowupManifold}, if one slightly perturbs
$\eta_{0}$ from $u_{0}\in\mathcal{M}$ with coordinates $(\lambda_{0},\gamma_{0},b_{0},\eta_{0},\epsilon_{0})$,
then the solution should escape the trapped regime by the uniqueness
of $\eta_{0}$. Our proof moreover shows that this solution exits
the trapped regime by $|\eta|\gtrsim o(b)$ at the exit time. The
nonlinear rotational instability requires the study of the forward-in-time
dynamics after the exit time.

It seems that the rotational instability is not particular for \eqref{eq:CSS}.
For the Landau--Lifschitz--Gilbert equation, which contains harmonic
heat flow and Schrödinger maps equations, authors in \cite{BergWilliams2013}
performed formal computations and provided numerical evidences for
a quick spatial rotation by the angle $\pi$, near the blow-up solution.

\subsection{\label{subsec:strategy}Strategy of the proof}

Our general scheme of the proof is the forward construction using
modulation analysis. As a main step of the proof, we perform a modified
energy method in higher derivatives. We found a \emph{conjugation
identity} (Proposition \ref{prop:ConjugationIdentity}) which enables
us to carry out the rest of the analysis. It also shows a deep connection
with the Schrödinger maps and wave maps.

To begin with, we decompose the solution into the blow-up profile
and the remainder. The former is a finite-dimensional object, whose
dynamics is described by a system of ODEs of modulation parameters.
The latter belongs to some finite codimension space described by orthogonality
conditions.

We first detect the evolution laws of the modulation parameters exhibiting
pseudoconformal blow-up and its rotational instability. Next, we control
the remainder part by the robust energy method combined with repulsivity
properties observed in higher Sobolev norms. This strategy was successfully
implemented in various contexts. To name a few relevant examples,
we refer to Rodnianski--Sterbenz \cite{RodnianskiSterbenz2010Ann.Math.}
and Raphaël--Rodnianski \cite{RaphaelRodnianski2012Publ.Math.} for
energy-critical wave maps, Merle--Raphaël--Rodnianski \cite{MerleRaphaelRodnianski2013InventMath}
for energy-critical Schrödinger maps. We also refer to \cite{HillairetRaphael2012AnalPDE,RaphaelSchweyer2013CPAM,RaphaelSchweyer2014AnalPDE}
for relevant works in the energy-critical equations. The strategy
also extends to energy-supercritical equations, for example in Merle--Raphaël--Rodnianski
\cite{MerleRaphaelRodnianski2015CambJMath} and Collot \cite{Collot2017AnalPDE,Collot2018MemAMS}.
This list is not exhaustive. The works \cite{MerleRaphaelRodnianski2013InventMath,MerleRaphaelRodnianski2015CambJMath,Collot2018MemAMS}
are the most relevant to our work.

\ 

1. \emph{Setup for the modulation analysis}. Let $(\lambda_{0},\gamma_{0},b_{0},\epsilon_{0})\in\tilde{\mathcal{U}}_{init}$
be given and let $\eta_{0}$ vary. Set our initial data 
\[
u_{0}(r)=\frac{e^{i\gamma_{0}}}{\lambda_{0}}[P(\cdot;b_{0},\eta_{0})+\epsilon_{0}]\Big(\frac{r}{\lambda_{0}}\Big),
\]
where $P(\cdot;b_{0},\eta_{0})$ is a modified profile to be used
in this paper, which deforms from $Q$. The construction of $P$ and
the roles of $b_{0}$ and $\eta_{0}$ will be explained soon. The
introduction of the modified profile $P$ and modulation parameters
$\lambda_{0},\gamma_{0},b_{0},\eta_{0}$ are motivated from the generalized
null space relations of the linearized operator.

Let $u$ be the forward-in-time evolution of $u_{0}$. Requiring certain
orthogonality conditions on $\epsilon$, we can decompose $u$ as
\[
u(t,r)=\frac{e^{i\gamma}}{\lambda}[P(\cdot;b,\eta)+\epsilon](t,\frac{r}{\lambda})
\]
as long as $u$ belongs to the \emph{trapped regime} \eqref{eq:def-H3-trapped}
or \eqref{eq:def-H5-trapped}, i.e. $\epsilon$ is kept small and
$|\eta|\ll b$. Here, $\lambda,\gamma,b,\eta$ are functions of $t$.
This decomposition will be clarified in Lemma \ref{lem:decomp}.

The proof of Theorem \ref{thm:Existence} essentially reduces to the
following assertions:
\begin{itemize}
\item (Main bootstrap) Smallness assumptions on $\epsilon$ is kept in the
trapped regime (Proposition \ref{prop:main-bootstrap}), and doing
so, the formal parameter equation \eqref{eq:FormalODE} approximately
holds;
\item There exists special $\eta_{0}$ such that $|\eta|\ll b$ holds for
the whole lifespan of $u$ (Proposition \ref{prop:Sets-I-pm}), so
$u$ is a trapped solution;
\item $u$ satisfies the statements of Theorem \ref{thm:Existence}.
\end{itemize}
The proof of Theorem \ref{thm:BlowupManifold} additionally requires
the difference estimates:
\begin{itemize}
\item The difference of the unstable parameter $\eta$ is controlled by
the difference of the stable parameters $b$ and $\epsilon$. (Proposition
\ref{prop:Lipschitz-estimate-modulo-scale-phase})
\end{itemize}
The heart of the proof is the main bootstrap part. Here we focus on
the main bootstrap argument.

\ 

2. \emph{Formal parameter ODEs and rotational instability}. Write
\eqref{eq:CSS} in the self-dual form \eqref{eq:SelfDualEquation}:
\[
\partial_{t}u+iL_{u}^{\ast}\D_{+}^{(u)}u=0.
\]
We renormalize $u$ by introducing the renormalized variables $(s,y)$
by 
\[
\frac{ds}{dt}=\frac{1}{\lambda^{2}},\quad y=\frac{r}{\lambda},\quad u^{\flat}(s,y)\coloneqq e^{-i\gamma}\lambda u(t,\lambda y).
\]
This yields 
\[
(\partial_{s}-\frac{\lambda_{s}}{\lambda}\Lambda+\gamma_{s}i)u^{\flat}+iL_{u^{\flat}}^{\ast}\D_{+}^{(u^{\flat})}u^{\flat}=0,
\]
where $\Lambda=r\partial_{r}+1$ is the $L^{2}$-scaling vector field.
From the identities (where we temporarily write $f_{b}(y)\coloneqq e^{-ib\frac{y^{2}}{4}}f(y)$)
\[
L_{f_{b}}^{\ast}\D_{+}^{(f_{b})}f_{b}=[L_{f}^{\ast}\D_{+}^{(f)}f]_{b}+ib\Lambda f_{b}-ib^{2}\partial_{b}(f_{b})\quad\text{and}\quad\D_{+}^{(Q)}Q=0,
\]
the pseudoconformal blow-up is encoded in the following ODE system
of modulation parameters 
\[
\frac{\lambda_{s}}{\lambda}+b=0,\qquad\gamma_{s}=0,\qquad b_{s}+b^{2}=0.
\]

If it were true that the decomposition of the form $\frac{e^{i\gamma}}{\lambda}[Q_{b}+\epsilon](\frac{r}{\lambda})$
is enough to prove that $\epsilon$ is kept small forward-in-time
and the above modulation equations are valid, then we would have a
\emph{stable blow-up}. However, this is not the case we observed in
\cite{KimKwon2019arXiv}, which asserts that pseudoconformal blow-up
solutions exhibit \emph{rotational instability}. This motivates us
to further introduce the fourth modulation parameter (denoted by $\eta$)
and modify our profile (denoted by $P(\cdot;b,\eta)$). The parameter
$\eta$ is chosen to generate the rotational instability. In \cite{KimKwon2019arXiv},
the authors introduced a fixed small parameter $\eta\geq0$ to generate
the rotational instability: 
\[
\frac{\lambda_{s}}{\lambda}+b=0,\quad\gamma_{s}=\eta\theta_{\eta},\quad b_{s}+b^{2}+\eta^{2}=0,
\]
where $\theta_{\eta}\approx m+1\neq0$.

In this work, we view $\eta$ as a dynamical parameter and write the
formal parameter ODE system as 
\begin{equation}
\frac{\lambda_{s}}{\lambda}+b=0,\quad\gamma_{s}=\eta\theta_{\eta},\quad b_{s}+b^{2}+\eta^{2}=0,\quad\eta_{s}=0.\label{eq:FormalODE}
\end{equation}
A typical example of solutions to this ODE system is 
\begin{gather*}
b(t)=|t|,\quad\lambda(t)=(t^{2}+\eta^{2})^{\frac{1}{2}},\quad\eta(t)=\eta_{0},\\
\gamma(t)=\begin{cases}
0 & \text{if }\eta_{0}=0,\\
\pm\mathrm{sgn}(\eta)(m+1)\tan^{-1}(\tfrac{t}{|\eta|}), & \text{if }\eta_{0}\neq0.
\end{cases}
\end{gather*}
The \emph{rotational instability} is exhibited when $\eta_{0}\neq0$,
where $\gamma(t)$ changes $(m+1)\pi$ on the time interval $|t|\lesssim|\eta|$.
When $\eta_{0}$ is small and nonnegative, an exact solution to \eqref{eq:CSS}
satisfying this law is constructed in \cite{KimKwon2019arXiv}. Such
solutions blow up only when $\eta_{0}=0$. This explains why we can
expect at best codimension one blow-up.

In our setting, $\eta$ is a dynamical parameter with $\eta_{s}\approx0$.
To ensure the pseudoconformal blow-up, we need $|\eta|\ll b$ for
trapped solutions. Due to $\eta_{s}\approx0$, the condition $|\eta|\ll b\to0$
does not propagate forward-in-time generically. In other words, $\eta$
is an \emph{unstable} parameter. The $\eta$-bound cannot be bootstrapped,
and we will show the existence of special $\eta_{0}$ by a soft connectivity
argument.

\ 

Our next goal is to construct the modified profile $P$ that admits
the formal ODEs \eqref{eq:FormalODE}.

3. \emph{Construction of the modified profile}. As many of the previous
works, one may try a Taylor expansion in small parameters $b$ and
$\eta$ to construct the modified profiles. However, this procedure
does not work very well, due to nonlocal nonlinearities; see Remarks
\ref{rem:TechnicalDifficulties} and \ref{rem:profile-nonlinearapproach}.
Instead, we use a nonlinear ansatz introduced in \cite{KimKwon2019arXiv}.
This is possible due to the explicit pseudoconformal symmetry and
self-duality. For the parameter $b$, we use the pseudoconformal phase
$e^{-ib\frac{y^{2}}{4}}$. For the parameter $\eta$, the authors
in \cite{KimKwon2019arXiv} found a remarkable nonlinear ansatz 
\begin{equation}
\D_{+}^{(Q^{(\eta)})}Q^{(\eta)}=-\eta\tfrac{y}{2}Q^{(\eta)}\quad\Rightarrow\quad L_{Q^{(\eta)}}^{\ast}\D_{+}^{(Q^{(\eta)})}Q^{(\eta)}+\eta\theta_{\eta}Q^{(\eta)}+\eta^{2}\tfrac{y^{2}}{4}Q^{(\eta)}=0\label{eq:ModifiedBogomol'nyi}
\end{equation}
motivated from the self-dual form \eqref{eq:SelfDualEquation} and
the Bogomol'nyi equation \eqref{eq:Bogomolnyi-Eq}. The profile $Q^{(\eta)}$
can be constructed (without divergence at infinity) when $\eta\geq0$.
Due to the factor $-\eta\tfrac{y}{2}Q^{(\eta)}$, it is easy to expect
that $Q^{(\eta)}$ has an exponential decay $e^{-\eta\frac{y^{2}}{4}}$.

In this work, as $\eta$ is a dynamical parameter, we need to consider
both positive and negative $\eta$. Thus we only use $Q^{(\eta)}(y)$
in the linearization regime $y\ll|\eta|^{-\frac{1}{2}}$, i.e. $e^{-\eta\frac{y^{2}}{4}}\approx1$.
As a consequence, we will use the profile 
\[
P(y;b,\eta)\coloneqq e^{-ib\frac{y^{2}}{4}}Q^{(\eta)}(y)\chi_{b^{-1/2}}(y),
\]
where $Q^{(\eta)}$ solves \eqref{eq:ModifiedBogomol'nyi} in the
region $y\ll|\eta|^{-\frac{1}{2}}$ and $\chi_{b^{-1/2}}$ is a smooth
cutoff to the region $y\lesssim b^{-\frac{1}{2}}$. As we will only
consider the regime $|\eta|\ll b$, this definition makes sense.

\ 

4. \emph{Propagation of smallness of $\epsilon$}. Having fixed the
profile $P$, we decompose 
\[
u(t,r)=\frac{e^{i\gamma}}{\lambda}[P(\cdot;b,\eta)+\epsilon](t,\frac{r}{\lambda})
\]
by the orthogonality conditions adapted to the generalized null space.
The equation of $\epsilon$ is given as 
\[
(\partial_{s}-\frac{\lambda_{s}}{\lambda}\Lambda+\gamma_{s}i)\epsilon+i\mathcal{L}_{Q}\epsilon=\mathbf{Mod}\cdot\mathbf{v}-iR_{\mathrm{L-L}}-iR_{\mathrm{NL}}-i\Psi,
\]
where $\mathcal{L}_{Q}=L_{Q}^{\ast}L_{Q}$ is the linearized operator
(see Section \ref{subsec:Linearization-of-CSS}), $\mathbf{Mod}$,
$\mathbf{v}$, $R_{\mathrm{L-L}}$, and $R_{\mathrm{NL}}$ are given
in \eqref{eq:prelim-eqn-e}, and $\Psi$ is given in \eqref{eq:decomp-P}.
Roughly speaking, $\mathbf{Mod}\cdot\mathbf{v}$ contains the modulation
equations, $R_{\mathrm{L-L}}$ is the difference of the linearized
operators at $P$ and $Q$, $R_{\mathrm{NL}}$ consists of quadratic
and higher order terms in $\epsilon$, and $\Psi$ is the error generated
from the profile $P$.

We will perform the energy method in the $H^{3}$-level of $\epsilon$.
To motivate this, the scaling argument says that we can expect $\|\epsilon\|_{\dot{H}^{k}}\lesssim\lambda^{k}$
at best. Since $\lambda\sim b$ in the pseudoconformal regime, this
reads $\|\epsilon\|_{\dot{H}^{k}}\lesssim b^{k}$. The standard modulation
estimate says $|b_{s}+b^{2}|\lesssim\|\epsilon\|_{\dot{H}^{k}}$,
so we need to work with $k\geq3$ not to disturb the law $b_{s}+b^{2}\approx0$.

When we go up to higher order derivatives of $\epsilon$, we will
take \emph{adaptive derivatives}. In view of the factorization $\mathcal{L}_{Q}=L_{Q}^{\ast}L_{Q}$,
we write the equation of $\epsilon_{1}=L_{Q}\epsilon$:
\[
\partial_{s}\epsilon_{1}+L_{Q}iL_{Q}^{\ast}\epsilon_{1}=\dots
\]
Here comes one of the novelties of this work. We observe a \emph{conjugation
identity} (see \eqref{eq:Def-AQ}) 
\[
L_{Q}iL_{Q}^{\ast}=iA_{Q}^{\ast}A_{Q}.
\]
Using this, we see that $\epsilon_{1}$ solves a Hamiltonian equation
\[
\partial_{s}\epsilon_{1}+iA_{Q}^{\ast}A_{Q}\epsilon_{1}=\dots
\]
with the associated energy $(\epsilon_{1},A_{Q}^{\ast}A_{Q}\epsilon_{1})_{r}=\|A_{Q}\epsilon_{1}\|_{L^{2}}^{2}$.
In contrast to $[L_{Q},i]\neq0$, we have $[A_{Q},i]=0$ so that we
can proceed further conjugation by $A_{Q}$ and $A_{Q}^{\ast}$ alternately.
In particular, $\epsilon_{2}=A_{Q}\epsilon_{1}$, $\epsilon_{3}=A_{Q}^{\ast}\epsilon_{2}$,
and so on. This also suggests how to define adapted function spaces
for higher derivatives of $\epsilon$, $\dot{\mathcal{H}}_{m}^{3}$
or $\dot{\mathcal{H}}_{m}^{5}$. More remarkably, this $A_{Q}$ is
equal to the linearized Bogomol'nyi operator in the wave maps and
Schrödinger maps; see Remark \ref{rem:ConnectionSMWM}.

With the above conjugation, the equation for $\epsilon_{2}=A_{Q}L_{Q}\epsilon$
enjoys a \emph{repulsive} dynamics. In the linearized dynamics $(\partial_{s}+i\mathcal{L}_{Q})\epsilon=0$,
there are four enemies preventing the repulsivity: the generalized
kernel elements $\{\Lambda Q,iQ,iy^{2}Q,\rho\}$ (See \eqref{eq:GenNullSpaceRelation}).
Interestingly enough, $A_{Q}L_{Q}$ kills all these elements and there
are no nontrivial static solutions to $(\partial_{s}+iA_{Q}A_{Q}^{\ast})\epsilon_{2}=0$.
Moreover, we have a formal monotonicity formula from a virial type
computation (see \eqref{eq:FormalMonotonicity}) 
\[
\tfrac{1}{2}\partial_{s}(\epsilon_{2},-i\Lambda\epsilon_{2})_{r}\geq(\epsilon_{2},A_{Q}A_{Q}^{\ast}\epsilon_{2})_{r}=\|\epsilon_{3}\|_{L^{2}}^{2}.
\]
We note that this type of repulsivity is already observed in \cite{RodnianskiSterbenz2010Ann.Math.}.
We will use a truncated version of this monotonicity to define the
corrective terms of our modified energy in a similar spirit of \cite{MerleRaphaelRodnianski2015CambJMath}.
See Sections \ref{subsec:Conjugation-identities} and \ref{subsec:energy-identity}
for more detailed exposition.

On the other hand, it is still required to extract nontrivial contributions
from $R_{\mathrm{L-L}}$ and $R_{\mathrm{NL}}$: 
\[
R_{\mathrm{L-L}}=-\theta_{\mathrm{L-L}}P+\tilde R_{\mathrm{L-L}}\quad\text{and}\quad R_{\mathrm{NL}}=-\theta_{\mathrm{NL}}P+\tilde R_{\mathrm{NL}}.
\]
The contributions $\theta_{\mathrm{L-L}}P$ and $\theta_{\mathrm{NL}}P$
are absorbed into the phase corrections. Such corrections arise from
the nonlocality of the gauge potential $A_{0}$. See Section \ref{subsec:Estimates-of-remainders-H3}
for more details. A similar idea was presented in \cite{KimKwon2019arXiv}.

\ 

Wrapping up the above strategies, we can close the bootstrap procedure
on $\epsilon$ by the modified energy method.

5. \emph{Existence of trapped solutions}. After closing the bootstrap
for $\epsilon$, we show the existence of special $\eta_{0}$ such
that $|\eta|\ll b$ holds on the whole forward-in-time lifespan of
$u$. Recall that $\eta$ is an unstable parameter, so the $\eta$-bound
cannot be bootstrapped. We show the existence of $\eta_{0}$ by a
standard connectivity argument. This shows the existence of trapped
solutions in a codimension one sense. Finally, the sharp asymptotics
of the blow-up rate $\lambda(t)$ and convergence of the phase parameter
$\gamma(t)$ for trapped solutions easily follow by integrating \eqref{eq:FormalODE}.
This concludes Theorem \ref{thm:Existence}.

\ 

6. \emph{Lipschitz regularity of the blow-up manifold} $\mathcal{M}$.
The strategy is quite similar to that of the previous bootstrap argument.
This time, we take two trapped solutions (say $u$ and $u'$) from
$\mathcal{M}$, and estimate the differences of their modulation parameters
and $\epsilon$'s. An important point is how we measure the differences
because the blow-up times of $u$ and $u'$ are in general different.
For this, we introduce the adapted time (as in Collot \cite{Collot2018MemAMS})
to compare these differences of modulation parameters and $\epsilon$'s
(say $\delta\epsilon(s)=\epsilon(s)-\epsilon'(s'(s))$ and similarly
for $\delta b,\delta\eta,\dots$). A trade-off to introducing this
adapted time is that we need to work with trapped solutions with two
higher derivatives (due to the linear term $(\frac{ds'}{ds}-1)\mathcal{L}_{Q}\epsilon'$).

For the stable parameters $b$ and $\epsilon$, we write the equation
of $\delta\epsilon$, prove modulation estimates for the difference
of modulation parameters, and perform energy estimates for adaptive
derivatives of $\delta\epsilon$. This yields forward-in-time controls
of $\delta b$ and $\delta\epsilon$. We will control the unstable
parameter difference $\delta\eta$ by integrating the modulation estimates
backwards in time. As a consequence, we will see that $\delta\eta_{0}$
is controlled by $\delta b_{0}$ and $\delta\epsilon_{0}$ in the
\emph{Lipschitz} sense. This essentially implies Theorem \ref{thm:BlowupManifold}.
\begin{rem}[Parameter dependence]
\label{rem:ParameterDependence}To realize the above strategy, we
will use a sequence of large (and small) parameters. We will presume
the order of choosing large parameters as follows: 
\[
1\ll K\ll M\ll M_{1}\ll M_{2}\ll(b^{\ast})^{-1}.
\]
The statements of our lemmas and propositions will omit the dependence
of these parameters. More precisely, our statements hold after adding
the following: ``for sufficiently large $K$, for sufficiently large
$M$ (depending on $K$), for sufficiently large $M_{1}$ (depending
on $M$), for sufficiently large $M_{2}$ (depending on $M_{1}$),
for sufficiently small $b^{\ast}>0$ (depending on $M_{2}$), the
following hold:''
\end{rem}

\subsection*{Notation}

For nonnegative quantities $A$ and $B$, we write $A\lesssim B$
if $A\leq CB$ holds for some implicit constant $C$. We write $A\gtrsim B$
if $B\lesssim A$; we write $A\sim B$ if $A\lesssim B$ and $B\lesssim A$.
If $C$ is allowed to depend on some parameters, then we write them
as subscripts of $\lesssim,\sim,\gtrsim$ to indicate the dependence.
We write $\langle x\rangle\coloneqq(1+|x|^{2})^{\frac{1}{2}}$. Classical
$L^{p}$ spaces and Sobolev spaces $H^{s}$ are also used on the domain
$\R^{2}$.

We use a smooth spherically symmetric cutoff function $\chi_{R}$,
such that $\chi_{R}(y)=\chi(R^{-1}y)$, $\chi(x)=1$ for $|x|\leq1$,
and $\chi(x)=0$ for $|x|\geq2$. We also denote the sharp cutoff
function on a set $A$ by $\mathbf{1}_{A}$.

For $x$ in a metric space $X$ and $\delta>0$, we use the notation
$B_{\delta}(x)$ to denote the ball of radius $\delta$ centered at
$x$.

Recall that a function $f:\R^{2}\to\C$ is \emph{$m$-equivariant}
if $f(x)=g(r)e^{im\theta}$ on the polar coordinates $(r,\theta)$
for some $g:\R_{+}\to\C$. We call $g$ as the \emph{radial part}
of $f$. We use an abuse of notation that $g$ is often considered
as an $m$-equivariant function. For example, we say that $g$ belongs
to some $m$-equivariant function space if its $m$-equivariant extension
belongs to that. We will later introduce what function spaces we use
in this paper.

We will use a shorthand for the integration of radial functions $f:(0,\infty)\to\C$
as 
\[
\int f\coloneqq\int_{\R^{2}}f(|x|)dx=2\pi\int_{0}^{\infty}f(r)rdr.
\]
We also use the \emph{real} $L^{2}$-inner product 
\[
(f,g)_{r}\coloneqq\Re\int\overline{f}g.
\]
For $s\in\R$, we need the\emph{ $\dot{H}^{s}$-scaling generator}
\begin{align*}
\Lambda_{s}f & \coloneqq\frac{d}{d\lambda}\Big|_{\lambda=1}\lambda^{1-s}f(\lambda\cdot)=[1-s+r\partial_{r}]f,\\
\Lambda & \coloneqq\Lambda_{0}.
\end{align*}

For $k\in\Z_{\geq0}$ and functions $f:(0,\infty)\to\C$, we use the
notation 
\begin{align*}
|f|_{k}(y) & \coloneqq\sup_{0\leq\ell\leq k}|y^{\ell}\partial_{y}^{\ell}f|,\\
|f|_{-k}(y) & \coloneqq\sup_{0\leq\ell\leq k}|y^{-\ell}\partial_{y}^{k-\ell}f|=y^{-k}|f|_{k}.
\end{align*}
The following  Leibniz rules hold: 
\[
|fg|_{k}\lesssim_{k}|f|_{k}|g|_{k}\quad\text{and}\quad|fg|_{-k}\lesssim|f|_{-k}|g|_{k}.
\]

\subsection*{Organization of the paper}

In Section \ref{sec:ConjugationIdentities}, we review the linearization
of \eqref{eq:CSS}, derive the conjugation identities, and develop
the adapted function spaces. In Section \ref{sec:ModifiedProfile},
we introduce the modified profile. In Section \ref{sec:TrappedSolutions},
we specify the decomposition of the solutions and reduce the proof
of the existence of trapped solutions to the main bootstrap proposition.
In Section \ref{sec:Main-bootstrap}, we close this boostrap procedure,
and finish the proof of Theorem \ref{thm:Existence} and the first
part of Theorem \ref{thm:BlowupManifold}. Finally in Section \ref{sec:Lipschitz-blow-up-manifold},
we establish the Lipschitz regularity of the blow-up manifold, thus
finishing the proof of Theorem \ref{thm:BlowupManifold}.

There is one appendix. In Appendix \ref{sec:HardyInequalities}, we
prove weighted Hardy's inequality, properties of the adapted function
spaces, and (sub-)coercivity estimates of the adaptive derivatives.

\subsection*{Acknowledgements}

The authors appreciate Sung-Jin Oh for helpful discussions and encouragement
to this work. The authors are partially supported by Samsung Science
\& Technology Foundation BA1701-01 and NRF-2019R1A5A1028324. Part
of this work was done while the first author was visiting Bielefeld
University through IRTG 2235. He would like to appreciate its kind
hospitality. The authors are grateful to anonymous referees for their
careful reading of this manuscript.

\section{\label{sec:ConjugationIdentities}Conjugation identities}

In this section, we discuss the linearized operators and their adapted
function spaces. In Section \ref{subsec:Linearization-of-CSS}, we
recall some facts on the linearization of \eqref{eq:CSS} at $Q$.
For instance, the self-dual factorization $i\mathcal{L}_{Q}=iL_{Q}^{\ast}L_{Q}$
and its generalized null space $\{\Lambda Q,iQ,ir^{2}Q,\rho\}$. In
Section \ref{subsec:Conjugation-identities}, we introduce \emph{conjugation
identities} and adapted derivatives. We observe a new factorization
$L_{Q}iL_{Q}^{\ast}=iA_{Q}^{\ast}A_{Q}$, which plays a key role in
our analysis. This suggests us how we proceed to higher adapted derivatives
and enables us to obtain the repulsivity. In Section \ref{subsec:Adapted-function-spaces},
we construct adapted function spaces and prove (sub-)coercivity estimates
for adapted derivatives.

\subsection{\label{subsec:Linearization-of-CSS}Linearization of \eqref{eq:CSS}}

In this subsection, we briefly recall some facts on linearization
made in \cite[Section 3]{KimKwon2019arXiv}. Motivated from the self-duality
\eqref{eq:SelfDualEnergy}, we first linearize the Bogomol'nyi operator:
\begin{align*}
\D_{+}^{(w+\epsilon)}(w+\epsilon) & =\D_{+}^{(w)}w+L_{w}\epsilon+N_{w}(\epsilon),\\
L_{w}\epsilon & \coloneqq\D_{+}^{(w)}+wB_{w},\\
N_{w}(\epsilon) & \coloneqq\epsilon B_{w}\epsilon+\tfrac{1}{2}wB_{\epsilon}\epsilon+\tfrac{1}{2}\epsilon B_{\epsilon}\epsilon,
\end{align*}
where 
\[
B_{f}g\coloneqq\frac{1}{r}\int_{0}^{r}\Re(\overline{f}g)r'dr'.
\]
One can rewrite \eqref{eq:CSS} as 
\[
i\partial_{t}u=\frac{\delta E}{\delta u}=L_{u}^{\ast}\D_{+}^{(u)}u,
\]
where $L_{u}^{\ast}$ denotes the adjoint of $L_{u}$. Now one can
linearize $L_{u}^{\ast}\D_{+}^{(u)}u$, as 
\begin{align*}
L_{w+\epsilon}^{\ast}\D_{+}^{(w+\epsilon)}(w+\epsilon) & =L_{w}^{\ast}\D_{+}^{(w)}w+\mathcal{L}_{w}\epsilon+(\text{h.o.t}),\\
\mathcal{L}_{w}\epsilon & =L_{w}^{\ast}L_{w}\epsilon+[(B_{w}\epsilon)+B_{w}^{\ast}[\overline{\epsilon}\cdot]+B_{\epsilon}^{\ast}[\overline{w}\cdot]](\D_{+}^{(w)}w),
\end{align*}
where $(\text{h.o.t})$ denotes quadratic and higher order terms in
$\epsilon$. In particular, we observe the self-dual factorization
of $\mathcal{L}_{Q}$ from $\D_{+}^{(Q)}Q=0$:
\[
\mathcal{L}_{Q}=L_{Q}^{\ast}L_{Q}.
\]
This identity was first observed by Lawrie, Oh, and Shahshahani in
their unpublished note and its derivation can be found in \cite{KimKwon2019arXiv}.
Then the linearized equation at $Q$ is given as 
\[
\partial_{t}\epsilon+i\mathcal{L}_{Q}\epsilon=0.
\]

Next, we recall the generalized null space relations of $i\mathcal{L}_{Q}$:
\cite[Proposition 3.4]{KimKwon2019arXiv}
\begin{equation}
\begin{aligned}i\mathcal{L}_{Q}\rho & =iQ; & i\mathcal{L}_{Q}ir^{2}Q & =4\Lambda Q;\\
i\mathcal{L}_{Q}iQ & =0; & i\mathcal{L}_{Q}\Lambda Q & =0;
\end{aligned}
\label{eq:GenNullSpaceRelation}
\end{equation}
where $\rho$ is given in Lemma \ref{lem:Def-rho} below. We remark
that $\mathcal{L}_{Q}$ is only $\R$-linear, \emph{not} $\C$-linear.
Note that except the relation $\mathcal{L}_{Q}\rho=Q$, other relations
can be derived either by direct computations or differentiating the
phase/scaling/pseudoconformal symmetries applied to the static solution
$Q$.
\begin{lem}[The generalized eigenmode $\rho$]
\label{lem:Def-rho}There exists a unique smooth function $\rho:(0,\infty)\to\R$
satisfying the following properties:
\begin{enumerate}
\item (Smoothness on the ambient space) The $m$-equivariant extension $\rho(x)\coloneqq\rho(r)e^{im\theta}$,
$x=re^{i\theta}$, is smooth on $\R^{2}$.
\item (Equation) $\rho(r)$ satisfies 
\[
L_{Q}\rho=\frac{1}{2(m+1)}rQ\quad\text{and}\quad\mathcal{L}_{Q}\rho=Q.
\]
\item (Pointwise bounds) We have 
\[
|\rho|_{k}\lesssim_{k}r^{2}Q,\qquad\forall k\in\N.
\]
\item (Nondegeneracy) We have 
\[
(\rho,Q)_{r}=\|L_{Q}\rho\|_{L^{2}}^{2}\neq0.
\]
\end{enumerate}
\end{lem}

\begin{proof}
We first recall the results proved in \cite[Section 3]{KimKwon2019arXiv}.
As $L_{Q}^{\ast}rQ=2(m+1)Q$ by an explicit computation, it suffices
to construct $\rho:(0,\infty)\to\R$ such that $L_{Q}\rho=\frac{1}{2(m+1)}rQ$.
In the proof of \cite[Lemma 3.6]{KimKwon2019arXiv}, such $\rho$
is contructed by solving the following integral equation for $\tilde{\rho}\coloneqq Q^{-1}\rho$
\[
\partial_{r}\tilde{\rho}+\frac{1}{r}\int_{0}^{r}Q^{2}\tilde{\rho}r'dr'=\frac{r}{2(m+1)},\qquad\forall r\in(0,\infty),
\]
in the class where $|\tilde{\rho}(r)|\lesssim r^{2}$. As the proof
relies on the contraction principle, uniqueness is shown in the same
class. Applying the above equation and coming back to $\rho$, we
have $\rho\in\dot{H}_{m}^{1}$, $L_{Q}\rho=\frac{1}{2(m+1)}rQ$, and
$|\rho|\lesssim r^{2}Q$. The nondegeneracy follows from 
\[
(\rho,Q)_{r}=(\rho,L_{Q}^{\ast}L_{Q}\rho)_{r}=\|L_{Q}\rho\|_{L^{2}}^{2}=\tfrac{1}{4(m+1)^{2}}\|rQ\|_{L^{2}}^{2}\neq0.
\]

To complete the proof, we first deal with smoothness. Writing the
equation $L_{Q}\rho=rQ$ in the ambient space $\R^{2}$ (see the proof
of Lemma \ref{lem:mapping-L-Appendix}, for instance), utilizing the
ellipticity of $\partial_{1}+i\partial_{2}$, and starting from $\rho\in\dot{H}_{m}^{1}$,
the $m$-equivariant extension $\rho(x)=\rho(r)e^{im\theta}$ is smooth.
To show the pointwise estimates, we come back to the radial part and
view $L_{Q}\rho=\frac{1}{2(m+1)}rQ$ of the form $r\partial_{r}\rho=-Q\int_{0}^{r}Q\rho r'dr'+\frac{1}{2(m+1)}r^{2}Q$.
This shows for any $k\geq0$ the recursive estimate 
\[
|\rho|_{k+1}\lesssim|Q|_{k+1}{\textstyle \int_{0}^{r}}Q\rho r'dr'+|r^{2}Q^{2}\rho|_{k}+r^{2}Q\lesssim r^{2}Q+|\rho|_{k}.
\]
Starting from the initial bound $|\rho|\lesssim r^{2}Q$, we get $|\rho|_{k}\lesssim_{k}r^{2}Q$
as desired.
\end{proof}
In the linearized equation 
\[
\partial_{t}\epsilon+i\mathcal{L}_{Q}\epsilon=0,
\]
there are two invariant subspaces: the generalized null space\footnote{When $m=1$, $ir^{2}Q$ and $\rho$ do not belong to $L^{2}$. We
mean by $ir^{2}Q,\rho\in N_{g}(i\mathcal{L}_{Q})$ the algebraic relations
\eqref{eq:GenNullSpaceRelation}.} of $i\mathcal{L}_{Q}$ 
\[
N_{g}(i\mathcal{L}_{Q})=\mathrm{span}_{\R}\{\Lambda Q,iQ,ir^{2}Q,\rho\}
\]
and the orthogonal complement of the generalized null space of the
adjoint $\mathcal{L}_{Q}i$
\[
N_{g}(\mathcal{L}_{Q}i)^{\perp}\coloneqq\{i\rho,r^{2}Q,Q,i\Lambda Q\}^{\perp}.
\]
Since the $4\times4$ matrix formed by taking the inner products\footnote{The inner products may not well-defined when $m=1$ due to the slow
decay of $r^{2}Q$ and $\rho$. We will resolve this technical issue
with introducing cutoffs. See Section \ref{subsec:decomposition}.} of $\Lambda Q,iQ,ir^{2}Q,\rho$ and $i\rho,r^{2}Q,Q,i\Lambda Q$
has nonzero determinant (c.f. \eqref{eq:Transversality}), $N_{g}(i\mathcal{L}_{Q})$
and $N_{g}(\mathcal{L}_{Q}i)^{\perp}$ are transversal. In other words,
$N_{g}(i\mathcal{L}_{Q})\oplus N_{g}(\mathcal{L}_{Q}i)^{\perp}$ is
equal to the whole space. The linearized evolution is decoupled into
its restriction on the subspaces $N_{g}(i\mathcal{L}_{Q})$ and $N_{g}(\mathcal{L}_{Q}i)^{\perp}$.
This motivates how we decompose the solution. When we write 
\[
u=\frac{e^{i\gamma}}{\lambda}[P(\cdot;b,\eta)+\epsilon]\Big(\frac{r}{\lambda}\Big),
\]
we introduce parameters for scaling $\lambda$, phase rotation $\gamma$,
pseudoconformal phase $b$, and the additional parameter $\eta$ associated
with $\rho$ responsible for the rotational instability. Then, our
modulated profile $P(\cdot;b,\eta)$ (applied with phase rotations
$\gamma$ and scalings $\lambda$) is tangent to $N_{g}(i\mathcal{L}_{Q})$.
On the other hand, we require orthogonality conditions for $\epsilon$
to lie in $N_{g}(\mathcal{L}_{Q}i)^{\perp}$.

\subsection{\label{subsec:Conjugation-identities}Conjugation identities}

In view of modulation analysis, one of the main ingredients of the
proof of our main theorems is the energy estimates of $\epsilon$.
More precisely, we need energy estimates for higher Sobolev norms
of $\epsilon$. Due to scaling considerations ($\lambda\to0$ as approaching
to the blow-up time), we expect that higher order derivatives of $\epsilon$
become smaller. For further discussions, see Section \ref{subsec:energy-identity}.

To perform an energy estimate for higher order derivatives, let us
start from the Hamiltonian flow 
\begin{align*}
\partial_{t}\epsilon+iH\epsilon & =0,
\end{align*}
for a symmetric differential operator $H$. The Hamiltonian structure
suggests the energy functional $(\epsilon,H\epsilon)_{r}$. To control
the higher derivatives, commuting $\partial_{r}$ (or the gradients
$\nabla$) will not work very well, because it simply does not commute
with $H$ and loses the Hamiltonian structure. In order to keep the
Hamiltonian structure, a natural conjugation is through $(iH)^{k}$:
\[
(\partial_{t}+iH)(iH)^{k}\epsilon=0.
\]
The associated higher order energy is $((iH)^{k}\epsilon,H(iH)^{k}\epsilon)_{r}$.

As in our case, when $H$ enjoys a special factorization $\mathcal{L}_{Q}=L_{Q}^{\ast}L_{Q}$,
the energy functional at the $\dot{H}^{1}$-level can be written as
$\|L_{Q}\epsilon\|_{L^{2}}^{2}$. To control the higher derivatives,
one can try the conjugation by one derivative (or, \emph{half} of
$H$) to get
\begin{align}
\partial_{t}\epsilon_{1}+L_{Q}iL_{Q}^{\ast}\epsilon_{1} & =0,\label{eq:LinEq-e1}\\
\epsilon_{1}\coloneqq L_{Q}\epsilon.\nonumber 
\end{align}
Due to $[L_{Q},i]\neq0$, this does not seem to be of Hamiltonian
form. However, surprisingly enough, it \emph{is} of Hamiltonian form.
\begin{prop}[Conjugation identity]
\label{prop:ConjugationIdentity}Let $\phi$ solve \eqref{eq:CovariantCSS}.
Then, $\tilde{\D}_{+}\phi$ solves 
\begin{equation}
(i\D_{t}+\tfrac{1}{2}|\phi|^{2})\tilde{\D}_{+}\phi-\tilde{\D}_{+}^{\ast}\tilde{\D}_{+}\tilde{\D}_{+}\phi=0.\label{eq:NonlinearConjugation}
\end{equation}
At the linearized level, we have the conjugation identity 
\begin{equation}
L_{Q}iL_{Q}^{\ast}=iA_{Q}^{\ast}A_{Q},\label{eq:ConjugationIdentity}
\end{equation}
where $A_{Q}$ is the radial part of $\tilde{\D}_{+}^{(Q)}$ acting
on $m+1$ equivariant functions, i.e. 
\begin{align}
\tilde{\D}_{+}^{(Q)}(f(r)e^{i(m+1)\theta}) & =[A_{Q}f](r)e^{i(m+2)\theta},\label{eq:Def-AQ}\\
A_{Q}\coloneqq\partial_{r} & -\tfrac{m+1+A_{\theta}[Q]}{r}.\nonumber 
\end{align}
\end{prop}

Let us postpone the proof for a moment.

This new factorization is one of our novelties of this work. One advantage
of \eqref{eq:ConjugationIdentity} is that, as alluded to above, \eqref{eq:LinEq-e1}
is now written as a Hamiltonian form. There is another advantage,
which is more important. We will observe \emph{repulsivity} in the
variable $\epsilon_{2}\coloneqq A_{Q}L_{Q}\epsilon$, which ultimately
relies on the fact that $A_{Q}L_{Q}$ kills all the elements of the
generalized null space.

Using \eqref{eq:ConjugationIdentity}, we rewrite the $\epsilon_{1}$-equation
as 
\begin{equation}
\partial_{t}\epsilon_{1}+iA_{Q}^{\ast}A_{Q}\epsilon_{1}=0.\label{eq:LinEq-e1-ConjIdty}
\end{equation}
The associated energy functional is $\|A_{Q}\epsilon_{1}\|_{L^{2}}^{2}$
at the $\dot{H}^{2}$-level. Moreover, in contrast to $[L_{Q},i]\neq0$,
we have 
\[
[A_{Q},i]=0.
\]
Thus we can proceed by the method of supersymmetric conjugates (i.e.
alternating the conjugations $A_{Q}^{\ast}$ and $A_{Q}$) to keep
the Hamiltonian form $(\partial_{t}+iH)\epsilon_{k}=0$. For example,
$\epsilon_{1}=L_{Q}\epsilon$, $\epsilon_{2}=A_{Q}\epsilon_{1}$,
$\epsilon_{3}=A_{Q}^{\ast}\epsilon_{2}$, $\epsilon_{4}=A_{Q}\epsilon_{3}$,
$\epsilon_{5}=A_{Q}^{\ast}\epsilon_{4}$, and so on: 
\begin{align}
(\partial_{t}+iA_{Q}A_{Q}^{\ast})\epsilon_{2} & =0,\label{eq:LinEq-e2}\\
(\partial_{t}+iA_{Q}^{\ast}A_{Q})\epsilon_{3} & =0.\nonumber 
\end{align}
The associated energy functional at the $\dot{H}^{3}$-level is $\|\epsilon_{3}\|_{L^{2}}^{2}$.
In this fashion, we can naturally define associated higher order energies.

More remarkably, the conjugation identity \eqref{eq:ConjugationIdentity}
is crucial to obtain repulsivity properties. In earlier works of wave
maps and Schr\"odinger maps, which also enjoy self-duality, the authors
in \cite{RodnianskiSterbenz2010Ann.Math.,RaphaelRodnianski2012Publ.Math.,MerleRaphaelRodnianski2013InventMath}
used the method of supersymmetric conjugates to obtain the repulsivity
in the conjugated dynamics. We will later use repulsivity as a corrective
term for the energy identity, in the the main bootstrap argument.
See Section \ref{subsec:energy-identity}.

In order to obtain repulsivity properties of the linearized equation
$\partial_{t}\epsilon+i\mathcal{L}_{Q}\epsilon=0$, we need to restrict
ourselves on $N_{g}(\mathcal{L}_{Q}i)^{\perp}$ as the elements of
$N_{g}(i\mathcal{L}_{Q})$ do not decay in time. However, it is not
clear that one can obtain repulsivity working directly on $\partial_{t}\epsilon+i\mathcal{L}_{Q}\epsilon=0$,
as one has to use somehow the relation $\epsilon\in N_{g}(\mathcal{L}_{Q}i)^{\perp}$.

Motivated from the works \cite{RodnianskiSterbenz2010Ann.Math.,RaphaelRodnianski2012Publ.Math.,MerleRaphaelRodnianski2013InventMath},
we conjugate $L_{Q}$ to the linearized equation get 
\[
\partial_{t}\epsilon_{1}+L_{Q}iL_{Q}^{\ast}\epsilon_{1}=0.
\]
Unlike the wave and Schr\"odinger maps, taking $L_{Q}$ does not
kill all the elements of $N_{g}(i\mathcal{L}_{Q})$, so the repulsivity
of \eqref{eq:LinEq-e1} will not hold for general $\epsilon_{1}$.
For example, $\epsilon_{1}(t,r)=rQ=2(m+1)L_{Q}\rho$ is a static solution
to \eqref{eq:LinEq-e1}. Thus it is still not sufficient to obtain
repulsivity directly from \eqref{eq:LinEq-e1}.

Perhaps it would be natural to further conjugate by $iL_{Q}^{\ast}$.
However, the resulting equation is merely the original linearized
equation $(\partial_{t}+iL_{Q}^{\ast}L_{Q})\tilde{\epsilon}_{2}=(\partial_{t}+i\mathcal{L}_{Q})\tilde{\epsilon}_{2}=0$,
where $\tilde{\epsilon}_{2}=iL_{Q}^{\ast}\epsilon_{1}$.

Rather than conjugating $iL_{Q}^{\ast}$, we conjugate $A_{Q}$, which
is naturally suggested from the new factorization of $L_{Q}iL_{Q}^{\ast}=iA_{Q}^{\ast}A_{Q}$.
A further conjugation by $A_{Q}$ (and using $[A_{Q},i]=0$) to \eqref{eq:LinEq-e1-ConjIdty}
yields 
\begin{align*}
(\partial_{t}+iA_{Q}A_{Q}^{\ast})\epsilon_{2} & =0,\\
\epsilon_{2}=A_{Q}\epsilon_{1} & =A_{Q}L_{Q}\epsilon.
\end{align*}
The main enemy to \eqref{eq:LinEq-e1}, the static solution $rQ$,
is now ruled out in the $\epsilon_{2}$-equation, thanks to the identity
\[
A_{Q}(rQ)=0.
\]
To put it differently, $A_{Q}L_{Q}$ kills all the elements of $N_{g}(i\mathcal{L}_{Q})$,
which are the main enemies for the repulsivity. A crucial observation
now is that $A_{Q}A_{Q}^{\ast}$ has a repulsive potential. More precisely,
we can write 
\begin{align*}
A_{Q}A_{Q}^{\ast} & =-\partial_{rr}-\tfrac{1}{r}\partial_{r}+\tfrac{\tilde V}{r^{2}},\\
\tilde V & \coloneqq(m+2+A_{\theta}[Q])^{2}+\tfrac{1}{2}r^{2}Q^{2}
\end{align*}
satisfying 
\[
\tilde V\gtrsim1\qquad\text{and}\qquad r\partial_{r}\tilde V=-r^{2}Q^{2}\leq0.
\]
The repulsivity yields \emph{monotonicity} by the virial-type computation:
\begin{align}
\tfrac{1}{2}\partial_{t}(\epsilon_{2},-i\Lambda\epsilon_{2})_{r}=(\epsilon_{2},A_{Q}A_{Q}^{\ast}\epsilon_{2})_{r} & +(\epsilon_{2},\tfrac{-\partial_{r}\tilde V}{2r}\epsilon_{2})_{r}\label{eq:FormalMonotonicity}\\
 & \geq(\epsilon_{2},A_{Q}A_{Q}^{\ast}\epsilon_{2})_{r}=\|\epsilon_{3}\|_{L^{2}}^{2}.\nonumber 
\end{align}
In Section \ref{subsec:energy-identity}, we will use a localized
version of this monotonicity as a corrective term for the energy identity.

\begin{rem}[Connection with wave maps and Schr\"odinger maps]
\label{rem:ConnectionSMWM}It is remarkable that \eqref{eq:CSS}
has a connection with the wave and Schr\"odinger maps. It turns out
that $A_{Q}$ is identical to the linearized Bogomol'nyi operator
appearing in the wave and Schr\"odinger maps. A direct computation
using 
\[
A_{\theta}[Q]=-2(m+1)\frac{r^{2m+2}}{1+r^{2m+2}}
\]
shows that 
\[
A_{Q}=\partial_{r}-\frac{m+1+A_{\theta}[Q]}{r}=\partial_{r}-\frac{m+1}{r}\frac{1-r^{2m+2}}{1+r^{2m+2}}.
\]
This is equal to the linearized Bogomol'nyi operator in the wave and
Schr\"odinger maps with equivariance index $k=m+1$. (See \cite[(34) and (47)]{RodnianskiSterbenz2010Ann.Math.}
for the wave maps and \cite[(2.11)]{MerleRaphaelRodnianski2013InventMath}
for the Schr\"odinger maps with $k=1$) It tells us that at least
in the linearized level, \eqref{eq:CSS}-dynamics for $\epsilon_{1}=L_{Q}\epsilon$
\eqref{eq:LinEq-e1-ConjIdty} is closely related to the dynamics for
wave and Schr\"odinger maps. The shift of equivariance from $m$
to $m+1$ comes from \eqref{eq:Def-D+} and the fact that our variable
$\epsilon_{1}=L_{Q}\epsilon$ is obtained from linearizing it. Thus
$\epsilon_{1}$ is regarded as the radial part of an $(m+1)$-equivariant
function.
\end{rem}

\begin{proof}[Proof of Proposition \ref{prop:ConjugationIdentity}]
Note that \eqref{eq:ConjugationIdentity} can be proved by direct
computations. However, we will derive \eqref{eq:ConjugationIdentity}
by linearizing \eqref{eq:NonlinearConjugation}.

We first prove \eqref{eq:NonlinearConjugation}. We start from writing
\eqref{eq:CovariantCSS} as 
\begin{equation}
(i\D_{t}+\tfrac{1}{2}|\phi|^{2})\phi-\tilde{\D}_{+}^{\ast}\tilde{\D}_{+}\phi=0.\label{eq:CovariantCSS-selfdual}
\end{equation}
This form is adapted to the expression $\phi(t,x)=[Q+\epsilon(t,\cdot)](r)e^{im\theta}$.
Indeed, if $\epsilon$ is small, then $i\partial_{t}\phi\approx0$
and $\tilde{\D}_{+}\phi\approx0$ so we can expect that $i\D_{t}+\tfrac{1}{2}|\phi|^{2}\approx i\partial_{t}$.
Now we conjugate $\tilde{\D}_{+}$ to the equation \eqref{eq:CovariantCSS-selfdual}.
From the formulae 
\begin{align*}
[\tilde{\D}_{+},i\D_{t}+\tfrac{1}{2}|\phi|^{2}] & =i[\tilde{\D}_{+},\D_{t}]+(\partial_{1}+i\partial_{2})(\tfrac{1}{2}|\phi|^{2})\\
 & =(F_{01}+iF_{02})+(\Re(\overline{\phi}\D_{1}\phi)+i\Re(\overline{\phi}\D_{2}\phi))=\overline{\phi}\tilde{\D}_{+}\phi,\\{}
[\tilde{\D}_{+}^{\ast},\tilde{\D}_{+}] & =2i[\D_{2},\D_{1}]=2F_{12}=-|\phi|^{2},
\end{align*}
we get 
\begin{align*}
0 & =\tilde{\D}_{+}(i\D_{t}+\tfrac{1}{2}|\phi|^{2})\phi-\tilde{\D}_{+}\tilde{\D}_{+}^{\ast}\tilde{\D}_{+}\phi\\
 & =(i\D_{t}+\tfrac{1}{2}|\phi|^{2})\tilde{\D}_{+}\phi+(|\phi|^{2}-\tilde{\D}_{+}\tilde{\D}_{+}^{\ast})\tilde{\D}_{+}\phi\\
 & =(i\D_{t}+\tfrac{1}{2}|\phi|^{2})\tilde{\D}_{+}\phi-\tilde{\D}_{+}^{\ast}\tilde{\D}_{+}\tilde{\D}_{+}\phi,
\end{align*}
which shows \eqref{eq:NonlinearConjugation}.

We will write \eqref{eq:NonlinearConjugation} under the Coulomb gauge
with equivariance, $\phi(t,x)=u(t,r)e^{im\theta}$. For notational
convenience, let $q=\D_{+}^{(u)}u$ so that $[\tilde{\D}_{+}\phi](t,x)=q(t,r)e^{i(m+1)\theta}$,
where we recall by \eqref{eq:Def-D+} that $\tilde{\D}_{+}\phi$ is
$(m+1)$-equivariant. We need to know how $i\D_{t}+\frac{1}{2}|\phi|^{2}$
and $\tilde{\D}_{+}^{\ast}\tilde{\D}_{+}$ acts on $(m+1)$-equivariant
functions. First, the potential part of $i\D_{t}+\frac{1}{2}|\phi|^{2}=i\partial_{t}-A_{0}+\frac{1}{2}|\phi|^{2}$
is spherically symmetric and has the formula 
\[
[-A_{0}+\tfrac{1}{2}|\phi|^{2}](x)=-\int_{|x|}^{\infty}\Re(\overline{u}\D_{+}^{(u)}u)dr.
\]
In particular, $i\D_{t}+\frac{1}{2}|\phi|^{2}$ is merely a small
perturbation of $i\partial_{t}$, provided that $\D_{+}^{(u)}u$ is
small as observed above. We thus have 
\[
(i\D_{t}+\tfrac{1}{2}|\phi|^{2})\tilde{\D}_{+}\phi=\Big[i\partial_{t}q-\Big(\int_{r}^{\infty}\Re(\overline{u}q)dr'\Big)q\Big]e^{i(m+1)\theta}.
\]
Next, if we define $A_{u}$ by the radial part of $\tilde{\mathbf{D}}_{+}=\tilde{\D}_{+}^{(\phi)}$
acting on $(m+1)$-equivariant functions,\footnote{Note that $\tilde{\D}_{+}$ shifts the equivariance index from $m+1$
to $m+2$. Similarly $\tilde{\D}_{+}^{\ast}$ shifts the equivariance
index from $m+2$ to $m+1$. The following display shows how $\tilde{\D}_{+}$
acts on the radial parts of $m+1$ equivariant functions, which can
be seen as a variant of \eqref{eq:Def-D+}.} i.e., 
\begin{align*}
\tilde{\D}_{+}^{(\phi)}(f(r)e^{i(m+1)\theta}) & =[A_{u}f](r)e^{i(m+2)\theta},\\
A_{u}\coloneqq\partial_{r} & -\tfrac{m+1+A_{\theta}[u]}{r},
\end{align*}
then we have 
\[
\tilde{\D}_{+}^{\ast}\tilde{\D}_{+}\tilde{\D}_{+}\phi=\tilde{\D}_{+}^{\ast}\tilde{\D}_{+}(q(r)e^{i(m+1)\theta})=[A_{u}^{\ast}A_{u}q]e^{i(m+1)\theta}.
\]
Combining altogether, \eqref{eq:NonlinearConjugation} reads 
\begin{equation}
i\partial_{t}q-\Big(\int_{r}^{\infty}\Re(\overline{u}q)dr\Big)q-A_{u}^{\ast}A_{u}q=0.\label{eq:conj-idty-temp2}
\end{equation}

Now we consider the linearization 
\begin{align*}
u & =Q+\epsilon,\\
q & =L_{Q}\epsilon+O(\epsilon^{2}),\\
A_{u} & =A_{Q}+O(\epsilon).
\end{align*}
Substituting these into \eqref{eq:conj-idty-temp2} yields 
\[
(i\partial_{t}-A_{Q}^{\ast}A_{Q})L_{Q}\epsilon=O(\epsilon^{2}).
\]
Comparing this with \eqref{eq:LinEq-e1} explains why the conjugation
identity \eqref{eq:ConjugationIdentity} should hold.
\end{proof}

\subsection{\label{subsec:Adapted-function-spaces}Adapted function spaces}

So far, our discussion was purely algebraic. In this subsection, we
investigate the function spaces adapted to the aforementioned operators
like $L_{Q}$, $A_{Q}$, and $A_{Q}^{\ast}$. Our function spaces
will be defined for equivariant functions on $\R^{2}$, but by an
abuse of notations, we will identify equivariant functions (defined
on $\R^{2}$) with their radial parts (defined on $(0,\infty)$).
Strictly speaking, $L_{Q}$, $A_{Q}$, and $A_{Q}^{\ast}$ act on
radial parts. With this abuse of notations, we will regard $\D_{+}^{(u)}$
sending $m$-equivariant functions to $m+1$-equivariant functions
by \eqref{eq:Def-D+}. Similarly, $L_{Q}$, $A_{Q}$, $A_{Q}^{\ast}$
send $m$, $m+1$, $m+2$-equivariant functions to $m+1$, $m+2$,
$m+1$-equivariant functions, respectively. (c.f. \eqref{eq:Def-AQ})

The operators will be shown to be well-defined on \emph{adapted function
spaces}. Moreover, we study the (sub-)coercivity estimates, which
allow us to estimate $\epsilon$ from its adapted derivatives $\epsilon_{1}=L_{Q}\epsilon$,
$\epsilon_{2}=A_{Q}\epsilon_{1}$, or $\epsilon_{3}=A_{Q}^{\ast}\epsilon_{2}$
and motivate the choice of our function spaces. Of course, because
$A_{Q}$ and $L_{Q}$ have nontrivial kernels, coercivity estimates
are obtained after imposing suitable orthogonality conditions. Note
that we need to work at least in the $\dot{H}^{3}$-level, as explained
in Section \ref{subsec:strategy}. The choices of these adapted derivatives
are motivated in Section \ref{subsec:Conjugation-identities}.

Moreover, we will need to take into account the fact that $\epsilon$,
$\epsilon_{1}=L_{Q}\epsilon$, $\epsilon_{2}=A_{Q}\epsilon_{1}$,
and $\epsilon_{3}=A_{Q}^{\ast}\epsilon_{2}$ correspond to the radial
parts of $m$, $m+1$, $m+2$, and $m+1$ equivariant functions on
$\R^{2}$. Thus we need to develop our function spaces with various
regularities and equivariance indices.

\subsubsection*{Equivariant Sobolev spaces}

$ $\\ For $s\geq0$, we denote by $H_{m}^{s}$ the restriction of
the usual Sobolev space $H^{s}(\R^{2})$ on $m$-equivariant functions.
The set of $m$-equivariant Schwartz functions, denoted by $\mathcal{S}_{m}$,
is dense in $H_{m}^{s}$. We use the $H_{m}^{s}$-norms and $\dot{H}_{m}^{s}$-norms
to mean the usual $H^{s}$-norms and $\dot{H}^{s}$-norms, but the
subscript $m$ is used to emphasize that we are applying these norms
to $m$-equivariant functions.

One of the advantages of using equivariant Sobolev spaces is the generalized
Hardy's inequality \cite[Lemma A.7]{KimKwon2019arXiv}: whenever $0\leq k\leq m$,
we have 
\begin{equation}
\||f|_{-k}\|_{L^{2}}\sim\|f\|_{\dot{H}_{m}^{k}},\qquad\forall f\in\mathcal{S}_{m}.\label{eq:GenHardySection2}
\end{equation}
In particular, when $0\leq k\leq m$, we can define the \emph{homogeneous
equivariant Sobolev space} $\dot{H}_{m}^{k}$ by taking the completion
of $\mathcal{S}_{m}$ under the $\dot{H}_{m}^{k}$-norm and have the
embeddings 
\[
\mathcal{S}_{m}\hookrightarrow H_{m}^{k}\hookrightarrow\dot{H}_{m}^{k}\hookrightarrow L_{\mathrm{loc}}^{2}.
\]
Specializing this to $k=1$ and applying the fundamental theorem of
calculus to $\partial_{r}|f|^{2}=2\Re(\overline{f}\partial_{r}f)$,
we have the \emph{Hardy-Sobolev inequality} \cite[Lemma A.6]{KimKwon2019arXiv}:
whenever $m\geq1$, we have 
\begin{equation}
\|r^{-1}f\|_{L^{2}}+\|f\|_{L^{\infty}}\lesssim\|f\|_{\dot{H}_{m}^{1}}.\label{eq:HardySobolevSection2}
\end{equation}
The estimates \eqref{eq:GenHardySection2} and \eqref{eq:HardySobolevSection2}
also hold for negative $m$, requiring either $0\leq k\leq|m|$ or
$|m|\geq1$. Note that in general $H^{1}\hookrightarrow L^{\infty}$
is \emph{false} on $\R^{2}$.

\subsubsection*{Adapted function space at $\dot{H}^{1}$-level}

$ $\\ Here, we investigate the relation $\epsilon_{1}=L_{Q}\epsilon$
such that $\epsilon_{1}$ lies in $L^{2}$. Analyzing at this level
is well-suited for the original energy functional $E[u]$.
\begin{lem}[Boundedness and subcoercivity for $L_{Q}$ on $\dot{H}_{m}^{1}$]
\label{lem:Mapping-L-Section2}For $v\in\dot{H}_{m}^{1}$, we have
\[
\|L_{Q}v\|_{L^{2}}+\|Qv\|_{L^{2}}\sim\|v\|_{\dot{H}_{m}^{1}}.
\]
Moreover, the kernel of $L_{Q}:\dot{H}_{m}^{1}\to L^{2}$ is $\mathrm{span}_{\R}\{\Lambda Q,iQ\}$.
\end{lem}

The equivalence in Lemma \ref{lem:Mapping-L-Section2} explains why
the space $\dot{H}_{m}^{1}$ is the right function space at $\dot{H}^{1}$-level.
The main tool is the weighted Hardy's inequality adapted to the operator
$L_{Q}$, written in the form of Corollary \ref{cor:WeightedHardy}.
We postpone the proof in Appendix \ref{sec:HardyInequalities}. Here
we explain the heuristics of the subcoercivity estimate $(\gtrsim)$
and how it suggests the right function space.

First, the $QB_{Q}$-contribution of $L_{Q}$ is perturbative. We
focus on the $\D_{+}^{(Q)}$-part of $L_{Q}$. Notice that $\D_{+}^{(Q)}\approx\partial_{r}-\frac{m}{r}$
when $r$ is small and $\D_{+}^{(Q)}\approx\partial_{r}+\frac{m+2}{r}$
when $r$ is large. When $r$ is small, an application of the noncritical
case of Corollary \ref{cor:WeightedHardy} with $\ell=m$ and $k=0$
says that $\|\mathbf{1}_{r\leq1}\D_{+}^{(Q)}f\|_{L^{2}}$ plus a boundary
term at $r\sim1$ has a lower bound $\|\mathbf{1}_{r\leq1}r^{-1}f\|_{L^{2}}$.
Note that the assumption $m\geq1$ is necessary to control $r^{-1}f$
in the small $r$ regime. Similarly, when $r$ is large, Corollary
\ref{cor:WeightedHardy} with $\ell=-m-2$ and $k=0$ says that $\|\mathbf{1}_{r\geq1}\D_{+}^{(Q)}f\|_{L^{2}}$
plus a boundary term at $r\sim1$ has a lower bound $\|\mathbf{1}_{r\geq1}r^{-1}f\|_{L^{2}}$.

The $Qv$ term of Lemma \ref{lem:Mapping-L-Section2} can be safely
deleted after ruling out the kernel elements $\Lambda Q,iQ$ of $L_{Q}$.
\begin{lem}[Coercivity for $L_{Q}$ on $\dot{H}_{m}^{1}$]
\label{lem:Coercivity-L-Section2}Let $\psi_{1}$ and $\psi_{2}$
be elements of $(\dot{H}_{m}^{1})^{\ast}$, which is the dual space
of $\dot{H}_{m}^{1}$. If the $2\times2$ matrix $(a_{ij})$ defined
by $a_{i1}=(\psi_{i},\Lambda Q)_{r}$ and $a_{i2}=(\psi_{i},iQ)_{r}$
has nonzero determinant, then we have a coercivity estimate 
\[
\|v\|_{\dot{H}_{m}^{1}}\gtrsim\|L_{Q}v\|_{L^{2}}\gtrsim_{\psi_{1},\psi_{2}}\|v\|_{\dot{H}_{m}^{1}},\qquad\forall v\in\dot{H}_{m}^{1}\cap\{\psi_{1},\psi_{2}\}^{\perp}.
\]
\end{lem}

We postpone the proof in Appendix \ref{sec:HardyInequalities}.

\subsubsection*{Adapted function space at $\dot{H}^{3}$-level}

$ $\\ We will perform an energy method for $\epsilon_{3}=A_{Q}^{\ast}\epsilon_{2}=A_{Q}^{\ast}A_{Q}L_{Q}\epsilon$.
This is due to a scaling consideration, which is motivated in Section
\ref{subsec:strategy} and also detailed in Section \ref{subsec:energy-identity}.
Henceforth, we look for adapted function space for $\epsilon$ such
that $\epsilon_{3}=A_{Q}^{\ast}A_{Q}L_{Q}\epsilon$ lies in $L^{2}$.

It turns out that the space $\dot{H}_{m}^{3}$ is not the right choice
for $m\in\{1,2\}$. We need to introduce an adapted function space
$\dot{\mathcal{H}}_{m}^{3}$. It turns out that $\dot{\mathcal{H}}_{m}^{3}$
is slightly smaller than $\dot{H}_{m}^{3}$ when $m\in\{1,2\}$. For
a domain $\Omega\subseteq\R^{2}$, we define the (semi-)norm
\begin{multline}
\|f\|_{\dot{\mathcal{H}}_{m}^{3}(\Omega)}\coloneqq\||\partial_{+}f|_{-2}\|_{L^{2}(\Omega)}\\
+\begin{cases}
\||f|_{-3}\|_{L^{2}(\Omega)} & \text{if }m\geq3,\\
\|\partial_{rrr}f\|_{L^{2}(\Omega)}+\|r^{-1}\langle\log_{-}r\rangle^{-1}|f|_{-2}\|_{L^{2}(\Omega)} & \text{if }m=2,\\
\||\partial_{rr}f|_{-1}\|_{L^{2}(\Omega)}+\|r^{-1}\langle r\rangle^{-1}\langle\log_{-}r\rangle^{-1}|f|_{-1}\|_{L^{2}(\Omega)} & \text{if }m=1.
\end{cases}\label{eq:Def-Hdot3-Section2}
\end{multline}
We will use $\Omega\in\{\R^{2},B_{R}(0),B_{2R}(0)\setminus B_{R}(0)\}$
and denote 
\begin{align*}
\|f\|_{\dot{\mathcal{H}}_{m}^{3}} & \coloneqq\|f\|_{\dot{\mathcal{H}}_{m}^{3}(\R^{2})},\\
\|f\|_{\dot{\mathcal{H}}_{m,\leq R}^{3}} & \coloneqq\|f\|_{\dot{\mathcal{H}}_{m}^{3}(B(0;R))},\\
\|f\|_{\dot{\mathcal{H}}_{m,\sim R}^{3}} & \coloneqq\|f\|_{\dot{\mathcal{H}}_{m}^{3}(B(0;2R)\setminus B(0;R))},
\end{align*}
Define the space $\dot{\mathcal{H}}_{m}^{3}$ by taking the completion
of $\mathcal{S}_{m}$ under the $\dot{\mathcal{H}}_{m}^{3}$-norm.
Here, we recall that $\partial_{+}=\partial_{1}+i\partial_{2}$, which
acts on $m$-equivariant functions by $\partial_{r}-\frac{m}{r}$
and shifts the equivariance index from $m$ to $m+1$. Note by \eqref{eq:GenHardySection2}
that $\||\partial_{+}f|_{-2}\|_{L^{2}}$ is equivalent to $\|\partial_{+}f\|_{\dot{H}_{m+1}^{2}}$.
The terms $\partial_{rrr}f$ for $m=2$ and $|\partial_{rr}f|_{-1}$
for $m=1$ are in fact redundant, because $|\partial_{rrr}f|\lesssim|\partial_{+}f|_{-2}$
when $m=2$ and $|\partial_{rr}f|_{-1}\lesssim|\partial_{+}f|_{-2}$
when $m=1$. For convenience in referring, we keep them in the definition.
When $f$ is supported in $r\geq1$, we indeed have 
\[
\|f\|_{\dot{\mathcal{H}}_{m}^{3}}\sim\||f|_{-3}\|_{L^{2}}.
\]

We first compare $\dot{\mathcal{H}}_{m}^{3}$ with $\dot{H}_{m}^{3}$.
As the Laplacian $\Delta$ on $\R^{2}$ admits the decomposition $\Delta=\partial_{-}\partial_{+}$
and $\partial_{+}f$ is now $m+1$-equivariant (with $m+1\geq2$),
one has 
\[
\|f\|_{\dot{H}_{m}^{3}}\sim\|\partial_{+}f\|_{\dot{H}_{m+1}^{2}},\qquad\forall f\in\mathcal{S}_{m}.
\]
Thus the $\dot{\mathcal{H}}_{m}^{3}$-norm is stronger than (or equal
to) the $\dot{H}_{m}^{3}$-norm. When $m\geq3$, these are same, thanks
to \eqref{eq:GenHardySection2}. When $m\in\{1,2\}$, however, the
$\dot{\mathcal{H}}_{m}^{3}$-norm turns out to be strictly stronger
than the $\dot{H}_{m}^{3}$-norm. Despite this fact, we have 
\[
L^{2}\cap\dot{\mathcal{H}}_{m}^{3}=H_{m}^{3}.
\]
See Lemma \ref{lem:Comparison-H3-Hdot3} for details.

We now motivate the definition of $\dot{\mathcal{H}}_{m}^{3}$. We
again use the weighted Hardy's inequality, but we will use both the
noncritical and critical case of Corollary \ref{cor:WeightedHardy}.

Now recall that $\epsilon$, $\epsilon_{1}=L_{Q}\epsilon$, $\epsilon_{2}=A_{Q}\epsilon_{1}$,
and $\epsilon_{3}=A_{Q}^{\ast}\epsilon_{2}$ correspond to $m$, $m+1$,
$m+2$, and $m+1$ equivariant functions. We first consider $\epsilon_{3}=A_{Q}^{\ast}\epsilon_{2}$
with $\epsilon_{3}\in L^{2}$. Thanks to the positivity of $A_{Q}A_{Q}^{\ast}$,
we have 
\begin{equation}
\|A_{Q}^{\ast}v\|_{L^{2}}\sim\|v\|_{\dot{H}_{m+2}^{1}}.\label{eq:non-M-dep-Hardy}
\end{equation}
Thus the natural function space for $\epsilon_{2}$ is $\dot{H}_{m+2}^{1}$.
Next, we consider $\epsilon_{2}=A_{Q}\epsilon_{1}$ with $\epsilon_{2}\in\dot{H}_{m+2}^{1}$.
We will see that the natural function space for $\epsilon_{1}$ is
$\dot{H}_{m+1}^{2}$. Indeed, as $m+1\geq2$, we can apply \eqref{eq:GenHardySection2}
to see that $A_{Q}:\dot{H}_{m+1}^{2}\to\dot{H}_{m+2}^{1}$ bounded.
On the other hand, $A_{Q}\approx\partial_{r}-\frac{m+1}{r}$ for small
$r$ and $A_{Q}\approx\partial_{r}+\frac{m+1}{r}$ for large $r$,
by a similar application of the noncritical case of Corollary \ref{cor:WeightedHardy}
with $\ell=\pm(m+1)$ and $k=1$ says that $\|r^{-1}A_{Q}f\|_{L^{2}}$
plus a boundary term localized at $r\sim1$ has a lower bound $\|r^{-2}f\|_{L^{2}}$.
Here, the assumption $m\geq1$ is necessary to control $r^{-2}f$
in the small $r$ regime.

Finally, we consider $\epsilon_{1}=L_{Q}\epsilon$ with $\epsilon_{1}\in\dot{H}_{m+1}^{2}$.
Treating $QB_{Q}$ term perturbatively, it suffices to study the subcoercivity
property of $\|\D_{+}^{(Q)}v\|_{\dot{H}_{m+1}^{2}}$. Recall that
$\D_{+}^{(Q)}\approx\partial_{r}-\frac{m}{r}$ for small $r$ and
$\D_{+}^{(Q)}\approx\partial_{r}+\frac{m+2}{r}$ for large $r$. When
$r$ is large, $\|\D_{+}^{(Q)}v\|_{\dot{H}_{m+1}^{2}}$ plus a boundary
term at $r\sim1$ can control $\|\mathbf{1}_{r\geq1}r^{-3}v\|_{L^{2}}$
by Corollary \ref{cor:WeightedHardy} with $\ell=-m-2$ and $k=2$,
for all $m\geq1$. In view of 
\[
|(\partial_{r}+\tfrac{m+2}{r})v|_{-2}+r^{-3}|v|\sim|v|_{-3}\gtrsim|\partial_{+}v|_{-2},
\]
we conclude that $\|\D_{+}^{(Q)}v\|_{\dot{H}_{m+1}^{2}}$ plus a perturbative
term at $r\sim1$ can control the $r\geq1$ portion of \eqref{eq:Def-Hdot3-Section2}.

However, when $r$ is small, the situation is delicate and we should
introduce logarithmic terms in \eqref{eq:Def-Hdot3-Section2} when
$m\in\{1,2\}$. The main source of these logarithmic terms is the
\emph{logarithmic Hardy's inequality}, the critical case of Corollary
\ref{cor:WeightedHardy}. Note that the lower bound $\|\partial_{+}f\|_{\dot{H}_{m+1}^{2}}$
in \eqref{eq:Def-Hdot3-Section2} comes from $\D_{+}^{(Q)}\approx\partial_{+}=\partial_{r}-\frac{m}{r}$
for small $r$. We also saw before that the special terms $\partial_{rrr}f$
(when $m=2$) and $|\partial_{rr}f|_{-1}$ (when $m=1$) can be absorbed
into $|\partial_{+}f|_{-2}$. Henceforth, we focus on deriving the
Hardy terms of \eqref{eq:Def-Hdot3-Section2}:
\begin{align*}
\|\mathbf{1}_{r\leq1}r^{-3}f\|_{L^{2}} & \qquad\text{if }m\geq3,\\
\|\mathbf{1}_{r\leq1}r^{-3}\langle\log r\rangle^{-1}f\|_{L^{2}} & \qquad\text{if }m=2,\\
\|\mathbf{1}_{r\leq1}r^{-2}\langle\log r\rangle^{-1}f\|_{L^{2}} & \qquad\text{if }m=1.
\end{align*}
When $m\geq3$, we simply apply the noncritical case of Corollary
\ref{cor:WeightedHardy} with $\ell=m$ and $k=2$. When $m=2$, the
critical case of Corollary \ref{cor:WeightedHardy} with $\ell=k=2$
applies and yields a logarithmic loss. When $m=1$, $\ell=1$ and
$k=2$ will belong to the noncritical case, but $\ell<k$ \emph{does
not }give the boundary term localized at $r\sim1$. Instead, we get
a weaker control by applying $\ell=1$ and $k=1$. This is how we
construct the adapted function space $\dot{\mathcal{H}}_{m}^{3}$,
which gives a subcoercivity estimate for $A_{Q}^{\ast}A_{Q}L_{Q}$.

The precise version of the subcoercivity estimate of $A_{Q}^{\ast}A_{Q}L_{Q}$
is as follows.
\begin{lem}[Boundedness and subcoercivity for $A_{Q}^{\ast}A_{Q}L_{Q}$ on $\dot{\mathcal{H}}_{m}^{3}$]
\label{lem:Mapping-AAL-Section2}For $v\in\dot{\mathcal{H}}_{m}^{3}$,
we have 
\[
\|A_{Q}^{\ast}A_{Q}L_{Q}v\|_{L^{2}}+\|Q|v|_{-2}\|_{L^{2}}\sim\|v\|_{\dot{\mathcal{H}}_{m}^{3}}.
\]
Moreover, the kernel of $A_{Q}^{\ast}A_{Q}L_{Q}:\dot{\mathcal{H}}_{m}^{3}\to L^{2}$
is $\mathrm{span}_{\R}\{\Lambda Q,iQ,\rho,ir^{2}Q\}$.
\end{lem}

Ruling out the elements $\Lambda Q,iQ,\rho,ir^{2}Q$ of the kernel,
the coercivity estimate follows.
\begin{lem}[Coercivity for $A_{Q}^{\ast}A_{Q}L_{Q}$ on $\dot{\mathcal{H}}_{m}^{3}$]
\label{lem:Coercivity-AAL-Section2}Let $\psi_{1},\psi_{2},\psi_{3},\psi_{4}$
be elements of $(\dot{\mathcal{H}}_{m}^{3})^{\ast}$, which is the
dual space of $\dot{\mathcal{H}}_{m}^{3}$. If the $4\times4$ matrix
$(a_{ij})$ defined by $a_{i1}=(\psi_{i},\Lambda Q)_{r}$, $a_{i2}=(\psi_{i},iQ)_{r}$,
$a_{i3}=(\psi_{i},ir^{2}Q)_{r}$, and $a_{i4}=(\psi_{i},\rho)_{r}$
has nonzero determinant, then we have a coercivity estimate 
\[
\|v\|_{\dot{\mathcal{H}}_{m}^{3}}\gtrsim\|A_{Q}^{\ast}A_{Q}L_{Q}v\|_{L^{2}}\gtrsim_{\psi_{1},\psi_{2},\psi_{3},\psi_{4}}\|v\|_{\dot{\mathcal{H}}_{m}^{3}},\qquad\forall v\in\dot{\mathcal{H}}_{m}^{3}\cap\{\psi_{1},\psi_{2},\psi_{3},\psi_{4}\}^{\perp}.
\]
\end{lem}

We postpone the proofs of the above lemmas to Appendix \ref{sec:HardyInequalities}.

\subsubsection*{Adapted Function Spaces at $\dot{H}^{5}$ when $m\protect\geq3$}

From now on, we assume $m\geq3$ until the end of this section.

For the difference estimate of trapped solutions, we will need not
only the $H_{m}^{3}$-controls, but also the $\dot{H}^{5}$-controls.
More precisely, to control the difference at the $\dot{H}^{3}$-level,
we will need a priori bounds of trapped solutions at the $\dot{H}^{5}$-level.
For this purpose, we will perform the construction analysis and develop
the adapted function spaces at the $\dot{H}^{5}$-level.

To construct the adapted function space at the $\dot{H}^{5}$-level,
say $\dot{\mathcal{H}}_{m}^{5}$, we need to prove subcoercivity estimates
for $A_{Q}^{\ast}A_{Q}A_{Q}^{\ast}A_{Q}L_{Q}:\dot{\mathcal{H}}_{m}^{5}\to L^{2}$.
We follow a similar analysis as in $\dot{\mathcal{H}}_{m}^{3}$. However,
we restrict to $m\geq3$. See Remark \ref{rem:Positivity-AAA-small-m}.

For a domain $\Omega\subseteq\R^{2}$, we define the (semi-)norm 
\begin{multline}
\|f\|_{\dot{\mathcal{H}}_{m}^{5}(\Omega)}\coloneqq\||\partial_{+}f|_{-4}\|_{L^{2}(\Omega)}\\
+\begin{cases}
\||f|_{-5}\|_{L^{2}(\Omega)} & \text{if }m\geq5,\\
\|\partial_{rrrrr}f\|_{L^{2}(\Omega)}+\|r^{-1}\langle\log_{-}r\rangle^{-1}|f|_{-4}\|_{L^{2}(\Omega)} & \text{if }m=4,\\
\||\partial_{rrrr}f|_{-1}\|_{L^{2}(\Omega)}+\|r^{-1}\langle r\rangle^{-1}\langle\log_{-}r\rangle^{-1}|f|_{-3}\|_{L^{2}(\Omega)} & \text{if }m=3.
\end{cases}\label{eq:Def-Hdot5-Section2}
\end{multline}
We define $\|f\|_{\dot{\mathcal{H}}_{m}^{5}}$, $\|f\|_{\dot{\mathcal{H}}_{m,\leq R}^{5}}$,
and $\|f\|_{\dot{\mathcal{H}}_{m,\sim R}^{5}}$ as in $\dot{\mathcal{H}}_{m}^{3}$.
Define the space $\dot{\mathcal{H}}_{m}^{5}$ by taking the completion
of $\mathcal{S}_{m}$ under the $\dot{\mathcal{H}}_{m}^{5}$-norm.

As like \eqref{eq:non-M-dep-Hardy}, we will have boundedness and
positivity of $A_{Q}^{\ast}A_{Q}A_{Q}^{\ast}:\dot{H}_{m+2}^{3}\to L^{2}$
when $m\geq3$: 
\begin{equation}
\|A_{Q}^{\ast}A_{Q}A_{Q}^{\ast}v\|_{L^{2}}\sim\|v\|_{\dot{H}_{m+2}^{3}}.\label{eq:positivity-AAA}
\end{equation}
Thanks to \eqref{eq:positivity-AAA}, the construction of $\dot{\mathcal{H}}_{m}^{5}$
and (sub-)coercivity estimates for $A_{Q}^{\ast}A_{Q}A_{Q}^{\ast}A_{Q}L_{Q}$
is very similar to those in the $\dot{\mathcal{H}}_{m}^{3}$ case.
When $m\in\{1,2\}$, it turns out that \eqref{eq:positivity-AAA}
becomes \emph{false}. This not only weakens the subcoercivity estimate,
but also seems to prevent modulation analysis with current profiles
and four modulation parameters. This is the main reason for the restriction
$m\geq3$.
\begin{rem}[Discussion when $m\in\{1,2\}$]
\label{rem:Positivity-AAA-small-m}When $m\in\{1,2\}$, $\|A_{Q}^{\ast}A_{Q}A_{Q}^{\ast}v\|_{L^{2}}\lesssim\|v\|_{\dot{H}_{m+2}^{3}}$
is still true, but the $(\gtrsim)$-direction becomes \emph{false}.
To see why, let us note that $A_{Q}^{\ast}A_{Q}A_{Q}^{\ast}h=0$ for
\[
h(r)=\frac{1}{r^{2}Q}\int_{0}^{r}(r')^{3}Q^{2}dr'\sim\begin{cases}
r^{m+2} & \text{if }r\leq1,\\
r^{m} & \text{if }r\geq1.
\end{cases}
\]

When $m=1$, we have $h\in\dot{H}_{3}^{3}$ due to its slow growth
$r$. In other words, $A_{Q}^{\ast}A_{Q}A_{Q}^{\ast}:\dot{H}_{3}^{3}\to L^{2}$
has the kernel $\mathrm{span}_{\C}\{h\}$. One can have \eqref{eq:Positivity-AAA}
only after ruling out this kernel.

When $m=2$, $h$ does not belong to $\dot{H}_{4}^{3}$ due to the
growth $r^{2}$. However, we \emph{cannot} expect that $\|A_{Q}^{\ast}A_{Q}A_{Q}^{\ast}v\|_{L^{2}}$
controls $\|\frac{1}{r^{3}}v\|_{L^{2}}$ near the infinity, due to
the example $v(x)=(1-\chi)(x_{1}+ix_{2})^{2}$. In other words, $\dot{H}_{4}^{3}$
is not the right choice for the subcoercivity estimates of $A_{Q}^{\ast}A_{Q}A_{Q}^{\ast}$.
To remedy this, we need to introduce a new adapted function space
by \emph{logarithmically weakening} the $\dot{H}_{4}^{3}$-norm (or,
enlarge the function space) near the infinity to have subcoercivity
estimates of $A_{Q}^{\ast}A_{Q}A_{Q}^{\ast}$. Then the problem is
that $h$ then belongs to this enlarged space, saying that $A_{Q}^{\ast}A_{Q}A_{Q}^{\ast}$
has the kernel $\mathrm{span}_{\C}\{h\}$.

Overall, $A_{Q}^{\ast}A_{Q}A_{Q}^{\ast}A_{Q}L_{Q}$ has the kernel
of six (real) dimensions. Thus four orthogonality conditions are not
sufficient to have coercivity of $A_{Q}^{\ast}A_{Q}A_{Q}^{\ast}A_{Q}L_{Q}$.
In the modulation analysis, this requires us to add two more modulation
parameters and change profiles.
\end{rem}

Thanks to \eqref{eq:positivity-AAA}, we obtain the following subcoercivity
estimate of $A_{Q}^{\ast}A_{Q}A_{Q}^{\ast}A_{Q}L_{Q}:\dot{\mathcal{H}}_{m}^{5}\to L^{2}$
(with the kernel of four dimensions).
\begin{lem}[Boundedness and subcoercivity for $A_{Q}^{\ast}A_{Q}A_{Q}^{\ast}A_{Q}L_{Q}$
on $\dot{\mathcal{H}}_{m}^{5}$]
\label{lem:Mapping-AAAAL-Section2}For $v\in\dot{\mathcal{H}}_{m}^{5}$,
we have 
\[
\|A_{Q}^{\ast}A_{Q}A_{Q}^{\ast}A_{Q}L_{Q}v\|_{L^{2}}+\|Q|v|_{-4}\|_{L^{2}}\sim\|v\|_{\dot{\mathcal{H}}_{m}^{5}}.
\]
Moreover, the kernel of $A_{Q}^{\ast}A_{Q}A_{Q}^{\ast}A_{Q}L_{Q}:\dot{\mathcal{H}}_{m}^{5}\to L^{2}$
is $\mathrm{span}_{\R}\{\Lambda Q,iQ,\rho,ir^{2}Q\}$.
\end{lem}

Ruling out the elements $\Lambda Q,iQ,\rho,ir^{2}Q$ of the kernel,
the coercivity estimate follows.
\begin{lem}[Coercivity for $A_{Q}^{\ast}A_{Q}A_{Q}^{\ast}A_{Q}L_{Q}$ on $\dot{\mathcal{H}}_{m}^{5}$]
\label{lem:Coercivity-AAAAL-Section2}Let $\psi_{1},\psi_{2},\psi_{3},\psi_{4}$
be elements of $(\dot{\mathcal{H}}_{m}^{5})^{\ast}$, which is the
dual space of $\dot{\mathcal{H}}_{m}^{5}$. If the $4\times4$ matrix
$(a_{ij})$ defined by $a_{i1}=(\psi_{i},\Lambda Q)_{r}$, $a_{i2}=(\psi_{i},iQ)_{r}$,
$a_{i3}=(\psi_{i},ir^{2}Q)_{r}$, and $a_{i4}=(\psi_{i},\rho)_{r}$
has nonzero determinant, then we have a coercivity estimate 
\[
\|v\|_{\dot{\mathcal{H}}_{m}^{5}}\gtrsim\|A_{Q}^{\ast}A_{Q}A_{Q}^{\ast}A_{Q}L_{Q}v\|_{L^{2}}\gtrsim_{\psi_{1},\psi_{2},\psi_{3},\psi_{4}}\|v\|_{\dot{\mathcal{H}}_{m}^{5}},\qquad\forall v\in\dot{\mathcal{H}}_{m}^{5}\cap\{\psi_{1},\psi_{2},\psi_{3},\psi_{4}\}^{\perp}.
\]
\end{lem}

For details, see Appendix \ref{sec:HardyInequalities}.

\section{\label{sec:ModifiedProfile}Modified profile}

Our main goal is to study the pseudoconformal blow-up using the modulation
analysis. Namely, we decompose a blow-up solution into modulated blow-up
profile $P$ and the error $\epsilon$: 
\[
u(t,r)=\frac{e^{i\gamma(t)}}{\lambda(t)}[P(\cdot;b(t),\eta(t))+\epsilon(t,\cdot)]\Big(\frac{r}{\lambda(t)}\Big).
\]
Here, $P(\cdot;b,\eta)$ is a deformation of the static solution $Q$.
This section aims to construct $P$ and derive evolution equations
of the parameters $\lambda,\gamma,b,\eta$ such that $\frac{e^{i\gamma}}{\lambda}P(\cdot;b,\eta)(\frac{r}{\lambda})$
becomes an approximate solution to \eqref{eq:CSS}. Moreover, we aim
to capture the pseudoconformal blow-up (and its instability) using
the evolution equations of these parameters.

Introduction of $\lambda$ and $\gamma$ renormalizes our solution
$u$ to lie near $Q$. These parameters $\lambda$ and $\gamma$ correspond
to the kernel elements of $i\mathcal{L}_{Q}$, which are $\Lambda Q$
and $iQ$. In other words, the $\lambda$-curve $\frac{1}{\lambda}P(\frac{\cdot}{\lambda})$
has tangent vector $\Lambda Q$ and the $\gamma$-curve $e^{i\gamma}P$
has tangent vector $iQ$.

There are two more elements, $ir^{2}Q$ and $\rho$, in the generalized
null space $N_{g}(i\mathcal{L}_{Q})$ of $i\mathcal{L}_{Q}$. Turning
on these elements triggers the growth in $\Lambda Q$ and $iQ$ directions
in the linearized equation. At the nonlinear level, this changes the
scaling and phase parameters.

The pseudoconformal blow-up occurs by turning on the $ir^{2}Q$ direction.
Due to $i\mathcal{L}_{Q}ir^{2}Q=4\Lambda Q$, it corresponds to the
change in scales. At the nonlinear level, this can be done by introducing
the pseudoconformal phase $e^{-ib\frac{r^{2}}{4}}$, parametrized
by $b$ as in the (NLS) setting \cite{MerleRaphael2003GAFA,MerleRaphaelSzeftel2013AJM}.
To see how the pseudoconformal phase leads to blow-up, we consider
$P=Q_{b}$ with 
\[
Q_{b}(y)\coloneqq Q(y)e^{-ib\frac{y^{2}}{4}}.
\]
Note that $\partial_{b=0}Q_{b}=-i\frac{y^{2}}{4}Q$. Substituting
this $P$ into \eqref{eq:CSS} written in the renormalized variables
$(s,y)$, we get 
\[
(i\partial_{s}-i\frac{\lambda_{s}}{\lambda}\Lambda-\gamma_{s})Q_{b}-L_{Q_{b}}^{\ast}\D_{+}^{(Q_{b})}Q_{b}=0.
\]
From the identity \cite[Lemma 4.2]{KimKwon2019arXiv}
\[
L_{Q_{b}}^{\ast}\D_{+}^{(Q_{b})}Q_{b}=[L_{Q}^{\ast}\D_{+}^{(Q)}Q+ib\Lambda Q+b^{2}\tfrac{r^{2}}{4}Q]e^{-ib\frac{y^{2}}{4}}=ib\Lambda Q_{b}-b^{2}\tfrac{y^{2}}{4}Q_{b},
\]
we are led to 
\[
-i\Big(\frac{\lambda_{s}}{\lambda}+b\Big)\Lambda Q_{b}-\gamma_{s}Q_{b}+(b_{s}+b^{2})\tfrac{y^{2}}{4}Q_{b}=0.
\]
The pseudoconformal blow-up is derived by 
\[
\frac{\lambda_{s}}{\lambda}+b=0,\qquad\gamma_{s}=0,\qquad b_{s}+b^{2}=0.
\]
The explicit pseudoconformal blow-up solution $S(t,r)$ can be expressed
as 
\[
S(t,r)=\frac{e^{i\gamma(t)}}{\lambda(t)}Q_{b(t)}\Big(\frac{r}{\lambda(t)}\Big),
\]
with
\[
\lambda(t)=|t|,\qquad\gamma(t)=0,\qquad b(t)=|t|.
\]

Let us remark that the presence of $b^{2}$ in the $b$-equation is
a consequence of nonlinear algebras, not detected in the linearized
dynamics itself. Without $b^{2}$-term, it would be the self-similar
blow-up. The pseudoconformal blow-up rate is derived from the $b^{2}$-term.

The other element in the generalized null space, $\rho$, leads to
the change in phase parameter. At the linearized level, this is due
to the relation $i\mathcal{L}_{Q}\rho=iQ$. This motivates us to introduce
the parameter $\eta$ for $\rho$ and set 
\[
\gamma_{s}=\eta.
\]
At the nonlinear level, introducing the parameter $\eta$ is responsible
for the \emph{rotational instability} of pseudoconformal blow-up solutions.
This is observed in authors' previous work \cite{KimKwon2019arXiv}.

Indeed, we add one more degree of freedom to $Q_{b}$. We assume $Q_{b}^{(\eta)}=Q^{(\eta)}e^{-ib\frac{y^{2}}{4}}$
is a modified profile satisfying $\partial_{\eta}Q^{(\eta)}=-\rho$.
We again start from the \eqref{eq:CSS} on the renormalized variables
\[
(i\partial_{s}-i\frac{\lambda_{s}}{\lambda}\Lambda-\gamma_{s})Q_{b}^{(\eta)}-L_{Q_{b}^{(\eta)}}^{\ast}\D_{+}^{(Q_{b}^{(\eta)})}Q_{b}^{(\eta)}=0.
\]
From the conjugation properties by the pseudoconformal phase $e^{-ib\frac{r^{2}}{4}}$,
this reduces to 
\[
-i\Big(\frac{\lambda_{s}}{\lambda}+b\Big)\Lambda Q_{b}^{(\eta)}-\gamma_{s}Q_{b}^{(\eta)}+(b_{s}+b^{2})\tfrac{y^{2}}{4}Q_{b}^{(\eta)}-[L_{Q^{(\eta)}}^{\ast}\D_{+}^{(Q^{(\eta)})}Q^{(\eta)}]e^{-ib\frac{r^{2}}{4}}+i\eta_{s}\partial_{\eta}Q^{(\eta)}=0.
\]
As before, we set $\frac{\lambda_{s}}{\lambda}+b=0$. As we assumed
$\partial_{\eta}Q^{(\eta)}=-\rho$, this leads us to set $\gamma_{s}=\eta$.
It is also possible to set $\eta_{s}=0$, by forcing $Q^{(\eta)}(y)$
to be real-valued. Thus we may assume $\eta$ is a nonzero \emph{constant}.
Now we are led to 
\begin{equation}
L_{Q^{(\eta)}}^{\ast}\D_{+}^{(Q^{(\eta)})}Q^{(\eta)}+\eta Q^{(\eta)}-(b_{s}+b^{2})\tfrac{y^{2}}{4}Q^{(\eta)}=0.\label{eq:Qeta-motiv-temp1}
\end{equation}
One of the crucial observations in \cite[Section 4.2]{KimKwon2019arXiv}
is that, in the nonlinear equation \eqref{eq:Qeta-motiv-temp1}, $b_{s}+b^{2}$
has a nontrivial $\eta^{2}$-order term as 
\[
b_{s}+b^{2}=-c\eta^{2},\qquad c=\tfrac{1}{(m+1)^{2}}+o_{\eta\to0}(1).
\]
This was derived by the Pohozaev-type identity, i.e. taking the inner
product of \eqref{eq:Qeta-motiv-temp1} with $\Lambda Q^{(\eta)}$.
Thus our modulation equation becomes 
\[
\frac{\lambda_{s}}{\lambda}+b=0,\quad\gamma_{s}=\eta,\quad b_{s}+b^{2}+c\eta^{2}=0,\quad\eta_{s}=0.
\]

A typical example of the solutions to this system is (when $\eta\neq0$)
\[
\left\{ \begin{aligned}\lambda(t) & =\sqrt{t^{2}+c\eta^{2}},\\
\gamma(t) & =\tfrac{1}{\sqrt{c}}\tan^{-1}(\tfrac{t}{\eta\sqrt{c}}),\\
b(t) & =-t
\end{aligned}
\right.\quad\forall t\in\R
\]
For fixed $\eta\neq0$, it not only says that the phase rotation takes
place by turning on the $\rho$-direction of $N_{g}(i\mathcal{L}_{Q})$,
but the solution \emph{does not} blow up. In fact, it is \emph{global}
and \emph{scattering}. It moreover turns out that the phase rotation
takes place at the fixed amount of angle, which is $\mathrm{sgn}(\eta)(m+1)\pi$.
For $m$-equivariant functions, the phase rotation by $\pm(m+1)\pi$
reflects a spatial rotation by $\pm(\frac{m+1}{m})\pi$. This is the
main mechanism for instability of pseudoconformal blow-up solutions
in \cite{KimKwon2019arXiv}. This is the reason why we can expect
at most the codimension-one blow-up.

Then the remaining question is to construct $Q^{(\eta)}$ of 
\begin{equation}
L_{Q^{(\eta)}}^{\ast}\D_{+}^{(Q^{(\eta)})}Q^{(\eta)}+\eta Q^{(\eta)}+c\eta^{2}\tfrac{y^{2}}{4}Q^{(\eta)}=0\label{eq:Qeta-motiv-temp2}
\end{equation}
such that $Q^{(\eta)}$ deforms from $Q$ with a small parameter $\eta$.
We note that \eqref{eq:Qeta-motiv-temp2} is a \emph{nonlocal} second-order
ODE.
\begin{rem}
\label{rem:TechnicalDifficulties}A standard approach to construct
$Q^{(\eta)}$ solving \eqref{eq:Qeta-motiv-temp2} would be to make
a Taylor expansion $Q^{(\eta)}=Q+\sum_{j}\eta^{j}T_{j}$ with $T_{1}=-\rho$
and solve for $T_{j}$ iteratively, by inverting the linearized operator
$\mathcal{L}_{Q}$. There are several technical difficulties for doing
this. First, in order to invert the operator $\mathcal{L}_{Q}$, we
need to check the \emph{solvability condition}: $\mathcal{L}_{Q}T_{j}=\varphi$
would be solvable when $\varphi$ is orthogonal to the kernel elements
of $\mathcal{L}_{Q}$. To check this solvability condition, one needs
to develop a generalization of the Pohozaev-type identity, for instance
as in \cite{RaphaelRodnianski2012Publ.Math.}. Second, the nonlocal
terms from the gauge potential $A_{0}$ contain $\int_{r}^{\infty}$-integral.
As one usually loses $y^{2}$-decay in the inversion procedure, the
$\int_{r}^{\infty}$-integral may not be defined for higher order
terms. Adding cutoffs in each of the inversion step will complicate
the analysis.
\end{rem}

In \cite[Section 4.3]{KimKwon2019arXiv}, the authors get around the
technical difficulties in Remark \ref{rem:TechnicalDifficulties}
by introducing a simple but remarkable nonlinear ansatz for $Q^{(\eta)}$.
Motivated from the self-duality, or the Bogomol'nyi equation \eqref{eq:Bogomolnyi-Eq},
solving \eqref{eq:Qeta-motiv-temp2} can be reduced to solving a first
order (nonlocal) differential equation. More precisely, if $Q^{(\eta)}$
is a solution to the \emph{modified Bogomol'nyi equation} 
\begin{equation}
\D_{+}^{(Q^{(\eta)})}Q^{(\eta)}=-\eta\tfrac{y}{2}Q^{(\eta)},\label{eq:Qeta-motiv-temp3}
\end{equation}
then it is a solution to 
\begin{gather*}
L_{Q^{(\eta)}}^{\ast}\D_{+}^{(Q^{(\eta)})}Q^{(\eta)}+\eta\theta_{\eta}Q^{(\eta)}+\eta^{2}\tfrac{y^{2}}{4}Q^{(\eta)}=0,\\
\theta_{\eta}=\frac{1}{4\pi}\int|Q^{(\eta)}|^{2}-(m+1)\approx m+1.
\end{gather*}
With this new $Q^{(\eta)}$,\footnote{We naturally redefine $Q^{(\eta)}$ under $\eta\leftrightarrow\eta\theta_{\eta}$.}
the formal parameter ODEs become 
\[
\frac{\lambda_{s}}{\lambda}+b=0,\quad\gamma_{s}-\eta\theta_{\eta}=0,\quad b_{s}+b^{2}+\eta^{2}=0,\quad\eta_{s}=0.
\]
The authors solved the modified Bogomol'nyi equation for $\eta\geq0$
in \cite[Proposition 4.4]{KimKwon2019arXiv}. In the construction
of global decaying $Q^{(\eta)}$, $\eta\ge0$ is necessary. Indeed,
$Q^{(\eta)}\approx e^{-\eta\frac{y^{2}}{4}}Q$ is expected due to
the factor $-\eta\frac{y}{2}Q^{(\eta)}$ of \eqref{eq:Qeta-motiv-temp3}.
When $\eta$ is negative, there is no decaying solution $Q^{(\eta)}$.
Indeed, if $Q^{(\eta)}$ were to decay, then $\frac{1}{y}(m+A_{\theta}[Q^{(\eta)}])$
would be dominated by the factor $-\eta\frac{y}{2}$. This yields
a contradicting fact: exponential growth of $Q^{(\eta)}$ for $y\gtrsim|\eta|^{-\frac{1}{2}}$.

In contrast to the previous work \cite{KimKwon2019arXiv}, in this
paper we will use the parameter $\eta$ as a \emph{modulation parameter}
that can vary in time. Moreover, we need to allow $\eta$ to be either
positive or negative. In particular, we need to construct $Q^{(\eta)}(y)$
even for $\eta<0$. This motivates us to truncate the profile in the
region $y\ll|\eta|^{-\frac{1}{2}}$, which is the regime where the
exponential factor $e^{-\eta\frac{y^{2}}{4}}\lesssim1$.

When we construct pseudoconformal blow-up dynamics, in view of 
\[
\frac{\lambda_{s}}{\lambda}+b=0\quad\text{and}\quad b_{s}+b^{2}+\eta^{2}=0,
\]
we can expect that the pseudoconformal blow-up occurs when $\eta\to0$
and $|\eta|\ll b$, to guarantee $b_{s}+b^{2}\approx0$.

We recall the construction of the modified profile $Q^{(\eta)}$.
This is originally done in \cite[Proposition 4.4]{KimKwon2019arXiv}
with $\eta\geq0$, but here we will \emph{truncate} $Q^{(\eta)}$
for $y\ll|\eta|^{-\frac{1}{2}}$ to allow $\eta<0$.\footnote{In this paper, our $Q^{(\eta)}$ for $\eta\geq0$ is equal to that
in \cite[Proposition 4.4]{KimKwon2019arXiv} on the region $y<\delta\eta^{-\frac{1}{2}}$.} We record basic properties of $Q^{(\eta)}$.
\begin{lem}[Modified profile in the linearization regime]
\label{lem:mod-profile-linear-regime}There exist universal constants
$0<\delta<1$ and $\eta^{\ast}>0$ and unique one-parameter family
$\{Q^{(\eta)}\}_{\eta\in[-\eta^{\ast},\eta^{\ast}]}$ of smooth real-valued
functions on $(0,\delta|\eta|^{-\frac{1}{2}})$ satisfying the following.
\begin{enumerate}
\item (Smoothness on $\R^{2}$) The $m$-equivariant extension $Q^{(\eta)}(x)\coloneqq Q^{(\eta)}(y)e^{im\theta}$,
$x=ye^{i\theta}$, defined on $\{x\in\R^{2}:|x|<\delta|\eta|^{-\frac{1}{2}}\}$
is smooth.
\item (Equation) We have 
\[
\D_{+}^{(Q^{(\eta)})}Q^{(\eta)}=-\eta\tfrac{y}{2}Q^{(\eta)},\qquad\forall y\in(0,\delta|\eta|^{-\frac{1}{2}}).
\]
\item (Uniform bounds) For each $k\in\N$, we have 
\[
|Q^{(\eta)}|_{k}\lesssim_{k}Q.
\]
\item (Differentiability in $\eta$) $Q^{(\eta)}(y)$ for $y\in(0,\delta|\eta|^{-\frac{1}{2}})$
is differentiable with respect to $\eta$. Moreover, the $m$-equivariant
extension of $\partial_{\eta}Q^{(\eta)}$ defined on $\{x\in\R^{2}:|x|<\delta|\eta|^{-\frac{1}{2}}\}$
is smooth. For each $k\in\N$, it also satisfies the pointwise estimates
\begin{align*}
|\partial_{\eta}Q^{(\eta)}+(m+1)\rho|_{k} & \lesssim_{k}|\eta|y^{4}Q,\\
|\partial_{\eta\eta}Q^{(\eta)}|_{k} & \lesssim_{k}y^{4}Q.
\end{align*}
In particular, 
\begin{align*}
|Q^{(\eta)}+\eta(m+1)\rho|_{k} & \lesssim_{k}|\eta|^{2}y^{4}Q,\\
|\partial_{\eta}Q^{(\eta)}|_{k} & \lesssim_{k}y^{2}Q.
\end{align*}
\end{enumerate}
\end{lem}

\begin{proof}
Most of the assertions are proved in \cite{KimKwon2019arXiv}. The
construction of $Q^{(\eta)}$ (for $\eta\geq0$) is already done in
\cite[Proposition 4.4]{KimKwon2019arXiv}, but in the regime $y\ll|\eta|^{-\frac{1}{2}}$,
the sign of $\eta$ is irrelevant. Here, we focus on deriving the
estimates both on $Q^{(\eta)}$ and $\partial_{\eta}Q^{(\eta)}$.
For this purpose, we further look at $\eta$-variated equation. Nevertheless,
the proof will be very similar to \cite[Proposition 4.4]{KimKwon2019arXiv}.
We only sketch the proof and refer \cite{KimKwon2019arXiv} for details.

Formally differentiating the equation in $\eta$, we get 
\begin{align*}
\D_{+}^{(Q^{(\eta)})}Q^{(\eta)} & =-\eta\tfrac{y}{2}Q^{(\eta)},\\
L_{Q^{(\eta)}}\partial_{\eta}Q^{(\eta)} & =-\tfrac{y}{2}Q^{(\eta)}-\eta\tfrac{y}{2}\partial_{\eta}Q^{(\eta)}.
\end{align*}
Substituting $\eta=0$ suggests that $\partial_{\eta=0}Q^{(\eta)}=-(m+1)\rho$.

Motivated from this, we first introduce the unknown $v_{1}$ such
that 
\begin{equation}
Q^{(\eta)}=Q-\eta(m+1)\rho+\eta^{2}Qv_{1},\label{eq:Qeta-def-temp1}
\end{equation}
and write the system of integral equations for $v_{1},v_{2}$. The
equation for $v_{1}$ is derived in \cite[Proposition 4.4]{KimKwon2019arXiv}
and given as 
\begin{align}
v_{1}(y) & ={\textstyle \int_{0}^{y}}(\tfrac{m+1}{2}\cdot\tfrac{y'\rho}{Q}-(m+1)^{2}\tfrac{\rho}{Q}B_{Q}\rho)dy'+{\textstyle \int_{0}^{y}}B_{Q}(Qv_{1})dy'\label{eq:Qeta-def-temp2}\\
 & \quad+\eta{\textstyle \int_{0}^{y}}(-\tfrac{1}{2}y'+(m+1)B_{Q}\rho)v_{1}dy'+\eta^{2}{\textstyle \int_{0}^{y}}(\tfrac{1}{2}B_{Q}(Qv_{1}))v_{1}dy'.\nonumber 
\end{align}
Following the proof of \cite[Proposition 4.4]{KimKwon2019arXiv},
there exist $\delta>0$, $\eta^{\ast}>0$ such that one can construct
$v_{1}$ for $|\eta|\leq\eta^{\ast}$ in the region $y<\delta|\eta|^{-\frac{1}{2}}$.
Note that the sign of $\eta$ is irrelevant. Moreover, one has the
estimate $|v_{1}|\lesssim y^{4}$ uniformly in $\eta$ (and region
$y<\delta|\eta|^{-\frac{1}{2}}$). Iterating the integral equation
\eqref{eq:Qeta-def-temp2} gives $|v_{1}|_{k}\lesssim_{k}y^{4}$ for
any $k\in\N$.

Next, we want to know differentiability in $\eta$. For this purpose,
we formally differentiate \eqref{eq:Qeta-def-temp1} to have 
\[
\partial_{\eta}Q^{(\eta)}=-(m+1)\rho+2\eta Qv_{1}+\eta^{2}Q\partial_{\eta}v_{1}.
\]
Of course, this makes sense if we knew that $v_{1}$ is $\eta$-differentiable.
To ensure that $Q^{(\eta)}$ is differentiable in $\eta$, we construct
$\partial_{\eta}v_{1}$ similarly as $v_{1}$. Now we set $v_{2}=\partial_{\eta}v_{1}$
as \emph{unknown}, and derive the equation for $v_{2}$. Formally
taking $\partial_{\eta}$ to \eqref{eq:Qeta-def-temp2} yields the
$v_{2}$-equation: 
\begin{align}
v_{2} & ={\textstyle \int_{0}^{y}}B_{Q}(Qv_{2})dy'\label{eq:Qeta-def-temp3}\\
 & +\eta{\textstyle \int_{0}^{y}}(-\tfrac{1}{2}y'+(m+1)B_{Q}\rho)v_{2}dy'+{\textstyle \int_{0}^{y}}(-\tfrac{1}{2}y'+(m+1)B_{Q}\rho)v_{1}dy'\nonumber \\
 & +\eta^{2}{\textstyle \int_{0}^{y}}(\tfrac{1}{2}B_{Q}(Qv_{2}))v_{1}dy'+\eta^{2}{\textstyle \int_{0}^{y}}(\tfrac{1}{2}B_{Q}(Qv_{1}))v_{2}dy'+2\eta{\textstyle \int_{0}^{y}}(\tfrac{1}{2}B_{Q}(Qv_{1}))v_{1}dy'.\nonumber 
\end{align}
This is similar to \eqref{eq:Qeta-def-temp2} with some additional
terms. Possibly shrinking $\delta$ and $\eta^{\ast}$, one can construct
$v_{2}$ for $|\eta|\leq\eta^{\ast}$ in the regime $y<\delta|\eta|^{-\frac{1}{2}}$,
by the mimicking the proof of \cite[Proposition 4.4]{KimKwon2019arXiv}.
Moreover, the pointwise estimates $|v_{2}|_{k}\lesssim_{k}y^{6}$
hold uniformly in $\eta$ and in the region $y<\delta|\eta|^{-\frac{1}{2}}$.

To ensure that $Q^{(\eta)}$ is twice differentiable in $\eta$, we
set $v_{3}=\partial_{\eta}v_{2}$ as unknown, and derive the equation
for $v_{3}$ by formally taking $\partial_{\eta}$ to \eqref{eq:Qeta-def-temp3}.
Similarly proceeding as before, the pointwise estimates $|v_{3}|_{k}\lesssim_{k}y^{8}$
hold uniformly in $\eta$ and in the region $y<\delta|\eta|^{-\frac{1}{2}}$.
We omit the details.

Now define $Q^{(\eta)}=Q-\eta(m+1)\rho+\eta^{2}Qv_{1}$. Due to the
derivation of integral equations, $Q^{(\eta)}$ satisfies $\D_{+}^{(Q^{(\eta)})}Q^{(\eta)}=-\eta\frac{y}{2}Q^{(\eta)}$.
By the definition of $v_{2}$, $Q+\int_{0}^{\eta}[-(m+1)\rho+2\eta'Qv_{1}+\eta'Qv_{2}]d\eta'$
also satisfies the same equation and pointwise estimates on $(0,\delta|\eta|^{-\frac{1}{2}})$,
thus by uniqueness it must be equal to $Q^{(\eta)}$. In particular,
$\partial_{\eta}Q^{(\eta)}=-(m+1)\rho+2\eta Qv_{1}+\eta Qv_{2}$.
From the pointwise bound $|v_{1}|_{k}+|v_{2}|_{k}\lesssim_{k}y^{4}$,
the bounds for $Q^{(\eta)}$ and $\partial_{\eta}Q^{(\eta)}$ follow
from those of $v_{1}$, $v_{2}$, and Lemma \ref{lem:Def-rho} (bounds
for $\rho$). A similar argument using $v_{3}$ shows that twice differentaibility
of $Q^{(\eta)}$ the bounds for $\partial_{\eta\eta}Q^{(\eta)}$ follow.

The smoothness of $Q^{(\eta)}(x)$ and $\partial_{\eta}Q^{(\eta)}(x)$
can be shown by writing the operators $\D_{+}^{(Q^{(\eta)})}$ and
$L_{Q^{(\eta)}}$ on the ambient space $\R^{2}$ and utilizing the
ellipticity $\partial_{1}+i\partial_{2}$; see for instance the proof
of Lemma \ref{lem:mapping-L-Appendix}. This completes the proof.
\end{proof}
Let $\delta$ and $\eta^{\ast}$ be fixed universal constants, given
in the above lemma. In the sequel, we will assume $0<b<b^{\ast}$
and $|\eta|<(\frac{\delta}{2})^{2}b$. Here, $b^{\ast}$ is some small
constant to be chosen later, but at this moment, let us only require
$b^{\ast}<\eta^{\ast}$. We define the modified profile $P(y)=P(y;b,\eta)$
by applying the pseudoconformal phase to $Q^{(\eta)}$ and make a
cutoff at $y\lesssim b^{-\frac{1}{2}}$:
\begin{equation}
P(y)\coloneqq P(y;b,\eta)=\chi_{b^{-\frac{1}{2}}}(y)Q^{(\eta)}(y)e^{-ib\frac{y^{2}}{4}}.\label{eq:def-mod-profile}
\end{equation}
We may also define $P(y;0,0)\coloneqq Q(y)$. Here, we recall the
notation that $\chi_{b^{-\frac{1}{2}}}(y)$ is a smooth cutoff function
supported in the region $y\leq2b^{-\frac{1}{2}}$.
\begin{prop}[Properties of the modified profile]
\label{prop:modified-profile}Let $0<b<b^{\ast}$ and $|\eta|<(\frac{\delta}{2})^{2}b$.
The modified profile $P(y)=P(y;b,\eta)$ satisfies the following properties.
\begin{enumerate}
\item (Bounds for $y\leq b^{-\frac{1}{2}}$) For $y\leq b^{-\frac{1}{2}}$,
we have 
\begin{align*}
|P-Q|_{5}+|\Lambda P-\Lambda Q|_{5} & \lesssim by^{2}Q,\\
|\partial_{b}P+i\tfrac{y^{2}}{4}Q|_{5}+|\partial_{\eta}P+(m+1)\rho|_{5} & \lesssim by^{4}Q.
\end{align*}
\item (Global bounds) We have 
\begin{align*}
|P|_{5}+|\Lambda P|_{5} & \lesssim\mathbf{1}_{y\leq2b^{-1/2}}Q,\\
|\partial_{b}P|_{5}+|\partial_{\eta}P|_{5} & \lesssim\mathbf{1}_{y\leq2b^{-1/2}}y^{2}Q,\\
|\partial_{bb}P|_{5}+|\partial_{b\eta}P|_{5}+|\partial_{\eta\eta}P|_{5} & \lesssim\mathbf{1}_{y\leq2b^{-1/2}}y^{4}Q.
\end{align*}
\item (Approximate generalized null space)
\begin{equation}
\||A_{Q}L_{Q}v|_{-1}\|_{L^{2}}+\||A_{Q}L_{Q}v|_{-3}\|_{L^{2}}\lesssim b,\qquad\forall v\in\{\Lambda P,iP,\partial_{b}P,\partial_{\eta}P\}.\label{eq:gen-null-rel}
\end{equation}
\item (Approximate pseudoconformal profile) One can write 
\begin{equation}
L_{P}^{\ast}\D_{+}^{(P)}P-ib\Lambda P+(\eta\theta_{\eta}+\theta_{\Psi})P+(b^{2}+\eta^{2})i\partial_{b}P=\Psi,\label{eq:decomp-P}
\end{equation}
where the scalars $\theta_{\eta}$, $\theta_{\Psi}$, and function
$\Psi(y)=\Psi(y;b,\eta)$ satisfy\footnote{Note that $\theta_{\eta}$ also depends on $b$ (the cutoff distance
$b^{-1/2}$, in fact).} 
\begin{align}
|\theta_{\eta}-(m+1)| & \lesssim|\eta|+b^{m+1},\label{eq:decomp-P-temp1}\\
|\theta_{\Psi}| & \lesssim b^{m+2},\label{eq:decomp-P-temp2}\\
|\Psi|_{5} & \lesssim b^{\frac{m}{2}+2}\mathbf{1}_{b^{-1/2}\leq y\leq2b^{-1/2}}.\label{eq:Psi-Pointwise}
\end{align}
Moreover, we have the charge estimate 
\[
\|P\|_{L^{2}}^{2}=\|Q\|_{L^{2}}^{2}+O(|\eta|+b^{m+1}).
\]
\item (Difference estimate) We have 
\begin{align}
|\partial_{b}\theta_{\eta}| & \lesssim b^{m},\label{eq:db-theta-eta-est}\\
|\partial_{\eta}\theta_{\eta}| & \lesssim1,\label{eq:deta-theta-eta-est}\\
|\partial_{b}\theta_{\Psi}|+|\partial_{\eta}\theta_{\Psi}| & \lesssim b^{m+1},\label{eq:Diff-theta-Psi-est}\\
|\partial_{b}\Psi|_{5}+|\partial_{\eta}\Psi|_{5} & \lesssim b^{\frac{m}{2}+1}\mathbf{1}_{b^{-1/2}\leq y\leq2b^{-1/2}}.\label{eq:Diff-Psi-est}
\end{align}
\item (Size of $\Psi$) We have 
\begin{align}
\|\Psi\|_{\dot{\mathcal{H}}_{m}^{3}} & \lesssim b^{\frac{m}{2}+3},\label{eq:Radiation-Hdot3-est}\\
\|\Psi\|_{\dot{\mathcal{H}}_{m}^{5}} & \lesssim b^{\frac{m}{2}+4},\quad\text{if }m\geq3.\label{eq:Radiation-Hdot5-est}
\end{align}
\end{enumerate}
\end{prop}

\begin{rem}
\label{rem:profile-nonlinearapproach}This \emph{nonlinear approach}
for the construction of modified profile is much simpler than the
linear expansion such as $P=Q-ib\tfrac{y^{2}}{4}Q-(m+1)\eta\rho+\text{(h.o.t)}$.
But it relies on a very special algebraic property, the \emph{self-duality}.
In view of the above discussions, we are led to solve \eqref{eq:Qeta-motiv-temp2}
for $Q^{(\eta)}$. However, it is a nonlocal second-order elliptic
equation. We used the self-duality to reduce it to a first-order nonlocal
equation \eqref{eq:Qeta-motiv-temp3}. This reduction is crucial in
our analysis. We are not sure if it is possible to solve \eqref{eq:Qeta-motiv-temp2}
without this reduction. There are similar nonlinear approaches for
construction of modified profiles in the context of mass-critical
NLS \cite{MerleRaphael2003GAFA,MerleRaphaelSzeftel2013AJM}. There,
the nonlinear terms are local and it is possible to solve it by variational
methods.
\end{rem}

\begin{rem}
There seems to be a slight flexibility for the cutoff radius $b^{-\frac{1}{2}}$.
We expect that one can use $b^{-\frac{1}{2}\pm}$. Here we chose $b^{-\frac{1}{2}}$
because it matches the radius where $|e^{-ib\frac{y^{2}}{4}}|_{k}\sim_{k}1$
and $|e^{-\eta\frac{y^{2}}{4}}|_{k}\lesssim_{k}1$.
\end{rem}

\begin{proof}[Proof of Proposition \ref{prop:modified-profile}]
For simplicity in notations, let us write 
\[
B\coloneqq b^{-\frac{1}{2}}.
\]
Note that $\partial_{b}\chi_{B}(y)=\frac{1}{2}b^{-\frac{1}{2}}y[\partial_{y}\chi](b^{\frac{1}{2}}y)$
so $|\partial_{b}\chi_{B}|_{4}\lesssim y^{2}\mathbf{1}_{B\leq y\leq2B}\sim b^{-1}\mathbf{1}_{B\leq y\leq2B}$.

(1) Since $\Lambda=y\partial_{y}+1$ and the cutoff function $\chi_{B}$
can be ignored for $y\leq B$, we have 
\begin{align*}
|P-Q|_{5}+|\Lambda P-\Lambda Q|_{5} & \lesssim|e^{-ib\frac{y^{2}}{4}}Q^{(\eta)}-Q|_{6}\\
 & \lesssim|e^{-ib\frac{y^{2}}{4}}-1|_{6}|Q^{(\eta)}|_{6}+|Q^{(\eta)}-Q|_{6}\lesssim by^{2}Q.
\end{align*}
Similarly, 
\begin{align*}
 & |\partial_{b}P+i\tfrac{y^{2}}{4}Q|_{5}+|\partial_{\eta}P+(m+1)\rho|_{5}\\
 & =|-i\tfrac{y^{2}}{4}e^{-ib\frac{y^{2}}{4}}Q^{(\eta)}+i\tfrac{y^{2}}{4}Q|_{5}+|e^{-ib\frac{y^{2}}{4}}\partial_{\eta}Q^{(\eta)}+(m+1)\rho|_{5}\\
 & \lesssim|e^{-ib\frac{y^{2}}{4}}-1|_{5}(|y^{2}Q^{(\eta)}|_{5}+|\partial_{\eta}Q^{(\eta)}|_{5})+|y^{2}(Q^{(\eta)}-Q)|_{5}+|\partial_{\eta}Q^{(\eta)}+(m+1)\rho|_{5}\\
 & \lesssim by^{4}Q.
\end{align*}

(2) This easily follows from the arguments in (1), uniform bounds
on $Q^{(\eta)}$ in Lemma \ref{lem:mod-profile-linear-regime}, $|\partial_{b}\chi_{B}|_{6}\lesssim y^{2}\mathbf{1}_{B\leq y\leq2B}$,
and $|\partial_{bb}\chi_{B}|_{6}\lesssim y^{4}\mathbf{1}_{B\leq y\leq2B}$.

(3) Using $L_{Q}\Lambda Q=0$, boundedness of $A_{Q}L_{Q}:\dot{\mathcal{H}}_{m}^{3}\to\dot{H}_{m+2}^{1}$
(Lemmas \ref{lem:mapping-AQ} and \ref{lem:mapping-LQ-H3}), (1) for
$y\leq B$, and (2) for $y>B$, we have 
\begin{align*}
\||A_{Q}L_{Q}\Lambda P|_{-1}\|_{L^{2}} & =\||A_{Q}L_{Q}(\Lambda P-\Lambda Q)|_{-1}\|_{L^{2}}\\
 & \lesssim\||\Lambda P-\Lambda Q|_{-3}\|_{L^{2}}\\
 & \lesssim\|\mathbf{1}_{y\leq B}by^{-1}Q+\mathbf{1}_{y>B}y^{-3}Q\|_{L^{2}}\lesssim b.
\end{align*}
A similar argument can be done for $iP$. Next, using $A_{Q}L_{Q}\rho=0$,
and similarly as before, we have 
\begin{align*}
\||A_{Q}L_{Q}\partial_{\eta}P|_{-1}\|_{L^{2}} & =\||A_{Q}L_{Q}(\partial_{\eta}P+(m+1)\rho)|_{-1}\|_{L^{2}}\\
 & \lesssim\||\partial_{\eta}P+(m+1)\rho|_{-3}\|_{L^{2}}\\
 & \lesssim\|\mathbf{1}_{y\leq B}byQ+\mathbf{1}_{y>B}y^{-1}Q\|_{L^{2}}\lesssim b.
\end{align*}
A similar argument can be done for $\partial_{b}P$. Estimates of
$\||A_{Q}L_{Q}v|_{-3}\|_{L^{2}}$ can be shown in a similar way using
boundedness of $A_{Q}L_{Q}:\dot{\mathcal{H}}_{m}^{5}\to\dot{H}_{m+2}^{3}$
(Lemmas \ref{lem:mapping-AQ-1} and \ref{lem:mapping-LQ-H5}).

(4, 5) By adding the cutoff $\chi_{B}$ to $Q^{(\eta)}$, we have
\begin{align*}
 & \D_{+}^{(\chi_{B}Q^{(\eta)})}[\chi_{B}Q^{(\eta)}]=-\eta\tfrac{y}{2}\chi_{B}Q^{(\eta)}+\phi Q^{(\eta)},\\
\phi & \coloneqq\partial_{y}\chi_{B}+\tfrac{1}{y}(A_{\theta}[Q^{(\eta)}]-A_{\theta}[\chi_{B}Q^{(\eta)}])\chi_{B}.
\end{align*}
We now take $L_{\chi_{B}Q^{(\eta)}}^{\ast}$. First we recall the
computation from \cite{KimKwon2019arXiv} 
\[
L_{\chi_{B}Q^{(\eta)}}^{\ast}(-\tfrac{y}{2}\chi_{B}Q^{(\eta)})=\tfrac{y}{2}\D_{+}^{(\chi_{B}Q^{(\eta)})}\chi_{B}Q^{(\eta)}+(m+1-\tfrac{1}{4\pi}{\textstyle \int}|\chi_{B}Q^{(\eta)}|^{2})\chi_{B}Q^{(\eta)}.
\]
Next we write 
\begin{align*}
L_{\chi_{B}Q^{(\eta)}}^{\ast}[\phi Q^{(\eta)}] & =\D_{+}^{(\chi_{B}Q^{(\eta)})\ast}[\phi Q^{(\eta)}]\\
 & \quad+({\textstyle \int_{0}^{\infty}}\chi_{B}\phi|Q^{(\eta)}|^{2}dy')\chi_{B}Q^{(\eta)}-({\textstyle \int_{0}^{y}}\chi_{B}\phi|Q^{(\eta)}|^{2}dy')\chi_{B}Q^{(\eta)}.
\end{align*}
Summing the above two displays, we have 
\[
L_{\chi_{B}Q^{(\eta)}}^{\ast}\D_{+}^{(\chi_{B}Q^{(\eta)})}\chi_{B}Q^{(\eta)}+(\eta\theta_{\eta}+\theta_{\Psi})\chi_{B}Q^{(\eta)}+\eta^{2}\tfrac{y^{2}}{4}\chi_{B}Q^{(\eta)}=\tilde{\Psi},
\]
where
\begin{align*}
\theta_{\eta} & =\tfrac{1}{4\pi}{\textstyle \int}|\chi_{B}Q^{(\eta)}|^{2}-(m+1),\\
\theta_{\Psi} & =-{\textstyle \int_{0}^{\infty}}\chi_{B}\phi|Q^{(\eta)}|^{2}dy,\\
\tilde{\Psi} & =(\eta\tfrac{y}{2}+\D_{+}^{(\chi_{B}Q^{(\eta)})\ast})\phi Q^{(\eta)}-({\textstyle \int_{0}^{y}}\chi_{B}\phi|Q^{(\eta)}|^{2}dy')\chi_{B}Q^{(\eta)}.
\end{align*}
Now we conjugate the pseudoconformal phase to get 
\[
L_{P}^{\ast}\D_{+}^{(P)}P-ib\Lambda P+(\eta\theta_{\eta}+\theta_{\Psi})P+(b^{2}+\eta^{2})\tfrac{y^{2}}{4}P=\tilde{\Psi}e^{-ib\frac{y^{2}}{4}}.
\]
The proof of \eqref{eq:decomp-P} follows by setting 
\begin{align*}
\Psi & =\tilde{\Psi}e^{-ib\frac{y^{2}}{4}}+(b^{2}+\eta^{2})(i\partial_{b}P-\tfrac{y^{2}}{4}P)\\
 & =[\tilde{\Psi}+i(b^{2}+\eta^{2})(\partial_{b}\chi_{B})Q^{(\eta)}]e^{-ib\frac{y^{2}}{4}}.
\end{align*}

We turn to show $\theta_{\eta}$ estimates. The estimate \eqref{eq:decomp-P-temp1}
follows from 
\begin{align*}
 & \big|{\textstyle \int}|\chi_{B}Q^{(\eta)}|^{2}-{\textstyle \int}\mathbf{1}_{y\leq B}|Q^{(\eta)}|^{2}\big|\leq{\textstyle \int}\mathbf{1}_{B\leq y\leq2B}|Q^{(\eta)}|^{2}\lesssim b^{m+1},\\
 & \big|{\textstyle \int}\mathbf{1}_{y\leq B}|Q^{(\eta)}|^{2}-{\textstyle \int}\mathbf{1}_{y\leq B}Q^{2}\big|\lesssim{\textstyle \int}\mathbf{1}_{y\leq B}Q|Q^{(\eta)}-Q|\lesssim{\textstyle \int}\mathbf{1}_{y\leq B}|\eta|y^{2}Q^{2}\lesssim|\eta|,\\
 & {\textstyle \int}\mathbf{1}_{y\leq B}Q^{2}={\textstyle \int}Q^{2}-{\textstyle \int}\mathbf{1}_{y>B}Q^{2}=8\pi(m+1)-O(b^{m+1}).
\end{align*}
The above display also shows the charge estimate. Next, \eqref{eq:db-theta-eta-est}
follows from 
\begin{align*}
|\partial_{b}\theta_{\eta}| & \lesssim{\textstyle \int}|\partial_{b}\chi_{B}|Q^{2}\lesssim b^{m},\\
|\partial_{\eta}\theta_{\eta}| & \lesssim{\textstyle \int}\chi_{B}Q\cdot y^{2}Q\lesssim1.
\end{align*}

We turn to show $\theta_{\Psi}$ estimates. We note that $\phi$ is
supported in the region $B\leq y\leq2B$ and satisfies 
\[
|\phi|_{6}\lesssim y^{-1}\mathbf{1}_{B\leq y\leq2B}\quad\text{and}\quad|\partial_{b}\phi|_{6}\lesssim y\mathbf{1}_{B\leq y\leq2B}.
\]
Thus we have \eqref{eq:decomp-P-temp2} and \eqref{eq:Diff-theta-Psi-est}:
\begin{align*}
|\theta_{\Psi}| & \lesssim{\textstyle \int_{0}^{\infty}}\mathbf{1}_{y\sim B}y^{-1}Q^{2}dy\lesssim b^{m+2},\\
|\partial_{b}\theta_{\Psi}|+|\partial_{\eta}\theta_{\Psi}| & \lesssim{\textstyle \int_{0}^{\infty}}\mathbf{1}_{y\sim B}yQ^{2}dy\lesssim b^{m+1}.
\end{align*}

Next, we show the $\Psi$ estimates. We note that $\tilde{\Psi}$
is supported in the region $B\leq y\leq2B$. We also note that 
\[
|A_{\theta}[\chi_{B}Q^{(\eta)}]|_{5}+|\partial_{b}A_{\theta}[\chi_{B}Q^{(\eta)}]|_{5}+|\partial_{\eta}A_{\theta}[\chi_{B}Q^{(\eta)}]|_{5}\lesssim1,
\]
thus 
\[
|\D_{+}^{(\chi_{B}Q^{(\eta)})}f|_{5}+|\partial_{b}\D_{+}^{(\chi_{B}Q^{(\eta)})}f|_{5}+|\partial_{\eta}\D_{+}^{(\chi_{B}Q^{(\eta)})}f|_{5}\lesssim y^{-1}|f|_{5}.
\]
Using $|\eta|\leq b$ and the above $|\phi|_{6}$-estimates, we have
\begin{align*}
|\tilde{\Psi}|_{5} & \lesssim\mathbf{1}_{B\leq y\leq2B}y^{-2}Q\lesssim b^{\frac{m}{2}+2}\mathbf{1}_{B\leq y\leq2B},\\
|\partial_{b}\tilde{\Psi}|_{5}+|\partial_{\eta}\tilde{\Psi}|_{5} & \lesssim\mathbf{1}_{B\leq y\leq2B}Q\lesssim b^{\frac{m}{2}+1}\mathbf{1}_{B\leq y\leq2B}.
\end{align*}
On the other hand, we have 
\begin{align*}
|(\partial_{b}\chi_{B})Q^{(\eta)}|_{5} & \lesssim b^{-1}\mathbf{1}_{B\leq y\leq2B}Q\lesssim b^{\frac{m}{2}}\mathbf{1}_{B\leq y\leq2B},\\
|(\partial_{bb}\chi_{B})Q^{(\eta)}|_{5}+|(\partial_{b}\chi_{B})\partial_{\eta}Q^{(\eta)}|_{5} & \lesssim b^{-1}\mathbf{1}_{B\leq y\leq2B}y^{2}Q\lesssim b^{\frac{m}{2}-1}\mathbf{1}_{B\leq y\leq2B}.
\end{align*}
Adding the pseudoconformal phase completes the proof of \eqref{eq:Psi-Pointwise}
and \eqref{eq:Diff-Psi-est}.

(6) By the definitions \eqref{eq:Def-Hdot3-Section2} and \eqref{eq:Def-Hdot5-Section2},
and pointwise estimates \eqref{eq:Psi-Pointwise}, we have 
\begin{align*}
\|\Psi\|_{\dot{\mathcal{H}}_{m}^{3}} & \lesssim\||\Psi|_{-3}\|_{L^{2}}\lesssim b^{\frac{m}{2}+3},\\
\|\Psi\|_{\dot{\mathcal{H}}_{m}^{5}} & \lesssim\||\Psi|_{-5}\|_{L^{2}}\lesssim b^{\frac{m}{2}+4},\qquad\text{if }m\geq3.
\end{align*}
\end{proof}

\section{\label{sec:TrappedSolutions}Trapped solutions}

So far, we have constructed the modified profile $P$, which formally
suggested the pseudoconformal blow-up. In this section, we present
how we decompose our solution $u$ into the blow-up part $P$ and
the error $\epsilon$. The dynamics of the blow-up part $P$ will
be governed by four modulation parameters $\lambda,\gamma,b,\eta$.
As illustrated above, we will fix the modulation parameters through
orthogonality conditions on $\epsilon$. The choice of orthogonality
conditions is motivated to almost decouple the dynamics of $P$ and
$\epsilon$. It will be shown that we can decompose our solution $u$
as long as $u$ stays in a soliton tube $Q$ \emph{with} $|\eta|\ll b$.
This is done in Lemma \ref{lem:decomp}.

Having fixed the decomposition, we discuss how to reduce the existence
part of our main theorems into Proposition \ref{prop:main-bootstrap}
(main bootstrap) and Proposition \ref{prop:Sets-I-pm} (existence
of special $\eta_{0}$). Here we motivate the notion of \emph{trapped
solutions} defined in Introduction. In fact the conditions of trapped
solutions are exactly the bootstrap hypotheses.

\subsection{\label{subsec:decomposition}Decomposition of solutions}

We will decompose our solution of the form 
\begin{equation}
u(t,r)=\frac{e^{i\gamma(t)}}{\lambda(t)}[P(\cdot;b(t),\eta(t))+\epsilon(t,\cdot)]\Big(\frac{r}{\lambda(t)}\Big).\label{eq:u-decomp}
\end{equation}
For $P$ and $\epsilon$, we denote the rescaled spatial variable
by $y=\frac{r}{\lambda(t)}$. To get an idea of this decomposition,
let us recall the discussion in Section \ref{subsec:Linearization-of-CSS}.
The linearized evolution has two invariant subspaces $N_{g}(i\mathcal{L}_{Q})$
and $N_{g}(\mathcal{L}_{Q}i)^{\perp}$, and the whole space decomposes
as $N_{g}(i\mathcal{L}_{Q})\oplus N_{g}(\mathcal{L}_{Q}i)^{\perp}$.
Note that this decomposition is not orthogonal, but is possible because
\eqref{eq:ortho-jacobian} is nonsingular. This fact will be a key
to our decomposition \eqref{eq:u-decomp} (Lemma \ref{lem:decomp}).
Roughly speaking, $\frac{e^{i\gamma}}{\lambda}P(\frac{\cdot}{\lambda};b,\eta)$
corresponds to the $N_{g}(i\mathcal{L}_{Q})$ part and $\epsilon$
corresponds to $N_{g}(\mathcal{L}_{Q}i)^{\perp}$ part of the solution.

For the $N_{g}(i\mathcal{L}_{Q})$ part, we introduce four modulation
parameters $\lambda(t)$, $\gamma(t)$, $b(t)$, and $\eta(t)$ corresponding
to the generalized null space $N_{g}(i\mathcal{L}_{Q})=\mathrm{span}_{\R}\{\Lambda Q,iQ,iy^{2}Q,\rho\}$.
The scaling $\lambda(t)$ and phase rotation $\gamma(t)$ correspond
to the kernel elements $\Lambda Q$ and $iQ$. As \eqref{eq:CSS}
has explicit scaling and phase rotation symmetries, we use $\lambda(t)$
and $\gamma(t)$ to renormalize our solution. More precisely, this
renormalization enables us to transform the blow-up dynamics to a
near-$Q$ dynamics. The pseudoconformal phase $b(t)$ and an additional
parameter $\eta(t)$ correspond to the generalized null space elements
$iy^{2}Q$ and $\rho$. These two parameters are used in the modified
profile in the previous section.\footnote{We note that \eqref{eq:CSS} has explicit pseudoconformal symmetry
(in continuous form), but we will not apply it directly to our solution
since it will require us to work with the weighted Sobolev spaces.
Instead, we apply the pseudoconformal phase to our modified profile.
The parameter $\eta(t)$ does not seem to correspond to an explicit
symmetry.} Due to the generalized null space relations $i\mathcal{L}_{Q}iy^{2}Q=4\Lambda Q$
and $i\mathcal{L}_{Q}\rho=iQ$, turning on the parameters $b(t)$
and $\eta(t)$ change the scaling $\lambda(t)$ and phase rotation
$\gamma(t)$. It is observed in \cite{KimKwon2019arXiv} that the
$\eta$-parameter accounts for the instability of pseudoconformal
blow-up. In \cite{KimKwon2019arXiv} we fixed $\eta$ as a small constant
in time, but in this paper we allow $\eta(t)$ to vary in time, to
take advantages of additional degree of freedom for choosing the modulation
parameters. This extra degree of freedom will sharpen our modulation
analysis.

For the $N_{g}(\mathcal{L}_{Q}i)^{\perp}$ part, we impose four orthogonality
conditions on $\epsilon$. Since we will carry out energy estimates
for $\epsilon$ in $\dot{\mathcal{H}}_{m}^{3}$ and $\dot{\mathcal{H}}_{m}^{5}$
levels, we will regard $\epsilon$ as an element of the adapted function
space $\dot{\mathcal{H}}_{m}^{3}$ and $\dot{\mathcal{H}}_{m}^{5}$,
respectively (see Section \ref{subsec:Adapted-function-spaces}).
However, if $m$ is small, $N_{g}(\mathcal{L}_{Q}i)^{\perp}=\{i\Lambda Q,Q,y^{2}Q,i\rho\}^{\perp}$
is not well-defined in the adapted function spaces, due to lack of
decay $|y^{2}Q|,|\rho|\sim y^{-m}$ for $y\gg1$. Thus we will use
localized versions of these orthogonality conditions defined as follows.

Define the codimension-four linear subspace $\mathcal{Z}^{\perp}\subset H_{m}^{3}$
by 
\begin{equation}
\mathcal{Z}^{\perp}\coloneqq\{\epsilon\in H_{m}^{3}:(\epsilon,\mathcal{Z}_{k})_{r}=0\text{ for all }k\in\{1,2,3,4\}\},\label{eq:ortho-cond}
\end{equation}
where 
\begin{equation}
\left\{ \begin{aligned}\mathcal{Z}_{1} & \coloneqq y^{2}Q\chi_{M}-\frac{(\rho,y^{2}Q\chi_{M})_{r}}{(Q,\rho\chi_{M})_{r}}\mathcal{L}_{Q}(\rho\chi_{M}),\\
\mathcal{Z}_{2} & \coloneqq i\rho\chi_{M}-\frac{(y^{2}Q,\rho\chi_{M})_{r}}{4(\Lambda Q,y^{2}Q\chi_{M})_{r}}\mathcal{L}_{Q}(iy^{2}Q\chi_{M}),\\
\mathcal{Z}_{3} & \coloneqq\mathcal{L}_{Q}i\mathcal{Z}_{1},\\
\mathcal{Z}_{4} & \coloneqq\mathcal{L}_{Q}i\mathcal{Z}_{2}.
\end{aligned}
\right.\label{eq:def-Zk}
\end{equation}
Thus in the decomposition \eqref{eq:u-decomp}, we will require $\epsilon\in\mathcal{Z}^{\perp}$.
For $H_{m}^{5}$ solutions, we will require $\epsilon\in\mathcal{Z}^{\perp}\cap H_{m}^{5}$.
Note that $\mathcal{Z}_{k}$'s deform from the basis elements of the
generalized null space $N_{g}(\mathcal{L}_{Q}i)$ by some large cutoff
parameter $M$. The additional factors in $\mathcal{Z}_{1}$ and $\mathcal{Z}_{2}$
help us to see transversality of $\mathcal{Z}^{\perp}$ and $N_{g}(i\mathcal{L}_{Q})$.
More precisely, this results in that the inner product matrix of $iQ,\Lambda Q,\rho,iy^{2}Q$
(which correspond to tangent vectors to the manifold of our modified
profiles) and $\mathcal{Z}_{1},\mathcal{Z}_{2},\mathcal{Z}_{3},\mathcal{Z}_{4}$
becomes diagonal; see \eqref{eq:Transversality}. In other words,
we can impose the orthogonality conditions \eqref{eq:ortho-cond}.
\begin{lem}[Estimates of $\mathcal{Z}_{k}$]
We have the following estimates of $\mathcal{Z}_{k}$:
\begin{enumerate}
\item (Pointwise estimates) We have 
\begin{equation}
\left.\begin{aligned}|\mathcal{Z}_{1}|+|\mathcal{Z}_{2}| & \lesssim\mathbf{1}_{y\leq2M}(y^{2}+\log M)Q+\mathbf{1}_{y\geq2M}(\log M)y^{-2}Q,\\
|\mathcal{Z}_{3}|+|\mathcal{Z}_{4}| & \lesssim\mathbf{1}_{y\leq2M}(1+\langle y\rangle^{-2}\log M)Q+\mathbf{1}_{y\geq2M}y^{-2}Q,
\end{aligned}
\right.\label{eq:Orthog-Pointwise}
\end{equation}
\item (Sobolev estimates) There exists a universal constant $C$ such that
\begin{align}
\sup_{k\in\{1,2,3,4\}}\|\mathcal{Z}_{k}\|_{L^{2}} & \lesssim\log M,\label{eq:Orthog-L2-est}\\
\sup_{k\in\{1,2,3,4\}}\|\mathcal{Z}_{k}\|_{(\dot{\mathcal{H}}_{m}^{3})^{\ast}} & \lesssim M^{C},\label{eq:Orthog-Hdot3-est}\\
\sup_{k\in\{1,2,3,4\}}\|\mathcal{Z}_{k}\|_{(\dot{\mathcal{H}}_{m}^{5})^{\ast}} & \lesssim M^{C},\qquad\text{if }m\geq3.\label{eq:Orthog-Hdot5-est}
\end{align}
\item (Approximate null space relations) We have 
\begin{align}
\sup_{k\in\{1,2\}}\|yA_{Q}L_{Q}i\mathcal{Z}_{k}\|_{L^{2}} & \lesssim M^{-1},\label{eq:Orthogonality-approx-gen-null-H3}\\
\sup_{k\in\{1,2\}}\|y^{3}A_{Q}L_{Q}i\mathcal{Z}_{k}\|_{L^{2}} & \lesssim M^{-1},\qquad\text{if }m\geq3.\label{eq:Orthogonality-approx-gen-null-H5}
\end{align}
\item (Transversality) We have 
\begin{equation}
\left.\begin{aligned}(\Lambda Q,\mathcal{Z}_{k})_{r} & =(-\|yQ\|_{L^{2}}^{2}+o_{M\to\infty}(1))\delta_{1k},\\
(iQ,\mathcal{Z}_{k})_{r} & =(\|L_{Q}\rho\|_{L^{2}}^{2}+o_{M\to\infty}(1))\delta_{2k},\\
(iy^{2}Q,\mathcal{Z}_{k})_{r} & =(4\|yQ\|_{L^{2}}^{2}+o_{M\to\infty}(1))\delta_{3k},\\
(\rho,\mathcal{Z}_{k})_{r} & =(-\|L_{Q}\rho\|_{L^{2}}^{2}+o_{M\to\infty}(1))\delta_{4k},
\end{aligned}
\right.\label{eq:Transversality}
\end{equation}
where $\delta_{jk}$ is the Kronecker-delta symbol.
\end{enumerate}
\end{lem}

\begin{proof}
(1) We first show pointwise estimates on $\mathcal{Z}_{k}$. We note
\begin{align*}
|\rho\chi_{M}|_{4}+|y^{2}Q\chi_{M}|_{4} & \lesssim\mathbf{1}_{y\leq2M}y^{2}Q,\\
|L_{Q}(\rho\chi_{M})|_{3}+|L_{Q}(iy^{2}Q\chi_{M})|_{3} & \lesssim\mathbf{1}_{y\leq2M}yQ+\mathbf{1}_{y\geq2M}y^{-1}Q,\\
|\mathcal{L}_{Q}(\rho\chi_{M})|_{2}+|\mathcal{L}_{Q}(iy^{2}Q\chi_{M})|_{2} & \lesssim\mathbf{1}_{y\leq2M}Q+\mathbf{1}_{y\geq2M}y^{-2}Q.
\end{align*}
Using $|(y^{2}Q,\rho\chi_{M})_{r}|=|(\rho,y^{2}Q\chi_{M})_{r}|\lesssim\log M$,
we get the first estimate of \eqref{eq:Orthog-Pointwise}. Next, we
take $L_{Q}i$ and use $L_{Q}iL_{Q}^{\ast}=iA_{Q}^{\ast}A_{Q}$ to
have\footnote{The gain at the origin relies on the following two facts: $\rho\chi_{M}$
and $y^{2}Q\chi_{M}$ are \emph{smooth} $m$-equivariant function
on $\R^{2}$, and $L_{Q}i\mathcal{L}_{Q}(\rho\chi_{M})$ and $L_{Q}i\mathcal{L}_{Q}(iy^{2}Q\chi_{M})$
must be smooth $(m+1)$-equivariant function.} 
\begin{align*}
|L_{Q}i\mathcal{L}_{Q}(\rho\chi_{M})|_{1}+|L_{Q}i\mathcal{L}_{Q}(iy^{2}Q\chi_{M})|_{1} & \lesssim y(|L_{Q}(\rho\chi_{M})|_{-3}+|L_{Q}(iy^{2}Q\chi_{M})|_{-3})\\
 & \lesssim\mathbf{1}_{y\leq2M}y\langle y\rangle^{-2}Q+\mathbf{1}_{y\geq2M}y^{-3}Q,\\
|\mathcal{L}_{Q}i\mathcal{L}_{Q}(\rho\chi_{M})|+|\mathcal{L}_{Q}i\mathcal{L}_{Q}(iy^{2}Q\chi_{M})| & \lesssim\mathbf{1}_{y\leq2M}\langle y\rangle^{-2}Q+\mathbf{1}_{y\geq2M}y^{-4}Q.
\end{align*}
Combining the above estimates completes the proof of \eqref{eq:Orthog-Pointwise}.

(2) The Sobolev estimates on $\mathcal{Z}_{k}$ follow from applying
the pointwise estimates to
\[
\|\mathcal{Z}_{k}\|_{(\dot{\mathcal{H}}_{m}^{\ell})^{\ast}}\lesssim\|\langle y\rangle^{\ell}\mathcal{Z}_{k}\|_{L^{2}},\qquad\forall\ell\in\{0,3,5\}.
\]

(3) We turn to \eqref{eq:Orthogonality-approx-gen-null-H3} and \eqref{eq:Orthogonality-approx-gen-null-H5}.
From $A_{Q}L_{Q}\rho=0$, we have 
\begin{align*}
|A_{Q}L_{Q}(\rho\chi_{M})| & =|A_{Q}L_{Q}(\rho-\rho\chi_{M})|\lesssim|L_{Q}(\rho-\rho\chi_{M})|_{-1}\lesssim\mathbf{1}_{y\geq M}Q,\\
|A_{Q}L_{Q}(iy^{2}Q\chi_{M})| & \lesssim|y^{2}Q-y^{2}Q\chi_{M}|_{-2}\lesssim\mathbf{1}_{y\geq M}Q.
\end{align*}
Further utilizing $L_{Q}iL_{Q}^{\ast}=iA_{Q}^{\ast}A_{Q}$, we have
\begin{align*}
|A_{Q}L_{Q}i\mathcal{L}_{Q}(\rho\chi_{M})| & \lesssim|L_{Q}(\rho-\rho\chi_{M})|_{-3}\lesssim\mathbf{1}_{y\geq M}y^{-2}Q,\\
|A_{Q}L_{Q}i\mathcal{L}_{Q}(iy^{2}Q-iy^{2}Q\chi_{M})| & \lesssim|y^{2}Q-y^{2}Q\chi_{M}|_{-4}\lesssim\mathbf{1}_{y\geq M}y^{-2}Q.
\end{align*}
Thus 
\[
\sup_{k\in\{1,2\}}|A_{Q}L_{Q}i\mathcal{Z}_{k}|\lesssim\mathbf{1}_{y\geq M}Q.
\]
This pointwise estimate immediately implies \eqref{eq:Orthogonality-approx-gen-null-H3}
and \eqref{eq:Orthogonality-approx-gen-null-H5}.

(4) Finally, we verify the transversality condition. First, 
\begin{align*}
(\Lambda Q,\mathcal{Z}_{k})_{r} & =(\Lambda Q,y^{2}Q\chi_{M})_{r}\delta_{1k},\\
(iQ,\mathcal{Z}_{k})_{r} & =(Q,\rho\chi_{M})_{r}\delta_{2k},
\end{align*}
because $\mathcal{Z}_{1},\mathcal{Z}_{4}$ are real, $\mathcal{Z}_{2},\mathcal{Z}_{3}$
are imaginary, and $\mathcal{L}_{Q}\Lambda Q=\mathcal{L}_{Q}iQ=0$.
Next, due to the additional terms in the definition of $\mathcal{Z}_{1}$
and $\mathcal{Z}_{2}$, and the algebras $\mathcal{L}_{Q}iy^{2}Q=-4i\Lambda Q$
and $\mathcal{L}_{Q}\rho=Q$, we have 
\begin{align*}
(iy^{2}Q,\mathcal{Z}_{k})_{r} & =(-4\Lambda Q,y^{2}Q\chi_{M})_{r}\delta_{3k},\\
(\rho,\mathcal{Z}_{k})_{r} & =(Q,-\rho\chi_{M})_{r}\delta_{4k}.
\end{align*}
Now the computations 
\begin{align*}
(\Lambda Q,y^{2}Q\chi_{M})_{r} & =(\Lambda Q,y^{2}Q)_{r}+O(M^{-2})=-\|yQ\|_{L^{2}}^{2}+O(M^{-2}),\\
(Q,\rho\chi_{M})_{r} & =(Q,\rho)_{r}+O(M^{-2})=\|L_{Q}\rho\|_{L^{2}}^{2}+O(M^{-2}),
\end{align*}
complete the proof.
\end{proof}
In the sequel, we will perform a bootstrap analysis. In the decomposition
\eqref{eq:u-decomp}, we will assume $|\eta|\ll b$ and $\epsilon$
stay small to guarantee the pseudoconformal blow-up. A quantitative
version of this smallness will be given as our bootstrap assumptions
and in fact used in the definition of trapped solutions.

Having fixed $\mathcal{Z}^{\perp}$ \eqref{eq:ortho-cond}, let $\mathcal{U}$
be a subset of $\R_{+}\times\R/2\pi\Z\times\R\times\R\times\mathcal{Z}^{\perp}$
consisting of $(\lambda,\gamma,b,\eta,\epsilon)$ satisfying 
\begin{equation}
b\in(0,b^{\ast}),\ |\eta|<Kb^{\frac{3}{2}},\ \|\epsilon\|_{L^{2}}<K(b^{\ast})^{\frac{1}{4}},\ \|\epsilon_{1}\|_{L^{2}}<Kb,\ \|\epsilon_{3}\|_{L^{2}}<Kb^{\frac{5}{2}}.\label{eq:bootstrap-hyp-H3}
\end{equation}
\eqref{eq:bootstrap-hyp-H3} will be used as a bootstrap hypothesis.
The largeness of $K$ and smallness of $b^{\ast}$ will be chosen
in the proof. See also Remark \ref{rem:ParameterDependence} for the
parameter dependence. We now define $\mathcal{O}\subseteq H_{m}^{3}$
by the set of images 
\begin{equation}
u(r)=\frac{e^{i\gamma}}{\lambda}[P(\cdot;b,\eta)+\epsilon]\Big(\frac{r}{\lambda}\Big)\in H_{m}^{3}.\label{eq:u-formula}
\end{equation}
By the definition \eqref{eq:Def-O-init}, $\mathcal{O}$ contains
$\mathcal{O}_{init}$. Moreover, by the definition \eqref{eq:def-H3-trapped},
the $H_{m}^{3}$-trapped solutions lies in $\mathcal{O}$ in their
forward lifespan. Next, we define $\overline{\mathcal{U}}$ and $\overline{\mathcal{O}}$
by the closure of $\mathcal{U}$ and $\mathcal{O}$, respectively.
Note that the case $b=\eta=0$ belongs to $\overline{\mathcal{U}}$,
and we will use $P(\cdot;0,0)=Q$. As Lemma \ref{lem:decomp} shows,
the set $\mathcal{O}$ is open and $(\lambda,\gamma,b,\eta,\epsilon)$
can be considered as coordinates of $u$.

Similarly, we can describe $H_{m}^{5}$ bootstrap assumptions when
$m\geq3$. Let $\mathcal{U}^{5}$ be a subset of $\R_{+}\times\R/2\pi\Z\times\R\times\R\times(\mathcal{Z}^{\perp}\cap H_{m}^{5})$
consisting of $(\lambda,\gamma,b,\eta,\epsilon)$ satisfying
\begin{equation}
b\in(0,b^{\ast}),\ |\eta|<Kb^{\frac{3}{2}},\ \|\epsilon\|_{L^{2}}<K(b^{\ast})^{\frac{1}{4}},\ \|\epsilon_{1}\|_{L^{2}}<Kb,\ \|\epsilon_{3}\|_{L^{2}}<Kb^{3},\ \|\epsilon_{5}\|_{L^{2}}<Kb^{\frac{9}{2}}.\label{eq:bootstrap-hyp-H5}
\end{equation}
Define $\mathcal{O}^{5}\subseteq H_{m}^{5}$ by the set of images
according to the formula \eqref{eq:u-formula}. Thus $H_{m}^{5}$-trapped
solutions lies in $\mathcal{O}^{5}$ in their forward lifespan.
\begin{lem}[Decomposition]
\label{lem:decomp}Consider 
\[
\Phi:\overline{\mathcal{U}}\to H_{m}^{3}
\]
that maps $(\lambda,\gamma,b,\eta,\epsilon)\in\overline{\mathcal{U}}$
to $u$ via \eqref{eq:u-formula}. Then, $\Phi(\overline{\mathcal{U}})=\overline{\mathcal{O}}$
and $\Phi$ is a homeomorphism onto $\overline{\mathcal{O}}$. The
subset $\mathcal{O}$ is open in $H_{m}^{3}$. Moreover, $\Phi|_{\mathcal{U}}$
is $C^{1}$ in $(\gamma,b,\eta,\epsilon)$ and continuous in $\lambda$.
The $(\lambda,\gamma,b,\eta)$-components of $\Phi^{-1}|_{\mathcal{O}}$
is $C^{1}$ and $\epsilon$-component is continuous. The analogous
statements also hold when $\mathcal{U},\mathcal{O},H_{m}^{3}$ are
replaced by $\mathcal{U}^{5},\mathcal{O}^{5},H_{m}^{5}$.
\end{lem}

\begin{proof}
Let us introduce some notations that will be used only in this proof.
For $\lambda\in\R_{+}$ and $\gamma\in\R/2\pi\Z$, let us denote 
\[
f_{\lambda,\gamma}(y)\coloneqq\frac{e^{i\gamma}}{\lambda}f\Big(\frac{y}{\lambda}\Big).
\]
We equip $\R_{+}$ with the metric $\mathrm{dist}(\lambda_{1},\lambda_{2})=|\log(\lambda_{1}/\lambda_{2})|$,
and equip $\R/2\pi\Z$ with the induced metric from $\R$. We will
use small parameters $\delta_{1},\delta_{2}>0$ (to be chosen later
in this proof) such that $0<\delta_{2}\ll\delta_{1}\ll\min\{\delta,\eta^{\ast},M^{-1}\}$.\footnote{Recall from Remark \ref{rem:ParameterDependence} that this implies
$\delta_{1}\leq M^{-C}$.}

Now we turn to $P(y;b,\eta)$. Notice that the profile $P(y;b,\eta)$
is only defined for $(b,\eta)$ with $|\eta|\ll b$, \emph{not} for
$|(b,\eta)|\ll1$. We extend it for $b<0$ with $|\eta|\leq(\tfrac{\delta}{2})^{2}|b|$
as $P(y;b,\eta)\coloneqq\chi_{|b|^{-\frac{1}{2}}}(y)Q^{(\eta)}(y)e^{-ib\frac{y^{2}}{4}}$.
We hope to apply the implicit function theorem near $Q=P(\cdot;0,0)$,
so we will define an artificial extension of the profile $P(y;b,\eta)$
for $(b,\eta)$ in a neighborhood of $(0,0)$. We introduce a smooth
function $\psi:\R\to\R$ such that $\psi(\tilde{\eta})=\tilde{\eta}$
for $|\tilde{\eta}|\leq1$ and $\sup(|\psi|+|\tilde{\eta}\psi'|)\leq2$.
For $|b|<\delta_{1}$, we define $\psi_{b}(\tilde{\eta})=|b|^{\frac{5}{4}}\psi(|b|^{-\frac{5}{4}}\tilde{\eta})$
if $b\neq0$ and $\psi_{0}(\tilde{\eta})=0$. Thus $\partial_{b}\psi_{b}(\tilde{\eta})=-\mathrm{sgn}(b)\tfrac{5}{4}|b|^{\frac{1}{4}}[\Lambda_{2}\psi](|b|^{-\frac{5}{4}}\tilde{\eta})$
if $b\neq0$ and $\partial_{b=0}\psi_{b}(\tilde{\eta})=0$. Finally,
we define an artificial extension $\tilde P(\cdot;b,\eta)$ of $P(y;b,\eta)$
by 
\[
\tilde P(\cdot;b,\eta)\coloneqq P(\cdot;b,\psi_{b}(\eta))-(m+1)(\eta-\psi_{b}(\eta))\rho\chi_{2M}
\]
for $|\eta|,|b|<\delta_{1}$.\footnote{Note here that smallness condition on $\delta_{1}$ depending on $\delta$
and $\eta^{\ast}$ (defined in Lemma \ref{lem:mod-profile-linear-regime})
is necessary to define the object $\tilde P(\cdot;b,\eta)$.} Thus $\tilde P(\cdot;b,\eta)=P(\cdot;b,\eta)$ for $|\eta|\leq|b|^{\frac{5}{4}}$
and hence for $(b,\eta)$ in our bootstrap hypothesis \eqref{eq:bootstrap-hyp-H3}.

We now come to our main part of the proof. We hope to apply the implicit
function theorem to the map $\mathbf{F}=(F_{1},F_{2},F_{3},F_{4})^{t}:\R_{+}\times\R/2\pi\Z\times B_{\delta_{1}}(0)\times B_{\delta_{1}}(0)\times L^{2}\to\R^{4}$
defined by 
\begin{align*}
F_{k}(\lambda,\gamma,b,\eta,u) & \coloneqq(\epsilon,\mathcal{Z}_{k})_{r},\qquad\forall k\in\{1,2,3,4\}\\
\epsilon & =u_{\lambda^{-1},-\gamma}(y)-\tilde P(y;b,\eta).
\end{align*}
We claim that $\mathbf{F}$ is $C^{1}$ and $\partial_{\lambda,\gamma,b,\eta}\mathbf{F}$
is invertible at $(\lambda,\gamma,b,\eta,u)=(1,0,0,0,Q)$. To see
this, we compute using \eqref{eq:Transversality} 
\begin{align*}
\partial_{\lambda}F_{k} & =([\Lambda Q]_{\lambda^{-1},-\gamma},\mathcal{Z}_{k})_{r}-(u-Q,[\Lambda\mathcal{Z}_{k}]_{\lambda,\gamma})_{r}\\
 & =-\delta_{1k}\|yQ\|_{L^{2}}^{2}+M^{C}O(\mathrm{dist}((\lambda,\gamma),(1,0))+\|u-Q\|_{L^{2}})+o_{M\to\infty}(1),\\
\partial_{\gamma}F_{k} & =-([iQ]_{\lambda^{-1},-\gamma},\mathcal{Z}_{k})_{r}+(u-Q,[i\mathcal{Z}_{k}]_{\lambda,\gamma})_{r}\\
 & =-\delta_{2k}\|L_{Q}\rho\|_{L^{2}}^{2}+M^{C}O(\mathrm{dist}((\lambda,\gamma),(1,0))+\|u-Q\|_{L^{2}})+o_{M\to\infty}(1).
\end{align*}
Next, by pointwise estimates of Proposition \ref{prop:modified-profile}
and using $\|\partial_{b}\psi_{b}\|_{L^{\infty}}\lesssim|b|^{\frac{1}{4}}$,
we have\footnote{Note that Proposition \ref{prop:modified-profile} holds for $0<|b|<b^{\ast}$.
When $b=0$, one directly computes the limit of $b^{-1}(\tilde P(y;b,\eta)-\tilde P(y;0,\eta))$
as $b\to0$, using $\tilde P(y;0,\eta)=Q-(m+1)\eta\rho\chi_{M}$.
As a result, one gets $\partial_{b=0}\tilde P(0;b,\eta)+i\tfrac{y^{2}}{4}Q=0$
pointwisely and also in $L_{\mathrm{loc}}^{2}$. In particular, $\partial_{b=0}F_{k}=(i\tfrac{y^{2}}{4}Q,\mathcal{Z}_{k})_{r}=\delta_{4k}(\Lambda Q,y^{2}Q\chi_{M})_{r}$.} 
\begin{align*}
 & |\partial_{b}\tilde P(0;b,\eta)+i\tfrac{y^{2}}{4}Q|\\
 & =\Big|\Big(\partial_{b}P(\cdot;b,\tilde{\eta})|_{\tilde{\eta}=\psi_{b}(\eta)}+i\tfrac{y^{2}}{4}Q\Big)+\partial_{b}\psi_{b}(\eta)\cdot\partial_{\tilde{\eta}=\psi_{b}(\eta)}P(\cdot;b,\tilde{\eta})+(m+1)\partial_{b}\psi_{b}(\eta)\rho\chi_{2M}\Big|\\
 & \lesssim\mathbf{1}_{y\leq2|b|^{-1/2}}(|b|y^{4}Q+|b|^{\frac{1}{4}}y^{2}Q)+\mathbf{1}_{y\geq2|b|^{-1/2}}y^{2}Q.
\end{align*}
Combining this with the pointwise estimates on $\mathcal{Z}_{k}$
and \eqref{eq:Transversality}, we have 
\[
\partial_{b}F_{k}=(i\tfrac{y^{2}}{4}Q,\mathcal{Z}_{k})_{r}+O(M^{C}|b|^{\frac{1}{4}})=\delta_{3k}\|yQ\|_{L^{2}}^{2}+O(M^{C}|b|^{\frac{1}{4}})+o_{M\to\infty}(1).
\]
Next, again by pointwise estimates of Proposition \ref{prop:modified-profile}
\begin{align*}
 & |\partial_{\eta}\tilde P(\cdot;b,\eta)+(m+1)\rho|\\
 & =\Big|\psi_{b}'(\eta)\Big(\partial_{\tilde{\eta}=\psi_{b}(\eta)}P(\cdot;b,\tilde{\eta})+(m+1)\rho\Big)+(1-\psi_{b}'(\eta))(m+1)(\rho-\rho\chi_{2M})\Big|\\
 & \lesssim\mathbf{1}_{y\leq2|b|^{-1/2}}|b|y^{4}Q+\mathbf{1}_{y\geq2M}y^{2}Q.
\end{align*}
Combining this with the above pointwise estimates on $\mathcal{Z}_{k}$
and \eqref{eq:Transversality}, we have 
\begin{align*}
\partial_{\eta}F_{k} & =(m+1)(\rho,\mathcal{Z}_{k})_{r}+O(M^{C}|b|+M^{-4}\log M)\\
 & =-\delta_{4k}(m+1)\|L_{Q}\rho\|_{L^{2}}^{2}+O(M^{C}|b|)+o_{M\to\infty}(1).
\end{align*}
Finally, we have 
\[
\frac{\delta F_{k}}{\delta u}=(\mathcal{Z}_{k})_{\lambda,\gamma}\in L^{2}.
\]
Gathering the above computations, we have $\mathbf{F}\in C^{1}$ and
moreover 
\begin{align}
 & \partial_{\lambda,\gamma,b,\eta}\mathbf{F}|_{(1,0,0,0,Q)}+o_{M\to\infty}(1)\nonumber \\
 & =\begin{pmatrix}-\|yQ\|_{L^{2}}^{2} & 0 & 0 & 0\\
0 & -\|L_{Q}\rho\|_{L^{2}}^{2} & 0 & 0\\
0 & 0 & \|yQ\|_{L^{2}}^{2} & 0\\
0 & 0 & 0 & -(m+1)\|L_{Q}\rho\|_{L^{2}}^{2}
\end{pmatrix}.\label{eq:ortho-jacobian}
\end{align}

Therefore, by the implicit function theorem, provided that $M\gg1$,
there exist $0<\delta_{2}\ll\delta_{1}$ and $C^{1}$-map $\mathbf{G}^{1,0}:B_{\delta_{2}}(Q)\to B_{\delta_{1}}(1,0,0,0)$
such that for given $u\in B_{\delta_{2}}(Q)\subseteq L^{2}$, $\mathbf{G}^{1,0}(u)$
is a unique solution to $\mathbf{F}(\mathbf{G}^{1,0}(u),u)=0$ in
$B_{\delta_{1}}(1,0,0,0)$. In particular, we have a Lipschitz bound
\[
\mathrm{dist}(\mathbf{G}^{1,0}(u),(1,0,0,0))\lesssim\|u-Q\|_{L^{2}}.
\]

We now use scale/phase invariances to cover the $\delta_{2}$-neighborhood
of $\{Q_{\lambda,\gamma}:\lambda\in\R_{+},\gamma\in\R/2\pi\Z\}$ in
$L^{2}$. For $\lambda\in\R_{+}$ and $\gamma\in\R/2\pi\Z$, define
$\mathbf{G}^{\lambda,\gamma}:B_{\delta_{2}}(Q_{\lambda,\gamma})\to B_{\delta_{1}}(\lambda,\gamma,0,0)$
in an obvious way, by applying the scale/phase invariances to $\mathbf{G}^{1,0}$.
Thus uniqueness property of $\mathbf{G}^{\lambda,\gamma}$ holds for
values in $B_{\delta_{1}}(\lambda,\gamma,0,0)$.

We now claim that the family $\{\mathbf{G}^{\lambda,\gamma}\}_{\lambda,\gamma}$
is compatible, i.e. 
\[
\mathbf{G}\coloneqq{\textstyle \bigcup_{\lambda,\gamma}}\mathbf{G}^{\lambda,\gamma}:{\textstyle \bigcup_{\lambda,\gamma}}B_{\delta_{2}}(Q_{\lambda,\gamma})\to\R_{+}\times\R/2\pi\Z\times B_{\delta_{1}}(0)\times B_{\delta_{1}}(0)
\]
is well-defined. Indeed, if $u\in B_{\delta_{2}}(Q_{\lambda_{1},\gamma_{1}})\cap B_{\delta_{2}}(Q_{\lambda_{2},\gamma_{2}})$,
then $\mathrm{dist}((\lambda_{1},\gamma_{1}),(\lambda_{2},\gamma_{2}))\lesssim\delta_{2}$
thus $\mathrm{dist}(\mathbf{G}^{\lambda_{2},\gamma_{2}}(u),(\lambda_{1},\gamma_{1},0,0))\lesssim\delta_{2}\ll\delta_{1}$.
Since $\mathbf{G}^{\lambda_{2},\gamma_{2}}(u)$ satisfies the equation
$\mathbf{F}(\mathbf{G}^{\lambda_{2},\gamma_{2}}(u),u)=0$, the uniqueness
property of $\mathbf{G}^{\lambda_{1},\gamma_{1}}(u)$ in $B_{\delta_{1}}(\lambda_{1},\gamma_{1},0,0)$
gives $\mathbf{G}^{\lambda_{2},\gamma_{2}}(u)=\mathbf{G}^{\lambda_{1},\gamma_{1}}(u)$,
showing the claim.

We further need the following uniqueness property of $\mathbf{G}$:
given $u\in B_{\delta_{2}}(Q)\subseteq L^{2}$, $\mathbf{G}(u)$ is
a unique solution to $\mathbf{F}(\mathbf{G}(u),u)=0$ in $\R_{+}\times\R/2\pi\Z\times B_{\delta_{1}}(0)\times B_{\delta_{1}}(0)$
with $\|\epsilon\|_{L^{2}}\lesssim\delta_{2}$. Indeed, suppose that
$\mathbf{G}'=(\lambda',\gamma',b',\eta')$ is a solution to $\mathbf{F}(\mathbf{G}',u)=0$
such that $\epsilon'=u_{(\lambda')^{-1},-\gamma'}-P(\cdot;b',\eta')$
satisfies $\|\epsilon'\|_{L^{2}}\lesssim\delta_{2}$. When $\mathrm{dist}(\mathbf{G}',\mathbf{G}(u))<\delta_{1}$,
then the uniqueness of $\mathbf{G}^{\lambda,\gamma}$ yields $\mathbf{G}'=\mathbf{G}(u)$.
If $\mathrm{dist}(\mathbf{G}',\mathbf{G}(u))>\delta_{1}$, then $\|\tilde P(y;b',\eta')_{\lambda',\gamma'}-\tilde P(y;b,\eta)_{\lambda,\gamma}\|_{L^{2}}\gtrsim\delta_{1}$
but $\|\epsilon'\|_{L^{2}},\|\epsilon\|_{L^{2}}\lesssim\delta_{2}$.
This contradicts to the identity $[\tilde P(y;b',\eta')+\epsilon']_{\lambda',\gamma'}=u=[P(\cdot;b,\eta)+\epsilon]_{\lambda,\gamma}$.

We now conclude the proof of this lemma on the $H_{m}^{3}$-topology.
Notice that the domain of $\mathbf{G}$ contains $\overline{\mathcal{O}}$.
As $\mathbf{G}$ is $C^{1}$ with respect to the $L^{2}$-topology
on $u$, it is also $C^{1}$ with respect to the $H_{m}^{3}$-topology
on $u$. Note that the map $u\mapsto(\mathbf{G}(u),u)\mapsto\epsilon$
is only continuous. Thus the map $(\mathbf{G},\epsilon)$ restricted
on $\overline{\mathcal{O}}$ is continuous. By the definition of $\epsilon$,
$(\mathbf{G},\epsilon)$ is the right inverse of $\Phi$. By the uniqueness
property of $\mathbf{G}$, one also has that $(\mathbf{G},\epsilon)$
is the left inverse of $\Phi$. These facts show that $\Phi(\overline{\mathcal{U}})=\overline{\mathcal{O}}$
and $\Phi$ is a homeomorphism with the inverse $(\mathbf{G},\epsilon)$.
Openness of $\mathcal{O}$ now follows from openness of $\mathcal{U}$.
Finally, when we restrict $\Phi$ on $\mathcal{U}$, then it is $C^{1}$
in $(\gamma,b,\eta,\epsilon)$ and continuous in $\lambda$.

The argument of the previous paragraph is still valid when $\mathcal{U},\mathcal{O},H_{m}^{3}$
are replaced by $\mathcal{U}^{5},\mathcal{O}^{5},H_{m}^{5}$. This
finishes the proof for the $H_{m}^{5}$-topology.
\end{proof}
\begin{rem}
Since the curve $\lambda\mapsto\lambda u(\lambda\cdot)$ is continuous
but not $C^{1}$, $\Phi$ is merely continuous in $\lambda$.
\end{rem}

\subsection{\label{subsec:Existence-of-trapped-solutions}Existence of trapped
solutions}

In this section, we reduce the existence of trapped solutions in Theorems
\ref{thm:Existence} and \ref{thm:BlowupManifold} into two propositions.
The heart of the proof is the main bootstrap argument, provided in
Proposition \ref{prop:main-bootstrap}. It roughly says that all the
assumptions except the $\eta$-bound of \eqref{eq:bootstrap-hyp-H3}
and \eqref{eq:bootstrap-hyp-H5} can be bootstrapped. Next, in Proposition
\ref{prop:Sets-I-pm}, we use a connectivity argument to show that
there exists a \emph{trapped solution}, which satisfies \eqref{eq:bootstrap-hyp-H3}
and \eqref{eq:bootstrap-hyp-H5} on its forward lifespan. Such a combination
of bootstrap and connectivity arguments is rather standard. Here we
follow \cite{MerleRaphaelRodnianski2013InventMath}.
\begin{proof}[Proof of the existence of trapped solutions, assuming Propositions
\ref{prop:main-bootstrap} and \ref{prop:Sets-I-pm}]
Here we prove the existence of $H_{m}^{3}$-trapped solutions in
Theorem \ref{thm:Existence} and $H_{m}^{5}$-trapped solutions in
Theorem \ref{thm:BlowupManifold}. Proofs of those are soft arguments
and very similar, so we only prove the former.

Let $(\lambda_{0},\gamma_{0},b_{0},\epsilon_{0})\in\tilde{\mathcal{U}}_{init}$
and consider $\eta_{0}$ which varies in the range $(-\frac{K}{2}b_{0}^{3/2},\frac{K}{2}b_{0}^{3/2})$.
Define $u_{0}\in\mathcal{O}_{init}$ via \eqref{eq:Def-u_0} and let
$u$ be the forward-in-time maximal solution to \eqref{eq:CSS} with
the initial data $u_{0}$. Let $[0,T)$ with $0<T\leq+\infty$ be
the lifespan of $u$.

We claim that $u(t)\in\mathcal{O}$ for all $t\in[0,T)$ for a well-chosen
$\eta_{0}$. As $u(0)=u_{0}\in\mathcal{O}$ initially and $\mathcal{O}$
is open, we can define the exit time 
\[
T_{\mathrm{exit}}\coloneqq\sup\{\tau\in[0,T):u(\tau')\in\mathcal{O}\text{ for }\tau'\in[0,\tau]\}\in(0,T].
\]
If $T_{\mathrm{exit}}=T$ for some $\eta_{0}$, then we are done.
Now assume $T_{\mathrm{exit}}<T$ for all $\eta_{0}$. Note that $u(T_{\mathrm{exit}})\in\overline{\mathcal{O}}\setminus\mathcal{O}$
due to openness of $\mathcal{O}$ and maximality of $T_{\mathrm{exit}}$.
We associate the modulation parameters and remainder $(\lambda,\gamma,b,\eta,\epsilon)$
with $u$ for each time $t\in[0,T_{\mathrm{exit}}]$ according to
Lemma \ref{lem:decomp}. The following proposition is shown in Section
\ref{sec:Main-bootstrap}, and is the heart of the proof of Theorem
\ref{thm:Existence}:
\begin{prop}[Main bootstrap]
\label{prop:main-bootstrap}\ 
\begin{itemize}
\item Let $m\geq1$; let $u$ be a solution to \eqref{eq:CSS} with initial
data $u_{0}\in\mathcal{O}_{init}$. If $u(t)\in\mathcal{O}$ (i.e.,
the $H^{3}$-boostrap hypothesis \eqref{eq:bootstrap-hyp-H3} is satisfied)
for $t\in[0,\tau_{\ast}]$ for some $\tau_{\ast}>0$, then the bootstrap
conclusion holds for $t\in[0,\tau_{\ast}]$: 
\[
b\in(0,b_{0}],\quad\|\epsilon\|_{L^{2}}<\tfrac{K}{2}(b^{\ast})^{\frac{1}{4}},\quad\|\epsilon_{1}\|_{L^{2}}<\tfrac{K}{2}b,\quad\|\epsilon_{3}\|_{L^{2}}<\tfrac{K}{2}b^{\frac{5}{2}}.
\]
\item Let $m\geq3$; let $u$ be a solution to \eqref{eq:CSS} with initial
data $u_{0}\in\mathcal{O}_{init}^{5}$. If $u(t)\in\mathcal{O}^{5}$
(i.e., the $H^{5}$-boostrap hypothesis \eqref{eq:bootstrap-hyp-H5}
is satisfied) for $t\in[0,\tau_{\ast}]$ for some $\tau_{\ast}>0$,
then the bootstrap conclusion holds for $t\in[0,\tau_{\ast}]$: 
\[
b\in(0,b_{0}],\ \|\epsilon\|_{L^{2}}<\tfrac{K}{2}(b^{\ast})^{\frac{1}{4}},\ \|\epsilon_{1}\|_{L^{2}}<\tfrac{K}{2}b,\ \|\epsilon_{3}\|_{L^{2}}<\tfrac{K}{2}b^{3},\ \|\epsilon_{5}\|_{L^{2}}<\tfrac{K}{2}b^{\frac{9}{2}}.
\]
\end{itemize}
\end{prop}

Compared to \eqref{eq:bootstrap-hyp-H3}, all the assumptions except
the $\eta$-bound are improved. It is indeed impossible to improve
the $\eta$-bound, because the $\eta$-parameter accounts of the instability.

Back to the proof, $u(T_{\mathrm{exit}})\in\overline{\mathcal{O}}\setminus\mathcal{O}$
and Proposition \ref{prop:main-bootstrap} says that either $b=0$
or $|\eta|=Kb^{3/2}$ at $t=T_{\mathrm{exit}}$. It is easy to exclude
the case $b=0$. If $b=0$ at $t=T_{\mathrm{exit}}$, then $u(T_{\mathrm{exit}})$
must be a scaled $Q$, which is a static solution. Hence $u_{0}=u(T_{\mathrm{exit}})$
is a scaled $Q$, which does not belong to $\mathcal{O}_{init}$,
yielding a contradiction.

Thus $|\eta|=Kb^{3/2}$ at $t=T_{\mathrm{exit}}$. To derive a contradiction,
we use a basic connectivity argument. Let $\mathcal{I}_{\pm}$ be
the set of $\eta_{0}$ such that $\eta(T_{\mathrm{exit}})=\pm Kb^{3/2}(T_{\mathrm{exit}})$.
Note that $\mathcal{I}_{\pm}$ partitions $(-\frac{K}{2}b_{0}^{3/2},\frac{K}{2}b_{0}^{3/2})$.
The following proposition is shown in Section \ref{subsec:Existence-of-special-eta}.
\begin{prop}[The sets $\mathcal{I}_{\pm}$]
\label{prop:Sets-I-pm}The sets $\mathcal{I}_{\pm}$ are nonempty
and open.
\end{prop}

We have a contradiction from connectivity of $(-\frac{K}{2}b_{0}^{3/2},\frac{K}{2}b_{0}^{3/2})$.
Thus our claim is proved: $u(t)\in\mathcal{O}$ for all $t\in[0,T)$,
for a well-chosen $\eta_{0}$.
\end{proof}
It remains to show that the trapped solutions are pseudoconformal
blow-up solutions satisfying the statements of Theorem \ref{thm:Existence}
(and those of Theorem \ref{thm:BlowupManifold}). We will complete
the proof in Section \ref{subsec:Sharp-description}.

\section{\label{sec:Main-bootstrap}Main Bootstrap}

In this section, we perform modulation analysis to prove Theorem \ref{thm:Existence}
(and the first part of Theorem \ref{thm:BlowupManifold}). The heart
of the proof is the main bootstrap procedure (Proposition \ref{prop:main-bootstrap}).
The proof of Proposition \ref{prop:main-bootstrap} will be finished
in Section \ref{subsec:ClosingBootstrap}. In Sections \ref{subsec:Existence-of-special-eta}
and \ref{subsec:Sharp-description}, we finish the proof of Theorem
\ref{thm:Existence} (and the first part of Theorem \ref{thm:BlowupManifold})\@.

In Section \ref{subsec:decomposition}, we decomposed our solution
$u$ at each time $t$ of the form 
\[
u(t,r)=\frac{e^{i\gamma(t)}}{\lambda(t)}[P(\cdot;b(t),\eta(t))+\epsilon(t,\cdot)]\Big(\frac{r}{\lambda(t)}\Big).
\]
The main goal of this section is to study the dynamics of $\epsilon$
and show that $\epsilon$ satisfies the bootstrap conclusions.

\subsection{Equation of $\epsilon$}

Here we derive the equation of $\epsilon$. As usual, we want to view
the dynamics of $u$ as a near soliton evolution. For this purpose,
we renormalize $u$ by setting the renormalized variables $(s,y)$
\[
y=\frac{r}{\lambda}\quad\text{and}\quad\frac{ds}{dt}=\frac{1}{\lambda^{2}},
\]
where $\frac{1}{\lambda^{2}}$ is motivated by the scaling symmetry
of \eqref{eq:CSS}. Under this relation, we will freely change the
variables either by $s=s(t)$ or $t=t(s)$, and by abuse of notations,
we denote functions of $t$ or $s$ identically. For example, we write
$\lambda(s)=\lambda(t(s))$, $\epsilon(s,y)=\epsilon(t(s),y)$, and
so on. We define 
\[
u^{\flat}(s,y)=e^{-i\gamma}\lambda u(t,\lambda y)=P(y;b,\eta)+\epsilon(s,y).
\]
Then $u^{\flat}$ satisfies\footnote{In \cite{KimKwon2019arXiv}, we intensively used $\sharp/\flat$ notations
when switching the dynamical variables $(t,x)\leftrightarrow(s,y)$.
See \cite{KimKwon2019arXiv} for more details. In this paper, we will
not heavily rely on $\sharp/\flat$ notations.} 
\[
(\partial_{s}-\frac{\lambda_{s}}{\lambda}\Lambda+\gamma_{s}i)u^{\flat}+iL_{u^{\flat}}^{\ast}\D_{+}^{(u^{\flat})}u^{\flat}=0.
\]
On the other hand, by \eqref{eq:decomp-P}, we have 
\begin{align*}
 & \partial_{s}P+iL_{P}^{\ast}\D_{+}^{(P)}P\\
 & =-b\Lambda P-(\eta\theta_{\eta}+\theta_{\Psi})iP+(b_{s}+b^{2}+\eta^{2})\partial_{b}P+\eta_{s}\partial_{\eta}P+i\Psi.
\end{align*}
Here, $\theta_{\eta},\theta_{\Psi},\Psi$ are defined in Proposition
\ref{prop:modified-profile}. Subtracting the second from the first,
we arrive at the preliminary version of $\epsilon$-equation:
\begin{equation}
(\partial_{s}-\frac{\lambda_{s}}{\lambda}\Lambda+\gamma_{s}i)\epsilon+i\mathcal{L}_{Q}\epsilon=\mathbf{Mod}\cdot\mathbf{v}-iR_{\mathrm{L-L}}-iR_{\mathrm{NL}}-i\Psi,\label{eq:prelim-eqn-e}
\end{equation}
where 
\begin{align*}
\mathbf{Mod} & \coloneqq(\frac{\lambda_{s}}{\lambda}+b,\gamma_{s}-\eta\theta_{\eta}-\theta_{\Psi},b_{s}+b^{2}+\eta^{2},\eta_{s})^{t},\\
\mathbf{v} & \coloneqq(\Lambda P,-iP,-\partial_{b}P,-\partial_{\eta}P)^{t}\\
R_{\mathrm{L-L}} & \coloneqq(\mathcal{L}_{P}-\mathcal{L}_{Q})\epsilon,\\
R_{\mathrm{NL}} & \coloneqq L_{u^{\flat}}^{\ast}\D_{+}^{(u^{\flat})}u^{\flat}-L_{P}^{\ast}\D_{+}^{(P)}P-\mathcal{L}_{P}\epsilon.
\end{align*}

The terms on the RHS of \eqref{eq:prelim-eqn-e} are viewed as the
remainder terms. The last term $i\Psi$ is the error from the modified
profile and is independent of $\epsilon$. It is already estimated
in Proposition \ref{prop:modified-profile} and determines the optimal
bootstrap assumptions on $\epsilon$. The term $iR_{\mathrm{L-L}}$
is the difference of two linearized operators $i\mathcal{L}_{P}\epsilon$
and $i\mathcal{L}_{Q}\epsilon$. It is linear in $\epsilon$ and enjoys
further degeneracy from $P-Q$, so roughly speaking it is $O((P-Q)\epsilon)=O(b\epsilon)$.
The term $iR_{\mathrm{NL}}$ collects the quadratic and higher terms
in $\epsilon$ of the expression $iL_{P+\epsilon}^{\ast}\D_{+}^{(P+\epsilon)}(P+\epsilon)$,
and is roughly of size $O(\epsilon^{2})$. Lastly, the smallness of
$\mathbf{Mod}\cdot\mathbf{v}$ indicates the formal evolution of the
modulation parameters \eqref{eq:FormalODE}.

In fact, we will reorganize the RHS of \eqref{eq:prelim-eqn-e} by
moving the phase correction terms from $iR_{\mathrm{L-L}}$ and $iR_{\mathrm{NL}}$
into $\mathbf{Mod}\cdot\mathbf{v}$. They will be written by $i\tilde R_{\mathrm{L-L}}$,
$i\tilde R_{\mathrm{NL}}$, and $\tilde{\mathbf{Mod}}\cdot\mathbf{v}$;
see \eqref{eq:e-eq-cor}. These remainder terms are estimated in the
following sections.

\subsection{\label{subsec:Estimates-of-remainders-H3}Estimates of remainders}

In this subsection, we provide estimates of remainders $R_{\mathrm{L-L}}$
and $R_{\mathrm{NL}}$. The term $R_{\mathrm{L-L}}$ contains $\epsilon$
and also $P-Q$ part from $\mathcal{L}_{P}-\mathcal{L}_{Q}$. The
term $R_{\mathrm{NL}}$ is quadratic or higher order in $\epsilon$.
These estimates are crucial ingeredients of the modulation estimates
and energy estimates. We treat $R_{\mathrm{L-L}}$ and $R_{\mathrm{NL}}$
in Lemmas \ref{lem:degenerate-linear} and \ref{lem:nonlinear}, respectively.

Let us introduce some notations to be used only in this subsection.
We recall the nonlinearity of \eqref{eq:CSS}
\[
\mathcal{N}(u)=-|u|^{2}u+\tfrac{2m}{y^{2}}A_{\theta}[u]+\tfrac{1}{y^{2}}A_{\theta}^{2}[u]u+A_{0}[u]u.
\]
Terms in $R_{\mathrm{L-L}}$ and $R_{\mathrm{NL}}$ are multilinear
terms as in $\mathcal{N}(u)$. We recognize them as a form of $V\psi$,
where $\psi\in\{\epsilon,P,P-Q\}$ and $V$ is a (nonlinear) potential
part of $\mathcal{N}$. We define the potential part of the cubic
nonlinearity by 
\[
V_{3}[\psi_{1},\psi_{2}]\coloneqq-\Re(\overline{\psi_{1}}\psi_{2})+\tfrac{2m}{y^{2}}A_{\theta}[\psi_{1},\psi_{2}]+A_{0,3}[\psi_{1},\psi_{2}]
\]
and that of the quintic nonlinearity by 
\[
V_{5}[\psi_{1},\psi_{2},\psi_{3},\psi_{4}]=\tfrac{1}{y^{2}}A_{\theta}[\psi_{1},\psi_{2}]A_{\theta}[\psi_{3},\psi_{4}]+A_{0,5}[\psi_{1},\psi_{2},\psi_{3},\psi_{4}].
\]
Here, 
\begin{align*}
A_{\theta}[\psi_{1},\psi_{2}] & \coloneqq-\frac{1}{2}\int_{0}^{y}\Re(\overline{\psi_{1}}\psi_{2})y'dy',\\
A_{0,3}[\psi_{1},\psi_{2}] & \coloneqq-m\int_{y}^{\infty}\Re(\overline{\psi_{1}}\psi_{2})\frac{dy'}{y'},\\
A_{0,5}[\psi_{1},\psi_{2},\psi_{3},\psi_{4}] & \coloneqq-\int_{y}^{\infty}A_{\theta}[\psi_{1},\psi_{2}]\Re(\overline{\psi_{3}}\psi_{4})\frac{dy'}{y'}.
\end{align*}
We further let 
\[
V_{3}[\psi]\coloneqq V_{3}[\psi,\psi]\quad\text{and}\quad V_{5}[\psi]\coloneqq V_{5}[\psi,\psi,\psi,\psi].
\]
Thus 
\begin{equation}
L_{u}^{\ast}\D_{+}^{(u)}u=-\Delta_{m}u+\mathcal{N}(u)=-\Delta_{m}u+V_{3}[u]u+V_{5}[u]u.\label{eq:nonlinearity-expansion}
\end{equation}

We start with estimates for $iR_{\mathrm{L-L}}$. \eqref{eq:CSS}
contains two types of long-range interactions, one from $A_{\theta}$
and the other from $A_{0}$. In this work, we need to recognize the
long-range effect of $A_{0}$, which contains $\int_{y}^{\infty}$-integral.
It turns out that this effect can be identified with the phase correction
to our solution. For instance, $iR_{\mathrm{L-L}}$ contains the term
\[
-\Big(\int_{y}^{\infty}m\Re((\overline{P-Q})\epsilon)\frac{dy'}{y'}\Big)iP.
\]
Due to the weak Hardy control of $\dot{\mathcal{H}}_{m}^{3}$ (or
$\dot{\mathcal{H}}_{m}^{5}$) at spatial infinity (see \eqref{eq:Def-Hdot3-Section2}
and \eqref{eq:Def-Hdot5-Section2}), it is more efficient that we
pull out the $\int_{0}^{\infty}$-integral and rewrite the above as
\[
\Big(\int_{0}^{y}m\Re((\overline{P-Q})\epsilon)\frac{dy'}{y'}\Big)iP-\Big(\int_{0}^{\infty}m\Re((\overline{P-Q})\epsilon)\frac{dy}{y}\Big)iP.
\]
For the first term, we can use the decay property of $iP$ outside
the integral. The latter term is of the form $-\theta\cdot iP$, whose
role is the \emph{phase correction}. One may compare with $\theta_{\mathrm{cor}}$
in \cite{KimKwon2019arXiv}. This phase correction also arises in
$R_{\mathrm{NL}}$.

For $k\in\{3,5\}$, we further introduce 
\[
\tilde A_{0,k}(y)\coloneqq A_{0,k}(y)-A_{0,k}(y=0)
\]
and define $\tilde V_{k}$ by replacing $A_{0,k}$-part of $V_{k}$
by $\tilde A_{0,k}$.
\begin{lem}[Various estimates for degenerate linear terms]
\label{lem:degenerate-linear}We decompose $R_{\mathrm{L-L}}$ as
\[
R_{\mathrm{L-L}}=-\theta_{\mathrm{L-L}}P+\tilde R_{\mathrm{L-L}}
\]
such that the following estimates hold.
\begin{itemize}
\item (Estimate of $\theta_{\mathrm{L-L}}$) We have 
\begin{equation}
|\theta_{\mathrm{L-L}}|\lesssim b^{\frac{1}{2}}\|\epsilon\|_{\dot{\mathcal{H}}_{m}^{3}}.\label{eq:LL-theta-est}
\end{equation}
\item ($\dot{\mathcal{H}}_{m}^{3}$-estimate of $\tilde R_{\mathrm{L-L}}$)
We have 
\begin{equation}
\|\tilde R_{\mathrm{L-L}}\|_{\dot{\mathcal{H}}_{m}^{3}}\lesssim b(\|\epsilon\|_{\dot{\mathcal{H}}_{m,\leq M_{2}}^{3}}+o_{M_{2}\to\infty}(1)\|\epsilon\|_{\dot{\mathcal{H}}_{m}^{3}}).\label{eq:LL-Hdot3-est}
\end{equation}
In particular, 
\begin{align}
\sup_{k\in\{1,2,3,4\}}|(i\tilde R_{\mathrm{L-L}},\mathcal{Z}_{k})_{r}| & \lesssim M^{C}b\|\epsilon\|_{\dot{\mathcal{H}}_{m}^{3}},\label{eq:LL-mod-est}\\
\||A_{Q}L_{Q}i\tilde R_{L-L}|_{-1}\|_{L^{2}} & \lesssim b(\|\epsilon\|_{\dot{\mathcal{H}}_{m,\leq M_{2}}^{3}}+o_{M_{2}\to\infty}(1)\|\epsilon\|_{\dot{\mathcal{H}}_{m}^{3}}),\label{eq:LL-E3-est}
\end{align}
where the implicit constant $M^{C}$ comes from the bounds of $\mathcal{Z}_{k}$
\eqref{eq:Orthog-Hdot3-est}.
\end{itemize}
In addition, if $m\geq3$, we have estimates in terms of the higher
Sobolev norm $\dot{\mathcal{H}}_{m}^{5}$:
\begin{itemize}
\item ($\dot{H}_{m}^{3}$-estimate of $\tilde R_{\mathrm{L-L}}$) We have
\begin{equation}
\|\tilde R_{\mathrm{L-L}}\|_{\dot{H}_{m}^{3}}\lesssim b\|\epsilon\|_{\dot{\mathcal{H}}_{m}^{5}}.\label{eq:LL-Hdot3-est-m-geq-3}
\end{equation}
In particular, 
\begin{align}
\sup_{k\in\{1,2,3,4\}}|(i\tilde R_{\mathrm{L-L}},\mathcal{Z}_{k})_{r}| & \lesssim M^{C}b\|\epsilon\|_{\dot{\mathcal{H}}_{m}^{5}},\label{eq:LL-mod-est-m-geq-3}\\
\||A_{Q}L_{Q}i\tilde R_{L-L}|_{-1}\|_{L^{2}} & \lesssim b\|\epsilon\|_{\dot{\mathcal{H}}_{m}^{5}}.\label{eq:LL-E3-est-m-geq-3}
\end{align}
\item ($\dot{\mathcal{H}}_{m}^{5}$-estimate of $\tilde R_{\mathrm{L-L}}$)
We have 
\begin{equation}
\|\tilde R_{\mathrm{L-L}}\|_{\dot{\mathcal{H}}_{m}^{5}}\lesssim b(\|\epsilon\|_{\dot{\mathcal{H}}_{m,\leq M_{2}}^{5}}+o_{M_{2}\to\infty}(1)\|\epsilon\|_{\dot{\mathcal{H}}_{m}^{5}}).\label{eq:LL-Hdot5-est}
\end{equation}
In particular, 
\begin{align}
\||A_{Q}L_{Q}i\tilde R_{L-L}|_{-3}\|_{L^{2}} & \lesssim b(\|\epsilon\|_{\dot{\mathcal{H}}_{m,\leq M_{2}}^{5}}+o_{M_{2}\to\infty}(1)\|\epsilon\|_{\dot{\mathcal{H}}_{m}^{5}}).\label{eq:LL-E5-est}
\end{align}
\end{itemize}
\end{lem}

\begin{rem}
According to the parameter dependence (Remark \ref{rem:ParameterDependence}),
more precise statement of \eqref{eq:LL-Hdot3-est} is as follows.
For all sufficiently large $M_{2}$, \eqref{eq:LL-Hdot3-est} holds
for all sufficiently small $b<b^{\ast}(M_{2})$. As mentioned in the
beginning of this section, we are also assuming the bootstrap hypothesis,
and in particular $|\eta|\lesssim b^{\frac{3}{2}}$. The same applies
for \eqref{eq:LL-Hdot5-est}. As explained in Introduction (Section
\ref{subsec:strategy}), having a smallness factor $o_{M_{2}\to\infty}(1)$
is crucial in our argument.
\end{rem}

\begin{rem}
In fact, \eqref{eq:LL-E3-est} also holds for $iR_{\mathrm{L-L}}$
instead of $i\tilde R_{\mathrm{L-L}}$, due to \eqref{eq:LL-theta-est}
and \eqref{eq:gen-null-rel} (which relies heavily on the fact $L_{Q}iP=O(b)$).
However, \eqref{eq:LL-Hdot3-est} does not hold for $iR_{\mathrm{L-L}}$.
\end{rem}

\begin{rem}
One can apply the bootstrap hypotheses \eqref{eq:bootstrap-hyp-H3}
or \eqref{eq:bootstrap-hyp-H5} to write the bounds for $\tilde R_{\mathrm{L-L}}$
in terms of $b$. The upper bounds are written of the form $O(b\epsilon)$
in order to keep the original form of $\tilde R_{\mathrm{L-L}}$.
A similar remark applies for $iR_{\mathrm{NL}}$ below.
\end{rem}

\begin{rem}
Instead of bounds of the form $b\|\epsilon\|_{\dot{\mathcal{H}}_{m}^{3}}$,
it is crucial to have improved upper bounds of the form $b(\|\epsilon\|_{\dot{\mathcal{H}}_{m,\leq M_{2}}^{3}}+o_{M_{2}\to\infty}(1)\|\epsilon\|_{\dot{\mathcal{H}}_{m}^{3}})$.
This will be detailed in Section \ref{subsec:energy-identity}.
\end{rem}

\begin{proof}
\textbf{Step 1.} Decomposition $R_{\mathrm{L-L}}=-\theta_{\mathrm{L-L}}P+\tilde R_{\mathrm{L-L}}$.

In view of \eqref{eq:nonlinearity-expansion}, we can write $\mathcal{L}_{P}\epsilon$
as 
\begin{align*}
\mathcal{L}_{P}\epsilon & =-\Delta_{m}\epsilon+V_{3}[P]\epsilon+2V_{3}[P,\epsilon]P\\
 & \quad+V_{5}[P]\epsilon+2V_{5}[P,\epsilon,P,P]P+2V_{5}[P,P,P,\epsilon]P
\end{align*}
and similarly for $\mathcal{L}_{Q}\epsilon$. Thus $R_{\mathrm{L-L}}=\mathcal{L}_{P}\epsilon-\mathcal{L}_{Q}\epsilon$
decomposes as
\begin{align*}
R_{\mathrm{L-L}} & =(V_{3}[P]-V_{3}[Q]+V_{5}[P]-V_{5}[Q])\epsilon\\
 & \quad+2(V_{3}[P,\epsilon]P-V_{3}[Q,\epsilon]Q)\\
 & \quad+2(V_{5}[P,\epsilon,P,P]P-V_{5}[Q,\epsilon,Q,Q]Q)\\
 & \quad+2(V_{5}[P,P,P,\epsilon]P-V_{5}[Q,Q,Q,\epsilon]Q).
\end{align*}
We rearrange the above as 
\begin{align*}
R_{\mathrm{L-L}} & =(V_{3}[P]-V_{3}[Q]+V_{5}[P]-V_{5}[Q])\epsilon\\
 & \quad+2(V_{3}[Q,\epsilon]+V_{5}[Q,\epsilon,Q,Q]+V_{5}[Q,Q,Q,\epsilon])(P-Q)\\
 & \quad+2V_{3}[P-Q,\epsilon]P+2(V_{5}[P,\epsilon,P,P]-V_{5}[Q,\epsilon,Q,Q])P\\
 & \quad+2(V_{5}[P,P,P,\epsilon]-V_{5}[Q,Q,Q,\epsilon])P
\end{align*}
We focus on the terms that end in $P$. As illustrated above, we capture
the phase correction term $\theta_{\mathrm{L-L}}P$. We replace the
outermost integral $\int_{y}^{\infty}$ of $A_{0}$ by $\int_{0}^{y}$
at cost of introducing $\theta_{\mathrm{L-L}}$: 
\begin{align}
\tilde R_{\mathrm{L-L}} & \coloneqq R_{\mathrm{L-L}}+\theta_{\mathrm{L-L}}P\nonumber \\
 & =(V_{3}[P]-V_{3}[Q])\epsilon+(V_{5}[P]-V_{5}[Q])\epsilon+2\tilde V_{3}[P-Q,\epsilon]P+2V_{3}[Q,\epsilon](P-Q)\label{eq:R-tilde-L-L}\\
 & \quad+2(\tilde V_{5}[P,\epsilon,P,P]-\tilde V_{5}[Q,\epsilon,Q,Q])P+2V_{5}[Q,\epsilon,Q,Q](P-Q)\nonumber \\
 & \quad+2(\tilde V_{5}[P,P,P,\epsilon]-\tilde V_{5}[Q,Q,Q,\epsilon])P+2V_{5}[Q,Q,Q,\epsilon](P-Q),\nonumber 
\end{align}
where 
\begin{align*}
\theta_{\mathrm{L-L}} & \coloneqq{\textstyle \int_{0}^{\infty}}2(m+A_{\theta}[P])\Re((\overline{P-Q})\epsilon)\tfrac{dy}{y}\\
 & \quad+{\textstyle \int_{0}^{\infty}}2(A_{\theta}[P]-A_{\theta}[Q])\Re(Q\epsilon)\tfrac{dy}{y}\\
 & \quad+{\textstyle \int_{0}^{\infty}}2(A_{\theta}[P,\epsilon]|P|^{2}-A_{\theta}[Q,\epsilon]Q^{2})\tfrac{dy}{y}.
\end{align*}

\ 

\textbf{Step 2.} The $\theta_{\mathrm{L-L}}$-estimate \eqref{eq:LL-theta-est}.

We use $|Q-P|\lesssim\min\{by^{2},1\}Q$ (see Proposition \ref{prop:modified-profile}),
$\frac{dy}{y}=y^{-2}\cdot ydy$, and 
\[
{\textstyle \int_{0}^{y}}\min\{b(y')^{2},1\}|f(y')|dy'\lesssim\min\{by^{2},1\}{\textstyle \int_{0}^{y}}|f(y')|dy'
\]
to reduce 
\begin{align*}
\|y^{-2}\min\{by^{2},1\}A_{\theta}[Q,\epsilon]Q^{2}\|_{L^{1}} & \lesssim b^{\frac{1}{2}}\|\epsilon\|_{\dot{\mathcal{H}}_{m}^{3}},\\
\|y^{-2}\min\{by^{2},1\}A_{\theta}[Q]Q\epsilon\|_{L^{1}} & \lesssim b^{\frac{1}{2}}\|\epsilon\|_{\dot{\mathcal{H}}_{m}^{3}}.
\end{align*}
The first estimate follows from 
\[
\|y^{-2}A_{\theta}[Q,\epsilon]\|_{L^{2}}\lesssim\|Q\epsilon\|_{L^{2}}\lesssim\|\epsilon\|_{\dot{\mathcal{H}}_{m}^{3}},\qquad\|\min\{by^{2},1\}Q^{2}\|_{L^{2}}\lesssim b.
\]
The second estimate follows from 
\[
\|y^{-2}\min\{by^{2},1\}A_{\theta}[Q]\|_{L^{2}}\lesssim b^{\frac{1}{2}},\qquad\|Q\epsilon\|_{L^{2}}\lesssim\|\epsilon\|_{\dot{\mathcal{H}}_{m}^{3}}.
\]
This completes the proof of \eqref{eq:LL-theta-est}.

\ 

\textbf{Step 3.} The estimates of $\tilde R_{\mathrm{L-L}}$.

We turn to the estimates of $\tilde R_{\mathrm{L-L}}$. Note that
\eqref{eq:LL-mod-est} follows from \eqref{eq:LL-Hdot3-est} and \eqref{eq:Orthog-Hdot3-est}.
Next, positivity of $A_{Q}A_{Q}^{\ast}$ and boundedness of $A_{Q}^{\ast}A_{Q}L_{Q}:\dot{\mathcal{H}}_{m}^{3}\to L^{2}$
(Lemma \ref{lem:Mapping-AAL-Section2}) say that \eqref{eq:LL-E3-est}
follows from \eqref{eq:LL-Hdot3-est}. Similarly, \eqref{eq:LL-mod-est-m-geq-3}
and \eqref{eq:LL-E3-est-m-geq-3} follow from \eqref{eq:LL-Hdot3-est-m-geq-3}.
Finally, using positivity of $(A_{Q}A_{Q}^{\ast})^{3}$ (Lemma \ref{lem:Positivity-AAA-Appendix})
and boundedness of $A_{Q}^{\ast}A_{Q}A_{Q}^{\ast}A_{Q}L_{Q}:\dot{\mathcal{H}}_{m}^{5}\to L^{2}$
(Lemma \ref{lem:Mapping-AAAAL-Section2}), \eqref{eq:LL-E5-est} follows
from \eqref{eq:LL-Hdot5-est}.

Henceforth, we focus on \eqref{eq:LL-Hdot3-est}, \eqref{eq:LL-Hdot3-est-m-geq-3},
and \eqref{eq:LL-Hdot5-est}. For simplicity of notations, we recognize
each term of \eqref{eq:R-tilde-L-L} as $V\psi$, where $V$ is the
potential part and $\psi\in\{\epsilon,P,P-Q\}$. For instance, the
first term of \eqref{eq:R-tilde-L-L} corresponds to $V=V_{3}[P]-V_{3}[Q]$
and $\psi=\epsilon$, the seventh term of \eqref{eq:R-tilde-L-L}
corresponds to $V=2(\tilde V_{5}[P,P,P,\epsilon]-\tilde V_{5}[Q,Q,Q,\epsilon])$
and $\psi=P$, and so on.

With these notations, we claim 
\begin{multline}
\|y^{-3}\mathbf{1}_{y\geq1}V\psi\|_{L^{2}}+\||V\partial_{+}\psi|_{-2}\|_{L^{2}}+\||(\partial_{y}V)\psi|_{-2}\|_{L^{2}}\\
\lesssim\begin{cases}
b(\|\epsilon\|_{\dot{\mathcal{H}}_{m,\leq M_{2}}^{3}}+o_{M_{2}\to\infty}(1)\|\epsilon\|_{\dot{\mathcal{H}}_{m}^{3}}), & \text{if }m\geq1,\\
b\|\epsilon\|_{\dot{\mathcal{H}}_{m}^{5}} & \text{if }m\geq3,
\end{cases}\label{eq:deg-lin-claim}
\end{multline}
and 
\begin{multline}
\|y^{-5}\mathbf{1}_{y\geq1}V\psi\|_{L^{2}}+\||V\partial_{+}\psi|_{-4}\|_{L^{2}}+\||(\partial_{y}V)\psi|_{-4}\|_{L^{2}}\\
\lesssim b(\|\epsilon\|_{\dot{\mathcal{H}}_{m,\leq M_{2}}^{5}}+o_{M_{2}\to\infty}(1)\|\epsilon\|_{\dot{\mathcal{H}}_{m}^{5}})\quad\text{if }m\geq3.\label{eq:deg-lin-claim-1}
\end{multline}
Let us finish the proof of \eqref{eq:LL-Hdot3-est}, \eqref{eq:LL-Hdot3-est-m-geq-3},
and \eqref{eq:LL-Hdot5-est} assuming the claims \eqref{eq:deg-lin-claim}
and \eqref{eq:deg-lin-claim-1}. As $P,Q,\epsilon\in H_{m}^{3}$ (for
each fixed time), we have $V\psi\in H_{m}^{3}$.\footnote{One can show that the nonlinearities of \eqref{eq:CSS} for $H_{m}^{3}$-solutions
also belong to $C_{T}H_{m}^{3}$ using Littlewood-Paley decomposition.
See for example \cite[Appendix B]{KimKwon2019arXiv}.} Thus we can apply Lemma \ref{lem:Comparison-H3-Hdot3} to get 
\begin{align*}
\|V\psi\|_{\dot{\mathcal{H}}_{m}^{3}} & \lesssim\|y^{-3}\mathbf{1}_{y\geq1}V\psi\|_{L^{2}}+\||\partial_{+}(V\psi)|_{-2}\|_{L^{2}}\\
 & \lesssim\|y^{-3}\mathbf{1}_{y\geq1}V\psi\|_{L^{2}}+\||V\partial_{+}\psi|_{-2}\|_{L^{2}}+\||(\partial_{y}V)\psi|_{-2}\|_{L^{2}}.
\end{align*}
Thus \eqref{eq:LL-Hdot3-est} and \eqref{eq:LL-Hdot3-est-m-geq-3}
follow from the claim \eqref{eq:deg-lin-claim}. Similarly, using
Lemma \ref{lem:ComparisonH5H5Appendix}, we see that \eqref{eq:LL-Hdot5-est}
follows from the claim \eqref{eq:deg-lin-claim-1}.

It remains to show the claims \eqref{eq:deg-lin-claim} and \eqref{eq:deg-lin-claim-1}.
We divide into three cases: when $\psi=\epsilon$, $\psi=P-Q$, or
$\psi=P$.

\uline{Case A:} $\psi=\epsilon$.

In this case, $V=V_{3}[P]-V_{3}[Q]$ or $V=V_{5}[P]-V_{5}[Q]$.

From the pointwise estimates 
\begin{align*}
|y^{-3}\mathbf{1}_{y\geq1}V\epsilon| & \lesssim|V|\cdot y^{-3}\mathbf{1}_{y\geq1}\epsilon,\\
|V\partial_{+}\epsilon|_{-2} & \lesssim|V|_{2}|\partial_{+}\epsilon|_{-2}\lesssim(|V|+y^{-1}\langle y\rangle^{2}|\partial_{y}V|_{1})\cdot|\partial_{+}\epsilon|_{-2},\\
|(\partial_{y}V)\epsilon|_{-2} & \lesssim|\partial_{y}V|_{2}\cdot|\epsilon|_{-2}\lesssim y^{-1}\langle y\rangle^{2}|\partial_{y}V|_{2}\cdot y\langle y\rangle^{-2}|\epsilon|_{-2},
\end{align*}
definition of the $\dot{\mathcal{H}}_{m}^{3}$-norm \eqref{eq:Def-Hdot3-Section2}
and $\dot{\mathcal{H}}_{m}^{5}$-norm \eqref{eq:Def-Hdot5-Section2}\footnote{As one can see in \eqref{eq:Def-Hdot3-Section2}, when $m\in\{1,2\}$,
$\|y^{-1}|\epsilon|_{-2}\|_{L^{2}}$ is not controlled by $\|\epsilon\|_{\dot{\mathcal{H}}_{m}^{3}}$
due to the singularity at the origin. Thus we use $y\langle y\rangle^{-2}|\epsilon|_{-2}$
instead of $y^{-1}|\epsilon|_{-2}$. The price to pay is that we need
to control $y^{-1}\langle y\rangle^{2}|\partial_{y}V|_{2}$, which
is more singular than $y|\partial_{y}V|_{2}$ at the origin.} 
\begin{align*}
\|\langle y\rangle^{-2}\cdot y\langle y\rangle^{-2}|\epsilon|_{-2}\|_{L^{2}} & +\|\langle y\rangle^{-2}\cdot|\partial_{+}\epsilon|_{-2}\|_{L^{2}}\\
 & \lesssim\begin{cases}
\|\epsilon\|_{\dot{\mathcal{H}}_{m,\leq M_{2}}^{3}}+o_{M_{2}\to\infty}(1)\|\epsilon\|_{\dot{\mathcal{H}}_{m}^{3}} & \text{if }m\geq1,\\
\|\epsilon\|_{\dot{\mathcal{H}}_{m}^{5}} & \text{if }m\geq3,
\end{cases}
\end{align*}
\eqref{eq:deg-lin-claim} would follow from the pointwise estimates\footnote{Indeed, one can easily improve the factor $b$ by $|\eta|+b^{2}$.
Notice that the pseudoconformal phase $e^{-ib\frac{y^{2}}{4}}$ plays
no role in the gauge potential, so $V_{k}[P]=V_{k}[Q^{(\eta)}\chi_{b^{-1/2}}]$
for $k\in\{3,5\}$. Thus one can use $|Q^{(\eta)}\chi_{b^{-1/2}}-Q|_{3}\lesssim(|\eta|y^{2}\mathbf{1}_{y\leq b^{-1/2}}+\mathbf{1}_{y\geq b^{-1/2}})Q$
instead of the cruder estimate $|P-Q|_{3}\lesssim by^{2}Q$. However,
we chose to use the crude bound $|P-Q|_{3}\lesssim by^{2}Q$ because
it is better suited for the difference estimate; see the proof of
Lemma \ref{lem:Diff-RLL-Est}.}
\[
|V|+y^{-1}\langle y\rangle^{2}|\partial_{y}V|_{2}\lesssim b\langle y\rangle^{-2}.
\]
Similarly, \eqref{eq:deg-lin-claim-1} would follow from the pointwise
estimates 
\[
|V|+y^{-1}\langle y\rangle^{2}|\partial_{y}V|_{4}\lesssim b\langle y\rangle^{-2}.
\]
Thus it suffices to show the latter.

Starting from the pointwise bound $|P-Q|_{5}\lesssim by^{2}Q$, we
have 
\[
|P^{2}-Q^{2}|_{5}\lesssim by^{2}Q^{2}\lesssim by^{2}\langle y\rangle^{-4}.
\]
Thus 
\begin{align*}
 & |V_{3}[P]-V_{3}[Q]|+y^{-1}\langle y\rangle^{2}|\partial_{y}(V_{3}[P]-V_{3}[Q])|_{4}\\
 & \lesssim y^{-2}\langle y\rangle^{2}(|P^{2}-Q^{2}|_{5}+\tfrac{1}{y^{2}}{\textstyle \int_{0}^{y}}|P^{2}-Q^{2}|y'dy')+{\textstyle \int_{y}^{\infty}}|P^{2}-Q^{2}|\tfrac{dy'}{y'}\\
 & \lesssim b\langle y\rangle^{-2}.
\end{align*}
Estimates of $V_{5}[P]-V_{5}[Q]$ follow from combining the bounds
\[
|A_{\theta}[P]|_{5}+|A_{\theta}[Q]|_{5}+|y^{2}P^{2}|_{5}+|y^{2}Q^{2}|_{5}\lesssim1
\]
and the previous estimates of $V_{3}[P]-V_{3}[Q]$ with the help of
Leibniz's rule.

\uline{Case B:} $\psi=P-Q$.

According to \eqref{eq:R-tilde-L-L}, $V$ is either $2V_{3}[Q,\epsilon]$,
$2V_{5}[Q,\epsilon,Q,Q]$, or $2V_{5}[Q,Q,Q,\epsilon]$.

Substituting the pointwise bounds $|P-Q|_{5}\lesssim by^{2}Q$ and
$Q\lesssim y^{m}\langle y\rangle^{-(2m+2)}$ into $\psi$ of \eqref{eq:deg-lin-claim}
and \eqref{eq:deg-lin-claim-1}, it suffices to show 
\begin{align*}
\|\langle y\rangle^{-4}|V|_{3}\|_{L^{2}} & \lesssim\|\epsilon\|_{\dot{\mathcal{H}}_{m,\leq M_{2}}^{3}}+o_{M_{2}\to\infty}(1)\|\epsilon\|_{\dot{\mathcal{H}}_{m}^{3}}\quad\text{if }m\geq1,\\
\|\langle y\rangle^{-4}|V|_{5}\|_{L^{2}} & \lesssim\|\epsilon\|_{\dot{\mathcal{H}}_{m,\leq M_{2}}^{5}}+o_{M_{2}\to\infty}(1)\|\epsilon\|_{\dot{\mathcal{H}}_{m}^{5}}\quad\text{if }m\geq3.
\end{align*}

We first estimate $V_{3}$: 
\begin{align*}
 & \|\langle y\rangle^{-4}|V_{3}[Q,\epsilon]|_{3}\|_{L^{2}}\\
 & \lesssim\|\langle y\rangle^{-4}(|Q\epsilon|_{3}+|y^{-2}A_{\theta}[Q,\epsilon]|_{3}+|A_{0,3}[Q,\epsilon]|_{3})\|_{L^{2}}\\
 & \lesssim\|\langle y\rangle^{-4}|Q\epsilon|_{3}\|_{L^{2}}+\|y^{-2}Q\epsilon\|_{L^{1}}\\
 & \lesssim\|y^{-2}\langle y\rangle^{\frac{3}{2}}|Q\epsilon|_{3}\|_{L^{2}}\lesssim\|\epsilon\|_{\dot{\mathcal{H}}_{m,\leq M_{2}}^{3}}+o_{M_{2}\to\infty}(1)\|\epsilon\|_{\dot{\mathcal{H}}_{m}^{3}},
\end{align*}
where in the second inequality we put $\langle y\rangle^{-4}$ inside
of $A_{\theta}$ and use $\|y^{-2}\int_{0}^{y}fy'dy'\|_{L^{2}}\lesssim\|f\|_{L^{2}}$
for $A_{\theta}$-term; we put $\langle y\rangle^{-4}\in L^{2}$ and
$\|\int_{y}^{\infty}f\frac{dy'}{y'}\|_{L^{\infty}}\lesssim\|y^{-2}f\|_{L^{1}}$
for $A_{0,3}$-term. Similarly, 
\begin{align*}
 & \|\langle y\rangle^{-4}|V_{3}[Q,\epsilon]|_{5}\|_{L^{2}}\\
 & \lesssim\|\langle y\rangle^{-4}|Q\epsilon|_{5}\|_{L^{2}}+\|y^{-2}Q\epsilon\|_{L^{1}}\\
 & \lesssim\|y^{-2}\langle y\rangle^{\frac{3}{2}}|Q\epsilon|_{5}\|_{L^{2}}\lesssim\|\epsilon\|_{\dot{\mathcal{H}}_{m,\leq M_{2}}^{5}}+o_{M_{2}\to\infty}(1)\|\epsilon\|_{\dot{\mathcal{H}}_{m}^{5}}\quad\text{if }m\geq3.
\end{align*}

We turn to $V_{5}$. For the $y^{-2}A_{\theta}^{2}$-type, we use
$|A_{\theta}[Q]|_{k}\lesssim1$ and the above $\langle y\rangle^{-4}|y^{-2}A_{\theta}[Q,\epsilon]|_{k}$
bound for $k\in\{3,5\}$ to get 
\begin{align*}
\|\langle y\rangle^{-4}|y^{-2}A_{\theta}[Q,\epsilon]A_{\theta}[Q]|_{3}\|_{L^{2}} & \lesssim\|\epsilon\|_{\dot{\mathcal{H}}_{m,\leq M_{2}}^{3}}+o_{M_{2}\to\infty}(1)\|\epsilon\|_{\dot{\mathcal{H}}_{m}^{3}}\quad\text{if }m\geq1,\\
\|\langle y\rangle^{-4}|y^{-2}A_{\theta}[Q,\epsilon]A_{\theta}[Q]|_{5}\|_{L^{2}} & \lesssim\|\epsilon\|_{\dot{\mathcal{H}}_{m,\leq M_{2}}^{5}}+o_{M_{2}\to\infty}(1)\|\epsilon\|_{\dot{\mathcal{H}}_{m}^{5}}\quad\text{if }m\geq3.
\end{align*}
For the $A_{0,5}$-type, 
\begin{align*}
 & \|\langle y\rangle^{-4}|A_{0,5}[Q,\epsilon,Q,Q]+A_{0,5}[Q,Q,Q,\epsilon]|_{3}\|_{L^{2}}\\
 & \lesssim\|y^{-2}(A_{\theta}[Q,\epsilon]Q^{2}+A_{\theta}[Q]Q\epsilon)\|_{L^{1}}+\|\langle y\rangle^{-4}|A_{\theta}[Q,\epsilon]Q^{2}+A_{\theta}[Q]Q\epsilon|_{2}\|_{L^{2}}\\
 & \lesssim\|\langle y\rangle^{-\frac{1}{2}}|Q\epsilon|_{2}\|_{L^{2}}\lesssim\|\epsilon\|_{\dot{\mathcal{H}}_{m,\leq M_{2}}^{3}}+o_{M_{2}\to\infty}(1)\|\epsilon\|_{\dot{\mathcal{H}}_{m}^{3}}\quad\text{if }m\geq1
\end{align*}
and similarly 
\begin{align*}
 & \|\langle y\rangle^{-4}|A_{0,5}[Q,\epsilon,Q,Q]+A_{0,5}[Q,Q,Q,\epsilon]|_{5}\|_{L^{2}}\\
 & \lesssim\|\langle y\rangle^{-\frac{1}{2}}|Q\epsilon|_{4}\|_{L^{2}}\lesssim\|\epsilon\|_{\dot{\mathcal{H}}_{m,\leq M_{2}}^{5}}+o_{M_{2}\to\infty}(1)\|\epsilon\|_{\dot{\mathcal{H}}_{m}^{5}}\quad\text{if }m\geq3.
\end{align*}

\uline{Case C:} $\psi=P$.

In this case, $V$ contains $P-Q$-part and $\epsilon$. According
to \eqref{eq:R-tilde-L-L}, $V$ is either $2\tilde V_{3}[P-Q,\epsilon]$,
$2(\tilde V_{5}[P,\epsilon,P,P]-\tilde V_{5}[Q,\epsilon,Q,Q])$, or
$2(\tilde V_{5}[P,P,P,\epsilon]-\tilde V_{5}[Q,Q,Q,\epsilon])$. Notice
that $A_{0,k}$ is replaced by $\tilde A_{0,k}$ for $k\in\{3,5\}$.
Using $|P|_{3}\lesssim Q$, it suffices to show 
\begin{align*}
\|y^{-2}\langle y\rangle^{-4}|V|_{3}\|_{L^{2}} & \lesssim b(\|\epsilon\|_{\dot{\mathcal{H}}_{m,\leq M_{2}}^{3}}+o_{M_{2}\to\infty}(1)\|\epsilon\|_{\dot{\mathcal{H}}_{m}^{3}})\quad\text{if }m\geq1,\\
\|y^{-2}\langle y\rangle^{-4}|V|_{5}\|_{L^{2}} & \lesssim b(\|\epsilon\|_{\dot{\mathcal{H}}_{m,\leq M_{2}}^{5}}+o_{M_{2}\to\infty}(1)\|\epsilon\|_{\dot{\mathcal{H}}_{m}^{5}})\quad\text{if }m\geq3.
\end{align*}
One can observe that there is an extra factor $y^{-2}$ in the LHS
compared to that of Case B. Since $V$ contains $P-Q$-parts, we can
exploit the degeneracy $|P-Q|_{3}\lesssim by^{2}Q$ to handle the
$y^{-2}$ factor. We can estimate $\tilde V_{3}$ by 
\begin{align*}
 & y^{-2}\langle y\rangle^{-4}|\tilde V_{3}[P-Q,\epsilon]|_{3}\\
 & \lesssim y^{-2}\langle y\rangle^{-4}(|(P-Q)\epsilon|_{3}+|y^{-2}A_{\theta}[P-Q,\epsilon]|+|\tilde A_{0,3}[P-Q,\epsilon]|)\\
 & \lesssim b\cdot\langle y\rangle^{-4}(|Q\epsilon|_{3}+\tfrac{1}{y^{2}}{\textstyle \int_{0}^{y}}Q|\epsilon|y'dy'+{\textstyle \int_{0}^{y}}Q|\epsilon|\tfrac{dy'}{y'}).
\end{align*}
In the last inequality, we pulled out the $by^{2}$ factor from $\int_{0}^{y}$-integrals
in $A_{\theta}$ and $\tilde A_{0,3}$. This was possible since we
replaced $A_{0,3}$ by $\tilde A_{0,3}$. The last line is essentially
same as $b\langle y\rangle^{-4}|V_{3}[Q,\epsilon]|_{3}$, though we
replaced $A_{0,3}$ by $\tilde A_{0,3}$. Now we can apply the arguments
in Case B to estimate $L^{2}$-norm of the last line. A similar argument
applies to $\tilde V_{5}$ and also $|V|_{5}$.
\end{proof}
We now turn to $R_{\mathrm{NL}}$.
\begin{lem}[Various estimates for nonlinear terms]
\label{lem:nonlinear}We decompose $R_{\mathrm{NL}}$ as 
\[
R_{\mathrm{NL}}=-\theta_{\mathrm{NL}}P+\tilde R_{\mathrm{NL}}
\]
such that the following estimates hold.
\begin{itemize}
\item (Estimate of $\theta_{\mathrm{NL}}$) We have 
\begin{equation}
|\theta_{\mathrm{NL}}|\lesssim\|\epsilon\|_{\dot{H}_{m}^{1}}^{2}.\label{eq:NL-theta-est}
\end{equation}
\item ($\dot{\mathcal{H}}_{m}^{3}$-estimate of $\tilde R_{\mathrm{NL}}$)
We have 
\begin{equation}
\|\tilde R_{\mathrm{NL}}\|_{\dot{\mathcal{H}}_{m}^{3}}\lesssim\|\epsilon\|_{\dot{\mathcal{H}}_{m}^{3}}^{2}+\|\epsilon\|_{\dot{H}_{m}^{1}}^{2}\|\epsilon\|_{\dot{\mathcal{H}}_{m}^{3}}+\|\epsilon\|_{\dot{H}_{m}^{1}}^{5}.\label{eq:NL-Hdot3-est}
\end{equation}
In particular, 
\begin{align}
\sup_{k\in\{1,2,3,4\}}|(i\tilde R_{\mathrm{NL}},\mathcal{Z}_{k})_{r}| & \lesssim M^{C}(\|\epsilon\|_{\dot{\mathcal{H}}_{m}^{3}}^{2}+\|\epsilon\|_{\dot{H}_{m}^{1}}^{2}\|\epsilon\|_{\dot{\mathcal{H}}_{m}^{3}}+\|\epsilon\|_{\dot{H}_{m}^{1}}^{5}),\label{eq:NL-mod-est}\\
\||A_{Q}L_{Q}i\tilde R_{\mathrm{NL}}|_{-1}\|_{L^{2}} & \lesssim\|\epsilon\|_{\dot{\mathcal{H}}_{m}^{3}}^{2}+\|\epsilon\|_{\dot{H}_{m}^{1}}^{2}\|\epsilon\|_{\dot{\mathcal{H}}_{m}^{3}}+\|\epsilon\|_{\dot{H}_{m}^{1}}^{5},\label{eq:NL-E3-est}
\end{align}
where the implicit constant $M^{C}$ comes from the bounds of $\mathcal{Z}_{k}$
\eqref{eq:Orthog-Hdot3-est}.
\end{itemize}
In addition, if $m\geq3$, we have estimates in terms of the higher
Sobolev norm $\dot{\mathcal{H}}_{m}^{5}$:
\begin{itemize}
\item ($\dot{\mathcal{H}}_{m}^{5}$-estimate of $\tilde R_{\mathrm{NL}}$)
We have
\begin{equation}
\|\tilde R_{\mathrm{NL}}\|_{\dot{\mathcal{H}}_{m}^{5}}\lesssim\|\epsilon\|_{\dot{\mathcal{H}}_{m}^{5}}(\|\epsilon\|_{\dot{H}_{m}^{3}}+\|\epsilon\|_{\dot{H}_{m}^{1}}^{2})+\|\epsilon\|_{\dot{H}_{m}^{3}}\|\epsilon\|_{\dot{H}_{m}^{1}}^{4}\label{eq:NL-Hdot5-est}
\end{equation}
In particular, 
\begin{align}
\sup_{k\in\{1,2,3,4\}}|(i\tilde R_{\mathrm{NL}},\mathcal{Z}_{k})_{r}| & \lesssim M^{C}(\|\epsilon\|_{\dot{\mathcal{H}}_{m}^{5}}(\|\epsilon\|_{\dot{H}_{m}^{3}}+\|\epsilon\|_{\dot{H}_{m}^{1}}^{2})+\|\epsilon\|_{\dot{H}_{m}^{3}}\|\epsilon\|_{\dot{H}_{m}^{1}}^{4}),\label{eq:NL-mod-est-m-geq-3}\\
\||A_{Q}L_{Q}i\tilde R_{\mathrm{NL}}|_{-3}\|_{L^{2}} & \lesssim\|\epsilon\|_{\dot{\mathcal{H}}_{m}^{5}}(\|\epsilon\|_{\dot{H}_{m}^{3}}+\|\epsilon\|_{\dot{H}_{m}^{1}}^{2})+\|\epsilon\|_{\dot{H}_{m}^{3}}\|\epsilon\|_{\dot{H}_{m}^{1}}^{4}.\label{eq:NL-E5-est}
\end{align}
\end{itemize}
\end{lem}

\begin{rem}
As will be detailed in Section \ref{subsec:energy-identity}, we need
bounds of the form $\||A_{Q}L_{Q}i\tilde R_{\mathrm{NL}}|_{-1}\|_{L^{2}}\ll b\|\epsilon_{3}\|_{L^{2}}$
after applying the bootstrap hypotheses on $\epsilon$. Similarly,
we need $\||A_{Q}A_{Q}^{\ast}A_{Q}L_{Q}i\tilde R_{\mathrm{NL}}|_{-1}\|_{L^{2}}\ll b\|\epsilon_{5}\|_{L^{2}}$
to close our bootstrap for $\|\epsilon_{5}\|_{L^{2}}$.
\end{rem}

\begin{rem}
In contrast to $iR_{\mathrm{L-L}}$, the estimates \eqref{eq:gen-null-rel}
and \eqref{eq:NL-theta-est} \emph{do not} imply that \eqref{eq:NL-E3-est}
for $iR_{\mathrm{NL}}$ instead of $i\tilde R_{\mathrm{NL}}$ holds.
In other words, it seems necessary to extract the phase correction
term $\theta_{\mathrm{NL}}iP$ from $iR_{\mathrm{NL}}$ to have \eqref{eq:NL-E3-est}.
\end{rem}

\begin{proof}
\textbf{Step 1.} Decomposition $R_{\mathrm{NL}}=-\theta_{\mathrm{NL}}P+\tilde R_{\mathrm{NL}}$.

To derive the decomposition, we proceed similarly as before. As $R_{\mathrm{NL}}$
collects the quadratic and higher terms in $\epsilon$ of the nonlinearity
$\mathcal{N}(P+\epsilon)$, we have 
\begin{align*}
R_{\mathrm{NL}} & =\sum_{k\in\{3,5\}}\sum_{\substack{\psi_{1},\dots,\psi_{k}\in\{P,\epsilon\}\\
\#\{j:\psi_{j}=\epsilon\}\geq2
}
}V_{k}[\psi_{1},\dots,\psi_{k-1}]\psi_{k}\\
 & =\sum_{k\in\{3,5\}}\sum_{\substack{\psi_{1},\dots,\psi_{k-1}\in\{P,\epsilon\}\\
\#\{j:\psi_{j}=\epsilon\}\geq1
}
}V_{k}[\psi_{1},\dots,\psi_{k-1}]\epsilon+\sum_{k\in\{3,5\}}\sum_{\substack{\psi_{1},\dots,\psi_{k-1}\in\{P,\epsilon\}\\
\#\{j:\psi_{j}=\epsilon\}\geq2
}
}V_{k}[\psi_{1},\dots,\psi_{k-1}]P.
\end{align*}
We focus on the term that ends in $P$. As in $R_{\mathrm{L-L}}$,
we capture the phase correction term $\theta_{\mathrm{NL}}P$. We
replace the outermost integral $\int_{y}^{\infty}$ of $A_{0}$ by
$\int_{0}^{y}$ at cost of introducing $\theta_{\mathrm{NL}}$: 
\begin{align*}
 & R_{\mathrm{NL}}+\theta_{\mathrm{NL}}P\\
 & =\sum_{k\in\{3,5\}}\sum_{\substack{\psi_{1},\dots,\psi_{k-1}\in\{P,\epsilon\}\\
\#\{j:\psi_{j}=\epsilon\}\geq1
}
}V_{k}[\psi_{1},\dots,\psi_{k-1}]\epsilon+\sum_{k\in\{3,5\}}\sum_{\substack{\psi_{1},\dots,\psi_{k-1}\in\{P,\epsilon\}\\
\#\{j:\psi_{j}=\epsilon\}\geq2
}
}\tilde V_{k}[\psi_{1},\dots,\psi_{k-1}]P,
\end{align*}
where 
\begin{align*}
\theta_{\mathrm{NL}} & \coloneqq{\textstyle \int_{0}^{\infty}}m|\epsilon|^{2}\tfrac{dy}{y}+\sum_{\substack{\psi_{1},\dots,\psi_{4}\in\{P,\epsilon\}\\
\#\{j:\psi_{j}=\epsilon\}\geq2
}
}{\textstyle \int_{0}^{\infty}}A_{\theta}[\psi_{1},\psi_{2}]\Re(\overline{\psi_{3}}\psi_{4})\tfrac{dy}{y}\\
 & ={\textstyle \int_{0}^{\infty}}(m+A_{\theta}[P+\epsilon])|\epsilon|^{2}\tfrac{dy}{y}\\
 & \quad+2{\textstyle \int_{0}^{\infty}}(A_{\theta}[P+\epsilon]-A_{\theta}[P])\Re(\overline{P}\epsilon)\tfrac{dy}{y}+{\textstyle \int_{0}^{\infty}}A_{\theta}[\epsilon]|P|^{2}\tfrac{dy}{y}.
\end{align*}

\ 

\textbf{Step 2.} The $\theta_{\mathrm{NL}}$-estimate \eqref{eq:NL-theta-est}.
\begin{align*}
|{\textstyle \int_{0}^{\infty}}(m+A_{\theta}[P+\epsilon])|\epsilon|^{2}\tfrac{dy}{y}| & \lesssim(1+\|P+\epsilon\|_{L^{2}}^{2})\|\tfrac{1}{y}\epsilon\|_{L^{2}}^{2}\lesssim\|\epsilon\|_{\dot{H}_{m}^{1}}^{2},\\
|{\textstyle \int_{0}^{\infty}}(A_{\theta}[P+\epsilon]-A_{\theta}[P])\Re(\overline{P}\epsilon)\tfrac{dy}{y}| & \lesssim\|\tfrac{1}{y^{2}}(A_{\theta}[P+\epsilon]-A_{\theta}[P])\|_{L^{2}}\|P\epsilon\|_{L^{2}}\lesssim\|\epsilon\|_{\dot{H}_{m}^{1}}^{2},\\
|{\textstyle \int_{0}^{\infty}}A_{\theta}[\epsilon]|P|^{2}\tfrac{dy}{y}| & \lesssim|{\textstyle \int_{0}^{\infty}}A_{\theta}[\tfrac{1}{y}\epsilon]Q^{2}ydy|\lesssim\|\epsilon\|_{\dot{H}_{m}^{1}}^{2}.
\end{align*}
This completes the proof of \eqref{eq:NL-theta-est}.

\ 

\textbf{Step 3.} The estimates of $\tilde R_{\mathrm{NL}}$.

For the $\dot{\mathcal{H}}_{m}^{3}$-estimate, as in the proof of
Lemma \ref{lem:degenerate-linear}, we need to show 
\begin{align}
\|y^{-3}\mathbf{1}_{y\geq1}V\psi\|_{L^{2}}+\||V\partial_{+}\psi|_{-2}\|_{L^{2}} & +\||(\partial_{y}V)\psi|_{-2}\|_{L^{2}}\label{eq:R-NL-Hdot3-claim1}\\
 & \lesssim\|\epsilon\|_{\dot{\mathcal{H}}_{m}^{3}}^{2}+\|\epsilon\|_{\dot{H}_{m}^{1}}^{2}\|\epsilon\|_{\dot{\mathcal{H}}_{m}^{3}}+\|\epsilon\|_{\dot{H}_{m}^{1}}^{5}.\nonumber 
\end{align}
In what follows, we will show a stronger estimate 
\begin{equation}
\||V\psi|_{-3}\|_{L^{2}}\lesssim\|\epsilon\|_{\dot{\mathcal{H}}_{m}^{3}}^{2}+\|\epsilon\|_{\dot{H}_{m}^{1}}^{2}\|\epsilon\|_{\dot{\mathcal{H}}_{m}^{3}}+\|\epsilon\|_{\dot{H}_{m}^{1}}^{5}\label{eq:R-NL-Hdot3-claim2}
\end{equation}
when $V$ is not of $A_{0}$-type, i.e. $V$ is of $|\phi|^{2}$,
$\frac{1}{y^{2}}A_{\theta}$, and $\frac{1}{y^{2}}A_{\theta}^{2}$.
For $V$ of $A_{0}$-type, we will directly show \eqref{eq:R-NL-Hdot3-claim1}.\footnote{In this case, we cannot hope to estimate $\||A_{0,3}\phi|_{-3}\|_{L^{2}}$.
Indeed, $A_{0,3}$ can have nonzero value at the origin, so we can
only expect $|A_{0,3}\phi|\lesssim y^{m}$ near the origin, even if
$\phi$ is smooth $m$-equivariant. Thus $y^{-3}A_{0,3}\phi$ may
not belong to $L^{2}$ when $m\in\{1,2\}$, due to the singularity
at the origin. The same discussion applies for $R_{\mathrm{L-L}}$
\eqref{eq:deg-lin-claim}.}

Similarly, for the $\dot{\mathcal{H}}_{m}^{5}$-estimate when $m\geq3$,
we need to show 
\begin{align}
\|y^{-5}\mathbf{1}_{y\geq1}V\psi\|_{L^{2}} & +\||V\partial_{+}\psi|_{-4}\|_{L^{2}}+\||(\partial_{y}V)\psi|_{-4}\|_{L^{2}}\label{eq:R-NL-Hdot5-claim1}\\
 & \lesssim\|\epsilon\|_{\dot{\mathcal{H}}_{m}^{5}}(\|\epsilon\|_{\dot{H}_{m}^{3}}+\|\epsilon\|_{\dot{H}_{m}^{1}}^{2})+\|\epsilon\|_{\dot{H}_{m}^{3}}\|\epsilon\|_{\dot{H}_{m}^{1}}^{4}\quad\text{if }m\geq3.\nonumber 
\end{align}
For $V$ of $|\phi|^{2}$, $\frac{1}{y^{2}}A_{\theta}$, and $\frac{1}{y^{2}}A_{\theta}^{2}$-types,
we will show a stronger estimate 
\begin{equation}
\||V\psi|_{-5}\|_{L^{2}}\lesssim\|\epsilon\|_{\dot{\mathcal{H}}_{m}^{5}}(\|\epsilon\|_{\dot{H}_{m}^{3}}+\|\epsilon\|_{\dot{H}_{m}^{1}}^{2})+\|\epsilon\|_{\dot{H}_{m}^{3}}\|\epsilon\|_{\dot{H}_{m}^{1}}^{4}\quad\text{if }m\geq3.\label{eq:R-NL-Hdot5-claim2}
\end{equation}

In the sequel, we prove the estimates \eqref{eq:R-NL-Hdot3-claim1}-\eqref{eq:R-NL-Hdot5-claim2}.
We proceed by types of nonlinearity. In order to reduce redundant
expressions, we use index $\ell\in\{3,5\}$ to denote the regularity
of $\dot{\mathcal{H}}_{m}^{\ell}$. When $\ell=5$, we further assume
$m\geq3$.

\uline{Case A:} $|\phi|^{2}\phi$-type.

It suffices to show \eqref{eq:R-NL-Hdot3-claim2} and \eqref{eq:R-NL-Hdot5-claim2},
i.e. 
\begin{align*}
\||P\epsilon^{2}|_{-\ell}\|_{L^{2}}+\||\epsilon^{3}|_{-\ell}\|_{L^{2}} & \lesssim\|\epsilon\|_{\dot{\mathcal{H}}_{m}^{\ell}}(\|\epsilon\|_{\dot{\mathcal{H}}_{m}^{3}}+\|\epsilon\|_{\dot{H}_{m}^{1}}^{2}).
\end{align*}
For $P\epsilon^{2}$, using $|P|_{5}\lesssim Q$, 
\begin{align*}
|P\epsilon^{2}|_{-3} & \lesssim Q|\partial_{yyy}\epsilon||\epsilon|+|\epsilon|_{-2}\cdot y^{-1}Q|\epsilon|_{1},\\
|P\epsilon^{2}|_{-5} & \lesssim Q|\partial_{yyyyy}\epsilon||\epsilon|+|\epsilon|_{-4}\cdot y^{-1}Q|\epsilon|_{2},
\end{align*}
thus 
\begin{align*}
\||P\epsilon^{2}|_{-\ell}\|_{L^{2}} & \lesssim(\|\partial_{y}^{\ell}\epsilon\|_{L^{2}}+\|y\langle y\rangle^{-2}|\epsilon|_{-(\ell-1)}\|_{L^{2}})\|y^{-2}\langle y\rangle^{2}Q|\epsilon|_{2}\|_{L^{\infty}}\lesssim\|\epsilon\|_{\dot{\mathcal{H}}_{m}^{\ell}}\|\epsilon\|_{\dot{\mathcal{H}}_{m}^{3}},
\end{align*}
where we estimated the $L^{\infty}$-term using \eqref{eq:L-infty-weighted-hardy}.

For $\epsilon^{3}$, we have 
\begin{align*}
|\epsilon^{3}|_{-3} & \lesssim|\partial_{yyy}\epsilon||\epsilon|^{2}+|\partial_{yy}\epsilon||\epsilon|_{-1}|\epsilon|+|\epsilon|_{-1}^{3},\\
|\epsilon^{3}|_{-5} & \lesssim|\partial_{yyyyy}\epsilon||\epsilon|^{2}+|\partial_{yyyy}\epsilon||\epsilon|_{-1}|\epsilon|+|\epsilon|_{-3}(|\epsilon|_{-2}|\epsilon|+|\epsilon|_{-1}^{2}).
\end{align*}
For $|\epsilon^{3}|_{-3}$, we estimate 
\begin{align*}
\||\epsilon^{3}|_{-3}\|_{L^{2}} & \lesssim\|\partial_{yyy}\epsilon\|_{L^{2}}\|\epsilon\|_{L^{\infty}}^{2}+\|\partial_{yy}\epsilon\|_{L^{2}}\||\epsilon|_{-1}\|_{L^{\infty}}\|\epsilon\|_{L^{\infty}}+\||\epsilon|_{-1}\|_{L^{\infty}}^{2}\||\epsilon|_{-1}\|_{L^{2}}\\
 & \lesssim\|\epsilon\|_{\dot{\mathcal{H}}_{m}^{3}}\|\epsilon\|_{\dot{H}_{m}^{1}}^{2},
\end{align*}
where in the last inequality we used \eqref{eq:HardySobolevSection2}
and \eqref{eq:interpolation}. For $|\epsilon^{3}|_{-5}$, we estimate
\begin{align*}
\||\epsilon^{3}|_{-5}\|_{L^{2}} & \lesssim\|\partial_{yyyyy}\epsilon\|_{L^{2}}\|\epsilon\|_{L^{\infty}}^{2}+\|\partial_{yyyy}\epsilon\|_{L^{2}}\||\epsilon|_{-1}\|_{L^{\infty}}\|\epsilon\|_{L^{\infty}}\\
 & \qquad\qquad+\||\epsilon|_{-3}\|_{L^{2}}\||\epsilon|_{-2}|\epsilon|+|\epsilon|_{-1}^{2}\|_{L^{\infty}}\lesssim\|\epsilon\|_{\dot{\mathcal{H}}_{m}^{5}}\|\epsilon\|_{\dot{H}_{m}^{1}}^{2},
\end{align*}
where in the last inequality we used \eqref{eq:interpolation}, \eqref{eq:interpolation-H5},
and the interpolation of $\dot{H}_{m}^{3}$ by $\dot{H}_{m}^{1}$
and $\dot{\mathcal{H}}_{m}^{5}$.

\uline{Case B:} $\frac{A_{\theta}}{y^{2}}\phi$-type.

It suffices to show \eqref{eq:R-NL-Hdot3-claim2} and \eqref{eq:R-NL-Hdot5-claim2}
\begin{align*}
\||\tfrac{A_{\theta}}{y^{2}}\phi|_{-\ell}\|_{L^{2}} & \lesssim\|\epsilon\|_{\dot{\mathcal{H}}_{m}^{\ell}}(\|\epsilon\|_{\dot{\mathcal{H}}_{m}^{3}}+\|\epsilon\|_{\dot{H}_{m}^{1}}^{2}).
\end{align*}
Since 
\[
|\tfrac{1}{y^{2}}A_{\theta}[\psi_{1},\psi_{2}]\psi_{3}|_{-\ell}\lesssim\tfrac{1}{y^{2}}A_{\theta}[\psi_{1},\psi_{2}]|\psi_{3}|_{-\ell}+\tfrac{1}{y}|\psi_{1}\psi_{2}\psi_{3}|_{-(\ell-1)}
\]
and $\tfrac{1}{y}|\psi_{1}\psi_{2}\psi_{3}|_{-(\ell-1)}\lesssim|\psi_{1}\psi_{2}\psi_{3}|_{-\ell}$
is treated in Case A, it suffices to show 
\begin{align*}
\|\tfrac{1}{y^{2}}A_{\theta}[\psi_{1},\psi_{2}]|\psi_{3}|_{-3}\|_{L^{2}} & \lesssim\|\epsilon\|_{\dot{\mathcal{H}}_{m}^{3}}(\|\epsilon\|_{\dot{\mathcal{H}}_{m}^{3}}+\|\epsilon\|_{\dot{H}_{m}^{1}}^{2}),\\
\|\tfrac{1}{y^{2}}A_{\theta}[\psi_{1},\psi_{2}]|\psi_{3}|_{-5}\|_{L^{2}} & \lesssim\|\epsilon\|_{\dot{\mathcal{H}}_{m}^{5}}(\|\epsilon\|_{\dot{H}_{m}^{3}}+\|\epsilon\|_{\dot{H}_{m}^{1}}^{2}).
\end{align*}
If $\psi_{3}=\epsilon$, then by \eqref{eq:L-infty-weighted-hardy}
and \eqref{eq:HardySobolevSection2} we have 
\begin{align*}
 & \|\tfrac{1}{y^{2}}A_{\theta}[\psi_{1},\psi_{2}]|\epsilon|_{-\ell}\|_{L^{2}}\\
 & \lesssim\|y^{-4}\langle y\rangle^{2}A_{\theta}[\psi_{1},\psi_{2}]\|_{L^{\infty}}\|y^{2}\langle y\rangle^{-2}|\epsilon|_{-\ell}\|_{L^{2}}\\
 & \lesssim\|y^{-2}\langle y\rangle^{2}\psi_{1}\psi_{2}\|_{L^{\infty}}\|\epsilon\|_{\dot{\mathcal{H}}_{m}^{\ell}}\\
 & \lesssim\|\epsilon\|_{\dot{\mathcal{H}}_{m}^{\ell}}\cdot\begin{cases}
\|y^{-2}\langle y\rangle^{2}Q\epsilon\|_{L^{\infty}} & \text{if }\{\psi_{1},\psi_{2}\}=\{P,\epsilon\},\\
\|y^{-1}\langle y\rangle\epsilon\|_{L^{\infty}}^{2} & \text{if }\psi_{1}=\psi_{2}=\epsilon,
\end{cases}\\
 & \lesssim\|\epsilon\|_{\dot{\mathcal{H}}_{m}^{\ell}}(\|\epsilon\|_{\dot{\mathcal{H}}_{m}^{3}}+\|\epsilon\|_{\dot{H}_{m}^{1}}^{2}).
\end{align*}
Note that we used $m\geq3$ when $\ell=5$ to get $\|y^{2}\langle y\rangle^{-2}|\epsilon|_{-5}\|_{L^{2}}\lesssim\|\epsilon\|_{\dot{\mathcal{H}}_{m}^{5}}$.
If $\psi_{3}=P$, then by $|P|_{-\ell}\lesssim y^{-2}\langle y\rangle^{-(1+\ell)}$
(here we also used $m\geq3$ when $\ell=5$ at the origin) and \eqref{eq:L-infty-weighted-hardy},
we have 
\begin{align*}
\|\tfrac{1}{y^{2}}A_{\theta}[\epsilon]|P|_{-\ell}\|_{L^{2}} & \lesssim\|y^{-2}\langle y\rangle^{-(1+\ell)}\epsilon^{2}\|_{L^{2}}\\
 & \lesssim\|y^{-1}\langle y\rangle^{-\ell}\epsilon\|_{L^{2}}\|y^{-1}\langle y\rangle^{-1}\epsilon\|_{L^{\infty}}\lesssim\|\epsilon\|_{\dot{\mathcal{H}}_{m}^{\ell}}\|\epsilon\|_{\dot{\mathcal{H}}_{m}^{3}}.
\end{align*}

\uline{Case C:} $A_{0,3}\phi$-type.

Here, when $\psi=P$, by subtracting $\theta_{\mathrm{NL}}P$, $A_{0,3}$
is replaced by $\tilde A_{0,3}$. In other words, $V\psi$ is of the
form $A_{0,3}[\psi_{1},\psi_{2}]\epsilon$ or $\tilde A_{0,3}[\epsilon]P$.

We directly show \eqref{eq:R-NL-Hdot3-claim1} and \eqref{eq:R-NL-Hdot5-claim1}.
We note that if $\partial_{y}$ hits $A_{0,3}$ (or $\tilde A_{0,3}$),
then it reduces to the $\tfrac{1}{y}|\phi|^{2}\phi$-type, which is
already estimated in Case A. Thus we may assume that $\partial_{y}$
never hits $A_{0,3}$ (or $\tilde A_{0,3}$). It suffices to show
\begin{align*}
\|A_{0,3}[\psi_{1},\psi_{2}](|\partial_{+}\epsilon|_{-(\ell-1)}+\mathbf{1}_{y\geq1}\tfrac{1}{y^{\ell}}\epsilon)\|_{L^{2}} & \lesssim\|\epsilon\|_{\dot{\mathcal{H}}_{m}^{\ell}}(\|\epsilon\|_{\dot{\mathcal{H}}_{m}^{3}}+\|\epsilon\|_{\dot{H}_{m}^{1}}^{2}),\\
\|\tilde A_{0,3}[\epsilon]|P|_{-\ell}\|_{L^{2}} & \lesssim\|\epsilon\|_{\dot{\mathcal{H}}_{m}^{\ell}}\|\epsilon\|_{\dot{\mathcal{H}}_{m}^{3}}.
\end{align*}
For the first one, we estimate by 
\begin{align*}
\|A_{0,3}[\psi_{1},\psi_{2}]\|_{L^{\infty}}\|\epsilon\|_{\dot{\mathcal{H}}_{m}^{\ell}} & \lesssim\|y^{-2}\psi_{1}\psi_{2}\|_{L^{1}}\|\epsilon\|_{\dot{\mathcal{H}}_{m}^{\ell}}\\
 & \lesssim\|\epsilon\|_{\dot{\mathcal{H}}_{m}^{\ell}}\cdot\begin{cases}
\|y^{-2}Q\epsilon\|_{L^{1}} & \text{if }\{\psi_{1},\psi_{2}\}=\{P,\epsilon\},\\
\|y^{-2}\epsilon^{2}\|_{L^{1}} & \text{if }\psi_{1}=\psi_{2}=\epsilon,
\end{cases}\\
 & \lesssim\|\epsilon\|_{\dot{\mathcal{H}}_{m}^{\ell}}(\|\epsilon\|_{\dot{\mathcal{H}}_{m}^{3}}+\|\epsilon\|_{\dot{H}_{m}^{1}}^{2}).
\end{align*}
For the second one, we simply use $|P|_{-\ell}\lesssim y^{-2}\langle y\rangle^{-(1+\ell)}$
and \eqref{eq:L-infty-weighted-hardy} to estimate 
\begin{align*}
\|\tilde A_{0,3}[\epsilon]|P|_{-\ell}\|_{L^{2}} & \lesssim\|y^{-2}\langle y\rangle^{-(1+\ell)}\epsilon^{2}\|_{L^{2}}\\
 & \lesssim\|y^{-1}\langle y\rangle^{-\ell}\epsilon\|_{L^{2}}\|y^{-1}\langle y\rangle^{-1}\epsilon\|_{L^{\infty}}\lesssim\|\epsilon\|_{\dot{\mathcal{H}}_{m}^{\ell}}\|\epsilon\|_{\dot{\mathcal{H}}_{m}^{3}}.
\end{align*}

\uline{Case D:} $\tfrac{1}{y^{2}}A_{\theta}^{2}\phi$-type.

It suffices to show \eqref{eq:R-NL-Hdot3-claim2} and \eqref{eq:R-NL-Hdot5-claim2}
\begin{align*}
\||\tfrac{1}{y^{2}}A_{\theta}[\psi_{1},\psi_{2}]A_{\theta}[\psi_{3},\psi_{4}]\psi_{5}|_{-3}\|_{L^{2}} & \lesssim\|\epsilon\|_{\dot{\mathcal{H}}_{m}^{3}}(\|\epsilon\|_{\dot{\mathcal{H}}_{m}^{3}}+\|\epsilon\|_{\dot{H}_{m}^{1}}^{2})+\|\epsilon\|_{\dot{H}_{m}^{1}}^{5},\\
\||\tfrac{1}{y^{2}}A_{\theta}[\psi_{1},\psi_{2}]A_{\theta}[\psi_{3},\psi_{4}]\psi_{5}|_{-5}\|_{L^{2}} & \lesssim\|\epsilon\|_{\dot{\mathcal{H}}_{m}^{5}}(\|\epsilon\|_{\dot{\mathcal{H}}_{m}^{3}}+\|\epsilon\|_{\dot{H}_{m}^{1}}^{2})+\|\epsilon\|_{\dot{\mathcal{H}}_{m}^{3}}\|\epsilon\|_{\dot{H}_{m}^{1}}^{4}.
\end{align*}
From the pointwise bound
\[
|\tfrac{1}{y^{2}}A_{\theta}^{2}\psi_{5}|_{-\ell}\lesssim\tfrac{1}{y^{2}}A_{\theta}^{2}|\psi_{5}|_{-\ell}+|\tfrac{1}{y}A_{\theta}\psi_{1}\psi_{2}\psi_{5}|_{-(\ell-1)}+|\tfrac{1}{y}A_{\theta}\psi_{3}\psi_{4}\psi_{5}|_{-(\ell-1)},
\]
it suffices to show 
\begin{align*}
\|\tfrac{1}{y^{2}}A_{\theta}^{2}|\psi_{5}|_{-\ell}\|_{L^{2}} & \lesssim\|\epsilon\|_{\dot{\mathcal{H}}_{m}^{\ell}}(\|\epsilon\|_{\dot{\mathcal{H}}_{m}^{3}}+\|\epsilon\|_{\dot{H}_{m}^{1}}^{2}),\\
\||\tfrac{1}{y}A_{\theta}\psi_{3}\psi_{4}\psi_{5}|_{-(\ell-1)}\|_{L^{2}} & \lesssim\|\epsilon\|_{\dot{\mathcal{H}}_{m}^{\ell}}(\|\epsilon\|_{\dot{\mathcal{H}}_{m}^{3}}+\|\epsilon\|_{\dot{H}_{m}^{1}}^{2})+\|\epsilon\|_{\dot{H}_{m}^{(\ell-2)}}\|\epsilon\|_{\dot{H}_{m}^{1}}^{4}.
\end{align*}

We first estimate $\tfrac{1}{y^{2}}A_{\theta}^{2}|\psi_{5}|_{-\ell}$.
By symmetry, it suffices to estimate 
\begin{align*}
 & \|\tfrac{1}{y^{2}}A_{\theta}[\epsilon,\psi_{2}]A_{\theta}[\psi_{3},\psi_{4}]|\epsilon|_{-\ell}\|_{L^{2}},\\
 & \|\tfrac{1}{y^{2}}A_{\theta}[\epsilon]A_{\theta}[\psi_{3},\psi_{4}]|P|_{-\ell}\|_{L^{2}},\\
 & \|\tfrac{1}{y^{2}}A_{\theta}^{2}[P,\epsilon]|P|_{-\ell}\|_{L^{2}}.
\end{align*}
For the the first two terms, we merely estimate $|A_{\theta}[\psi_{3},\psi_{4}]|\lesssim1$
and use the estimate in Case B. For the last term, we estimate 
\[
\|\tfrac{1}{y^{2}}A_{\theta}[P,\epsilon]A_{\theta}[P,\epsilon]|P|_{-\ell}\|_{L^{2}}\lesssim\|P\epsilon\|_{L^{\infty}}^{2}\|y^{2}|P|_{-\ell}\|_{L^{2}}\lesssim\|\epsilon\|_{\dot{\mathcal{H}}_{m}^{\ell}}\|\epsilon\|_{\dot{\mathcal{H}}_{m}^{3}},
\]
where in the last inequality we used \eqref{eq:L-infty-weighted-hardy}
for $\ell=3$ and \eqref{eq:interpolation-H5} for $\ell=5$.

Next, we estimate $|\tfrac{1}{y}A_{\theta}\psi_{3}\psi_{4}\psi_{5}|_{-(\ell-1)}$.
From the pointwise estimate 
\[
|\tfrac{1}{y}A_{\theta}\psi_{3}\psi_{4}\psi_{5}|_{-(\ell-1)}\lesssim|A_{\theta}[\psi_{1},\psi_{2}]|\cdot\tfrac{1}{y}|\psi_{3}\psi_{4}\psi_{5}|_{-(\ell-1)}+|\psi_{1}\psi_{2}\psi_{3}\psi_{4}\psi_{5}|_{-(\ell-2)},
\]
we estimate these two terms.

We estimate $|A_{\theta}[\psi_{1},\psi_{2}]|\cdot\tfrac{1}{y}|\psi_{3}\psi_{4}\psi_{5}|_{-(\ell-1)}$.
If $\psi_{3},\psi_{4},\psi_{5}$ contain at least two $\epsilon$,
then we bound $|A_{\theta}[\psi_{1},\psi_{2}]|\lesssim1$ and use
the $L^{2}$-estimate of $\tfrac{1}{y}|\psi_{3}\psi_{4}\psi_{5}|_{-(\ell-1)}$
done in Case A. Otherwise, there are $j,k\in\{3,4,5\}$ such that
$j\neq k$, $\psi_{1},\psi_{2},\psi_{j}$ contain at least two $\epsilon$,
and $\psi_{k}=P$. Say $j=3$ and $k=4$. We estimate as 
\begin{align*}
\|A_{\theta}[\psi_{1},\psi_{2}]\tfrac{1}{y}|\psi_{3}P\psi_{5}|_{-(\ell-1)}\|_{L^{2}} & \lesssim\|\tfrac{1}{y^{2}}A_{\theta}[\psi_{1},\psi_{2}]\tfrac{1}{y}|\psi_{3}|_{-(\ell-1)}\|_{L^{2}}\|y^{2}Q|\psi_{5}|_{\ell-1}\|_{L^{\infty}}\\
 & \lesssim\|\epsilon\|_{\dot{\mathcal{H}}_{m}^{\ell}}(\|\epsilon\|_{\dot{H}_{m}^{3}}+\|\epsilon\|_{\dot{H}_{m}^{1}}^{2}),
\end{align*}
where in the last inequality we used $\|y^{2}Q|\psi_{5}|_{\ell-1}\|_{L^{\infty}}\lesssim1$
and the $L^{2}$-estimate of $\tfrac{1}{y^{2}}A_{\theta}[\psi_{1},\psi_{2}]|\psi_{3}|_{-\ell}$
done in Case B.

Finally, we estimate $|\psi_{1}\psi_{2}\psi_{3}\psi_{4}\psi_{5}|_{-(\ell-2)}$.
If $\psi_{1}=\cdots=\psi_{5}=\epsilon$, then we have 
\begin{align*}
|\epsilon^{5}|_{-1} & \lesssim|\epsilon|_{-1}|\epsilon|^{4},\\
|\epsilon^{5}|_{-3} & \lesssim(|\partial_{yyy}\epsilon||\epsilon|^{2}+|\partial_{yy}\epsilon||\epsilon|_{-1}|\epsilon|+|\epsilon|_{-1}^{3})|\epsilon|^{2}.
\end{align*}
Thus as in Case A 
\begin{align*}
\||\epsilon^{5}|_{-1}\|_{L^{2}} & \lesssim\|\epsilon\|_{L^{\infty}}^{4}\||\epsilon|_{-1}\|_{L^{2}}\lesssim\|\epsilon\|_{\dot{H}_{m}^{1}}^{5},\\
\||\epsilon^{5}|_{-3}\|_{L^{2}} & \lesssim\||\partial_{yyy}\epsilon||\epsilon|^{2}+|\partial_{yy}\epsilon||\epsilon|_{-1}|\epsilon|+|\epsilon|_{-1}^{3}\|_{L^{2}}\|\epsilon\|_{L^{\infty}}^{2}\lesssim\|\epsilon\|_{\dot{H}_{m}^{3}}\|\epsilon\|_{\dot{H}_{m}^{1}}^{4}.
\end{align*}
Otherwise, there exist at least one $P$ and two $\epsilon$'s. When
$\ell=3$, we estimate this contribution by 
\[
\|Q\epsilon|\epsilon|_{-1}\|_{L^{\infty}}\|Q^{2}+Q|\epsilon|+|\epsilon|^{2}\|_{L^{2}}\lesssim\|\langle y\rangle^{-1}|\epsilon|_{-1}\|_{L^{\infty}}\|\langle y\rangle^{-2}\epsilon\|_{L^{\infty}}\lesssim\|\epsilon\|_{\dot{\mathcal{H}}_{m}^{3}}^{2},
\]
where we bounded the $L^{2}$-term by $1$ and used \eqref{eq:L-infty-weighted-hardy}.
When $\ell=5$, we further subdivide the case. If there is exactly
one $P$, then we estimate this contribution using $\||\epsilon^{3}|_{-3}\|_{L^{2}}$
estimate of Case A and \eqref{eq:Weighted-L-infty-H5} by 
\[
\||P\epsilon|_{3}\|_{L^{\infty}}\||\epsilon^{3}|_{-3}\|_{L^{2}}\lesssim\|\epsilon\|_{\dot{\mathcal{H}}_{m}^{5}}\|\epsilon\|_{\dot{H}_{m}^{3}}\|\epsilon\|_{\dot{H}_{m}^{1}}^{2}.
\]
Otherwise, there exist at least two $P$'s and two $\epsilon$'s,
say $\psi_{4}=\psi_{5}=P$. We estimate this contribution using $\||\psi_{1}\psi_{2}\psi_{3}|_{-5}\|_{L^{2}}$
estimate of Case A by 
\[
\|y^{2}|P^{2}|_{3}\|_{L^{\infty}}\|y^{-2}|\psi_{1}\psi_{2}\psi_{3}|_{-3}\|_{L^{2}}\lesssim\|\epsilon\|_{\dot{\mathcal{H}}_{m}^{5}}(\|\epsilon\|_{\dot{H}_{m}^{3}}+\|\epsilon\|_{\dot{H}_{m}^{1}}^{2}).
\]

\uline{Case E:} $A_{0,5}\phi$-type.

Here, when $\psi=P$, by subtracting $\theta_{\mathrm{NL}}P$, $A_{0,5}$
is replaced by $\tilde A_{0,5}$. In other words, $V\psi$ is of the
form $A_{0,5}[\psi_{1},\psi_{2},\psi_{3},\psi_{4}]\epsilon$ or $\tilde A_{0,5}[\psi_{1},\psi_{2},\psi_{3},\psi_{4}]P$.

We directly show \eqref{eq:R-NL-Hdot3-claim1} and \eqref{eq:R-NL-Hdot5-claim1}.
We note that if $\partial_{y}$ hits $A_{0,5}$ (or $\tilde A_{0,5}$),
then it reduces to the $\tfrac{1}{y}A_{\theta}|\phi|^{2}\phi$-type,
which is estimated in Case D. Thus we may assume that $\partial_{y}$
never hits $A_{0,5}$ (or $\tilde A_{0,5}$). It suffices to show
\begin{align*}
\|A_{0,5}[\psi_{1},\psi_{2},\psi_{3},\psi_{4}](|\partial_{+}\epsilon|_{-(\ell-1)}+\mathbf{1}_{y\geq1}\tfrac{1}{y^{\ell}}\epsilon)\|_{L^{2}} & \lesssim\|\epsilon\|_{\dot{\mathcal{H}}_{m}^{\ell}}(\|\epsilon\|_{\dot{\mathcal{H}}_{m}^{3}}+\|\epsilon\|_{\dot{H}_{m}^{1}}^{2}),\\
\|\tilde A_{0,5}[\psi_{1},\psi_{2},\psi_{3},\psi_{4}]|P|_{-\ell}\|_{L^{2}} & \lesssim\|\epsilon\|_{\dot{\mathcal{H}}_{m}^{\ell}}(\|\epsilon\|_{\dot{\mathcal{H}}_{m}^{3}}+\|\epsilon\|_{\dot{H}_{m}^{1}}^{2}).
\end{align*}

For the first one, it suffices to show (as in Case C)
\[
\|y^{-2}A_{\theta}[\psi_{1},\psi_{2}]\psi_{3}\psi_{4}\|_{L^{1}}\lesssim\|\epsilon\|_{\dot{\mathcal{H}}_{m}^{3}}+\|\epsilon\|_{\dot{H}_{m}^{1}}^{2}.
\]
If $\epsilon\in\{\psi_{3},\psi_{4}\}$, then 
\[
\|y^{-2}A_{\theta}[\psi_{1},\psi_{2}]\psi_{3}\psi_{4}\|_{L^{1}}\lesssim\|y^{-2}\psi_{3}\psi_{4}\|_{L^{1}}\lesssim\|\epsilon\|_{\dot{\mathcal{H}}_{m}^{3}}+\|\epsilon\|_{\dot{H}_{m}^{1}}^{2},
\]
where we used $\|A_{\theta}[\psi_{1},\psi_{2}]\|_{L^{\infty}}\lesssim1$
and the $\|y^{-2}\psi_{3}\psi_{4}\|_{L^{1}}$-estimate done in Case
C. Otherwise, $\psi_{3}=\psi_{4}=P$ and $\epsilon\in\{\psi_{1},\psi_{2}\}$.
We estimate 
\[
\|y^{-2}A_{\theta}[\psi_{1},\psi_{2}]P^{2}\|_{L^{1}}\lesssim\|A_{\theta}[y^{-2}\psi_{1},\psi_{2}]\|_{L^{\infty}}\|P^{2}\|_{L^{1}}\lesssim\|y^{-2}\psi_{1}\psi_{2}\|_{L^{1}}\lesssim\|\epsilon\|_{\dot{\mathcal{H}}_{m}^{3}}+\|\epsilon\|_{\dot{H}_{m}^{1}}^{2},
\]
where we used the $\|y^{-2}\psi_{1}\psi_{2}\|_{L^{1}}$-estimate done
in Case C.

For the second one, using $|P|_{-\ell}\lesssim y^{-2}\langle y\rangle^{-(1+\ell)}$,
it suffices to show 
\[
\|y^{-2}\langle y\rangle^{-(1+\ell)}A_{\theta}[\psi_{1},\psi_{2}]\psi_{3}\psi_{4}\|_{L^{2}}\lesssim\|\epsilon\|_{\dot{\mathcal{H}}_{m}^{\ell}}(\|\epsilon\|_{\dot{\mathcal{H}}_{m}^{3}}+\|\epsilon\|_{\dot{H}_{m}^{1}}^{2}).
\]
If $\psi_{3}=\psi_{4}=\epsilon$, then we simply bound $\|A_{\theta}[\psi_{1},\psi_{2}]\|_{L^{\infty}}\lesssim1$
to estimate 
\[
\|y^{-2}\langle y\rangle^{-(1+\ell)}A_{\theta}[\psi_{1},\psi_{2}]\epsilon^{2}\|_{L^{2}}\lesssim\|y^{-1}\langle y\rangle^{-(\ell-1)}\epsilon\|_{L^{2}}\|y^{-1}\langle y\rangle^{-2}\epsilon\|_{L^{\infty}}\lesssim\|\epsilon\|_{\dot{\mathcal{H}}_{m}^{\ell}}\|\epsilon\|_{\dot{\mathcal{H}}_{m}^{3}}.
\]
If $P\in\{\psi_{3},\psi_{4}\}$, say $\psi_{4}=P$, then we estimate
\[
\|y^{-2}\langle y\rangle^{-(1+\ell)}A_{\theta}[\psi_{1},\psi_{2}]\psi_{3}P\|_{L^{2}}\lesssim\|y^{-(2+\ell)}A_{\theta}[\psi_{1},\psi_{2}]\psi_{3}\|_{L^{2}}\lesssim\|\epsilon\|_{\dot{\mathcal{H}}_{m}^{\ell}}(\|\epsilon\|_{\dot{\mathcal{H}}_{m}^{3}}+\|\epsilon\|_{\dot{H}_{m}^{1}}^{2}),
\]
where in the last inequality we used $\|\tfrac{1}{y^{2}}A_{\theta}[\psi_{1},\psi_{2}]|\psi_{3}|_{-\ell}\|_{L^{2}}$-estimate
of Case B.
\end{proof}

\subsection{Modulation estimates}

As discussed in Section \ref{subsec:Estimates-of-remainders-H3},
we reorganize the remainder terms of \eqref{eq:prelim-eqn-e}. Our
$\epsilon$-equation becomes 
\begin{equation}
(\partial_{s}-\frac{\lambda_{s}}{\lambda}\Lambda+\gamma_{s}i)\epsilon+i\mathcal{L}_{Q}\epsilon=\tilde{\mathbf{Mod}}\cdot\mathbf{v}-i\tilde R_{\mathrm{L-L}}-i\tilde R_{\mathrm{NL}}-i\Psi,\label{eq:e-eq-cor}
\end{equation}
where
\[
\tilde{\mathbf{Mod}}\coloneqq(\frac{\lambda_{s}}{\lambda}+b,\gamma_{s}-\eta\theta_{\eta}-\theta_{\Psi}-\theta_{\mathrm{L-L}}-\theta_{\mathrm{NL}},b_{s}+b^{2}+\eta^{2},\eta_{s})^{t}.
\]

In this subsection, we estimate $\tilde{\mathbf{Mod}}$. This says
that our modulation parameters follow the formal parameter ODEs \eqref{eq:FormalODE}.
\begin{lem}[Modulation estimates]
\label{lem:mod-est}We have 
\[
|\tilde{\mathbf{Mod}}|\lesssim\begin{cases}
o_{M\to\infty}(1)\|\epsilon_{3}\|_{L^{2}}+M^{C}(b\|\epsilon\|_{\dot{\mathcal{H}}_{m}^{3}}+b^{4}) & \text{if }m\geq1,\\
o_{M\to\infty}(1)\|\epsilon_{5}\|_{L^{2}}+M^{C}(b\|\epsilon\|_{\dot{\mathcal{H}}_{m}^{5}}+b^{6}) & \text{if }m\geq3.
\end{cases}
\]
In particular, $\lambda$ and $b$ are decreasing, and 
\[
\Big|\frac{\lambda_{s}}{\lambda}\Big|\lesssim b,\quad|\gamma_{s}|\lesssim|\eta|+b^{\frac{3}{2}}.
\]
\end{lem}

\begin{rem}
As will be detailed in Section \ref{subsec:energy-identity}, we need
the bound $|\tilde{\mathbf{Mod}}|\ll\|\epsilon_{3}\|_{L^{2}}$ (or
$\|\epsilon_{5}\|_{L^{2}}$) to close the bootstrap.
\end{rem}

\begin{proof}
We differentiate the orthogonality conditions in the renormalized
time variable $s$. In other words, we take the inner product of \eqref{eq:e-eq-cor}
with each $\mathcal{Z}_{k}$, $k\in\{1,2,3,4\}$ to get 
\begin{align*}
 & \frac{\lambda_{s}}{\lambda}(\epsilon,\Lambda\mathcal{Z}_{k})_{r}-\gamma_{s}(\epsilon,i\mathcal{Z}_{k})_{r}-(\epsilon,\mathcal{L}_{Q}i\mathcal{Z}_{k})_{r}\\
 & =\tilde{\mathbf{Mod}}\cdot(\mathbf{v},\mathcal{Z}_{k})_{r}-(i\tilde R_{\mathrm{L-L}},\mathcal{Z}_{k})_{r}-(i\tilde R_{\mathrm{NL}},\mathcal{Z}_{k})_{r}-(i\Psi_{b}^{(\eta)},\mathcal{Z}_{k})_{r}.
\end{align*}
Using $\frac{\lambda_{s}}{\lambda}=\tilde{\mathrm{Mod}}_{1}-b$ and
$\gamma_{s}=\tilde{\mathrm{Mod}}_{2}+\eta\theta_{\eta}+\theta_{\Psi}+\theta_{\mathrm{L-L}}+\theta_{\mathrm{NL}}$,
we can rewrite the above as 
\begin{align}
 & \sum_{j=1}^{4}\big\{(v_{j},\mathcal{Z}_{k})_{r}-\delta_{j1}(\epsilon,\Lambda\mathcal{Z}_{k})_{r}+\delta_{j2}(\epsilon,i\mathcal{Z}_{k})_{r}\big\}\tilde{\mathrm{Mod}}_{j}\nonumber \\
 & =-(\epsilon,\mathcal{L}_{Q}i\mathcal{Z}_{k})_{r}-b(\epsilon,\Lambda\mathcal{Z}_{k})_{r}-(\eta\theta_{\eta}+\theta_{\Psi}+\theta_{\mathrm{L-L}}+\theta_{\mathrm{NL}})(\epsilon,i\mathcal{Z}_{k})_{r}\label{eq:mod-temp}\\
 & \quad+(i\tilde R_{\mathrm{L-L}},\mathcal{Z}_{k})_{r}+(i\tilde R_{\mathrm{NL}},\mathcal{Z}_{k})_{r}+(i\Psi,\mathcal{Z}_{k})_{r},\nonumber 
\end{align}
where $\delta_{jk}$ denotes the Kronecker-delta symbol. By the choice
of $\mathcal{Z}_{k}$, the matrix 
\[
((v_{j},\mathcal{Z}_{k})_{r}-\delta_{j1}(\epsilon,\Lambda\mathcal{Z}_{k})_{r}+\delta_{j2}(\epsilon,i\mathcal{Z}_{k})_{r})_{1\leq j,k\leq4}
\]
is invertible with uniformly bounded inverse (cf. \eqref{eq:ortho-jacobian}).
It now suffices to estimate all terms of the RHS of \eqref{eq:mod-temp}.

The main contribution comes from the linear part $(\epsilon,\mathcal{L}_{Q}i\mathcal{Z}_{k})_{r}$.
By the definition of $\mathcal{Z}_{3}$ and $\mathcal{Z}_{4}$, we
have $(\epsilon,\mathcal{L}_{Q}i\mathcal{Z}_{k})_{r}=0$ for $k\in\{1,2\}$.
For $k\in\{3,4\}$, $(\epsilon,\mathcal{L}_{Q}i\mathcal{Z}_{k})_{r}$
does not necessarily vanish, but we can exploit the \emph{degeneracy}
of $\mathcal{L}_{Q}i\mathcal{Z}_{k}$ \eqref{eq:Orthogonality-approx-gen-null-H3}
or \eqref{eq:Orthogonality-approx-gen-null-H5}. We start with 
\begin{align*}
(\epsilon,\mathcal{L}_{Q}i\mathcal{Z}_{k})_{r} & =(\epsilon,\mathcal{L}_{Q}i\mathcal{L}_{Q}i\mathcal{Z}_{k-2})_{r}\\
 & =(\epsilon,L_{Q}^{\ast}iA_{Q}^{\ast}A_{Q}L_{Q}i\mathcal{Z}_{k-2})_{r}=(\epsilon_{2},iA_{Q}L_{Q}i\mathcal{Z}_{k-2})_{r}.
\end{align*}
Applying \eqref{eq:Orthogonality-approx-gen-null-H3} and positivity
of $A_{Q}^{\ast}A_{Q}$ \eqref{eq:non-M-dep-Hardy}, we have 
\[
|(\epsilon,\mathcal{L}_{Q}i\mathcal{Z}_{k})_{r}|\lesssim M^{-1}\|\tfrac{1}{y}\epsilon_{2}\|_{L^{2}}\lesssim M^{-1}\|\epsilon_{3}\|_{L^{2}}=o_{M\to\infty}(1)\|\epsilon_{3}\|_{L^{2}}.
\]
Similarly applying \eqref{eq:Orthogonality-approx-gen-null-H5} and
the positivity Lemma \ref{lem:Positivity-AAA-Appendix}, we have 
\[
|(\epsilon,\mathcal{L}_{Q}i\mathcal{Z}_{k})_{r}|=o_{M\to\infty}(1)\|\epsilon_{5}\|_{L^{2}}\qquad\text{if }m\geq3.
\]

Next, as $\mathcal{Z}_{k}\in(\dot{\mathcal{H}}_{m}^{3})^{\ast}$ and
$\mathcal{Z}_{k}\in(\dot{\mathcal{H}}_{m}^{5})^{\ast}$ if $m\geq3$
(see either \eqref{eq:Orthog-Hdot3-est} or \eqref{eq:Orthog-Hdot5-est}),
we have 
\begin{align*}
|b(\epsilon,\Lambda\mathcal{Z}_{k})_{r}|+|(\eta\theta_{\eta}+\theta_{\Psi}+\theta_{\mathrm{L-L}}+\theta_{\mathrm{NL}})(\epsilon,i\mathcal{Z}_{k})_{r}| & \lesssim bM^{C}\|\epsilon\|_{\dot{\mathcal{H}}_{m}^{3}}.
\end{align*}
The remaining contributions are already estimated in Section \ref{subsec:Estimates-of-remainders-H3}
and Proposition \ref{prop:modified-profile}. Indeed, the contribution
from $i\tilde R_{\mathrm{L-L}}$ is estimated in \eqref{eq:LL-mod-est}
(or \eqref{eq:LL-mod-est-m-geq-3} if $m\geq3$). The contribution
from $i\tilde R_{\mathrm{NL}}$ can be estimated by \eqref{eq:NL-mod-est}
(or \eqref{eq:NL-mod-est-m-geq-3} if $m\geq3$), substituting the
bootstrap hypotheses, and using the parameter dependence (Remark \ref{rem:ParameterDependence}).
Finally, the contribution from $i\Psi$ can be estimated using pointwise
estimates of $i\Psi$ \eqref{eq:Psi-Pointwise} and $\mathcal{Z}_{k}$
\eqref{eq:Orthog-Pointwise}, yielding the bound $(\log M)b^{m+3}$.
\end{proof}

\subsection{\label{subsec:energy-identity}Local virial control and modified
energy inequality}

In this subsection, we prove monotonicity of $\epsilon$, which enables
us to control $\epsilon$ forward in time. The main idea is to prove
a modified energy inequality in higher Sobolev norms. As explained
in Section \ref{subsec:Conjugation-identities}, we are able to take
Hamiltonian equations of higher order adapted derivatives. Recall
that adapted derivatives are $\epsilon_{1}=L_{Q}\epsilon$, $\epsilon_{2}=A_{Q}\epsilon_{1}$,
$\epsilon_{3}=A_{Q}^{\ast}\epsilon_{2}$, $\epsilon_{4}=A_{Q}\epsilon_{3}$,
$\epsilon_{5}=A_{Q}^{\ast}\epsilon_{4}$, and so on. We also observed
repulsivity in the equations of $\epsilon_{2}$ (and $\epsilon_{4}$).
With this repulsivity, we will be able to obtain monotonicity for
the modified energies $\mathcal{F}_{3}$ and $\mathcal{F}_{5}$ at
the $\dot{H}^{3}$ and $\dot{H}^{5}$-levels, i.e. 
\begin{align*}
\mathcal{F}_{3} & \approx\|\epsilon_{3}\|_{L^{2}}^{2}=\|A_{Q}^{\ast}\epsilon_{2}\|_{L^{2}}^{2},\\
\mathcal{F}_{5} & \approx\|\epsilon_{5}\|_{L^{2}}^{2}=\|A_{Q}^{\ast}\epsilon_{4}\|_{L^{2}}^{2}.
\end{align*}

Let us explain why we work at least in the $\dot{H}^{3}$-level. Due
to scaling considerations, the optimal bound (what we can expect)
for $\|\epsilon_{k}\|_{L^{2}}$ is $O(\lambda^{k})$. In the pseudoconformal
blow-up regime, $\lambda\sim b$ is expected so this bound reads $\|\epsilon_{k}\|_{L^{2}}\lesssim b^{k}$.
We now recall that the main contribution to the modulation estimate
$b_{s}+b^{2}$ was the linear term, say $o_{M\to\infty}(1)\|\epsilon_{k}\|_{L^{2}}$.
This says that in order to justify $b_{s}+b^{2}\approx0$, we need
to get $\|\epsilon_{k}\|_{L^{2}}\lesssim b^{2+}$. In other words,
we need to work at least with $k>2$.\footnote{If $m$ is large, one may use untruncated orthogonality conditions
(i.e., putting $\epsilon\in N_{g}(\mathcal{L}_{Q}i)^{\perp}$) to
improve the modulation estimate for $b_{s}+b^{2}$ (say, by the factor
of $b$). Then,  working at the $\dot{H}^{2}$-level would be sufficient.
However, we do not know how to perform an energy estimate at the $\dot{H}^{2}$-level
due to the lack of the repulsivity structure (the truncated virial
functional in this paper is defined for the $\epsilon_{2}$-variable
and would require at least $\dot{H}^{1/2}$-regularity of $\epsilon_{2}$).}

Next, we explain how we choose the powers of $b$ in the bootstrap
hypothesis. One of the restrictions comes from the above scaling considerations
$\|\epsilon_{k}\|_{L^{2}}\lesssim b^{k}$. There is the other source
of the restrictions, the error $i\Psi$ from the modified profile.
To see this, let us consider the toy model 
\begin{align*}
\tfrac{1}{2}(\partial_{s}-2k\tfrac{\lambda_{s}}{\lambda})\|\epsilon_{k}\|_{L^{2}}^{2} & =(\epsilon_{k},-(i\Psi)_{k})_{r}\lesssim\|\epsilon_{k}\|_{L^{2}}\|(i\Psi)_{k}\|_{L^{2}},
\end{align*}
where $(i\Psi)_{k}$ is the $k$-th adapted derivative of $i\Psi$.
Let us formally rewrite this as 
\begin{align*}
(\partial_{s}-k\tfrac{\lambda_{s}}{\lambda})\|\epsilon_{k}\|_{L^{2}} & \lesssim\|(i\Psi)_{k}\|_{L^{2}}.
\end{align*}
Assuming the size $\|(i\Psi)_{k}\|_{L^{2}}\lesssim b^{p}$ and using
$\partial_{t}=\lambda^{-2}\partial_{s}$, we integrate $\lambda^{-k}\|\epsilon_{k}\|_{L^{2}}$
as
\[
\frac{\|\epsilon_{k}(t)\|_{L^{2}}}{\lambda^{k}(t)}=\frac{\|\epsilon_{k}(0)\|_{L^{2}}}{\lambda^{k}(0)}+O\Big(\int_{0}^{t}\frac{b^{p}(\tau)}{\lambda^{k+2}(\tau)}d\tau\Big).
\]
Using the pseudoconformal regime $\lambda\sim b$ and the ansatz $b=-\frac{\lambda_{s}}{\lambda}=-\lambda\lambda_{t}$,
we can integrate 
\[
\int_{0}^{t}\frac{b^{p}}{\lambda^{k+2}}d\tau\sim\int_{0}^{t}\frac{-\lambda_{t}}{\lambda^{k-p+2}}d\tau\sim\Big(\frac{1}{k-p+1}\Big)\frac{1}{\lambda^{k-p+1}}\Big|_{0}^{t}\lesssim\begin{cases}
b^{p-1-k}(t) & \text{if }k>p-1,\\
b^{p-1-k}(0) & \text{if }k<p-1.
\end{cases}
\]
Thus 
\[
\|\epsilon_{k}(t)\|_{L^{2}}\lesssim\begin{cases}
(\lambda^{-k}(0)\|\epsilon_{k}(0)\|_{L^{2}})b^{k}(t)+b^{p-1}(t) & \text{if }k>p-1,\\
(\lambda^{-k}(0)\|\epsilon_{k}(0)\|_{L^{2}}+b^{p-1-k}(0))b^{k}(t) & \text{if }k<p-1.
\end{cases}
\]
In other words, we can only expect 
\[
\|\epsilon_{k}(t)\|_{L^{2}}\lesssim b^{\min\{k,p-1\}}(t).
\]
Recalling that we need to get $\|\epsilon_{k}\|_{L^{2}}\lesssim b^{2+}$,
the error $i\Psi$ from the modified profile should satisfy $\|(i\Psi)_{k}\|_{L^{2}}\lesssim b^{3+}$.
Our choice of the modified profile satisfies this bound, see \eqref{eq:Radiation-Hdot3-est}
when $m\geq1$ and $k=3$. Our bootstrap bounds \eqref{eq:bootstrap-hyp-H3}
for $m\geq1$ and \eqref{eq:bootstrap-hyp-H5} for $m\geq3$ are motivated
from the $(i\Psi)_{k}$ bounds \eqref{eq:Radiation-Hdot3-est} and
\eqref{eq:Radiation-Hdot5-est}.

When we compute $\tfrac{1}{2}(\partial_{s}-6\frac{\lambda_{s}}{\lambda})\|\epsilon_{3}\|_{L^{2}}^{2}$
(and similarly for $\epsilon_{5}$), we will meet the error terms
coming from the $\epsilon_{3}$-equation. We hope that such error
terms do not disturb our bootstrap procedure. We show some heuristics
to determine how much errors we can allow. We claim that $Cb\|\epsilon_{3}\|_{L^{2}}^{2}$
error (with possibly large constant $C\gtrsim1$) is not perturbative.
Indeed, the previous computations say that we roughly lose one $b$
when we integrate in time. In other words, 
\[
\tfrac{1}{2}(\partial_{s}-6\tfrac{\lambda_{s}}{\lambda})\|\epsilon_{3}\|_{L^{2}}^{2}\leq Cb\|\epsilon_{3}\|_{L^{2}}^{2}+C_{2}b^{p}
\]
would yield 
\begin{align*}
\|\epsilon_{3}\|_{L^{2}}^{2} & \lesssim C\|\epsilon_{3}\|_{L^{2}}^{2}+C_{2}b^{\min\{p-1,3\}}.
\end{align*}
This says that we may not close the bootstrap if $C$ is large.

Such error terms of size $Cb\|\epsilon_{3}\|_{L^{2}}^{2}$ can appear
from $(\epsilon_{3},A_{Q}^{\ast}A_{Q}L_{Q}i\tilde R_{\mathrm{L-L}})_{r}$
because $i\tilde R_{\mathrm{L-L}}$ is roughly $O(b\epsilon)$-like
term (see \eqref{eq:LL-E3-est}). (In fact, there is another such
term induced from scalings.) From the above heuristics, we know that
using a cruder bound $b\|\epsilon\|_{\dot{\mathcal{H}}_{m}^{3}}$
for $i\tilde R_{\mathrm{L-L}}$ may not be sufficient.

To get around this difficulty, we need to use correction terms for
the energy. We now recall that the repulsive nature of $\epsilon_{2}$-equation
allows us to control the energy $\|\epsilon_{3}\|_{L^{2}}^{2}=\|A_{Q}^{\ast}\epsilon_{2}\|_{L^{2}}^{2}$
via the virial identity \eqref{eq:FormalMonotonicity}. So it would
be natural to use $M_{1}b(\epsilon_{2},-i\Lambda\epsilon_{2})_{r}$
with large $M_{1}$ to dominate the error $Cb\|\epsilon_{3}\|_{L^{2}}^{2}$.
However, the virial functional is not bounded on our function spaces.
Thus we need to localize it, say $M_{1}b(\epsilon_{2},-i\Lambda_{M_{2}}\epsilon_{2})_{r}$,
and guarantee that at least the local portion ($y\leq M_{2}$) of
$Cb\|\epsilon_{3}\|_{L^{2}}^{2}$ can be dominated. This motivates
the localized form of the estimates \eqref{eq:LL-E3-est}. Such idea
was used in \cite{MerleRaphaelRodnianski2015CambJMath,Collot2018MemAMS}.

Finally, to deal with the technical error coming from localizing virial
functionals, we use an averaging argument over the parameter $M_{2}$
as in \cite{KimKwon2019arXiv}.

From now on, we prove a modified energy inequality for $\dot{\mathcal{H}}_{m}^{3}$
and $\dot{\mathcal{H}}_{m}^{5}$. We start from the equation of $\epsilon_{2}$:
\begin{align}
 & (\partial_{s}-\frac{\lambda_{s}}{\lambda}\Lambda_{-2}+\gamma_{s}i)\epsilon_{2}+iA_{Q}A_{Q}^{\ast}\epsilon_{2}\nonumber \\
 & =\tilde{\mathbf{Mod}}\cdot A_{Q}L_{Q}\mathbf{v}+\frac{\lambda_{s}}{\lambda}\partial_{\lambda}(A_{Q_{\lambda}}L_{Q_{\lambda}})\epsilon-\gamma_{s}A_{Q}[L_{Q},i]\epsilon\label{eq:e2-eq}\\
 & \quad-A_{Q}L_{Q}i\tilde R_{\mathrm{L-L}}-A_{Q}L_{Q}i\tilde R_{\mathrm{NL}}-A_{Q}L_{Q}i\Psi,\nonumber 
\end{align}
where we denote $Q_{\lambda}(y)=\frac{1}{\lambda}Q(\frac{y}{\lambda})$
and used the computation (at $\lambda=1$)
\begin{align*}
\Lambda_{-2}A_{Q}L_{Q}\epsilon & =-\partial_{\lambda}A_{Q_{\lambda}}L_{Q_{\lambda}}\epsilon_{\lambda}=-\partial_{\lambda}(A_{Q_{\lambda}}L_{Q_{\lambda}})\epsilon+A_{Q}L_{Q}\Lambda\epsilon,\\{}
[A_{Q}L_{Q},i]\epsilon & =A_{Q}[L_{Q},i]\epsilon.
\end{align*}
Similarly, we have the equation of $\epsilon_{4}$: 
\begin{align}
 & (\partial_{s}-\frac{\lambda_{s}}{\lambda}\Lambda_{-4}+\gamma_{s}i)\epsilon_{4}+iA_{Q}A_{Q}^{\ast}\epsilon_{4}\nonumber \\
 & =\tilde{\mathbf{Mod}}\cdot A_{Q}A_{Q}^{\ast}A_{Q}L_{Q}\mathbf{v}+\frac{\lambda_{s}}{\lambda}\partial_{\lambda}(A_{Q_{\lambda}}A_{Q_{\lambda}}^{\ast}A_{Q_{\lambda}}L_{Q_{\lambda}})\epsilon-\gamma_{s}A_{Q}A_{Q}^{\ast}A_{Q}[L_{Q},i]\epsilon\label{eq:e4-eq}\\
 & \quad-A_{Q}A_{Q}^{\ast}A_{Q}L_{Q}i\tilde R_{\mathrm{L-L}}-A_{Q}A_{Q}^{\ast}A_{Q}L_{Q}i\tilde R_{\mathrm{NL}}-A_{Q}A_{Q}^{\ast}A_{Q}L_{Q}i\Psi.\nonumber 
\end{align}

We first estimate the commutator terms. From $|\frac{\lambda_{s}}{\lambda}|\approx b$,
the scaling terms $\frac{\lambda_{s}}{\lambda}\partial_{\lambda}(A_{Q_{\lambda}}L_{Q_{\lambda}})\epsilon$
and $\frac{\lambda_{s}}{\lambda}\partial_{\lambda}(A_{Q_{\lambda}}A_{Q_{\lambda}}^{\ast}A_{Q_{\lambda}}L_{Q_{\lambda}})\epsilon$
are of $O(b\epsilon)$-like terms. Thus we aim to estimate these  by
local $\dot{\mathcal{H}}_{m}^{3}$ or $\dot{\mathcal{H}}_{m}^{5}$
norms, as in $\tilde R_{\mathrm{L-L}}$.
\begin{lem}[Commutator terms]
\label{lem:commutator}We have 
\begin{align*}
 & \|\partial_{\lambda=1}(A_{Q_{\lambda}}^{\ast}A_{Q_{\lambda}}L_{Q_{\lambda}})\epsilon\|_{L^{2}}+\|y^{-1}\partial_{\lambda=1}(A_{Q_{\lambda}}L_{Q_{\lambda}})\epsilon\|_{L^{2}}+\||[L_{Q},i]\epsilon|_{-2}\|_{L^{2}}\\
 & \lesssim\begin{cases}
\|\epsilon\|_{\dot{\mathcal{H}}_{m,\leq M_{2}}^{3}}+o_{M_{2}\to\infty}(1)\|\epsilon\|_{\dot{\mathcal{H}}_{m}^{3}}, & \text{if }m\geq1,\\
\|\epsilon\|_{\dot{\mathcal{H}}_{m}^{5}} & \text{if }m\geq3,
\end{cases}
\end{align*}
and when $m\geq3$ 
\begin{multline*}
\|\partial_{\lambda=1}(A_{Q_{\lambda}}^{\ast}A_{Q_{\lambda}}A_{Q_{\lambda}}^{\ast}A_{Q_{\lambda}}L_{Q_{\lambda}})\epsilon\|_{L^{2}}+\|y^{-1}\partial_{\lambda=1}(A_{Q_{\lambda}}A_{Q_{\lambda}}^{\ast}A_{Q_{\lambda}}L_{Q_{\lambda}})\epsilon\|_{L^{2}}\\
+\||[L_{Q},i]\epsilon|_{-4}\|_{L^{2}}\lesssim\|\epsilon\|_{\dot{\mathcal{H}}_{m,\leq M_{2}}^{5}}+o_{M_{2}\to\infty}(1)\|\epsilon\|_{\dot{\mathcal{H}}_{m}^{5}}.
\end{multline*}
\end{lem}

\begin{proof}
Note the computations 
\begin{align*}
\partial_{\lambda=1}A_{Q_{\lambda}} & =-\tfrac{1}{2}yQ^{2},\\
\partial_{\lambda=1}(A_{Q_{\lambda}}^{\ast}A_{Q_{\lambda}}) & =2(m+1+A_{\theta}[Q])Q^{2},\\
\partial_{\lambda=1}L_{Q_{\lambda}} & =-\tfrac{1}{2}yQ^{2}-\Lambda QB_{Q}-QB_{\Lambda Q}.
\end{align*}
Moreover, 
\[
[L_{Q},i]f=-iQB_{Q}(\Re(f))-QB_{Q}(\Im(f)).
\]
Here, the operators $\Lambda QB_{Q}$ and $QB_{\Lambda Q}$ are amenable
to Lemma \ref{lem:Contribution-QBQ}, so we may regard them as $QB_{Q}$.
Thus 
\begin{align*}
 & \|\partial_{\lambda=1}(A_{Q_{\lambda}}^{\ast}A_{Q_{\lambda}}L_{Q_{\lambda}})\epsilon\|_{L^{2}}+\|y^{-1}\partial_{\lambda=1}(A_{Q_{\lambda}}L_{Q_{\lambda}})\epsilon\|_{L^{2}}+\||[L_{Q},i]\epsilon|_{-2}\|_{L^{2}}\\
 & \lesssim\|Q^{2}L_{Q}\epsilon\|_{L^{2}}+\||yQ^{2}\epsilon|_{-2}\|_{L^{2}}+\||QB_{Q}\epsilon|_{-2}\|_{L^{2}}\\
 & \lesssim\|\langle y\rangle^{-2}Q|\epsilon|_{-2}\|_{L^{2}}\\
 & \lesssim\begin{cases}
\|\epsilon\|_{\dot{\mathcal{H}}_{m,\leq M_{2}}^{3}}+o_{M_{2}\to\infty}(1)\|\epsilon\|_{\dot{\mathcal{H}}_{m}^{3}} & \text{if }m\geq1,\\
\|\epsilon\|_{\dot{\mathcal{H}}_{m}^{5}} & \text{if }m\geq3,
\end{cases}
\end{align*}
as desired. Similarly, we use Lemma \ref{lem:Contribution-QBQ-H5}
to get 
\begin{align*}
 & \|\partial_{\lambda=1}(A_{Q_{\lambda}}^{\ast}A_{Q_{\lambda}}A_{Q_{\lambda}}^{\ast}A_{Q_{\lambda}}L_{Q_{\lambda}})\epsilon\|_{L^{2}}+\|y^{-1}\partial_{\lambda=1}(A_{Q_{\lambda}}A_{Q_{\lambda}}^{\ast}A_{Q_{\lambda}}L_{Q_{\lambda}})\epsilon\|_{L^{2}}\\
 & \qquad+\||[L_{Q},i]\epsilon|_{-4}\|_{L^{2}}\\
 & \lesssim\|Q^{2}|L_{Q}\epsilon|_{-2}\|_{L^{2}}+\||yQ^{2}\epsilon|_{-4}\|_{L^{2}}+\||QB_{Q}\epsilon|_{-4}\|_{L^{2}}\\
 & \lesssim\|\langle y\rangle^{-2}Q|\epsilon|_{-4}\|_{L^{2}}\\
 & \lesssim\|\epsilon\|_{\dot{\mathcal{H}}_{m,\leq M_{2}}^{5}}+o_{M_{2}\to\infty}(1)\|\epsilon\|_{\dot{\mathcal{H}}_{m}^{5}}.
\end{align*}
This completes the proof.
\end{proof}
First, we prove preliminary energy estimates. We will compute 
\[
\Big(\partial_{s}-6\frac{\lambda_{s}}{\lambda}\Big)\|\epsilon_{3}\|_{L^{2}}^{2}\quad\text{and}\quad\Big(\partial_{s}-10\frac{\lambda_{s}}{\lambda}\Big)\|\epsilon_{5}\|_{L^{2}}^{2}.
\]

\begin{lem}[Energy identity for $\dot{\mathcal{H}}_{3}^{m}$]
\label{lem:E3-Identity}We have 
\begin{equation}
\Big|\Big(\partial_{s}-6\frac{\lambda_{s}}{\lambda}\Big)\|\epsilon_{3}\|_{L^{2}}^{2}\Big|\lesssim b\|\epsilon_{3}\|_{L^{2}}\Big(o_{M\to\infty}(1)\|\epsilon_{3}\|_{L^{2}}+\|\epsilon\|_{\dot{\mathcal{H}}_{m,\leq M_{2}}^{3}}+o_{M_{2}\to\infty}(1)\|\epsilon\|_{\dot{\mathcal{H}}_{m}^{3}}+b^{\frac{5}{2}}\Big).\label{eq:EnergyIdentityHdot3}
\end{equation}
If $m\geq3$, we further have 
\begin{equation}
\Big|\Big(\partial_{s}-6\frac{\lambda_{s}}{\lambda}\Big)\|\epsilon_{3}\|_{L^{2}}^{2}\Big|\lesssim b\|\epsilon_{3}\|_{L^{2}}\cdot b^{\frac{7}{2}}.\label{eq:EnergyIdentityHdot3-m-geq-3}
\end{equation}
\end{lem}

One can observe that the contributions of $O(b\epsilon)$-like terms,
$\tilde R_{\mathrm{L-L}}$ and scaling induced terms, are estimated
by 
\begin{equation}
Cb\|\epsilon_{3}\|_{L^{2}}(\|\epsilon\|_{\dot{\mathcal{H}}_{m,\leq M_{2}}^{3}}+o_{M_{2}\to\infty}(1)\|\epsilon\|_{\dot{\mathcal{H}}_{m}^{3}}).\label{eq:local H3}
\end{equation}
Later this will be dominated by adding a correction term into the
energy.

In the display \eqref{eq:EnergyIdentityHdot3}, the terms $o_{M\to\infty}(1)\|\epsilon_{3}\|_{L^{2}}$
and $o_{M_{2}\to\infty}(1)\|\epsilon\|_{\dot{\mathcal{H}}_{m}^{3}}$
can be absorbed by $b^{\frac{5}{2}}$ using the bootstrap hypothesis
\eqref{eq:bootstrap-hyp-H3} and the linear coercivity Lemma \ref{lem:Coercivity-AAL-Section2}.
However, we chose to keep those terms, which motivate our choice of
parameters as in Remark \ref{rem:ParameterDependence}.
\begin{proof}
Taking $A_{Q}^{\ast}$ to \eqref{eq:e2-eq}, we get the equation of
$\epsilon_{3}$: 
\begin{align}
 & (\partial_{s}-\frac{\lambda_{s}}{\lambda}\Lambda_{-3}+\gamma_{s}i)\epsilon_{3}+iA_{Q}^{\ast}A_{Q}\epsilon_{3}\nonumber \\
 & =\tilde{\mathbf{Mod}}\cdot A_{Q}^{\ast}A_{Q}L_{Q}\mathbf{v}+\frac{\lambda_{s}}{\lambda}\partial_{\lambda}(A_{Q_{\lambda}}^{\ast}A_{Q_{\lambda}}L_{Q_{\lambda}})\epsilon-\gamma_{s}A_{Q}^{\ast}A_{Q}[L_{Q},i]\epsilon\label{eq:e3-eq}\\
 & \quad-A_{Q}^{\ast}A_{Q}L_{Q}i\tilde R_{\mathrm{L-L}}-A_{Q}^{\ast}A_{Q}L_{Q}i\tilde R_{\mathrm{NL}}-A_{Q}^{\ast}A_{Q}L_{Q}i\Psi.\nonumber 
\end{align}
We take the inner product of \eqref{eq:e3-eq} with $\epsilon_{3}$
to get 
\[
\Big|\frac{1}{2}\Big(\partial_{s}-6\frac{\lambda_{s}}{\lambda}\Big)\|\epsilon_{3}\|_{L^{2}}^{2}\Big|=\Big|(\epsilon_{3},\text{RHS of }\eqref{eq:e3-eq})_{r}\Big|\lesssim\|\epsilon_{3}\|_{L^{2}}\|\text{RHS of }\eqref{eq:e3-eq}\|_{L^{2}}.
\]
To show \eqref{eq:EnergyIdentityHdot3}, we keep $\|\epsilon_{3}\|_{L^{2}}$
and estimate $\|\text{RHS of }\eqref{eq:e3-eq}\|_{L^{2}}$. To show
\eqref{eq:EnergyIdentityHdot3-m-geq-3} when $m\geq3$, it is enough
to show that $\|\text{RHS of }\eqref{eq:e3-eq}\|_{L^{2}}\lesssim b^{\frac{9}{2}}$.

Henceforth, we estimate $\|\text{RHS of }\eqref{eq:e3-eq}\|_{L^{2}}$.
By Lemma \ref{lem:mod-est} and \eqref{eq:gen-null-rel}, we have
\[
\|\tilde{\mathbf{Mod}}\cdot A_{Q}^{\ast}A_{Q}L_{Q}\mathbf{v}\|_{L^{2}}\lesssim\begin{cases}
b(o_{M\to\infty}(1)\|\epsilon_{3}\|_{L^{2}}+M^{C}(b\|\epsilon\|_{\dot{\mathcal{H}}_{m}^{3}}+b^{4})) & \text{if }m\geq1,\\
b(o_{M\to\infty}(1)\|\epsilon_{5}\|_{L^{2}}+M^{C}(b\|\epsilon\|_{\dot{\mathcal{H}}_{m}^{5}}+b^{6})) & \text{if }m\geq3,
\end{cases}
\]
which are enough in view of the bootstrap hypotheses \eqref{eq:bootstrap-hyp-H3}
and \eqref{eq:bootstrap-hyp-H5} and the coercivity Lemmas \ref{lem:Coercivity-AAL-Section2}
and \ref{lem:Coercivity-AAAAL-Section2}. By Lemmas \ref{lem:mod-est}
and \ref{lem:commutator}, we have 
\begin{align*}
 & \|\frac{\lambda_{s}}{\lambda}\partial_{\lambda}(A_{Q_{\lambda}}^{\ast}A_{Q_{\lambda}}L_{Q_{\lambda}})\epsilon\|_{L^{2}}+\|\gamma_{s}A_{Q}^{\ast}A_{Q}[L_{Q},i]\epsilon\|_{L^{2}}\\
 & \lesssim\begin{cases}
b(\|\epsilon\|_{\dot{\mathcal{H}}_{m,\leq M_{2}}^{3}}+o_{M_{2}\to\infty}(1)\|\epsilon\|_{\dot{\mathcal{H}}_{m}^{3}}) & \text{if }m\geq1,\\
b\|\epsilon\|_{\dot{\mathcal{H}}_{m}^{5}} & \text{if }m\geq3,
\end{cases}
\end{align*}
which is enough. The contribution from $\tilde R_{\mathrm{L-L}}$
is estimated by \eqref{eq:LL-E3-est} if $m\geq1$ and \eqref{eq:LL-E3-est-m-geq-3}
if $m\geq3$. The contribution from $\tilde R_{\mathrm{NL}}$ is estimated
in \eqref{eq:NL-E3-est}. Finally, the contribution from $\Psi$ is
estimated in \eqref{eq:Radiation-Hdot3-est} (take $m=1$ for $m\geq1$
and $m=3$ for $m\geq3$).
\end{proof}
When $\epsilon$ is a $H_{m}^{5}$-solution (i.e. $\epsilon\in\mathcal{O}^{5}$),
then the energy estimate \eqref{eq:EnergyIdentityHdot3-m-geq-3} for
$\epsilon_{3}$ suffices. No virial corrections are needed for $\|\epsilon_{3}\|_{L^{2}}^{2}$.
In fact, such a good estimate can be obtained because we were able
to use higher Sobolev norm $\dot{\mathcal{H}}_{m}^{5}$ (thus having
more power of $b$).

Similarly to \eqref{eq:EnergyIdentityHdot3}, we need an energy estimate
for $\epsilon_{5}$.
\begin{lem}[Energy identity for $\dot{\mathcal{H}}_{m}^{5}$]
\label{lem:E5-identity}Let $m\geq3$. We have
\begin{equation}
\Big|\Big(\partial_{s}-10\frac{\lambda_{s}}{\lambda}\Big)\|\epsilon_{5}\|_{L^{2}}^{2}\Big|\lesssim b\|\epsilon_{5}\|_{L^{2}}\Big(o_{M\to\infty}(1)\|\epsilon_{5}\|_{L^{2}}+\|\epsilon\|_{\dot{\mathcal{H}}_{m,\leq M_{2}}^{5}}+o_{M_{2}\to\infty}(1)\|\epsilon\|_{\dot{\mathcal{H}}_{m}^{5}}+b^{\frac{9}{2}}\Big).\label{eq:EnergyIdentityHdot5}
\end{equation}
\end{lem}

As above, the terms $o_{M\to\infty}(1)\|\epsilon_{5}\|_{L^{2}}$ and
$o_{M_{2}\to\infty}(1)\|\epsilon\|_{\dot{\mathcal{H}}_{m}^{5}}$ can
be absorbed into the term $b^{\frac{9}{2}}$.
\begin{proof}
One proceeds similarly as in the proof of \eqref{eq:EnergyIdentityHdot3}.
Taking $A_{Q}^{\ast}$ to the equation, we get the equation of $\epsilon_{5}$:
\begin{align}
 & (\partial_{s}-\frac{\lambda_{s}}{\lambda}\Lambda_{-5}+\gamma_{s}i)\epsilon_{5}+iA_{Q}^{\ast}A_{Q}\epsilon_{5}\nonumber \\
 & =\tilde{\mathbf{Mod}}\cdot A_{Q}^{\ast}A_{Q}A_{Q}^{\ast}A_{Q}L_{Q}\mathbf{v}+\frac{\lambda_{s}}{\lambda}\partial_{\lambda}(A_{Q}^{\ast}A_{Q_{\lambda}}A_{Q_{\lambda}}^{\ast}A_{Q_{\lambda}}L_{Q_{\lambda}})\epsilon\label{eq:e5-eq}\\
 & \quad-\gamma_{s}A_{Q}^{\ast}A_{Q}A_{Q}^{\ast}A_{Q}[L_{Q},i]\epsilon-A_{Q}^{\ast}A_{Q}A_{Q}^{\ast}A_{Q}L_{Q}(i\tilde R_{\mathrm{L-L}}+i\tilde R_{\mathrm{NL}}+i\Psi).\nonumber 
\end{align}
As before, it suffices to estimate $\|\text{RHS of }\eqref{eq:e5-eq}\|_{L^{2}}$.
For the modulation term, use $m\geq3$ case of Lemma \ref{lem:mod-est}.
For the commutator terms, use the last estimate of Lemma \ref{lem:commutator}.
For $\tilde R_{\mathrm{L-L}}$, $\tilde R_{\mathrm{NL}}$, and $\Psi$,
one may use \eqref{eq:LL-E5-est}, \eqref{eq:NL-E5-est}, and \eqref{eq:Radiation-Hdot5-est},
respectively.
\end{proof}
For the highest Sobolev norms, the estimates \eqref{eq:EnergyIdentityHdot3}
and \eqref{eq:EnergyIdentityHdot5} seem to be optimal. As explained
above, the local $\dot{\mathcal{H}}_{m}^{3}$ and $\dot{\mathcal{H}}_{m}^{5}$
terms (e.g. \eqref{eq:local H3}) can be controlled by the monotonicity,
namely the repulsivity from the virial functional.

Let us recall the formal computation \eqref{eq:FormalMonotonicity}
\[
\tfrac{1}{2}\partial_{s}(\epsilon_{2},-i\Lambda\epsilon_{2})_{r}\approx(\epsilon_{2},A_{Q}A_{Q}^{\ast}\epsilon_{2})_{r}+(\epsilon_{2},\tfrac{-\partial_{y}\tilde V}{2y}\epsilon_{2})_{r}\geq(\epsilon_{2},A_{Q}A_{Q}^{\ast}\epsilon_{2})_{r}=\|\epsilon_{3}\|_{L^{2}}^{2}.
\]
However, the virial functional $(\epsilon_{2},i\Lambda\epsilon_{2})_{r}$
is unbounded, so we will truncate this. In order to mimic the original
virial functional in a large compact region and keep symmetricity
of $i\Lambda$, we truncate as follows.

Let $\phi:\R\to\R$ be a smooth function such that $\phi'(y)=y$ for
$|y|\leq1$ and $\phi'(y)=0$ for $y\geq2$. For $A\geq1$, define
the smooth radial function $\phi_{A}:\R^{2}\to\R$ such that $\phi_{A}(x)\coloneqq A^{2}\phi(\frac{|x|}{A})$.
Denote $\phi_{A}'\coloneqq\partial_{r}\phi_{A}$. The \emph{deformed
$L^{2}$-scaling vector field} is defined by 
\[
\Lambda_{A}\coloneqq\phi_{A}'\partial_{r}+\frac{\Delta\phi_{A}}{2},
\]
where $\Delta$ is the Laplacian on $\R^{2}$ (thus $\Delta\phi_{A}=(\partial_{yy}+\frac{1}{y}\partial_{y})\phi_{A}$).
We note that $i\Lambda_{A}=i\Lambda$ in the region $y\leq A$ and
it is symmetric.

Now we prove a localized version of the repulsivity.
\begin{lem}[Localized repulsivity]
\label{lem:LocalizedRepulsivity}We have 
\begin{equation}
(A_{Q}A_{Q}^{\ast}\epsilon_{2},\Lambda_{M_{2}}\epsilon_{2})_{r}\geq c_{M}\|\epsilon\|_{\dot{\mathcal{H}}_{m,\leq M_{2}}^{3}}^{2}-O\big(\|\mathbf{1}_{y\sim M_{2}}|\epsilon|_{-3}\|_{L^{2}}^{2}\big)-o_{M_{2}\to\infty}(1)\|\epsilon\|_{\dot{\mathcal{H}}_{m}^{3}}^{2}.\label{eq:LocalizedRepulsivityH3}
\end{equation}
If $m\geq3$, we have 
\begin{equation}
(A_{Q}A_{Q}^{\ast}\epsilon_{4},\Lambda_{M_{2}}\epsilon_{4})_{r}\geq c_{M}\|\epsilon\|_{\dot{\mathcal{H}}_{m,\leq M_{2}}^{5}}^{2}-O\big(\|\mathbf{1}_{y\sim M_{2}}|\epsilon|_{-5}\|_{L^{2}}^{2}\big)-o_{M_{2}\to\infty}(1)\|\epsilon\|_{\dot{\mathcal{H}}_{m}^{5}}^{2}.\label{eq:LocalizedRepulsivityH5}
\end{equation}
\end{lem}

\begin{rem}
The error term in the region $y\sim M_{2}$ can be easily managed
using an averaging argument in $M_{2}$, see Lemma \ref{lem:LocalVirialControl}
below.
\end{rem}

\begin{proof}
By a direct computation, we have for $\ell\in\{2,4\}$ that 
\begin{align*}
(-\partial_{yy}\epsilon_{\ell}-\tfrac{1}{y}\partial_{y}\epsilon_{\ell},\Lambda_{M_{2}}\epsilon_{\ell})_{r} & ={\textstyle \int}\phi_{M_{2}}''|\partial_{y}\epsilon_{\ell}|^{2}-{\textstyle \int}\tfrac{\Delta^{2}\phi_{M_{2}}}{4}|\epsilon_{\ell}|^{2},\\
(\tfrac{\tilde V}{y^{2}}\epsilon_{\ell},\Lambda_{M_{2}}\epsilon_{\ell})_{r} & ={\textstyle \int}\tfrac{\phi_{M_{2}}'}{y}\cdot\tfrac{\tilde V}{y^{2}}|\epsilon_{\ell}|^{2}-{\textstyle \int}\tfrac{\phi_{M_{2}}'}{2y}\cdot\tfrac{y\partial_{y}\tilde V}{y^{2}}|\epsilon_{\ell}|^{2}.
\end{align*}
Using $-y\partial_{y}\tilde V\geq0$ and $|\Delta^{2}\phi_{M_{2}}|\lesssim\tfrac{1}{y^{2}}\mathbf{1}_{y\sim M_{2}}$,
we get 
\[
(A_{Q}A_{Q}^{\ast}\epsilon_{\ell},\Lambda_{M_{2}}\epsilon_{\ell})_{r}\geq{\textstyle \int}\phi_{M_{2}}''|\partial_{y}\epsilon_{\ell}|^{2}+{\textstyle \int}\tfrac{\phi_{M_{2}}'}{y}\cdot\tfrac{\tilde V}{y^{2}}|\epsilon_{\ell}|^{2}+O\big(\|\mathbf{1}_{y\sim M_{2}}\tfrac{1}{y}\epsilon_{\ell}\|_{L^{2}}^{2}\big).
\]
If one formally substitutes $M_{2}=\infty$, one would get the full
monotonicity as \eqref{eq:FormalMonotonicity}.

We now propagate the above lower bound on $\epsilon_{\ell}$ to that
in terms of $\epsilon$. Using $\phi_{M_{2}}''=\tfrac{1}{y}\phi_{M_{2}}'$
on the region $\{y\leq M_{2}\}$ and $\phi_{M_{2}}'=0$ on the region
$\{y\geq2M_{2}\}$, we have (when $\ell=2$) 
\begin{align*}
 & {\textstyle \int}\phi_{M_{2}}''|\partial_{y}\epsilon_{2}|^{2}+{\textstyle \int}\tfrac{\phi_{M_{2}}'}{y}\cdot\tfrac{\tilde V}{y^{2}}|\epsilon_{2}|^{2}\\
 & ={\textstyle \int}|\partial_{y}A_{Q}(\phi_{M_{2}}''\epsilon_{1})|^{2}+{\textstyle \int}\tfrac{\tilde V}{y^{2}}|A_{Q}(\phi_{M_{2}}''\epsilon_{1})|^{2}+O\big(\|\mathbf{1}_{y\sim M_{2}}|\epsilon_{1}|_{-2}\|_{L^{2}}^{2}\big)\\
 & ={\textstyle \int}|A_{Q}^{\ast}A_{Q}(\phi_{M_{2}}''\epsilon_{1})|^{2}+O\big(\|\mathbf{1}_{y\sim M_{2}}|\epsilon_{1}|_{-2}\|_{L^{2}}^{2}\big)
\end{align*}
and (when $\ell=4$)
\begin{align*}
 & {\textstyle \int}\phi_{M_{2}}''|\partial_{y}\epsilon_{4}|^{2}+{\textstyle \int}\tfrac{\phi_{M_{2}}'}{y}\cdot\tfrac{\tilde V}{y^{2}}|\epsilon_{4}|^{2}\\
 & ={\textstyle \int}|\partial_{y}A_{Q}A_{Q}^{\ast}A_{Q}(\phi_{M_{2}}''\epsilon_{1})|^{2}+{\textstyle \int}\tfrac{\tilde V}{y^{2}}|A_{Q}A_{Q}^{\ast}A_{Q}(\phi_{M_{2}}''\epsilon_{1})|^{2}+O\big(\|\mathbf{1}_{y\sim M_{2}}|\epsilon_{1}|_{-4}\|_{L^{2}}^{2}\big)\\
 & ={\textstyle \int}|A_{Q}^{\ast}A_{Q}A_{Q}^{\ast}A_{Q}(\phi_{M_{2}}''\epsilon_{1})|^{2}+O\big(\|\mathbf{1}_{y\sim M_{2}}|\epsilon_{1}|_{-4}\|_{L^{2}}^{2}\big).
\end{align*}
We now write $\epsilon_{1}=L_{Q}\epsilon$ and commute $\phi_{M_{2}}''$
and $L_{Q}$. Recall $L_{Q}=\D_{+}^{(Q)}+QB_{Q}$. Because of the
local nature of $\D_{+}^{(Q)}$, we can easily commute $\phi_{M_{2}}''$
and $\D_{+}^{(Q)}$ with an error localized in $\{y\sim M_{2}\}$.
Thus 
\[
\||[\phi_{M_{2}}'',\D_{+}^{(Q)}]\epsilon|_{-\ell}\|_{L^{2}}\lesssim\|\mathbf{1}_{y\sim M_{2}}|\epsilon|_{-(\ell+1)}\|_{L^{2}}.
\]
When we commute $\phi_{M_{2}}''$ and $QB_{Q}$, the error is not
necessarily supported in $\{y\sim M_{2}\}$. However, $[\phi_{M_{2}}'',QB_{Q}]\epsilon$
only uses the information of $\epsilon$ in $\{y\sim M_{2}\}$. Combining
this observation with either Lemma \ref{lem:Contribution-QBQ} or
\ref{lem:Contribution-QBQ-H5}, we have 
\[
\||[\phi_{M_{2}}'',QB_{Q}]\epsilon|_{-\ell}\|_{L^{2}}\lesssim\|\mathbf{1}_{y\sim M_{2}}|\epsilon|_{-(\ell+1)}\|_{L^{2}}.
\]
Thus we have 
\[
(A_{Q}A_{Q}^{\ast}\epsilon_{2},\Lambda_{M_{2}}\epsilon_{2})_{r}\geq\|A_{Q}^{\ast}A_{Q}L_{Q}(\phi_{M_{2}}''\epsilon)\|_{L^{2}}^{2}-O\big(\|\mathbf{1}_{y\sim M_{2}}|\epsilon|_{-3}\|_{L^{2}}^{2}\big)
\]
and when $m\geq3$ 
\begin{align*}
 & (A_{Q}A_{Q}^{\ast}\epsilon_{4},\Lambda_{M_{2}}\epsilon_{4})_{r}\\
 & \geq\|A_{Q}^{\ast}A_{Q}A_{Q}^{\ast}A_{Q}L_{Q}(\phi_{M_{2}}''\epsilon)\|_{L^{2}}^{2}-O\big(\|\mathbf{1}_{y\sim M_{2}}|\epsilon|_{-5}\|_{L^{2}}^{2}\big).
\end{align*}

To complete the proof, it suffices to derive a lower bound for the
adapted derivatives of $\phi_{M_{2}}''\epsilon$. The function $\phi_{M_{2}}''\epsilon$
does not necessarily satisfy the orthogonality conditions \eqref{eq:ortho-cond},
but it \emph{almost} satisfies them in the sense that 
\begin{align*}
|(\phi_{M_{2}}''\epsilon,\mathcal{Z}_{k})_{r}| & =|(\epsilon,(\phi_{M_{2}}''-1)\mathcal{Z}_{k})_{r}|\\
 & \lesssim\begin{cases}
\|\epsilon\|_{\dot{\mathcal{H}}_{m}^{3}}\|\langle y\rangle^{3}\mathbf{1}_{y\geq M_{2}}\mathcal{Z}_{k}\|_{L^{2}}=o_{M_{2}\to\infty}(1)\|\epsilon\|_{\dot{\mathcal{H}}_{m}^{3}} & \text{if }m\geq1,\\
\|\epsilon\|_{\dot{\mathcal{H}}_{m}^{5}}\|\langle y\rangle^{5}\mathbf{1}_{y\geq M_{2}}\mathcal{Z}_{k}\|_{L^{2}}=o_{M_{2}\to\infty}(1)\|\epsilon\|_{\dot{\mathcal{H}}_{m}^{5}} & \text{if }m\geq3,
\end{cases}
\end{align*}
for each $k\in\{1,2,3,4\}$. Thus we can apply the coercivity estimates
(Lemma \ref{lem:Coercivity-AAL-Section2} or \ref{lem:Coercivity-AAAAL-Section2})
with an additional error either $o_{M_{2}\to\infty}(1)\|\epsilon\|_{\dot{\mathcal{H}}_{m}^{3}}$
or $o_{M_{2}\to\infty}(1)\|\epsilon\|_{\dot{\mathcal{H}}_{m}^{5}}$.
This finishes the proof.
\end{proof}
In view of \eqref{eq:LocalizedRepulsivityH3} and \eqref{eq:LocalizedRepulsivityH5},
$\partial_{s}\{b(\epsilon_{2},i\Lambda_{M_{2}}\epsilon_{2})_{r}\}$
will have an error of the form 
\[
b\cdot O(\|\mathbf{1}_{y\sim M_{2}}|\epsilon|_{-3}\|_{L^{2}}^{2}).
\]
If one crudely discards the localization $\mathbf{1}_{y\sim M_{2}}$,
this error is on the borderline of the acceptable errors. In order
to make use of this localization, we use an averaging argument over
$M_{2}$, as in \cite{KimKwon2019arXiv}.
\begin{lem}[Local virial control]
\label{lem:LocalVirialControl}We have 
\begin{align}
 & \Big|\frac{b}{\log M_{2}}\int_{M_{2}}^{M_{2}^{2}}(\epsilon_{2},-i\Lambda_{M_{2}'}\epsilon_{2})_{r}\frac{dM_{2}'}{M_{2}'}\Big|\lesssim bM_{2}^{C}\|\epsilon_{3}\|_{L^{2}}^{2},\label{eq:LocalVirialBdryH3}\\
 & \Big(\partial_{s}-6\frac{\lambda_{s}}{\lambda}\Big)\Big[\frac{b}{\log M_{2}}\int_{M_{2}}^{M_{2}^{2}}(\epsilon_{2},-i\Lambda_{M_{2}'}\epsilon_{2})_{r}\frac{dM_{2}'}{M_{2}'}\Big]\label{eq:LocalVirialDerivH3}\\
 & \geq b\Big(c_{M}\|\epsilon\|_{\dot{\mathcal{H}}_{m,\leq M_{2}}^{3}}^{2}-o_{M_{2}\to\infty}(1)\|\epsilon\|_{\dot{\mathcal{H}}_{m}^{3}}^{2}-M_{2}^{C}\|\epsilon_{3}\|_{L^{2}}(b\|\epsilon\|_{\dot{\mathcal{H}}_{m}^{3}}+b^{\frac{7}{2}})\Big).\nonumber 
\end{align}
If $m\geq3$, then
\begin{align}
 & \Big|\frac{b}{\log M_{2}}\int_{M_{2}}^{M_{2}^{2}}(\epsilon_{4},-i\Lambda_{M_{2}'}\epsilon_{4})_{r}\frac{dM_{2}'}{M_{2}'}\Big|\lesssim bM_{2}^{C}\|\epsilon_{5}\|_{L^{2}}^{2},\label{eq:LocalVirialBdryH5}\\
 & \Big(\partial_{s}-10\frac{\lambda_{s}}{\lambda}\Big)\Big[\frac{b}{\log M_{2}}\int_{M_{2}}^{M_{2}^{2}}(\epsilon_{4},-i\Lambda_{M_{2}'}\epsilon_{4})_{r}\frac{dM_{2}'}{M_{2}'}\Big]\label{eq:LocalVirialDerivH5}\\
 & \geq b\Big(c_{M}\|\epsilon\|_{\dot{\mathcal{H}}_{m,\leq M_{2}}^{5}}^{2}-o_{M_{2}\to\infty}(1)\|\epsilon\|_{\dot{\mathcal{H}}_{m}^{5}}^{2}-M_{2}^{C}\|\epsilon_{5}\|_{L^{2}}(b\|\epsilon\|_{\dot{\mathcal{H}}_{m}^{5}}+b^{\frac{11}{2}})\Big).\nonumber 
\end{align}
\end{lem}

\begin{proof}
Here we only show how we can get \eqref{eq:LocalVirialBdryH3} and
\eqref{eq:LocalVirialDerivH3} from the computations in Lemma \ref{lem:E3-Identity},
\eqref{eq:LocalizedRepulsivityH3}, and the averaging argument. The
proof of \eqref{eq:LocalVirialBdryH5} and \eqref{eq:LocalVirialDerivH5}
would follow from the same argument using the computations in Lemma
\ref{lem:E5-identity} and \eqref{eq:LocalizedRepulsivityH5} instead
of Lemma \ref{lem:E3-Identity} and \eqref{eq:LocalizedRepulsivityH3}.

We claim the unaveraged version (and without $b$-factor) of the estimates
\begin{align*}
|(\epsilon_{2},-i\Lambda_{M_{2}}\epsilon_{2})_{r}| & \lesssim M_{2}^{2}\|\epsilon_{3}\|_{L^{2}}^{2},\\
\Big(\partial_{s}-6\frac{\lambda_{s}}{\lambda}\Big)(\epsilon_{2},-i\Lambda_{M_{2}}\epsilon_{2})_{r} & \geq c_{M}\|\epsilon\|_{\dot{\mathcal{H}}_{m,\leq M_{2}}^{3}}^{2}-O\big(\|\mathbf{1}_{y\sim M_{2}}|\epsilon|_{-3}\|_{L^{2}}^{2}\big)-o_{M_{2}\to\infty}(1)\|\epsilon\|_{\dot{\mathcal{H}}_{m}^{3}}^{2}\\
 & \quad-M_{2}\|\epsilon_{3}\|_{L^{2}}\cdot O(b\|\epsilon\|_{\dot{\mathcal{H}}_{m}^{3}}+b^{\frac{7}{2}})
\end{align*}
Assuming the claim, \eqref{eq:LocalVirialBdryH3} follows from 
\[
\Big|\frac{b}{\log M_{2}}\int_{M_{2}}^{M_{2}^{2}}(\epsilon_{2},-i\Lambda_{M_{2}'}\epsilon_{2})_{r}\frac{dM_{2}'}{M_{2}'}\Big|\lesssim\frac{b}{\log M_{2}}\int_{M_{2}}^{M_{2}^{2}}(M_{2}')^{2}\|\epsilon_{3}\|_{L^{2}}^{2}\frac{dM_{2}'}{M_{2}'}\lesssim M_{2}^{C}\|\epsilon_{3}\|_{L^{2}}^{2}.
\]
Next, \eqref{eq:LocalVirialDerivH3} follows from the identities 
\[
\Big(\partial_{s}-6\frac{\lambda_{s}}{\lambda}\Big)\Big(b(\epsilon_{2},-i\Lambda_{M_{2}'}\epsilon_{2})_{r}\Big)=b\Big(\partial_{s}-6\frac{\lambda_{s}}{\lambda}\Big)(\epsilon_{2},-i\Lambda_{M_{2}'}\epsilon_{2})_{r}+b_{s}(\epsilon_{2},-i\Lambda_{M_{2}'}\epsilon_{2})_{r},
\]
applying $|b_{s}|\lesssim b^{2}$, and the \emph{averaging argument}
(using Fubini):
\begin{align*}
\frac{1}{\log M_{2}}\int_{M_{2}}^{M_{2}^{2}}\frac{dM_{2}'}{M_{2}'} & =1,\\
\frac{1}{\log M_{2}}\int_{M_{2}}^{M_{2}^{2}}\Big(\int\mathbf{1}_{y\sim M_{2}}|\epsilon|_{-3}^{2}\Big)\frac{dM_{2}'}{M_{2}'} & \lesssim\frac{1}{\log M_{2}}\int\mathbf{1}_{M_{2}\lesssim y\lesssim M_{2}^{2}}|\epsilon|_{-3}^{2}=o_{M_{2}\to\infty}(1)\|\epsilon\|_{\dot{\mathcal{H}}_{m}^{3}}^{2}.
\end{align*}

From now on, we prove the above claim. We first note 
\[
\|y\Lambda_{M_{2}}\epsilon_{2}\|_{L^{2}}+\|[\Lambda_{M_{2}},\Lambda]\epsilon_{2}\|_{L^{2}}\lesssim M_{2}^{2}\|\epsilon_{3}\|_{L^{2}},
\]
which is a consequence of the crude estimate $|y\Lambda_{M_{2}}\epsilon_{2}|+|[\Lambda_{M_{2}},\Lambda]\epsilon_{2}|\lesssim\mathbf{1}_{y\leq2M_{2}}y^{2}|\epsilon_{2}|_{-1}$
and positivity of $A_{Q}^{\ast}A_{Q}$. In particular, the boundedness
property of the local virial functional is clear from 
\[
|(\epsilon_{2},-i\Lambda_{M_{2}}\epsilon_{2})_{r}|\lesssim\|\tfrac{1}{y}\epsilon_{2}\|_{L^{2}}\|y\Lambda_{M_{2}}\epsilon_{2}\|_{L^{2}}\lesssim M_{2}^{2}\|\epsilon_{3}\|_{L^{2}}^{2}.
\]
We now turn to the monotonicity estimate. As $-i\Lambda_{M_{2}}$
is symmetric, we have 
\begin{align*}
 & \frac{1}{2}\Big(\partial_{s}-6\frac{\lambda_{s}}{\lambda}\Big)(\epsilon_{2},-i\Lambda_{M_{2}}\epsilon_{2})_{r}=(\partial_{s}\epsilon_{2},-i\Lambda_{M_{2}}\epsilon_{2})_{r}-3\frac{\lambda_{s}}{\lambda}(\epsilon_{2},-i\Lambda_{M_{2}}\epsilon_{2})_{r}\\
 & =(A_{Q}^{\ast}A_{Q}\epsilon_{2},\Lambda_{M_{2}}\epsilon_{2})_{r}+\frac{\lambda_{s}}{\lambda}(y\partial_{y}\epsilon_{2},-i\Lambda_{M_{2}}\epsilon_{2})_{r}+(\text{RHS of }\eqref{eq:e2-eq},-i\Lambda_{M_{2}}\epsilon_{2})_{r}
\end{align*}
As seen in Lemma \ref{lem:LocalizedRepulsivity}, the monotonicity
comes from the first term with acceptable errors. We now show that
the remaining terms can be treated as errors. First, from the estimate
\begin{align*}
(y\partial_{y}\epsilon_{2},-i\Lambda_{M_{2}}\epsilon_{2})_{r} & =(\Lambda\epsilon_{2},-i\Lambda_{M_{2}}\epsilon_{2})_{r}-(\epsilon_{2},-i\Lambda_{M_{2}}\epsilon_{2})_{r}\\
 & =\tfrac{1}{2}([-i\Lambda_{M_{2}},\Lambda]\epsilon_{2},\epsilon_{2})_{r}-(\epsilon_{2},-i\Lambda_{M_{2}}\epsilon_{2})_{r}=O(M_{2}^{2}\|\epsilon_{3}\|_{L^{2}}^{2}),
\end{align*}
we have 
\[
\Big|\frac{\lambda_{s}}{\lambda}(y\partial_{y}\epsilon_{2},-i\Lambda_{M_{2}}\epsilon_{2})_{r}\Big|\lesssim b(M_{2})^{2}\|\epsilon_{3}\|_{L^{2}}^{2}.
\]
Next, proceeding as in the proof of Lemma \ref{lem:E3-Identity},
we have 
\[
\|y^{-1}(\text{RHS of }\eqref{eq:e2-eq})\|_{L^{2}}\lesssim b\|\epsilon\|_{\dot{\mathcal{H}}_{m}^{3}}+b^{\frac{7}{2}}.
\]
Thus 
\begin{align*}
|(\text{RHS of }\eqref{eq:e2-eq},-i\Lambda_{M_{2}}\epsilon_{2})_{r}| & \lesssim\|y^{-1}(\text{RHS of }\eqref{eq:e2-eq})\|_{L^{2}}\cdot M_{2}\|\epsilon_{3}\|_{L^{2}}\\
 & \lesssim M_{2}\|\epsilon_{3}\|_{L^{2}}(b\|\epsilon\|_{\dot{\mathcal{H}}_{m}^{3}}+b^{\frac{7}{2}}).
\end{align*}
Summing up the above estimates, we obtain the monotonicity.
\end{proof}
As explained above, the lower bounds 
\[
c_{M}b\|\epsilon\|_{\dot{\mathcal{H}}_{m,\leq M_{2}}^{3}}^{2}\quad\text{and}\quad c_{M}b\|\epsilon\|_{\dot{\mathcal{H}}_{m,\leq M_{2}}^{5}}^{2}
\]
from the local virial controls dominate the dangerous contributions
\[
O(b\|\epsilon_{3}\|_{L^{2}}\|\epsilon\|_{\dot{\mathcal{H}}_{m,\leq M_{2}}^{3}})\quad\text{and}\quad O(b\|\epsilon_{5}\|_{L^{2}}\|\epsilon\|_{\dot{\mathcal{H}}_{m,\leq M_{2}}^{5}})
\]
of the preliminary energy estimates \eqref{eq:EnergyIdentityHdot3}
and \eqref{eq:EnergyIdentityHdot5}. For this purpose, let us define
the \emph{modified energies} by 
\begin{align*}
\mathcal{F}_{3} & \coloneqq\|\epsilon_{3}\|_{L^{2}}^{2}-M_{1}\frac{b}{\log M_{2}}\int_{M_{2}}^{M_{2}^{2}}(\epsilon_{2},-i\Lambda_{M_{2}'}\epsilon_{2})_{r}\frac{dM_{2}'}{M_{2}'},\\
\mathcal{F}_{5} & \coloneqq\|\epsilon_{5}\|_{L^{2}}^{2}-M_{1}\frac{b}{\log M_{2}}\int_{M_{2}}^{M_{2}^{2}}(\epsilon_{4},-i\Lambda_{M_{2}'}\epsilon_{4})_{r}\frac{dM_{2}'}{M_{2}'}.
\end{align*}

\begin{prop}[Modified energy inequality for $\dot{\mathcal{H}}_{m}^{3}$ and $\dot{\mathcal{H}}_{m}^{5}$]
\label{prop:ModifiedEnergyInequalityH3}There exists some universal
constant $C$ such that the following estimates hold.
\begin{enumerate}
\item (Modified energy inequality for $\dot{\mathcal{H}}_{m}^{3}$)
\begin{align}
|\mathcal{F}_{3}-\|\epsilon_{3}\|_{L^{2}}^{2}| & \leq\frac{1}{100}\|\epsilon_{3}\|_{L^{2}}^{2},\label{eq:ModifiedEnergyBdryH3}\\
\Big(\partial_{s}-6\frac{\lambda_{s}}{\lambda}\Big)\mathcal{F}_{3} & \leq b\Big(\frac{1}{100}\|\epsilon_{3}\|_{L^{2}}^{2}+Cb^{5}\Big).\label{eq:ModifiedEnergyDerivH3}
\end{align}
\item (Energy identity for $\dot{H}_{m}^{3}$ when $m\geq3$)
\begin{equation}
\Big|\Big(\partial_{s}-6\frac{\lambda_{s}}{\lambda}\Big)\|\epsilon_{3}\|_{L^{2}}^{2}\Big|\leq Cb\|\epsilon_{3}\|_{L^{2}}\cdot b^{\frac{7}{2}}.\label{eq:EnergyIdentityH3-low-Sobolev}
\end{equation}
\item (Modified energy inequality for $\dot{\mathcal{H}}_{m}^{5}$ when
$m\geq3$)
\begin{align}
|\mathcal{F}_{5}-\|\epsilon_{5}\|_{L^{2}}^{2}| & \leq\frac{1}{100}\|\epsilon_{5}\|_{L^{2}}^{2},\label{eq:ModifiedEnergyBdryH5}\\
\Big(\partial_{s}-10\frac{\lambda_{s}}{\lambda}\Big)\mathcal{F}_{5} & \leq b\Big(\frac{1}{100}\|\epsilon_{5}\|_{L^{2}}^{2}+Cb^{9}\Big).\label{eq:ModifiedEnergyDerivH5}
\end{align}
\end{enumerate}
\end{prop}

\begin{rem}
The number $\frac{1}{100}$ can be replaced by any small number.
\end{rem}

\begin{rem}
In the proof, we will see why we require $M\ll M_{1}\ll M_{2}$ for
the large parameters $M$, $M_{1}$, $M_{2}$. (Remark \ref{rem:ParameterDependence})
\end{rem}

\begin{proof}
(1) \eqref{eq:ModifiedEnergyBdryH3} follows from $M_{2}^{C}b\ll1$.
We now show \eqref{eq:ModifiedEnergyDerivH3}. Using \eqref{eq:EnergyIdentityHdot3},
\eqref{eq:LocalVirialDerivH3}, and parameter dependence (Remark \ref{rem:ParameterDependence}),
we have 
\begin{align*}
\Big(\partial_{s}-6\frac{\lambda_{s}}{\lambda}\Big)\mathcal{F}_{3} & \leq b\Big(C-M_{1}c_{M}\Big)\|\epsilon\|_{\dot{\mathcal{H}}_{m,\leq M_{2}}^{3}}^{2}\\
 & \quad+b\Big((\frac{1}{200}+o_{M\to\infty}(1))\|\epsilon_{3}\|_{L^{2}}^{2}+M_{1}o_{M_{2}\to\infty}(1)\|\epsilon\|_{\dot{\mathcal{H}}_{m}^{3}}^{2}+Cb^{\frac{5}{2}}\Big),
\end{align*}
where $C$ is some universal constant. Notice that the first term
of the RHS becomes \emph{negative} due to $M\ll M_{1}$. The proof
is now finished by applying the coercivity estimate (Lemma \ref{lem:Coercivity-AAL-Section2})
and parameter dependence (Remark \ref{rem:ParameterDependence}) for
the remaining terms.

(2) This is merely a restatement of \eqref{eq:EnergyIdentityHdot3-m-geq-3}.

(3) This follows as the proof of (1) using \eqref{eq:EnergyIdentityHdot5},
\eqref{eq:LocalVirialDerivH5}, and Lemma \ref{lem:Coercivity-AAAAL-Section2}
instead of their $\dot{\mathcal{H}}_{m}^{3}$-versions. We omit the
proof.
\end{proof}

\subsection{\label{subsec:EnergyEstimateL2}Energy estimate for $L^{2}$}

In this subsection, we consider the time-variation of $\|\epsilon\|_{L^{2}}^{2}$.
To close the bootstrap hypothesis for $\|\epsilon\|_{L^{2}}$, we
use the energy method for the $\epsilon$-equation. Thanks to the
mass conservation, one can simply close the bootstrap for $\|\epsilon\|_{L^{2}}$.
But here we provide the proof not relying on the mass conservation.
There are two reasons for this. First, in Section \ref{sec:Lipschitz-blow-up-manifold},
we will need $L^{2}$-estimates of the difference of $\epsilon$'s
(say $\delta\epsilon=\epsilon-\epsilon'$), where the conservation
laws cannot be applied. Second, the argument in this subsection will
in fact estimate $\partial_{t}\epsilon^{\sharp}$ (i.e. the time-variation
of $\epsilon$ in the original scalings $(t,r)$). This will be used
in Section \ref{subsec:Sharp-description}. Henceforth, we start with
measuring the time-variation of $\|\epsilon\|_{L^{2}}$.

Recall the equation \eqref{eq:e-eq-cor} of $\epsilon$. In contrast
to the equations for $\epsilon_{1},\dots,\epsilon_{5}$, the linear
part $i\mathcal{L}_{Q}\epsilon$ is \emph{not} anti-symmetric due
to $[\mathcal{L}_{Q},i]\neq0$. We will merely replace the linear
part of $\epsilon$-equation by $-i\Delta_{m}\epsilon$, where $\Delta_{m}=\partial_{yy}+\tfrac{1}{y}\partial_{y}-\tfrac{m^{2}}{y^{2}}$
is the Laplacian acting on $m$-equivariant functions. The resulting
error $i(\mathcal{L}_{Q}+\Delta_{m})\epsilon$ is linear in $\epsilon$
so it should be worse than $R_{\mathrm{L-L}}$ or $R_{\mathrm{NL}}$.
However, we have a good \emph{higher Sobolev control} so we can roughly
estimate this by $\|\epsilon\|_{\dot{H}_{m}^{2}}\lesssim b^{\frac{7}{4}}$
(by interpolation). Since $b^{\frac{7}{4}}\ll b$, this is safe.

In this regard, we will look at a rougher version of the $\epsilon$-equation:
\begin{equation}
(\partial_{s}-\frac{\lambda_{s}}{\lambda}\Lambda+\gamma_{s}i)\epsilon-i\Delta_{m}\epsilon=\mathbf{Mod}\cdot\mathbf{v}-i[\mathcal{N}(P+\epsilon)-\mathcal{N}(P)]-i\Psi.\label{eq:e-equation-L2}
\end{equation}
We estimate the RHS of \eqref{eq:e-equation-L2}.
\begin{lem}[$L^{2}$-estimate for the remainder terms]
\label{lem:L2-EstimateRemainders}We have 
\begin{align}
\|\mathbf{Mod}\cdot\mathbf{v}\|_{L^{2}} & \lesssim|\log b|^{\frac{1}{2}}(\|\epsilon\|_{\dot{\mathcal{H}}_{m}^{3}}+\|\epsilon\|_{\dot{H}_{m}^{1}}^{2}),\label{eq:L2-Mod-est}\\
\|\mathcal{N}(P+\epsilon)-\mathcal{N}(P)\|_{L^{2}} & \lesssim\|\epsilon\|_{\dot{H}_{m}^{1}}^{\frac{1}{2}}\|\epsilon\|_{\dot{\mathcal{H}}_{m}^{3}}^{\frac{1}{2}}+\|\epsilon\|_{\dot{H}_{m}^{1}}^{2},\label{eq:L2-nonlinear-est}\\
\|\Psi\|_{L^{2}} & \lesssim b^{\frac{m}{2}+\frac{3}{2}}.\label{eq:L2-Radiation-est}
\end{align}
\end{lem}

\begin{rem}
It is only important to have bounds of the form $b^{1+}$.
\end{rem}

\begin{proof}
The estimate \eqref{eq:L2-Mod-est} follows from 
\begin{align}
|\mathbf{Mod}| & \lesssim|\tilde{\mathbf{Mod}}|+|\theta_{\mathrm{L-L}}+\theta_{\mathrm{NL}}|\lesssim\|\epsilon\|_{\dot{\mathcal{H}}_{m}^{3}}+\|\epsilon\|_{\dot{H}_{m}^{1}}^{2},\label{eq:non-tilde-mod-est}\\
\|\mathbf{v}\|_{L^{2}} & \lesssim|\log b|^{\frac{1}{2}}.\nonumber 
\end{align}
The estimate \eqref{eq:L2-Radiation-est} follows from the pointwise
estimates of $\Psi$ \eqref{eq:Psi-Pointwise}.

It remains to show \eqref{eq:L2-nonlinear-est}. View 
\[
\mathcal{N}(P+\epsilon)-\mathcal{N}(P)=\sum_{\substack{\psi_{1},\psi_{2},\psi_{3}\in\{P,\epsilon\}\\
\#\{j:\psi_{j}=\epsilon\}\geq1
}
}\mathcal{N}_{3}(\psi_{1},\psi_{2},\psi_{3})+\sum_{\substack{\psi_{1},\dots,\psi_{5}\in\{P,\epsilon\}\\
\#\{j:\psi_{j}=\epsilon\}\geq1
}
}\mathcal{N}_{5}(\psi_{1},\dots,\psi_{5}).
\]
Next, we recall the result of \cite[Lemma 2.3]{KimKwon2019arXiv}:
the $L^{2}$-norm of $\mathcal{N}$ can be estimated by the product
of $\dot{H}_{m}^{1}$-norms of any two of its arguments and $L^{2}$-norms
for the remaining arguments. In particular, the nonlinear terms with
$\#\{j:\psi_{j}=\epsilon\}\geq2$ can be estimated by $\|\epsilon\|_{\dot{H}_{m}^{1}}^{2}$.

The argument of the previous paragraph does not apply for the terms
with $\#\{j:\psi_{j}=\epsilon\}=1$. In this case, we view $\mathcal{N}$
as a $V\psi$ form as in Section \ref{subsec:Estimates-of-remainders-H3}.
If $\psi=\epsilon$, then $V$ is either $V_{3}[P]$ or $V_{5}[P]$,
which are estimated by $|V|\lesssim\langle y\rangle^{-2}$. Therefore,
\[
\|V\epsilon\|_{L^{2}}\lesssim\|\langle y\rangle^{-3}\epsilon\|_{L^{2}}^{\frac{1}{2}}\|\langle y\rangle^{-1}\epsilon\|_{L^{2}}^{\frac{1}{2}}\lesssim\|\epsilon\|_{\dot{\mathcal{H}}_{m}^{3}}^{\frac{1}{2}}\|\epsilon\|_{\dot{H}_{m}^{1}}^{\frac{1}{2}}.
\]
If $\psi=P$, then $V$ is either $V_{3}[P,\epsilon]$, $V_{5}[P,P,P,\epsilon]$,
or $V_{5}[P,\epsilon,P,P]$, where we always find the product $\Re(P\epsilon)$
in the expression of $V$. Thus we estimate using \eqref{eq:interpolation}
\[
\|VP\|_{L^{2}}\lesssim\|V\|_{L^{\infty}}\lesssim\|P\epsilon\|_{L^{\infty}}+\|y^{-2}P\epsilon\|_{L^{1}}\lesssim\|\epsilon\|_{\dot{H}_{m}^{1}}^{\frac{1}{2}}\|\epsilon\|_{\dot{\mathcal{H}}_{m}^{3}}^{\frac{1}{2}}.
\]
This completes the proof.
\end{proof}
\begin{lem}[$L^{2}$-energy estimate]
\label{lem:L2-EnergyEstimate}We have 
\[
\partial_{s}\|\epsilon\|_{L^{2}}^{2}\lesssim b^{\frac{3}{2}}.
\]
\end{lem}

\begin{proof}
Note that 
\[
\tfrac{1}{2}\partial_{s}\|\epsilon\|_{L^{2}}^{2}=(\epsilon,\text{RHS of }\eqref{eq:e-equation-L2})_{r}\lesssim\|\epsilon\|_{L^{2}}\|\text{RHS of }\eqref{eq:e-equation-L2}\|_{L^{2}}.
\]
Applying Lemma \ref{lem:L2-EstimateRemainders}, substituting the
bootstrap hypotheses, and using $K\lesssim b^{0-}$, the proof follows.
\end{proof}

\subsection{\label{subsec:ClosingBootstrap}Closing the bootstrap}

Here we finish the proof of Proposition \ref{prop:main-bootstrap}
by gathering the modulation estimates and (modified) energy estimates
for adaptive derivatives of $\epsilon$.
\begin{lem}[Consequences of modulation estimates]
We have 
\begin{align}
\int_{0}^{t}\frac{b}{\lambda^{2}}\cdot\frac{b^{5}}{\lambda^{6}}d\tau & \leq\Big(1+O((b^{\ast})^{\frac{1}{2}})\Big)\frac{b^{5}(t)}{\lambda^{6}(t)},\label{eq:closing-bootstrap-claim1}\\
\frac{\lambda^{\frac{6}{5}}(t)}{b(t)} & \leq\frac{\lambda_{0}^{\frac{6}{5}}}{b_{0}},\label{eq:closing-bootstrap-claim2}\\
\int_{0}^{t}\frac{b^{\frac{3}{2}}}{\lambda^{2}}d\tau & \leq2\Big(1+O((b^{\ast})^{\frac{1}{2}}))\Big)b_{0}^{\frac{1}{2}},\label{eq:closing-bootstrap-claim3}\\
\frac{b(t)}{\lambda(t)} & =\Big(1+O((b^{\ast})^{\frac{1}{2}})\Big)\frac{b_{0}}{\lambda_{0}}.\label{eq:closing-bootstrap-claim4}
\end{align}
\end{lem}

\begin{proof}
The assertions will be obtained by integrating in time using the modulation
estimates 
\[
b=-\frac{\lambda_{s}}{\lambda}+O(b^{\frac{5}{2}})\quad\text{and}\quad b_{s}=-b^{2}+O(b^{\frac{5}{2}}),
\]
which are obtained by Lemma \ref{lem:mod-est} and applying the bootstrap
hypothesis with $K\ll M$.

Now \eqref{eq:closing-bootstrap-claim1} follows from $\frac{b}{\lambda^{2}}=-\frac{\lambda_{t}}{\lambda}+O(\frac{b^{5/2}}{\lambda^{2}})$
and integration by parts: 
\begin{align*}
\int_{0}^{t}\frac{b}{\lambda^{2}}\cdot\frac{b^{5}}{\lambda^{6}}d\tau & =\frac{1}{6}\Big[\frac{b^{5}}{\lambda^{6}}\Big]_{0}^{t}-\frac{5}{6}\int_{0}^{t}\frac{b_{t}b^{4}}{\lambda^{6}}d\tau+O\Big((b^{\ast})^{\frac{3}{2}}\int_{0}^{t}\frac{b}{\lambda^{2}}\cdot\frac{b^{5}}{\lambda^{6}}d\tau\Big)\\
 & =\frac{1}{6}\Big[\frac{b^{5}}{\lambda^{6}}\Big]_{0}^{t}+\frac{5}{6}\int_{0}^{t}\frac{b}{\lambda^{2}}\cdot\frac{b^{5}}{\lambda^{6}}d\tau+O\Big((b^{\ast})^{\frac{1}{2}}\int_{0}^{t}\frac{b}{\lambda^{2}}\cdot\frac{b^{5}}{\lambda^{6}}d\tau\Big).
\end{align*}
\eqref{eq:closing-bootstrap-claim2} follows from 
\[
\partial_{s}\log\Big(\frac{\lambda^{\frac{6}{5}}}{b}\Big)=\frac{\lambda_{s}}{5\lambda}+\Big(\frac{\lambda_{s}}{\lambda}+b\Big)-\Big(\frac{b_{s}+b^{2}}{b}\Big)=-\frac{b}{5}+O(b^{\frac{3}{2}})\leq0.
\]
\eqref{eq:closing-bootstrap-claim3} follows from 
\[
\int_{0}^{t}\frac{b^{\frac{3}{2}}}{\lambda^{2}}d\tau=\int_{0}^{t}\frac{1}{b^{\frac{1}{2}}}\Big(-b_{t}+\frac{b_{s}+b^{2}}{\lambda^{2}}\Big)d\tau=-2[b^{\frac{1}{2}}]_{0}^{t}+O\Big((b^{\ast})^{\frac{1}{2}}\int_{0}^{t}\frac{b^{\frac{3}{2}}}{\lambda^{2}}d\tau\Big).
\]
\eqref{eq:closing-bootstrap-claim4} follows from integrating 
\begin{equation}
\Big|\partial_{t}\log\Big(\frac{b}{\lambda}\Big)\Big|=\frac{1}{\lambda^{2}}\Big|\Big(\frac{b_{s}+b^{2}}{b}\Big)-\Big(\frac{\lambda_{s}}{\lambda}+b\Big)\Big|\lesssim\frac{b^{\frac{3}{2}}}{\lambda^{2}}\label{eq:closing-bootstrap-claim4-1}
\end{equation}
using \eqref{eq:closing-bootstrap-claim3}.
\end{proof}
We finish the proof of Proposition \ref{prop:main-bootstrap}.
\begin{proof}[End of the proof of Proposition \ref{prop:main-bootstrap}]
We note that $b\leq b_{0}$ is clear from $b_{s}=-b^{2}+O(b^{\frac{5}{2}})$.
Thus we focus on the estimates on $\epsilon$.

We first close the $\|\epsilon_{3}\|_{L^{2}}$-bound for $H_{m}^{3}$-solutions.
By the modified energy inequality \eqref{eq:ModifiedEnergyDerivH3},
bootstrap hypothesis \eqref{eq:bootstrap-hyp-H3}, and \eqref{eq:closing-bootstrap-claim1},
we have 
\begin{align*}
\frac{\mathcal{F}_{3}(t)}{\lambda^{6}(t)} & \leq\frac{\mathcal{F}_{3}(0)}{\lambda_{0}^{6}}+\Big(\frac{K^{2}}{100}+C\Big)\int_{0}^{t}\frac{b}{\lambda^{2}}\cdot\frac{b^{5}}{\lambda^{6}}d\tau\\
 & \leq\frac{\mathcal{F}_{3}(0)}{\lambda_{0}^{6}}+\big(1+O((b^{\ast})^{\frac{1}{2}})\big)\Big(\frac{K^{2}}{100}+C\Big)\frac{b^{5}(t)}{\lambda^{6}(t)}.
\end{align*}
Applying \eqref{eq:ModifiedEnergyBdryH3} and \eqref{eq:closing-bootstrap-claim2},
we have 
\[
\frac{99}{100}\|\epsilon_{3}(t)\|_{L^{2}}^{2}\leq\frac{101}{100}\Big(\frac{b(t)}{b_{0}}\Big)^{5}\|\epsilon_{3}(0)\|_{L^{2}}^{2}+\big(1+O((b^{\ast})^{\frac{1}{2}})\big)\Big(\frac{K^{2}}{100}+C\Big)b^{5}.
\]
Using the initial bound $\|\epsilon_{3}(0)\|_{L^{2}}\leq b_{0}^{3}$,
this closes the $\|\epsilon_{3}\|_{L^{2}}$-bound for $H_{m}^{3}$-solutions,
due to $K\gg1$.

Next, we close the $\|\epsilon_{3}\|_{L^{2}}$-bound for $H_{m}^{5}$-solutions
when $m\geq3$. Here, we merely apply the energy identity \eqref{eq:EnergyIdentityH3-low-Sobolev}
and the claims \eqref{eq:closing-bootstrap-claim4} and \eqref{eq:closing-bootstrap-claim3}:
\[
\Big|\frac{\|\epsilon_{3}(t)\|_{L^{2}}^{2}}{\lambda^{6}(t)}-\frac{\|\epsilon_{3}(0)\|_{L^{2}}^{2}}{\lambda^{6}(0)}\Big|\lesssim K\int_{0}^{t}\frac{b}{\lambda^{8}}\cdot b^{6+\frac{1}{2}}d\tau\lesssim K\Big(\frac{b_{0}}{\lambda_{0}}\Big)^{6}\int_{0}^{t}\frac{b^{\frac{3}{2}}}{\lambda^{2}}d\tau\lesssim K\Big(\frac{b_{0}}{\lambda_{0}}\Big)^{6}b_{0}^{\frac{1}{2}}.
\]
Again by \eqref{eq:closing-bootstrap-claim4}, we get 
\[
\|\epsilon_{3}(t)\|_{L^{2}}^{2}\lesssim\Big(\frac{b(t)}{b_{0}}\Big)^{6}\|\epsilon_{3}(0)\|_{L^{2}}^{2}+Kb_{0}^{\frac{1}{2}}\cdot b^{6}(t).
\]
Using the initial bound $\|\epsilon_{3}(0)\|_{L^{2}}\leq b_{0}^{3}$,
this closes the $\|\epsilon_{3}\|_{L^{2}}$-bound for $H_{m}^{5}$-solutions
when $m\geq3$, due to $K\gg1$.

Next, we close the $\|\epsilon_{5}\|_{L^{2}}$-bound for $H_{m}^{5}$-solutions
when $m\geq3$. Applying the modified energy inequality \eqref{eq:ModifiedEnergyDerivH3}
and analogues of the claims \eqref{eq:closing-bootstrap-claim1} and
\eqref{eq:closing-bootstrap-claim2} (replace $\frac{b^{5}}{\lambda^{6}}$
by $\frac{b^{9}}{\lambda^{10}}$), we have 
\[
\frac{99}{100}\|\epsilon_{5}(t)\|_{L^{2}}^{2}\leq\frac{101}{100}\Big(\frac{b(t)}{b_{0}}\Big)^{9}\|\epsilon_{5}(0)\|_{L^{2}}^{2}+\big(1+O((b^{\ast})^{\frac{1}{2}})\big)\Big(\frac{K^{2}}{100}+C\Big)b^{9}.
\]
Using the initial bound $\|\epsilon_{5}(0)\|_{L^{2}}\leq b_{0}^{5}$,
this closes the $\|\epsilon_{3}\|_{L^{2}}$-bound for $H_{m}^{3}$-solutions,
due to $K\gg1$.

Next, we close the $\|\epsilon_{1}\|_{L^{2}}$-bound. We use the energy
conservation 
\[
\frac{1}{\lambda_{0}^{2}}\int|\D_{+}^{(P(\cdot;b_{0},\eta_{0})+\epsilon_{0})}(P(\cdot;b_{0},\eta_{0})+\epsilon_{0})|^{2}=\frac{1}{\lambda^{2}}\int|\D_{+}^{(P+\epsilon)}(P+\epsilon)|^{2}.
\]
We write 
\[
\D_{+}^{(P+\epsilon)}(P+\epsilon)=\D_{+}^{(P)}P+L_{Q}\epsilon+(L_{P}\epsilon-L_{Q}\epsilon)+N_{P}(\epsilon).
\]
Note that $\|\D_{+}^{(P)}P\|_{L^{2}}\lesssim b$. We note that the
last two terms are linear combinations of the expression 
\begin{equation}
\tfrac{1}{y}({\textstyle \int_{0}^{y}}\Re(\overline{\psi_{1}}\psi_{2})y'dy')\psi_{3},\label{eq:H1-expression}
\end{equation}
whose $L^{2}$-norm can be bounded by (see for instance \cite[Lemma 2.3]{KimKwon2019arXiv})
\[
\|\psi_{j_{1}}\|_{L^{2}}\|\psi_{j_{2}}\|_{L^{2}}\|\psi_{j_{3}}\|_{\dot{H}_{m}^{1}},
\]
where $\{j_{1},j_{2},j_{3}\}=\{1,2,3\}$ can be arbitrarily chosen.
We note that each term of $L_{P}\epsilon-L_{Q}\epsilon$ is of the
form \eqref{eq:H1-expression} with $\psi_{j_{1}}=P-Q$, $\psi_{j_{2}}\in\{P,Q\}$,
$\psi_{j_{3}}=\epsilon$, so its contribution is bounded by $b|\log b|^{\frac{1}{2}}\|\epsilon\|_{\dot{H}_{m}^{1}}=O(b)$.
Similarly, each term of $N_{P}(\epsilon)$ is of the form \eqref{eq:H1-expression}
with $\psi_{j_{1}}\in\{P,\epsilon\}$ and $\psi_{j_{2}}=\psi_{j_{3}}=\epsilon$,
so its contribution is bounded by $\|\epsilon\|_{\dot{H}_{m}^{1}}\|\epsilon\|_{L^{2}}\lesssim K(b^{\ast})^{\frac{1}{4}}\|\epsilon\|_{\dot{H}_{m}^{1}}=O(b)$.
Therefore, we have 
\begin{align*}
\lambda\|\epsilon_{1}\|_{L^{2}} & =\|\D_{+}^{(P+\epsilon)}(P+\epsilon)\|_{L^{2}}+O(b)\\
 & =\|\D_{+}^{(P(\cdot;b_{0},\eta_{0})+\epsilon_{0})}(P(\cdot;b_{0},\eta_{0})+\epsilon_{0})\|_{L^{2}}+O(b_{0})\\
 & =\lambda_{0}\|\epsilon_{1}(0)\|_{L^{2}}+O(b_{0}).
\end{align*}
Applying \eqref{eq:closing-bootstrap-claim4} yields
\[
\|\epsilon_{1}\|_{L^{2}}\lesssim b.
\]
As $K\gg1$, this closes the $\|\epsilon_{1}\|_{L^{2}}$-bound.

Finally, we close the $\|\epsilon\|_{L^{2}}$-bound. One can use the
mass conservation (similarly as in the $\dot{H}_{m}^{1}$-estimate
above), but as mentioned in Section \ref{subsec:EnergyEstimateL2},
we will rely on $\partial_{t}\|\epsilon\|_{L^{2}}^{2}$. By Lemma
\ref{lem:L2-EnergyEstimate}, we have 
\[
\partial_{t}\|\epsilon(t)\|_{L^{2}}^{2}\lesssim\lambda^{-2}b^{\frac{3}{2}}.
\]
Integrating this using \eqref{eq:closing-bootstrap-claim3} and the
initial bound yields 
\[
\|\epsilon(t)\|_{L^{2}}^{2}\lesssim\|\epsilon_{0}\|_{L^{2}}^{2}+\int_{0}^{t}\frac{b^{\frac{3}{2}}}{\lambda^{2}}d\tau\lesssim b_{0}^{\frac{1}{2}}.
\]
This closes the $\|\epsilon\|_{L^{2}}$-bound due to $K\gg1$.
\end{proof}

\subsection{\label{subsec:Existence-of-special-eta}Existence of special $\eta_{0}$}

Here we prove Proposition \ref{prop:Sets-I-pm}.
\begin{proof}[Proof of Proposition \ref{prop:Sets-I-pm}]
To show nonempty-ness of $\mathcal{I}_{\pm}$, we show $\pm\frac{K}{10}b_{0}^{3/2}\in\mathcal{I}_{\pm}$.
We compute the variation of the ratio $\frac{\eta}{b^{3/2}}$: 
\[
\partial_{s}\Big(\frac{\eta}{b^{\frac{3}{2}}}\Big)=\frac{3}{2}\frac{\eta}{b^{\frac{3}{2}}}\Big(b-\frac{b_{s}+b^{2}}{b}\Big)+\frac{\eta_{s}}{b^{\frac{3}{2}}}=\frac{3}{2}\Big(\frac{\eta}{b^{\frac{3}{2}}}\Big)b(1+O(b^{\frac{1}{2}}))+O(b).
\]
Thus if $|\frac{\eta}{b^{3/2}}|\geq\frac{K}{10}$ (recall that $K\gg1$)
holds at some time, then $|\frac{\eta}{b^{3/2}}|$ starts to increase.
In particular, if $\eta_{0}=\pm\frac{K}{10}b_{0}^{3/2}$, then $\eta(T_{\mathrm{exit}})$
must have same sign with $\eta_{0}$, saying that $\pm\frac{K}{10}b_{0}^{3/2}\in\mathcal{I}_{\pm}$.

We now show that $\mathcal{I}_{\pm}$ is open. Fix $\eta_{0}\in\mathcal{I}_{\pm}$
and let us denote by $u^{(\eta_{0})}$ the corresponding evolution.
Let us add a superscript $(\eta_{0})$ for clarification. Since $\eta_{0}\in\mathcal{I}_{\pm}$,
there exists $t^{(\eta_{0})}\in[0,T_{\mathrm{exit}}^{(\eta_{0})})$
such that $\pm\eta^{(\eta_{0})}(t^{(\eta_{0})})>\frac{K}{2}b^{\frac{3}{2}}(t^{(\eta_{0})})$.
Now, by continuous dependence, (obtained by combining the local well-posedness
and Lemma \ref{lem:decomp}) we should have $t^{(\eta_{0})}\in[0,T_{\mathrm{exit}}^{(\tilde{\eta}_{0})})$
and $\pm\eta^{(\tilde{\eta}_{0})}(t^{(\eta_{0})})>\frac{K}{2}b^{\frac{3}{2}}(t^{(\eta_{0})})$
for all $\tilde{\eta}_{0}$ near $\eta_{0}$. Such $\tilde{\eta}_{0}$
belongs to $\mathcal{I}_{\pm}$ due to the argument in the previous
paragraph. This completes the proof.
\end{proof}

\subsection{\label{subsec:Sharp-description}Pseudoconformal blow-up of trapped
solutions}

Here, we conclude the proof of Theorem \ref{thm:Existence} and the
first part of Theorem \ref{thm:BlowupManifold}. By the reduction
in Section \ref{subsec:Existence-of-trapped-solutions}, it only remains
to show that trapped solutions are pseudoconformal blow-up solutions.
Both the proofs for $H_{m}^{3}$-trapped solutions and $H_{m}^{5}$-trapped
solutions are quite similar. The only difference for $H_{m}^{3}$-trapped
solutions and $H_{m}^{5}$-trapped solutions is the regularity of
the radiation $u^{\ast}$, i.e. $u^{\ast}\in H_{m}^{1}$ or $u^{\ast}\in H_{m}^{3}$.
Here we prove it for $H_{m}^{3}$-trapped solutions. For $H_{m}^{5}$-solutions,
see Remark \ref{rem:forH5solutions}.
\begin{proof}[End of the proof of Theorem \ref{thm:Existence}]
By the claim \eqref{eq:closing-bootstrap-claim4}, we have 
\[
\partial_{t}\lambda=-\frac{b}{\lambda}+\frac{1}{\lambda}\Big(\frac{\lambda_{s}}{\lambda}+b\Big)=-\frac{b}{\lambda}(1+O((b^{\ast})^{\frac{1}{2}}))=-\frac{b_{0}}{\lambda_{0}}(1+O((b^{\ast})^{\frac{1}{2}}))<-\frac{b_{0}}{2\lambda_{0}}.
\]
This implies $T<+\infty$. By the standard blow-up criterion,\footnote{The standard Cauchy theory of \eqref{eq:CSS} says that the solution
blows up at finite time $T<+\infty$ if and only if $\lim_{t\uparrow T}\|u(t)\|_{\dot{H}_{m}^{1}}=\infty$.} we have $\lambda(T)\coloneqq\lim_{t\uparrow T}\lambda(t)=0$.

We now rewrite the claim \eqref{eq:closing-bootstrap-claim4} as 
\begin{equation}
\frac{b}{\lambda}=\ell(1+O(b^{\frac{1}{2}})),\qquad\ell\coloneqq\lim_{t\uparrow T}\frac{b(t)}{\lambda(t)}\in(0,\infty).\label{eq:closing-bootstrap-claim5}
\end{equation}
The existence of $\ell\in(0,\infty)$ follows from \eqref{eq:closing-bootstrap-claim4-1}
and \eqref{eq:closing-bootstrap-claim3}. The error bound $O(b^{\frac{1}{2}})$
follows from integrating \eqref{eq:closing-bootstrap-claim4-1} on
$[t,T)$ instead of $[0,t]$. In particular, $b(T)\coloneqq\lim_{t\uparrow T}b(t)=0$.

We now derive the asymptotics of the modulation parameters. We again
compute $\partial_{t}\lambda$ and $\partial_{t}b$, but with the
help of \eqref{eq:closing-bootstrap-claim5}: 
\begin{align*}
\partial_{t}\lambda & =-\frac{b}{\lambda}(1+O(b^{\frac{1}{2}}))=-\ell(1+O(b^{\frac{1}{2}})),\\
\partial_{t}b & =\frac{b_{s}+b^{2}}{\lambda^{2}}-\frac{b^{2}}{\lambda^{2}}=\frac{-b^{2}(1+O(b^{\frac{1}{2}}))}{\lambda^{2}}=-\ell^{2}(1+O(b^{\frac{1}{2}})).
\end{align*}
Integrating the above relations from backward in time shows 
\[
\lim_{t\uparrow T}\frac{\lambda(t)}{T-t}=\ell,\qquad\lim_{t\uparrow T}\frac{b(t)}{T-t}=\ell^{2}.
\]
From the modulation estimate (Lemma \ref{lem:mod-est}) and the estimates
of $\theta_{\eta}$, $\theta_{\Psi}$, $\theta_{\mathrm{L-L}}$, and
$\theta_{\mathrm{NL}}$ estimates (see \eqref{eq:decomp-P-temp1},
\eqref{eq:decomp-P-temp2}, \eqref{eq:LL-theta-est}, and \eqref{eq:NL-theta-est}),
we have 
\[
|\partial_{t}\gamma|\lesssim\lambda^{-2}(Kb^{\frac{3}{2}}+K^{2}b^{2})\lesssim\lambda^{-2}Kb^{\frac{3}{2}}\lesssim K\ell(T-t)^{-\frac{1}{2}}
\]
for $t$ near $T$. This shows that $\gamma(t)$ converges as $t\to T$,
say $\gamma^{\ast}$. Finally, we note that $\eta(t)\to0$ as $t\to T$
due to $b(t)\to0$ and the definition of $\mathcal{O}_{trap}$.

It only remains to show that $u$ decomposes as in Theorem \ref{thm:Existence}.
Let us define 
\[
\epsilon^{\sharp}(t,r)\coloneqq\frac{e^{i\gamma(t)}}{\lambda(t)}\epsilon\Big(t,\frac{r}{\lambda(t)}\Big).
\]
We should show that 
\[
\tilde{\epsilon}^{\sharp}(t)\coloneqq\bigg\{\frac{e^{i\gamma(t)}}{\lambda(t)}P\Big(\frac{r}{\lambda(t)};b(t),\eta(t)\Big)-\frac{e^{i\gamma^{\ast}}}{\ell(T-t)}Q\Big(\frac{r}{\ell(T-t)}\Big)\bigg\}+\epsilon^{\sharp}(t)
\]
converges in $L^{2}$ as $t\to T$ and the limit belongs to $H_{m}^{1}$.
Since $(\gamma,b,\eta)\to(\gamma^{\ast},0,0)$ and $\frac{\ell(T-t)}{\lambda(t)}\to1$
as $t\to T$, the first term of the above display converges to $0$
in $L^{2}$. The second term $\epsilon^{\sharp}(t)$ is uniformly
bounded in $H_{m}^{1}$, thanks to the boundedness of $\frac{b}{\lambda}$
shown in \eqref{eq:closing-bootstrap-claim4} and $\|\epsilon^{\sharp}\|_{\dot{H}_{m}^{1}}=\lambda^{-1}\|\epsilon\|_{\dot{H}_{m}^{1}}$.
Thus it only remains to show that $\{\epsilon^{\sharp}(t)\}_{t\to T}$
is Cauchy in $L^{2}$.

To show that $\{\epsilon^{\sharp}(t)\}_{t\to T}$ is Cauchy in $L^{2}$,
it suffices to show that $\|\partial_{t}\epsilon^{\sharp}\|_{L^{2}}$
is integrable in time. From the $\epsilon$-equation \eqref{eq:e-equation-L2},
we have 
\[
\partial_{t}\epsilon^{\sharp}=\frac{1}{\lambda^{2}}\cdot\frac{e^{i\gamma(t)}}{\lambda(t)}[i\Delta_{m}\epsilon+\mathbf{Mod}\cdot\mathbf{v}-i(\mathcal{N}(u^{\flat})-\mathcal{N}(P))-i\Psi]\Big(t,\frac{r}{\lambda(t)}\Big).
\]
By Lemma \ref{lem:L2-EstimateRemainders}, we have 
\begin{align*}
\|\partial_{t}\epsilon^{\sharp}\|_{L^{2}} & \lesssim\lambda^{-2}(\|\Delta_{m}\epsilon\|_{L^{2}}+\|\text{RHS of }\eqref{eq:e-equation-L2}\|_{L^{2}})\\
 & \lesssim\lambda^{-2}(\|\epsilon\|_{\dot{H}_{m}^{1}}^{\frac{1}{2}}\|\epsilon\|_{\dot{\mathcal{H}}_{m}^{3}}^{\frac{1}{2}}+b^{\frac{3}{2}})\lesssim\ell(T-t)^{-\frac{1}{2}}.
\end{align*}
Thus $\|\partial_{t}\epsilon^{\sharp}\|_{L^{2}}$ is integrable in
time. This ends the proof of Theorem \ref{thm:Existence}.
\end{proof}
\begin{rem}
\label{rem:forH5solutions}For $H_{m}^{5}$-trapped solutions, thanks
to the improved bound for $\epsilon_{3}$, we see that $\epsilon^{\sharp}(t)$
is uniformly bounded in $H_{m}^{3}$. This says that the radiation
enjoys further regularity $H_{m}^{3}$.
\end{rem}

\section{\label{sec:Lipschitz-blow-up-manifold}Lipschitz blow-up manifold}

Throughout this section, we assume 
\[
m\geq3
\]
and only deal with $H_{m}^{5}$-trapped solutions.

So far, we constructed trapped solutions. Due to a soft connectivity
argument, we were only able to guarantee the existence of pseudoconformal
blow-up solutions, with a weak codimension one condition. Recall that
$\mathcal{M}$ is the set of initial data in $\mathcal{O}_{init}^{5}$
yielding $H_{m}^{5}$-trapped solutions. Hence $\mathcal{M}$ contains
the blow-up solutions constructed in the first part of Theorem \ref{thm:BlowupManifold}.
In this section we aim to prove more quantitative information on $\mathcal{M}$.
First, it includes the uniqueness statement saying that $\mathcal{M}$
is equal to the solutions constructed by the proof of the first part
of Theorem \ref{thm:BlowupManifold}. It also includes the Lipschitz
dependence on initial data. These will also finish the proof of Theorem
\ref{thm:BlowupManifold}.

Among many works on establishing the regularity of blow-up manifolds
(or, stable/unstable manifolds), the most relevant one to this paper
is the work of Collot \cite{Collot2018MemAMS}. Our proof uses the
idea of \cite{Collot2018MemAMS}.

Recall from the proof of Theorem \ref{thm:Existence} that the codimension
one condition stems from the \emph{unstable parameter} $\eta$. In
order to finish the proof of Theorem \ref{thm:BlowupManifold}, we
have to ensure \emph{uniqueness} of $\eta_{0}$ yielding trapped solutions
and also the Lipschitz dependence on (the stable modes of) initial
data. Thus the heart of the proof will be to control the difference
of unstable parameter $\eta$ by the difference of \emph{stable parameters}
($b$ and $\epsilon$ if we modulo out the scaling and phase rotation
symmetries).

\subsection{Reduction of Theorem \ref{thm:BlowupManifold}}

In this subsection, we reduce the proof of Theorem \ref{thm:BlowupManifold}
into Proposition \ref{prop:Lipschitz-estimate-modulo-scale-phase},
which is the control of the difference of unstable parameters by the
difference of stable parameters.

Let $u$ and $u'$ be two $H_{m}^{5}$-trapped solutions. By Lemma
\ref{lem:decomp}, $u$ has associated modulation parameters $\lambda,\gamma,b,\eta$
and the error part $\epsilon$. Similarly, $u'$ has parameters $\lambda',\gamma',b',\eta'$
and the error $\epsilon'$. At the initial time $t=0$, we add a subscript
$0$ to these: $b(0)=b_{0}$, $\eta'(0)=\eta_{0}'$, and so on.
\begin{prop}[Lipschitz estimate modulo scaling/phase invariances]
\label{prop:Lipschitz-estimate-modulo-scale-phase}Let $u$ and $u'$
be two $H_{m}^{5}$-trapped solutions. If $b_{0}'$ is sufficiently
close to $b_{0}$, then 
\[
|\eta_{0}-\eta_{0}'|\lesssim_{b_{0}}|b_{0}-b_{0}'|+\|\epsilon_{0}-\epsilon_{0}'\|_{H_{m}^{3}}.
\]
In particular, for given $b_{0}\in(0,b^{\ast})$ and $\|\epsilon_{0}\|_{H_{m}^{5}}<b_{0}^{5}$,
there exists unique $\tilde{\eta}_{prem}=\tilde{\eta}_{prem}(b_{0},\epsilon_{0})$
such that $P(\cdot;b_{0},\tilde{\eta}_{prem})+\epsilon_{0}\in\mathcal{M}$.
Moreover, $\tilde{\eta}$ is locally Lipschitz in $b_{0},\epsilon_{0}$.
\end{prop}

We remark that Proposition \ref{prop:Lipschitz-estimate-modulo-scale-phase}
implies the \emph{uniqueness} of $\eta_{0}$ (for given $b_{0}$ and
$\epsilon_{0}$) of Theorem \ref{thm:BlowupManifold}. Hence $\mathcal{M}$
is equal to the solutions constructed by the proof of the first part
of Theorem \ref{thm:BlowupManifold}. Due to the scaling and phase
rotation symmetries, once $P(\cdot;b_{0},\eta_{0})+\epsilon_{0}\in\mathcal{M}$,
we have $\tfrac{e^{i\gamma_{0}}}{\lambda_{0}}[P(\cdot;b_{0},\eta_{0})+\epsilon_{0}(\cdot)](\frac{\cdot}{\lambda_{0}})\in\mathcal{M}$.
Thus $\eta_{0}$ does not depend on the scaling and phase parameters.

Notice also that the statement of Proposition \ref{prop:Lipschitz-estimate-modulo-scale-phase}
is insensitive to the scales and phases of $u$ and $u'$. Thus this
measures the difference of solutions modulo scaling/phase symmetries.

If one fixes the scale $1$ and phase $0$, then one can construct
a codimension three blow-up manifold, say $\mathcal{M}_{1,0}$. A
naive try to recover the codimension one manifold $\mathcal{M}$ is
to apply scaling and phase rotation symmetries to $\mathcal{M}_{1,0}$
since $\mathcal{M}=\bigcup_{\lambda,\gamma}\mathcal{M}_{\lambda,\gamma}$,
where $\mathcal{M}_{\lambda,\gamma}=\{\frac{e^{i\gamma}}{\lambda}u_{0}(\frac{\cdot}{\lambda}):u_{0}\in\mathcal{M}_{1,0}\}$.
However, the scaling symmetry is not Lipschitz continuous on any Sobolev
spaces. More precisely, the map
\[
(\lambda_{0},\gamma_{0},b_{0},\eta_{0},\epsilon_{0})\mapsto\frac{e^{i\gamma_{0}}}{\lambda_{0}}[P(\cdot;b_{0},\eta_{0})+\epsilon_{0}]\Big(\frac{\cdot}{\lambda_{0}}\Big)
\]
is not Lipschitz continuous in $\lambda_{0}$, due to the $\frac{1}{\lambda_{0}}\epsilon_{0}(\frac{\cdot}{\lambda_{0}})$-part.
This says that Lipschitz property of $\mathcal{M}$ does not immediately
follow from applying the scaling and phase rotation symmetries to
$\mathcal{M}_{1,0}$.

The remedy is to consider the decomposition formula 
\[
(\lambda_{0},\gamma_{0},b_{0},\eta_{0},\epsilon_{0}^{\sharp})\mapsto\frac{e^{i\gamma_{0}}}{\lambda_{0}}P(\frac{\cdot}{\lambda_{0}};b_{0},\eta_{0})+\epsilon_{0}^{\sharp}
\]
instead, where we denoted $\epsilon_{0}^{\sharp}(r)\coloneqq\frac{e^{i\gamma_{0}}}{\lambda_{0}}\epsilon_{0}(\frac{r}{\lambda_{0}})$.
Once we use $\epsilon_{0}^{\sharp}$ as an independent variable instead
of $\epsilon_{0}$, the above formula is Lipschitz continuous in terms
of $(\lambda_{0},\gamma_{0},b_{0},\eta_{0},\epsilon_{0}^{\sharp})$.
Viewing $\epsilon_{0}=e^{-i\gamma_{0}}\lambda_{0}\epsilon_{0}^{\sharp}(\lambda_{0}\cdot)$,
we can write $\tilde{\eta}_{prem}$ as a function of $\lambda_{0},\gamma_{0},b_{0},\epsilon_{0}^{\sharp}$,
say $\tilde{\eta}_{prem}(b_{0},\epsilon_{0})=\tilde{\eta}(\lambda_{0},\gamma_{0},b_{0},\epsilon_{0}^{\sharp})$.
Then the Lipschitz regularity of $\mathcal{M}$ would follow if we
show that $\tilde{\eta}=\tilde{\eta}(\lambda_{0},\gamma_{0},b_{0},\epsilon_{0}^{\sharp})$
is Lipschitz continuous.

From this, we need a variant of Proposition \ref{prop:Lipschitz-estimate-modulo-scale-phase},
where we measure the difference in the original variables.
\begin{prop}[Lipschitz estimate in the original variables]
\label{prop:Lipschitz-diff-orig-variables}Let $u$ and $u'$ two
$H_{m}^{5}$-trapped solutions. If $b_{0}'$ is sufficiently close
to $b_{0}$, and $\lambda_{0}$ and $\lambda_{0}'$ are sufficiently
close to $1$, then 
\[
|\eta_{0}-\eta_{0}'|\lesssim_{b_{0}}|\lambda_{0}-\lambda_{0}'|+|\gamma_{0}-\gamma_{0}'|+|b_{0}-b_{0}'|+\|\epsilon_{0}^{\sharp}-(\epsilon_{0}')^{\sharp'}\|_{H_{m}^{3}},
\]
where $\epsilon_{0}^{\sharp}=\frac{e^{i\gamma_{0}}}{\lambda_{0}}\epsilon_{0}(\frac{\cdot}{\lambda_{0}})$
and $(\epsilon_{0}')^{\sharp'}=\frac{e^{i\gamma'_{0}}}{\lambda_{0}'}\epsilon_{0}'(\frac{\cdot}{\lambda_{0}'})$.
In particular, for given $(\lambda_{0},\gamma_{0},b_{0},\epsilon_{0})\in\tilde{\mathcal{U}}_{init}^{5}$,
there exists unique $\tilde{\eta}=\tilde{\eta}(\lambda_{0},\gamma_{0},b_{0},\epsilon_{0}^{\sharp})$\footnote{In fact, $\tilde{\eta}(\lambda_{0},\gamma_{0},b_{0},\epsilon_{0}^{\sharp})=\tilde{\eta}_{prem}(b_{0},\epsilon_{0})$
under the relation $\epsilon_{0}^{\sharp}=\frac{e^{i\gamma_{0}}}{\lambda_{0}}\epsilon_{0}(\frac{\cdot}{\lambda_{0}})$.} such that $\frac{e^{i\gamma_{0}}}{\lambda_{0}}P(\frac{\cdot}{\lambda_{0}};b_{0},\tilde{\eta})+\epsilon_{0}^{\sharp}\in\mathcal{M}$.
Moreover, $\tilde{\eta}$ is a locally Lipschitz function of $\lambda_{0},\gamma_{0},b_{0},\epsilon_{0}^{\sharp}$.
\end{prop}

Let us show how one can prove Theorem \ref{thm:BlowupManifold} assuming
Proposition \ref{prop:Lipschitz-diff-orig-variables}.
\begin{proof}[Proof of Theorem \ref{thm:BlowupManifold} assuming Proposition \ref{prop:Lipschitz-diff-orig-variables}]
We remind the reader that the existence of $H_{m}^{5}$-trapped solutions
is proved in Section \ref{subsec:Existence-of-trapped-solutions}.
The uniqueness of $\eta_{0}$ is an immediate consequence of Proposition
\ref{prop:Lipschitz-estimate-modulo-scale-phase}. Henceforth, we
show that $\mathcal{M}$ has the Lipschitz regularity.

Let $\hat u\in\mathcal{M}$ be a reference element. It suffices to
show that a neighborhood of $\hat u$ in $\mathcal{M}$ can be expressed
as a Lipschitz graph of some codimension one subspace of $H_{m}^{5}$.

Denote by $(\hat{\lambda},\hat{\gamma},\hat b,\hat{\eta},\hat{\epsilon})$
the associated parameters and error part. Denote by $\hat{\epsilon}^{\sharp}\coloneqq\frac{e^{i\hat{\gamma}}}{\hat{\lambda}}\hat{\epsilon}(\frac{\cdot}{\hat{\lambda}})$.
By scaling and phase rotation symmetries, we may assume that $\hat{\lambda}=1$
and $\hat{\gamma}=0$.\footnote{Thus $\hat{\epsilon}^{\sharp}=\hat{\epsilon}$, but we will use the
notation $\hat{\epsilon}^{\sharp}$ to emphasize that we are working
with the original variable $r$.} Denote by $\hat P=P(\cdot;\hat b,\hat{\eta})$, $\partial_{b}\hat P=[\partial_{b}P](\cdot;\hat b,\hat{\eta})$,
and $\partial_{\eta}\hat P=[\partial_{\eta}P](\cdot;\hat b,\hat{\eta})$.
Similarly denote the modulation vector at this reference point by
$\hat{\mathbf{v}}=(\Lambda\hat P,-i\hat P,-\partial_{b}\hat P,-\partial_{\eta}\hat P)^{t}$.
Define the stable/unstable subspaces 
\begin{align*}
X_{s} & \coloneqq\mathrm{span}_{\R}\{\Lambda\hat P,i\hat P,\partial_{b}\hat P\}\oplus(\mathcal{Z}^{\perp}\cap H_{m}^{5}),\\
X_{u} & \coloneqq\mathrm{span}_{\R}\{\partial_{\eta}\hat P\},
\end{align*}
such that 
\[
H_{m}^{5}=X_{s}\oplus X_{u}.
\]
In the proof we also use the notation 
\[
f_{\lambda,\gamma}\coloneqq\frac{e^{i\gamma}}{\lambda}f\Big(\frac{\cdot}{\lambda}\Big).
\]

On a small neighborhood of $\hat u$ in $\hat u+X_{s}$, define the
map $h$ taking values in $\mathcal{M}$ by 
\[
\hat u+(\delta\lambda)\Lambda\hat P+(\delta\gamma)i\hat P+(\delta b)\partial_{b}\hat P+(\delta\epsilon_{0}^{\sharp})\mapsto P(\cdot;b_{0},\tilde{\eta})_{\lambda_{0},\gamma_{0}}+\epsilon_{0}^{\sharp},
\]
where $\delta\lambda,\delta\gamma,\delta b,\delta\epsilon_{0}^{\sharp}$
are considered as independent variables and 
\begin{gather}
\lambda_{0}\coloneqq1+\delta\lambda,\quad\gamma_{0}\coloneqq\delta\gamma,\quad b_{0}\coloneqq\hat b+\delta b,\quad\tilde{\eta}\coloneqq\tilde{\eta}(\lambda_{0},\gamma_{0},b_{0},\epsilon_{0}^{\sharp}),\nonumber \\
\epsilon_{0}^{\sharp}\coloneqq\hat{\epsilon}^{\sharp}+\delta\epsilon_{0}^{\sharp}-\sum_{j,k}(A^{-1})_{jk}(\hat{\epsilon}^{\sharp}+\delta\epsilon_{0}^{\sharp},[\mathcal{Z}_{k}]_{\lambda_{0},\gamma_{0}})_{r}[\hat v_{j}]_{\lambda_{0},\gamma_{0}}.\label{eq:Diff-est-e-sharp-formula}
\end{gather}
Here, $A^{-1}$ is the inverse matrix of $A$ whose components are
defined by $A_{jk}\coloneqq(\mathcal{Z}_{j},\hat v_{k})_{r}$,\footnote{$A$ is an almost diagonal matrix, as seen in \eqref{eq:ortho-jacobian}.}
and $\tilde{\eta}(\lambda_{0},\gamma_{0},b_{0},\epsilon_{0}^{\sharp})$
is defined in Proposition \ref{prop:Lipschitz-diff-orig-variables}.
The (seemingly complicated) additional term in the definition \eqref{eq:Diff-est-e-sharp-formula}
is designed to guarantee that $\epsilon_{0}^{\sharp}$ satisfies the
orthogonality conditions at the scale $\lambda_{0}$ and phase $\gamma_{0}$,
i.e. 
\[
(\epsilon_{0}^{\sharp},[\mathcal{Z}_{k}]_{\lambda_{0},\gamma_{0}})_{r}=0\qquad\forall k\in\{1,2,3,4\}.
\]
By the definition of $\tilde{\eta}$, we verify that $h$ maps into
$\mathcal{M}$ and $h(\hat u)=\hat u$.

On a small neighborhood of $\hat u$ in $H_{m}^{5}=(\hat u+X_{s})\oplus X_{u}$,
we define a $C^{1}$-diffeomorphism $H$ by 
\[
\hat u+(\delta\lambda)\Lambda\hat P+(\delta\gamma)i\hat P+(\delta b)\partial_{b}\hat P+(\delta\eta)\partial_{\eta}\hat P+(\delta\epsilon_{0}^{\sharp})\mapsto P(\cdot;b_{0},\eta_{0})_{\lambda_{0},\gamma_{0}}+\epsilon_{0}^{\sharp},
\]
where $b_{0},\lambda_{0},\gamma_{0},\epsilon_{0}^{\sharp}$ are as
in the definition of $h$ and $\eta_{0}\coloneqq\hat{\eta}+\delta\eta$.
Note that $H(\hat u)=\hat u$. To see the $C^{1}$-diffeomorphism
property, we observe the following. The partial derivatives of $H$
along the $\delta\lambda_{0},\delta\gamma_{0},\delta b_{0},\delta\eta_{0}$
directions are $-\frac{1}{\lambda_{0}}[\Lambda P]_{\lambda_{0},\gamma_{0}}$,
$iP_{\lambda_{0},\gamma_{0}}$, $[\partial_{b}P]_{\lambda_{0},\gamma_{0}}$,
$[\partial_{\eta}P]_{\lambda_{0},\gamma_{0}}$ evaluated at $(b_{0},\eta_{0})$
plus some $O(\epsilon_{0}^{\sharp})$ error. Next, the functional
derivative of $H$ along the $\delta\epsilon_{0}^{\sharp}$ is $\mathrm{id}_{(\mathcal{Z}^{\perp}\cap H_{m}^{5})}$
plus a rank four operator whose range lies in the span of $[\Lambda P]_{\lambda_{0},\gamma_{0}},iP_{\lambda_{0},\gamma_{0}},[\partial_{b}P]_{\lambda_{0},\gamma_{0}},[\partial_{\eta}P]_{\lambda_{0},\gamma_{0}}$.
In summary, 
\begin{align*}
 & \delta H/\delta((\delta\lambda),(\delta\gamma),(\delta b),(\delta\eta),(\delta\epsilon_{0}^{\sharp}))\\
 & =\begin{pmatrix}-1 & 0 & 0 & 0 & \ast\\
0 & 1 & 0 & 0 & \ast\\
0 & 0 & 1 & 0 & \ast\\
0 & 0 & 0 & 1 & \ast\\
0 & 0 & 0 & 0 & \mathrm{id}_{(\mathcal{Z}^{\perp}\cap H_{m}^{5})}
\end{pmatrix}+O_{\mathcal{L}(H_{m}^{5},H_{m}^{5})}(|\delta\lambda|+\cdots+|\delta\eta|+\|\epsilon_{0}^{\sharp}\|_{L^{2}}),
\end{align*}
where the first four rows are written in the directions $\Lambda\hat P,i\hat P,\partial_{b}\hat P,\partial_{\eta}\hat P$.
Applying the inverse function theorem at $\hat u$ says that $H$
is a $C^{1}$-diffeomorphism on a neighborhood of $\hat u$ in $H_{m}^{5}$.

Therefore, it suffices to show that there is an open neighborhood
$\hat{\mathcal{O}}$ of $\hat u$ such that $H^{-1}(\mathcal{M}\cap\hat{\mathcal{O}})$
is a locally Lipschitz codimension one manifold. We note that $H^{-1}\circ h$
is given by 
\begin{align*}
 & \hat u+(\delta\lambda)\Lambda\hat P+(\delta\gamma)i\hat P+(\delta b)\partial_{b}\hat P+(\delta\epsilon_{0}^{\sharp})\\
 & \mapsto\hat u+(\delta\lambda)\Lambda\hat P+(\delta\gamma)i\hat P+(\delta b)\partial_{b}\hat P+(\delta\epsilon_{0}^{\sharp})+(\tilde{\eta}-\hat{\eta})\partial_{\eta}\hat P,
\end{align*}
where $\tilde{\eta}=\tilde{\eta}(\lambda_{0},\gamma_{0},b_{0},\epsilon_{0}^{\sharp})$.
Notice that this looks like the form $\mathrm{id}_{\hat u+X_{s}}\oplus f$,
where $f=(\tilde{\eta}-\hat{\eta})\partial_{\eta}\hat P$. According
to Definition \ref{def:localLipschitz}, it suffices to show that
(1) there are small neighborhoods $\hat{\mathcal{O}}$ of $\hat u$
in $H_{m}^{5}$ and $\tilde{\mathcal{O}}$ of $\hat u$ in $\hat u+X_{s}$
such that $\mathcal{M}\cap\hat{\mathcal{O}}=h(\tilde{\mathcal{O}})$,
and (2) the map $(\delta\lambda,\delta\gamma,\delta b,\delta\epsilon_{0}^{\sharp})\mapsto(\tilde{\eta}-\hat{\eta})\partial_{\eta}\hat P$
is Lipschitz continuous.

(1) For $\hat{\mathcal{O}}$ to be chosen small, let $u_{0}\in\mathcal{M}\cap\hat{\mathcal{O}}$.
Because $u_{0}$ is near $\hat u$, we can perform the decomposition
according to Lemma \ref{lem:decomp} and the associated $\lambda_{0},\gamma_{0},\eta_{0},b_{0},\epsilon_{0}^{\sharp}$
are all close to the reference data $\hat{\lambda},\dots,\hat{\epsilon}^{\sharp}$.
Since $u_{0}\in\mathcal{M}$ and $u_{0}$ is near $\hat u$, the uniqueness
statement of Proposition \ref{prop:Lipschitz-diff-orig-variables}
says that $\eta_{0}=\tilde{\eta}(\lambda_{0},\gamma_{0},b_{0},\epsilon_{0}^{\sharp})$.
We then define $\delta\epsilon_{0}^{\sharp}$ by inverting the formula
\eqref{eq:Diff-est-e-sharp-formula}. Note that the inversion is possible
because the summation part is small, thanks to the orthogonality condition.
If $\hat{\mathcal{O}}$ is sufficiently small, then this $(\lambda_{0},\gamma_{0},b_{0},\delta\epsilon_{0}^{\sharp})$
lies in the domain of $h$. In other words, $u_{0}$ lies in the image
of $h$. Thus $\mathcal{M}\cap\hat{\mathcal{O}}$ is contained in
the image of $h$. Finally set $\tilde{\mathcal{O}}\coloneqq h^{-1}(\hat{\mathcal{O}})$.

(2) It suffices to show that $(\delta\lambda,\delta\gamma,\delta b,\delta\epsilon_{0}^{\sharp})\mapsto\tilde{\eta}(\lambda_{0},\gamma_{0},b_{0},\epsilon_{0}^{\sharp})$
is Lipschitz continuous. This follows from Proposition \ref{prop:Lipschitz-diff-orig-variables}
and the $\epsilon_{0}^{\sharp}$-formula \eqref{eq:Diff-est-e-sharp-formula}.

This ends the proof of Theorem \ref{thm:BlowupManifold}.
\end{proof}
One can also deduce Proposition \ref{prop:Lipschitz-diff-orig-variables}
from Proposition \ref{prop:Lipschitz-estimate-modulo-scale-phase}.
\begin{proof}[Proof of Proposition \ref{prop:Lipschitz-diff-orig-variables} assuming
Proposition \ref{prop:Lipschitz-estimate-modulo-scale-phase}]
Without loss of generality, we may assume that $\lambda_{0}\geq\lambda_{0}'$.
As $\lambda(t)$ is strictly decreasing to $0$ as $t$ goes to the
blow-up time, there exists unique $\tilde t_{0}\geq0$ such that $\lambda(\tilde t_{0})=\lambda_{0}'$.
Now $u(\tilde t_{0})$ and $u_{0}'$ are located at the \emph{same
scale} $\lambda_{0}'\approx1$ so that 
\[
\|\epsilon(\tilde t_{0})-\epsilon_{0}'\|_{H_{m}^{3}}\lesssim\|\epsilon^{\sharp}(\tilde t_{0})-(\epsilon_{0}')^{\sharp'}\|_{H_{m}^{3}}+|\gamma_{0}-\gamma_{0}'|.
\]
Applying Proposition \ref{prop:Lipschitz-estimate-modulo-scale-phase}
for $u(\tilde t_{0})$ and $u_{0}'$ yields 
\begin{equation}
|\eta(\tilde t_{0})-\eta_{0}'|\lesssim_{b_{0}}|b(\tilde t_{0})-b_{0}'|+|\gamma_{0}-\gamma_{0}'|+\|\epsilon^{\sharp}(\tilde t_{0})-(\epsilon_{0}')^{\sharp'}\|_{H_{m}^{3}},\label{eq:Lipschitz-diff-temp1}
\end{equation}
provided that $b_{0}(\tilde t_{0})$ is near $b_{0}$.

We claim that 
\begin{equation}
|\eta(\tilde t_{0})-\eta_{0}|+|b(\tilde t_{0})-b_{0}|+\|\epsilon^{\sharp}(\tilde t_{0})-\epsilon_{0}^{\sharp}\|_{H_{m}^{3}}\lesssim\lambda_{0}-\lambda_{0}'.\label{eq:Lipschitz-diff-claim1}
\end{equation}
Let us assume the claim \eqref{eq:Lipschitz-diff-claim1} and finish
the proof. Note that the claim guarantees $b_{0}(\tilde t_{0})\approx b_{0}$,
because we assumed that $\lambda_{0}$ and $\lambda_{0}'$ are sufficiently
close. The proof is completed by \eqref{eq:Lipschitz-diff-temp1}
and the claim \eqref{eq:Lipschitz-diff-claim1}: 
\begin{align*}
|\eta_{0}-\eta_{0}'| & \leq|\eta_{0}-\eta(\tilde t_{0})|+|\eta(\tilde t_{0})-\eta_{0}'|\\
 & \lesssim_{b_{0}}|\eta(\tilde t_{0})-\eta_{0}|+|b(\tilde t_{0})-b_{0}'|+|\gamma_{0}-\gamma_{0}'|+\|\epsilon^{\sharp}(\tilde t_{0})-(\epsilon_{0}')^{\sharp'}\|_{H_{m}^{3}}\\
 & \lesssim_{b_{0}}|\eta(\tilde t_{0})-\eta_{0}|+|b(\tilde t_{0})-b_{0}|+\|\epsilon^{\sharp}(\tilde t_{0})-\epsilon_{0}^{\sharp}\|_{H_{m}^{3}}\\
 & \qquad+|b_{0}-b_{0}'|+|\gamma_{0}-\gamma_{0}'|+\|\epsilon_{0}^{\sharp}-(\epsilon_{0}')^{\sharp'}\|_{H_{m}^{3}}\\
 & \lesssim_{b_{0}}|\lambda_{0}-\lambda_{0}'|+|b_{0}-b_{0}'|+|\gamma_{0}-\gamma_{0}'|+\|\epsilon_{0}^{\sharp}-(\epsilon_{0}')^{\sharp'}\|_{H_{m}^{3}}.
\end{align*}
It now remains to prove the claim \eqref{eq:Lipschitz-diff-claim1}.

We first show $|\eta(\tilde t_{0})-\eta_{0}|+|b(\tilde t_{0})-b_{0}|\lesssim\lambda_{0}-\lambda_{0}'$
by a bootstrap argument. From the modulation estimate, we have 
\[
(\lambda^{2})_{t}=-(2+O(b^{\frac{3}{2}}))b.
\]
Since both $\lambda_{0}$ and $\lambda_{0}'$ are close to $1$, we
have $\lambda_{0}^{2}-(\lambda_{0}')^{2}\approx2(\lambda_{0}-\lambda_{0}')$.
Thus 
\[
\tilde t_{0}\approx b_{0}^{-1}(\lambda_{0}-\lambda_{0}'),
\]
\emph{provided that} $b(t)\approx b_{0}$ on $[0,\tilde t_{0}]$.
Then the modulation estimates $|\eta_{t}|\lesssim b^{\frac{5}{2}}$
and $|b_{t}|\lesssim b^{2}$ (recall that $\lambda\approx1$ on $[0,\tilde t_{0}]$)
say that 
\[
|\eta(\tilde t_{0})-\eta_{0}|+\sup_{t\in[0,\tilde t_{0}]}|b(t)-b_{0}|\lesssim b^{2}\tilde t_{0}\lesssim\lambda_{0}-\lambda_{0}',
\]
\emph{provided that} $b(t)\approx b_{0}$ on $[0,\tilde t_{0}]$.
By a standard bootstrap argument, we obtain the following: if $\lambda_{0}$
and $\lambda_{0}'$ are close to $1$, then $|\eta(\tilde t_{0})-\eta_{0}|+|b(\tilde t_{0})-b_{0}|\lesssim\lambda_{0}-\lambda_{0}'$,
$b(t)\approx b_{0}$ on $[0,\tilde t_{0}]$, and $\tilde t_{0}\approx b_{0}^{-1}(\lambda_{0}-\lambda_{0}')$.

Finally, we need to measure the difference of $\epsilon_{0}^{\sharp}$
and $\epsilon^{\sharp}(\tilde t_{0})$. For this, we will rewrite
the equations of $\epsilon$ and $\epsilon_{3}$ in the $(t,r)$-variables.
Let 
\[
\epsilon^{\sharp}(t,r)\coloneqq\frac{e^{i\gamma(t)}}{\lambda(t)}\epsilon\Big(t,\frac{r}{\lambda(t)}\Big)\qquad\text{and}\qquad\epsilon_{3}^{\sharp_{-3}}(t,r)\coloneqq\frac{e^{i\gamma(t)}}{\lambda^{4}(t)}\epsilon_{3}\Big(t,\frac{r}{\lambda(t)}\Big),
\]
where $\sharp_{-3}$ means the $\dot{H}^{-3}$-scaling. Then we have
\begin{align*}
\partial_{t}\epsilon^{\sharp} & =\frac{1}{\lambda^{2}}\Big(i\Delta_{m}\epsilon+\text{RHS of }\eqref{eq:e-equation-L2}\Big)^{\sharp},\\
\partial_{t}\epsilon_{3}^{\sharp_{-3}} & =\frac{1}{\lambda^{2}}\Big(-iA_{Q}^{\ast}A_{Q}\epsilon_{3}+\text{RHS of }\eqref{eq:e3-eq}\Big)^{\sharp_{-3}}.
\end{align*}
Using $\lambda\approx1$, we get 
\begin{align*}
\|\partial_{t}\epsilon^{\sharp}\|_{L^{2}} & \lesssim\|\Delta_{m}\epsilon\|_{L^{2}}+\|\text{RHS of }\eqref{eq:e-equation-L2}\|_{L^{2}},\\
\|\partial_{t}\epsilon_{3}^{\sharp_{-3}}\|_{L^{2}} & \lesssim\|\epsilon_{5}\|_{L^{2}}+\|\text{RHS of }\eqref{eq:e3-eq}\|_{L^{2}}.
\end{align*}
For the first term of each RHS, we can apply a priori $H_{m}^{5}$-estimates
of $\epsilon$. For the second term of each RHS, we recall that these
are estimated in Lemma \ref{lem:L2-EnergyEstimate} and \ref{lem:E3-Identity},
respectively. As a result, 
\[
\|\partial_{t}\epsilon^{\sharp}\|_{L^{2}}+\|\partial_{t}\epsilon_{3}^{\sharp_{-3}}\|_{L^{2}}\lesssim b^{\frac{3}{2}}.
\]
Integrating this on $[0,\tilde t_{0}]$ and applying the coercivity
Lemma \ref{lem:Coercivity-AAL-Section2} for $\epsilon_{3}$, we finally
get 
\[
\|\epsilon^{\sharp}(\tilde t_{0})-\epsilon_{0}^{\sharp}\|_{H_{m}^{3}}\lesssim\tilde t_{0}b\lesssim\lambda_{0}-\lambda_{0}'.
\]
This completes the proof of \eqref{eq:Lipschitz-diff-claim1}.
\end{proof}
Therefore, it only remains to show Proposition \ref{prop:Lipschitz-estimate-modulo-scale-phase}.

\subsection{Further reduction of Proposition \ref{prop:Lipschitz-estimate-modulo-scale-phase}}

So far, we have seen how the Lipschitz control on the difference of
unstable parameter $\eta$ by that of stable parameters $b$ and $\epsilon$
yields Lipschitz regularity of the manifold $\mathcal{M}$. From now
on, we focus on proving this Lipschitz control, Proposition \ref{prop:Lipschitz-estimate-modulo-scale-phase}.

Roughly speaking, the difference $|\eta_{0}-\eta_{0}'|$ can be controlled
backwards in time because $\eta$ is an unstable parameter, and $\eta$
and $\eta'$ have zero limit (with sufficient decay) at their blow-up
times. Thus $|\eta_{0}-\eta_{0}'|$ will be controlled by the differences
of all parameters in the future times. The differences of $b$ and
$\epsilon$ can be controlled forwards in time, because they are stable
parameters. Thus they will be controlled by $|b_{0}-b_{0}'|$ and
$\|\epsilon_{0}-\epsilon_{0}'\|_{H_{m}^{3}}$ (propagation of the
initial data difference), and the differences of all parameters in
the past times. Combining the above controls, one may expect that
$|\eta_{0}-\eta_{0}'|$ can be controlled by $|b_{0}-b_{0}'|$ and
$\|\epsilon_{0}-\epsilon_{0}'\|_{H_{m}^{3}}$. This idea is formulated
in Proposition \ref{prop:ForwardBackwardControl}. In this subsection,
we reduce Proposition \ref{prop:Lipschitz-estimate-modulo-scale-phase}
to Proposition \ref{prop:ForwardBackwardControl}.

In order to realize this idea, an important issue is the meaning of
the difference. Given any two $H_{m}^{5}$-trapped solutions $u$
and $u'$, their blow-up times are in general different. Thus taking
the difference in the original time variables $t$ will not work.
However, as observed in \cite{Collot2018MemAMS}, there is a natural
choice of the time for measuring the difference of $u$ and $u'$.

To be more precise, we first apply the dynamic rescaling for $u$
and $u'$. First, fix any $s_{0}\in\R$ and define $s=s(t):[0,T_{+}(u))\to[s_{0},\infty)$
such that $\frac{ds}{dt}=\frac{1}{\lambda^{2}(t)}$ and $s(0)=s_{0}$.
Write $\lambda,\gamma,b,\eta,\epsilon$ as functions of the renormalized
time $s\in[s_{0},\infty)$. Then, $\epsilon=\epsilon(s,y):[s_{0},\infty)\times(0,\infty)\to\C$
satisfies the equation 
\[
\partial_{s}\epsilon-\frac{\lambda_{s}}{\lambda}\Lambda\epsilon=-i\mathcal{L}_{Q}\epsilon+\dots.
\]
Similarly, fix any $s_{0}'$ and define $s'=s'(t):[0,T_{+}(u'))\to[s_{0}',\infty)$
such that $\frac{ds'}{dt}=\frac{1}{(\lambda')^{2}(t)}$ and $s'(0)=s_{0}'$.
Write $\lambda',\gamma',b',\eta',\epsilon'$ as functions of $s'\in[s_{0}',\infty)$.
Then, $\epsilon'=\epsilon'(s',y):[s_{0}',\infty)\times(0,\infty)\to\C$
satisfies the equation 
\[
\partial_{s'}\epsilon'-\frac{\lambda_{s'}'}{\lambda'}\Lambda\epsilon'=-i\mathcal{L}_{Q}\epsilon'+\dots.
\]

We now forget the original time variable $t$ and regard $s$ and
$s'$ as independent time variables. A priori $s\in[s_{0},\infty)$
and $s'\in[s_{0}',\infty)$ live in different spaces. In order to
compare $\epsilon$ and $\epsilon'$, we introduce a transform $s'=s'(s):[s_{0},\infty)\to[s_{0}',\infty)$
with $s'(s_{0})=s_{0}'$ to put $\epsilon$ and $\epsilon'$ on the
same time domain $[s_{0},\infty)$ and measure their differences by
$\epsilon(s)-\epsilon'(s'(s))$. To choose a right transform $s'=s'(s)$,
let us rewrite the equation of $\epsilon'$ in the $s$-variable (abusing
the notation $\epsilon'(s,y)=\epsilon'(s'(s),y)$) using $\partial_{s}=\frac{ds'}{ds}\partial_{s'}$:
\[
\partial_{s}\epsilon'-\frac{ds'}{ds}\frac{\lambda_{s'}'}{\lambda'}\Lambda\epsilon'=\frac{ds'}{ds}(-i\mathcal{L}_{Q}\epsilon'+\dots).
\]
When we write the equation for the difference $\delta\epsilon(s)=\epsilon(s)-\epsilon'(s'(s))$,
the difference of $\frac{\lambda_{s}}{\lambda}\Lambda\epsilon$ and
$\frac{ds'}{ds}\frac{\lambda_{s'}'}{\lambda'}\Lambda\epsilon'$ will
contain the term $(\frac{\lambda_{s}}{\lambda}-\frac{ds'}{ds}\frac{\lambda_{s'}'}{\lambda'})\Lambda\epsilon$.
As the scaling vector field $\Lambda$ is unbounded, we want to delete
this term. This motivates us to choose $s'=s'(s)$ such that 
\[
\frac{ds'}{ds}\frac{\lambda_{s'}'}{\lambda'}=\frac{\lambda_{s}}{\lambda}.
\]
Integrating this, we have 
\[
\frac{\lambda'(s'(s))}{\lambda(s)}=\frac{\lambda_{0}'}{\lambda_{0}}.
\]
This says that we choose $s'=s'(s)$ such that the ratio between $\lambda'(s')$
and $\lambda(s)$ remains constant. The price to pay is that the difference
equation has the term $(\frac{ds'}{ds}-1)i\mathcal{L}_{Q}\epsilon'$.
In the energy estimate of $\delta\epsilon$, this term requires a
priori control of $i\mathcal{L}_{Q}\epsilon'$ having two more derivatives.

From now on, we reduce Proposition \ref{prop:Lipschitz-estimate-modulo-scale-phase}
to Proposition \ref{prop:ForwardBackwardControl}. Set 
\[
s_{0}=s_{0}'=b_{0}^{-1}.
\]
As the statement of Proposition \ref{prop:Lipschitz-estimate-modulo-scale-phase}
is invariant under scalings, we may assume that 
\[
\lambda_{0}=\lambda_{0}'=s_{0}^{-1}.
\]
As motivated from the previous paragraph, we define a strictly increasing
function $s'=s'(s):[s_{0},\infty)\to[s_{0}',\infty)$ such that $s'(s_{0})=s_{0}$
and
\[
\lambda'(s'(s))=\lambda(s).
\]
Differentiating this relation, we have 
\[
\frac{ds'}{ds}=\frac{\lambda_{s}}{\lambda'_{s'}}=\Big(\frac{\lambda'_{s'}}{\lambda'}\Big)^{-1}\Big(\frac{\lambda_{s}}{\lambda}\Big).
\]

\begin{lem}[Rough asymptotics]
\label{lem:RoughAsymptotics}For sufficiently large $s_{0}$,\footnote{As we have set $s_{0}=b_{0}^{-1}$, this is equivalent to saying ``for
sufficiently large $b_{0}^{-1}>(b^{\ast})^{-1}$.''} we have for $s\in[s_{0},\infty)$
\begin{align*}
s'(s) & =s(1+O(s_{0}^{-\frac{1}{2}}))\\
b(s) & =s^{-1}(1+O(s_{0}^{-\frac{1}{2}}))=b'(s'(s)),\\
\lambda(s) & =s^{-1}(1+O(s_{0}^{-\frac{1}{2}}))=\lambda'(s'(s)).
\end{align*}
\end{lem}

\begin{proof}
We first prove asymptotics of $b(s)$ and $\lambda(s)$. We start
from writing the modulation equation $b_{s}+b^{2}=O(b^{\frac{5}{2}})$
in the form 
\[
|(b^{-1})_{s}-1|\lesssim b^{\frac{1}{2}}.
\]
Using $b^{\frac{1}{2}}=-b^{-\frac{3}{2}}b_{s}+O(b)$ and integration
by parts, we have $\int_{s_{0}}^{s}b^{\frac{1}{2}}(\sigma)d\sigma\leq b^{-\frac{1}{2}}(s)+O((b^{\ast})^{\frac{1}{2}})\int_{s_{0}}^{s}b^{\frac{1}{2}}(\sigma)d\sigma$.
Thus $\int_{s_{0}}^{s}b^{\frac{1}{2}}(\sigma)d\sigma\lesssim b^{-\frac{1}{2}}(s)$
so $|b^{-1}(s)-s|\lesssim b^{-\frac{1}{2}}(s)$. In other words, $b(s)=s^{-1}(1+O(s^{-\frac{1}{2}}))$.
Next, from 
\[
|(\log(s\lambda))_{s}|=\Big|\Big(\frac{1}{s}-b\Big)+\Big(\frac{\lambda_{s}}{\lambda}+b\Big)\Big|\lesssim s^{-\frac{3}{2}},
\]
we get $|s\lambda(s)-1|\lesssim s_{0}^{-\frac{1}{2}}$.

We now focus on $s'(s)$ and $b'(s'(s))$. Applying the argument in
the previous paragraph for $b'(s')$, we have $b'(s')=(s')^{-1}(1+O((s')^{-\frac{1}{2}}))$
and $|s'\lambda'(s')-1|\lesssim s_{0}^{-\frac{1}{2}}$. Combining
these with $|s\lambda(s)-1|\lesssim s_{0}^{-\frac{1}{2}}$ and $\lambda'(s'(s))=\lambda(s)$,
we have 
\begin{align*}
|s'(s)-s|\lambda(s) & \leq|s'(s)\lambda(s)-1|+|s\lambda(s)-1|\\
 & =|s'(s)\lambda'(s'(s))-1|+|s\lambda(s)-1|\lesssim s_{0}^{-\frac{1}{2}}.
\end{align*}
Since $\lambda(s)\approx s^{-1}$, we get $|s'(s)-s|\lesssim s_{0}^{-\frac{1}{2}}s$.
Substituting this time difference estimate into $b'(s')=(s')^{-1}(1+O((s')^{-\frac{1}{2}}))$,
we get $b'(s'(s))=s^{-1}(1+O(s_{0}^{-\frac{1}{2}}))$ as desired.
\end{proof}
For $s\in[s_{0},\infty)$, let 
\begin{align*}
\delta b(s) & \coloneqq b(s)-b'(s'(s)),\\
\delta\eta(s) & \coloneqq\eta(s)-\eta'(s'(s)),\\
\delta\epsilon(s) & \coloneqq\epsilon(s)-\epsilon'(s'(s)).
\end{align*}
The adapted derivatives are defined similarly. We note that $(\delta\epsilon)_{k}=\delta(\epsilon_{k})$.

We define the distances of stable/unstable parameters 
\begin{align*}
D^{s}(s) & \coloneqq s|\delta b|+\|\delta\epsilon\|_{L^{2}}+s^{\frac{5}{2}}\|\delta\epsilon_{3}\|_{L^{2}},\\
D^{u}(s) & \coloneqq s^{\frac{3}{2}}|\delta\eta|,\\
D(s) & \coloneqq D^{s}(s)+D^{u}(s).
\end{align*}
Taking the supremum over $[s_{0},\infty)$, we define 
\begin{align*}
D^{s}(s_{0},\infty) & \coloneqq\sup_{s\in[s_{0},\infty)}D^{s}(s),\\
D^{u}(s_{0},\infty) & \coloneqq\sup_{s\in[s_{0},\infty)}D^{u}(s).
\end{align*}
We note that $D^{s}(s_{0},\infty)$ and $D^{u}(s_{0},\infty)$ are
finite due to \eqref{eq:def-H5-trapped}.

We now motivate the weights in front of the differences $|\delta b|,|\delta\eta|,\|\delta\epsilon\|_{L^{2}},\|\delta\epsilon_{3}\|_{L^{2}}$.
Indeed, it is designed to have an $s$-independent bound 
\[
D^{s}(s)+D^{u}(s)\lesssim D^{s}(s_{0}).
\]
First, due to scalings, $\|\delta\epsilon\|_{L^{2}}\lesssim\|\epsilon_{0}-\epsilon_{0}'\|_{L^{2}}$
and $\|\delta\epsilon_{3}\|_{L^{2}}\lesssim(\frac{s_{0}}{s})^{3}\|\epsilon_{3}(s_{0})-\epsilon_{3}'(s_{0})\|_{L^{2}}$
are expected at best. However, due to a limitation of the method,
we will only have $\|\delta\epsilon_{3}\|_{L^{2}}\lesssim s^{-(3-)}\|\epsilon_{3}(0)-\epsilon_{3}'(0)\|_{L^{2}}$,
so we choose $s^{-\frac{5}{2}}$ to work with. Next, it will turn
out that $(\delta b)_{s}+b(\delta b)\approx0$ so $(\delta b)\approx s^{-1}(b_{0}-b_{0}')$
using $b(s)\approx s^{-1}$. Finally, since $|(\delta\eta)_{s}|\lesssim\|\delta\epsilon_{3}\|_{L^{2}}+error$,
and $\eta(s)$ and $\eta'(s'(s))$ go to zero as $s\to\infty$, integrating
backwards in time roughly yields $|\delta\eta(s)|\lesssim s\|\delta\epsilon_{3}\|_{L^{2}}\lesssim s^{-\frac{3}{2}}$.
\begin{prop}[Forward/backward controls]
\label{prop:ForwardBackwardControl}If $s_{0}=b_{0}^{-1}$ is chosen
sufficiently large, then we have
\begin{enumerate}
\item (Forward-in-time control for stable modes) 
\begin{equation}
D^{s}(s_{0},\infty)\lesssim D^{s}(s_{0})+s_{0}^{-\frac{1}{4}}D(s_{0},\infty).\label{eq:ForwardControlStableModes}
\end{equation}
\item (Backward-in-time control for unstable modes)
\begin{equation}
D^{u}(s_{0},\infty)\lesssim D^{s}(s_{0},\infty)+s_{0}^{-\frac{1}{2}}D(s_{0},\infty).\label{eq:BackwardControlUnstableModes}
\end{equation}
\end{enumerate}
In particular, we have 
\[
D(s_{0},\infty)\lesssim D^{s}(s_{0}).
\]
\end{prop}

\begin{rem}
In fact, we can show 
\begin{align*}
D^{s}(s_{0},s) & \lesssim D^{s}(s_{0})+s_{0}^{-\frac{1}{4}}D(s_{0},s),\\
D^{u}(s,\infty) & \lesssim D^{s}(s,\infty)+s^{-\frac{1}{2}}D(s,\infty),
\end{align*}
where we denoted $D^{s}(I)\coloneqq\sup_{s\in I}D^{s}(s)$ and $D^{u}(I)\coloneqq\sup_{s\in I}D^{u}(s)$
for an interval $I\subseteq[s_{0},\infty)$. The first estimate is
obtained by the forward-in-time integration on $[s_{0},s]$. The second
estimate is obtained by the backward-in-time integration on $[s,\infty)$
with $\delta\eta(\infty)=0$. We note that it suffices to have negative
powers for $s_{0}$ in the second term (time integral term) of each
RHS. These terms are perturbative.
\end{rem}

Proposition \ref{prop:ForwardBackwardControl} implies the conclusion
of Proposition \ref{prop:Lipschitz-estimate-modulo-scale-phase} from
the inequality 
\[
|\delta\eta(s_{0})|\lesssim s_{0}^{-\frac{1}{2}}|\delta b(s_{0})|+s_{0}\|\delta\epsilon(s_{0})\|_{H_{m}^{3}}\lesssim_{b_{0}}|\delta b(s_{0})|+\|\delta\epsilon(s_{0})\|_{H_{m}^{3}}.
\]

Therefore, it remains to prove Proposition \ref{prop:ForwardBackwardControl}.
The strategy is quite similar to that of Proposition \ref{prop:main-bootstrap}.
For the forward-in-time control, we write the equation of $\delta\epsilon$,
prove modulation estimates for the difference of modulation parameters,
and perform energy estimates for adaptive derivatives of $\delta\epsilon$.
The control of $\delta\eta$ will follow from integrating the modulation
estimates backwards in time.

\subsection{Equation of $\delta\epsilon$}

We start by writing the equation for $\delta\epsilon(s)=\epsilon(s)-\epsilon'(s'(s))$.
We recall the equation of $\epsilon$: 
\[
(\partial_{s}-\frac{\lambda_{s}}{\lambda}\Lambda+\gamma_{s}i)\epsilon+i\mathcal{L}_{Q}\epsilon=\tilde{\mathbf{Mod}}\cdot\mathbf{v}-i\tilde R_{\mathrm{L-L}}-i\tilde R_{\mathrm{NL}}-i\Psi.
\]
Next, we write the equation of $\epsilon'$ in the \emph{$s$-variable}
using $\partial_{s}=\frac{ds'}{ds}\partial_{s'}$: 
\[
(\partial_{s}-\frac{\lambda_{s}}{\lambda}\Lambda+\gamma'_{s}i)\epsilon'+\frac{ds'}{ds}\cdot i\mathcal{L}_{Q}\epsilon'=\frac{ds'}{ds}\Big(\tilde{\mathbf{Mod}}'\cdot\mathbf{v}'-i\tilde R_{\mathrm{L-L}}'-i\tilde R_{\mathrm{NL}}'-i\Psi'\Big).
\]
Here, we use prime notations canonically; for an expression $f=f(b,\eta,\epsilon,\dots)$
for $u$, we denote the corresponding object for $u'$ by $f'\coloneqq f(b',\eta',\epsilon',\dots)$
evaluated at time $s'=s'(s)$. Now we can write the equation of $\delta\epsilon(s)=\epsilon(s)-\epsilon'(s'(s))$:
\begin{align}
(\partial_{s}-\frac{\lambda_{s}}{\lambda}\Lambda+\gamma_{s}i)\delta\epsilon+i\mathcal{L}_{Q}\delta\epsilon & =-(\delta\gamma)_{s}i\epsilon'+\Big(\frac{ds'}{ds}-1\Big)i\mathcal{L}_{Q}\epsilon'\label{eq:Diff-e-Eq}\\
 & \quad+\tilde{\mathbf{Mod}}\cdot(\delta\mathbf{v})+(\tilde{\mathbf{Mod}}-\frac{ds'}{ds}\tilde{\mathbf{Mod}}')\cdot\mathbf{v}'\nonumber \\
 & \quad-i\delta\tilde R{}_{\mathrm{L-L}}-i\delta\tilde R_{\mathrm{NL}}-i\delta\Psi\nonumber \\
 & \quad+\Big(\frac{ds'}{ds}-1\Big)(i\tilde R_{\mathrm{L-L}}'+i\tilde R_{\mathrm{NL}}'+i\Psi').\nonumber 
\end{align}
Here, we use the notation $\delta f(s)\coloneqq f(s)-f'(s'(s))$ to
denote the difference of the objects $f$ for $u$ at time $s$ and
$f'$ for $u'$ at time $s'(s)$. For example, 
\[
(\delta\gamma)_{s}\coloneqq\tfrac{d}{ds}\big(\gamma(s)-\gamma'(s'(s))\big),\quad\delta\mathbf{v}\coloneqq\mathbf{v}(s)-\mathbf{v}'(s'(s)),\quad\delta\Psi\coloneqq\Psi(s)-\Psi'(s'(s)).
\]

Due to $\partial_{s}=\frac{ds'}{ds}\partial_{s'}$, $\frac{ds'}{ds}\tilde{\mathbf{Mod}}'$
encodes the evolution laws of $\lambda',\gamma',b',\eta'$ in the
$s$-variable: 
\[
\frac{ds'}{ds}\tilde{\mathbf{Mod}}'=(\frac{\lambda_{s}}{\lambda}+\frac{ds'}{ds}b',\tilde{\gamma}_{s}'-\frac{ds'}{ds}\eta'\theta_{\eta'},b'_{s}+\frac{ds'}{ds}((b')^{2}+(\eta')^{2}),\eta'_{s})^{t}.
\]
Thus $\tilde{\mathbf{Mod}}-\frac{ds'}{ds}\tilde{\mathbf{Mod}}'$ encodes
the evolution laws of $\frac{ds'}{ds},\delta\tilde{\gamma},\delta b,\delta\eta$
in the $s$-variable: 
\begin{align*}
\tilde{\mathrm{Mod}}_{1}-\tfrac{ds'}{ds}\tilde{\mathrm{Mod}}_{1}' & =\delta b-(\tfrac{ds'}{ds}-1)b',\\
\tilde{\mathrm{Mod}}_{2}-\tfrac{ds'}{ds}\tilde{\mathrm{Mod}}_{2}' & =(\delta\tilde{\gamma})_{s}-\eta\theta_{\eta}+\tfrac{ds'}{ds}\eta'\theta_{\eta'},\\
\tilde{\mathrm{Mod}}_{3}-\tfrac{ds'}{ds}\tilde{\mathrm{Mod}}_{3}' & =(\delta b)_{s}+b^{2}+\eta^{2}-\tfrac{ds'}{ds}((b')^{2}+(\eta')^{2}),\\
\tilde{\mathrm{Mod}}_{4}-\tfrac{ds'}{ds}\tilde{\mathrm{Mod}}_{4}' & =(\delta\eta)_{s}.
\end{align*}

\subsection{Difference estimates}

We begin the difference estimates of the RHS of \eqref{eq:Diff-e-Eq}.
\begin{lem}[Difference estimates for profiles, modulation vectors, radiation terms]
\label{lem:Diff-Est-ProfilesModRad}We have 
\begin{align*}
|\delta P|_{3} & \lesssim(|\delta b|+|\delta\eta|)y^{2}Q\mathbf{1}_{y\lesssim b^{-1/2}},\\
|\delta\mathbf{v}|_{3} & \lesssim(|\delta b|+|\delta\eta|)y^{4}Q\mathbf{1}_{y\lesssim b^{-1/2}},\\
|\delta(\eta\theta_{\eta})| & \lesssim|\eta b^{m}(\delta b)|+|\delta\eta|\lesssim s^{-\frac{3}{2}}D,\\
|\delta\theta_{\Psi}| & \lesssim b^{m+1}(|\delta b|+|\delta\eta|),\\
|\delta\Psi|_{3} & \lesssim b^{\frac{m}{2}+1}(|\delta b|+|\delta\eta|)\mathbf{1}_{y\sim b^{-1/2}}.
\end{align*}
\end{lem}

\begin{proof}
For each of the above quantitites, we use 
\[
|\delta f|\lesssim|\delta b||\partial_{b}f|+|\delta\eta||\partial_{\eta}f|.
\]
Thus the estimate for $|\delta P|_{3}$ and $|\delta\mathbf{v}|_{3}$
follow from global bounds of Proposition \ref{prop:modified-profile}.
The estimates for $\delta(\eta\theta_{\eta})$, $\delta\theta_{\Psi}$,
and $|\delta\Psi|_{3}$ follow from \eqref{eq:db-theta-eta-est},
\eqref{eq:deta-theta-eta-est}, \eqref{eq:Diff-theta-Psi-est}, and
\eqref{eq:Diff-Psi-est}.
\end{proof}
\begin{lem}[Difference estimate for $R_{\mathrm{L-L}}$]
\label{lem:Diff-RLL-Est}We have
\begin{align}
|\delta\theta_{\mathrm{L-L}}| & \lesssim s^{-3}D,\label{eq:Diff-theta-LL-est}\\
\|\delta\tilde R_{\mathrm{L-L}}\|_{\dot{H}_{m}^{3}} & \lesssim b\Big(\|\delta\epsilon\|_{\dot{\mathcal{H}}_{m,\leq M_{2}}^{3}}+o_{M_{2}\to\infty}(1)\|\delta\epsilon\|_{\dot{H}_{m}^{3}}+s^{-\frac{1}{2}}(s^{-\frac{5}{2}}D)\Big).\label{eq:Diff-RLL-est}
\end{align}
\end{lem}

\begin{rem}
Indeed $s^{-\frac{1}{2}}(s^{-\frac{5}{2}}D)=s^{-3}D$. But we want
to emphasize that $s^{-\frac{1}{2}}(s^{-\frac{5}{2}}D)$ is perturbative
relative to $\|\delta\epsilon\|_{\dot{H}_{m}^{3}}$-part ($\sim s^{-\frac{5}{2}}D$).
\end{rem}

\begin{proof}
Lemma \ref{lem:Diff-RLL-Est} is an analogue of Lemma \ref{lem:degenerate-linear}
for the difference. The proof relies on that of Lemma \ref{lem:degenerate-linear}.

We start with $\delta\theta_{\mathrm{L-L}}$-estimate. In the proof
of the $\theta_{\mathrm{L-L}}$ estimate \eqref{eq:LL-theta-est},
we viewed $\theta_{\mathrm{L-L}}$ as a linear combination of 
\[
{\textstyle \int_{0}^{\infty}}m\Re(\overline{\psi_{1}}\psi_{2})\tfrac{dy}{y}\quad\text{and}\quad{\textstyle \int_{0}^{\infty}}A_{\theta}[\psi_{1},\psi_{2}]\Re(\overline{\psi_{3}}\psi_{4})\tfrac{dy}{y},
\]
where $\psi_{j}=\epsilon$ and $\psi_{k}=P-Q$ for some $j$ and $k$,
and $\psi_{\ell}\in\{P,Q\}$ for $\ell\neq j,k$. To show the $\delta\theta_{\mathrm{L-L}}$
estimate, we view $\delta\theta_{\mathrm{L-L}}$ as the same linear
combination, where
\begin{itemize}
\item $\psi_{j}\in\{\epsilon,\epsilon',\delta\epsilon\}$ and $\psi_{k}\in\{P-Q,P'-Q,\delta P\}$
for some $j$ and $k$,
\item $\psi_{\ell}\in\{P,P',Q,\delta P\}$ for $\ell\neq j,k$, and
\item $\delta$ should appear \emph{exactly once} among $\psi_{i}$'s.
\end{itemize}
If $\delta$ hits $\epsilon$, i.e. $\psi_{j}=\delta\epsilon$, then
the proof of \eqref{eq:LL-theta-est} works by replacing $\epsilon$
by $\delta\epsilon$ and yields the bound $b^{\frac{1}{2}}\|\delta\epsilon\|_{\dot{H}_{m}^{3}}$.
If $\delta$ hits $P-Q$, i.e. $\psi_{k}=\delta P$, then $|\delta P|\lesssim(|\delta b|+|\delta\eta|)y^{2}Q\mathbf{1}_{y\lesssim b^{-1/2}}\lesssim b^{-1}(|\delta b|+|\delta\eta|)\min\{by^{2},1\}Q$.
Compared to $|P-Q|\lesssim\min\{by^{2},1\}Q$ used in the proof of
\eqref{eq:LL-theta-est}, it suffices to multiply $b^{-1}(|\delta b|+|\delta\eta|)$
to the bound \eqref{eq:LL-theta-est}. This yields the bound $b^{-1}(|\delta b|+|\delta\eta|)(\|\epsilon\|_{\dot{H}_{m}^{3}}+\|\epsilon'\|_{\dot{H}_{m}^{3}})$.
If $\delta$ hits $P$, i.e. $\psi_{\ell}=\delta P$, then $|\delta P|\lesssim b^{-1}(|\delta b|+|\delta\eta|)Q$.
Thus we need to multiply $b^{-1}(|\delta b|+|\delta\eta|)$ to the
bound obtained in \eqref{eq:LL-theta-est}. As a result, 
\[
|\delta\theta_{\mathrm{L-L}}|\lesssim b^{\frac{1}{2}}\|\delta\epsilon\|_{\dot{H}_{m}^{3}}+b^{-\frac{1}{2}}(|\delta b|+|\delta\eta|)(\|\epsilon\|_{\dot{H}_{m}^{3}}+\|\epsilon'\|_{\dot{H}_{m}^{3}}).
\]
Substituting the a priori $\dot{H}_{m}^{3}$-bound for $\epsilon$
and $\epsilon'$ \eqref{eq:bootstrap-hyp-H5} and using the definition
of $D$ yield the $\delta\theta_{\mathrm{L-L}}$ bound.

For the $\delta\tilde R_{\mathrm{L-L}}$-estimate, we similarly argue
as in the $\delta\theta_{\mathrm{L-L}}$ estimate. We view $\delta\tilde R_{\mathrm{L-L}}$
as a linear combination of 
\[
V_{3}[\psi_{1},\psi_{2}]\psi_{3}\quad\text{and}\quad V_{5}[\psi_{1},\psi_{2},\psi_{3},\psi_{4}]\psi_{5},
\]
(replace $V_{3}$ by $\tilde V_{3}$ if $\psi_{3}\in\{P,P'\}$ and
similarly for $V_{5}$) where $\psi_{j}\in\{\epsilon,\epsilon',\delta\epsilon\}$
and $\psi_{k}\in\{P-Q,P'-Q,\delta P\}$ for some $j$ and $k$, $\psi_{\ell}\in\{P,P',Q,\delta P\}$
for $\ell\neq j,k$, and $\delta$ should appear \emph{exactly once}
among $\psi_{i}$'s. If $\psi_{j}=\delta\epsilon$, the proof of \eqref{eq:LL-Hdot3-est}
works by replacing $\epsilon$ by $\delta\epsilon$. If $\delta$
hits $\psi_{k}$ or $\psi_{\ell}$, we need to multiply $b^{-1}(|\delta b|+|\delta\eta|)$
to the bound \eqref{eq:LL-Hdot3-est}. As a result,
\[
\|\delta\tilde R_{\mathrm{L-L}}\|_{\dot{H}_{m}^{3}}\lesssim b(\|\delta\epsilon\|_{\dot{\mathcal{H}}_{m,\leq M_{2}}^{3}}+o_{M_{2}\to\infty}(1)\|\delta\epsilon\|_{\dot{H}_{m}^{3}})+(|\delta b|+|\delta\eta|)(\|\epsilon\|_{\dot{H}_{m}^{3}}+\|\epsilon'\|_{\dot{H}_{m}^{3}}).
\]
Substituting the a priori $\dot{H}_{m}^{3}$-bound \eqref{eq:bootstrap-hyp-H5}
and using the definition of $D$, we get the conclusion.
\end{proof}
\begin{lem}[Difference estimate for $R_{\mathrm{NL}}$]
\label{lem:Diff-RNL-Est}We have 
\begin{align}
|\delta\theta_{\mathrm{NL}}| & \lesssim s^{-\frac{3}{2}}D,\label{eq:Diff-theta-NL-est}\\
\|\delta\tilde R_{\mathrm{NL}}\|_{\dot{H}_{m}^{3}} & \lesssim bs^{-\frac{1}{2}}(s^{-\frac{5}{2}}D).\label{eq:Diff-RNL-est}
\end{align}
\end{lem}

\begin{proof}
We argue as in the proof of Lemma \ref{lem:Diff-RLL-Est}. We view
$\delta\theta_{\mathrm{NL}}$ as a linear combination of 
\[
{\textstyle \int_{0}^{\infty}}m\Re(\overline{\psi_{1}}\psi_{2})\tfrac{dy}{y}\quad\text{and}\quad{\textstyle \int_{0}^{\infty}}A_{\theta}[\psi_{1},\psi_{2}]\Re(\overline{\psi_{3}}\psi_{4})\tfrac{dy}{y},
\]
where at least two $\psi_{j}$'s belong to $\{\epsilon,\epsilon',\delta\epsilon\}$,
remaining $\psi_{\ell}$'s are filled with $P,P',\delta P$, and $\delta$
should appear \emph{exactly once} among all $\psi_{i}$'s.

Let us rewrite the bound \eqref{eq:NL-theta-est} of $\theta_{\mathrm{NL}}$
as 
\[
|\theta_{\mathrm{NL}}|\lesssim\|\epsilon\|_{\dot{H}_{m}^{1}}^{2}(1+\|\epsilon\|_{L^{2}})^{2},
\]
by inspecting the proof of \eqref{eq:NL-theta-est} without assuming
$\|\epsilon\|_{L^{2}}\leq1$. If $\delta$ hits $\epsilon$, i.e.
one of $\psi_{j}$ is $\delta\epsilon$, the proof of \eqref{eq:NL-theta-est}
works by replacing \emph{one} $\epsilon$ by $\delta\epsilon$. Thus
the resulting bound is \footnote{In fact, the upper bound should contain the contributions of $\epsilon'$.
We omit them for readability. The same remark applies for $\delta\tilde R_{\mathrm{NL}}$.} 
\[
\|\epsilon\|_{\dot{H}_{m}^{1}}\|\delta\epsilon\|_{\dot{H}_{m}^{1}}(1+\|\epsilon\|_{L^{2}})^{2}+\|\epsilon\|_{\dot{H}_{m}^{1}}^{2}\|\delta\epsilon\|_{L^{2}}(1+\|\epsilon\|_{L^{2}}).
\]
Applying the interpolation and coercivity (Lemma \ref{lem:Coercivity-AAL-Section2}),
we estimate $\|\delta\epsilon\|_{\dot{H}_{m}^{1}}\lesssim\|\delta\epsilon\|_{L^{2}}^{\frac{2}{3}}\|\delta\epsilon\|_{\dot{H}_{m}^{3}}^{\frac{1}{3}}\lesssim_{M}s^{-\frac{5}{6}}D$.
To remove the $M$-dependence, we lose a small power of $s$. Applying
a priori bounds of $\epsilon$, the above bound is dominated by $s^{-\frac{3}{2}}D$,
as desired. If $\delta$ hits $P$, i.e. one of $\psi_{\ell}$ is
$\delta P$, then we lose $b^{-1}(|\delta b|+|\delta\eta|)$ to the
bound \eqref{eq:NL-theta-est}. This yields $s^{-2}D$.

We turn to $\delta\tilde R_{\mathrm{NL}}$. We view $\delta\tilde R_{\mathrm{NL}}$
as a linear combination of 
\[
V_{3}[\psi_{1},\psi_{2}]\psi_{3}\quad\text{and}\quad V_{5}[\psi_{1},\psi_{2},\psi_{3},\psi_{4}]\psi_{5},
\]
(replace $V_{3}$ by $\tilde V_{3}$ if $\psi_{3}\in\{P,P'\}$ and
similarly for $V_{5}$) where at least two $\psi_{j}$'s belong to
$\{\epsilon,\epsilon',\delta\epsilon\}$, remaining $\psi_{\ell}$'s
are filled with $P,P',\delta P$, and $\delta$ should appear \emph{exactly
once} among all $\psi_{i}$'s.

Let us rewrite the bound \eqref{eq:NL-Hdot3-est} of $\tilde R_{\mathrm{NL}}$
as 
\[
\|\tilde R_{\mathrm{NL}}\|_{\dot{H}_{m}^{3}}\lesssim\|\epsilon\|_{\dot{H}_{m}^{3}}^{2}(1+\|\epsilon\|_{H_{m}^{3}})^{3}+\|\epsilon\|_{\dot{H}_{m}^{1}}^{2}\|\epsilon\|_{\dot{H}_{m}^{3}}(1+\|\epsilon\|_{H_{m}^{3}})^{2}+\|\epsilon\|_{\dot{H}_{m}^{1}}^{5}.
\]
If $\psi_{j}=\delta\epsilon$, the proof of \eqref{eq:NL-Hdot3-est}
works by replacing \emph{one} $\epsilon$ by $\delta\epsilon$, i.e.
the resulting bound (with applying $\|\epsilon\|_{H_{m}^{3}}\leq1$)
is 
\begin{multline*}
(\|\epsilon\|_{\dot{H}_{m}^{3}}^{2}+\|\epsilon\|_{\dot{H}_{m}^{1}}^{2}\|\epsilon\|_{\dot{H}_{m}^{3}})\|\delta\epsilon\|_{H_{m}^{3}}+(\|\epsilon\|_{\dot{H}_{m}^{1}}\|\epsilon\|_{\dot{H}_{m}^{3}}+\|\epsilon\|_{\dot{H}_{m}^{1}}^{4})\|\delta\epsilon\|_{\dot{H}_{m}^{1}}\\
+(\|\epsilon\|_{\dot{H}_{m}^{3}}+\|\epsilon\|_{\dot{H}_{m}^{1}}^{2})\|\delta\epsilon\|_{\dot{H}_{m}^{3}}.
\end{multline*}
Using $\|\delta\epsilon\|_{H_{m}^{3}}\leq\|\delta\epsilon\|_{L^{2}}+\|\delta\epsilon\|_{\dot{H}_{m}^{3}}$,
$\|\delta\epsilon\|_{\dot{H}_{m}^{1}}\lesssim s^{-\frac{5}{6}+}D$,
$\|\delta\epsilon\|_{\dot{H}_{m}^{3}}\lesssim s^{-\frac{5}{2}+}D$
(by coercivity), and applying a priori bounds of $\epsilon$, the
above bound is dominated by $s^{-\frac{9}{2}+}D$, which is enough.
If $\psi_{\ell}=\delta P$, then from $|\delta P|_{3}\lesssim b^{-1}(|\delta b|+|\delta\eta|)Q$,
so we lose $b^{-1}(|\delta b|+|\delta\eta|)$ to the bound \eqref{eq:NL-Hdot3-est}.
This yields $s^{-5+}D$, which is enough.
\end{proof}

\subsection{Modulation estimates}

In this subsection, we prove the difference analogue of the modulation
estimates (Lemma \ref{lem:mod-est}). More precisely, we estimate
$\tilde{\mathbf{Mod}}-\frac{ds'}{ds}\tilde{\mathbf{Mod}}'$, which
encodes the evolution laws of the difference of the modulation parameters.
\begin{lem}[Modulation estimates]
\label{lem:DiffModEst}We have 
\begin{equation}
\Big|\tilde{\mathbf{Mod}}-\frac{ds'}{ds}\tilde{\mathbf{Mod}}'\Big|\lesssim o_{M\to\infty}(1)\|\delta\epsilon_{3}\|_{L^{2}}+s^{-\frac{1}{2}}(s^{-\frac{5}{2}}D)\label{eq:DiffModEst}
\end{equation}
In particular, 
\begin{align}
\Big|\frac{ds'}{ds}-1\Big| & \lesssim\frac{|\delta b|}{b'}+s^{-\frac{3}{2}}D\lesssim D,\label{eq:Diff-dsds-Estimate}\\
|(\delta\gamma)_{s}| & \lesssim s^{-\frac{3}{2}}D\label{eq:DiffGammaEstimate}
\end{align}
\end{lem}

\begin{proof}
Let us denote by $\theta_{cor}\coloneqq-(\theta_{\Psi}+\theta_{\mathrm{L-L}}+\theta_{\mathrm{NL}})$.
Then, we have 
\begin{align*}
(\delta\gamma)_{s} & =\Big(\tilde{\mathrm{Mod}}_{2}-\frac{ds'}{ds}\tilde{\mathrm{Mod}}_{2}'\Big)+\delta(\eta\theta_{\eta}+\theta_{cor})-\Big(\frac{ds'}{ds}-1\Big)(\eta'\theta_{\eta'}+\theta_{cor}'),\\
\frac{ds'}{ds}-1 & =\frac{1}{b'}\Big(\delta b-\Big(\tilde{\mathrm{Mod}}_{1}-\frac{ds'}{ds}\tilde{\mathrm{Mod}}_{1}'\Big)\Big).
\end{align*}
Thus \eqref{eq:Diff-dsds-Estimate} and \eqref{eq:DiffGammaEstimate}
follow from \eqref{eq:DiffModEst}, the estimates for $\delta(\eta\theta_{\eta})$
and $\delta\theta_{cor}$ (Lemma \ref{lem:Diff-Est-ProfilesModRad},
\eqref{eq:Diff-theta-LL-est}, and \eqref{eq:Diff-theta-NL-est}),
and the estimates for $\eta\theta_{\eta}$ and $\theta_{cor}$ (\eqref{eq:decomp-P-temp1},
\eqref{eq:decomp-P-temp2}, \eqref{eq:LL-theta-est}, and \eqref{eq:NL-theta-est}).

Henceforth, we focus on the proof of \eqref{eq:DiffModEst}. We rewrite
\eqref{eq:Diff-e-Eq} as 
\begin{align}
 & (\tilde{\mathbf{Mod}}-\frac{ds'}{ds}\tilde{\mathbf{Mod}}')\cdot(\mathbf{v}'-\frac{e_{1}}{b'}((\eta'\theta_{\eta'}+\theta_{cor}')i\epsilon'+i\mathcal{L}_{Q}\epsilon'+i\tilde R_{\mathrm{L-L}}'+i\tilde R_{\mathrm{NL}}'+i\Psi')-e_{2}(i\epsilon'))\label{eq:DiffModEst-tmp1}\\
 & =(\partial_{s}-\frac{\lambda_{s}}{\lambda}\Lambda+\gamma_{s}i)\delta\epsilon+i\mathcal{L}_{Q}\delta\epsilon\nonumber \\
 & \quad+\delta(\eta\theta_{\eta}+\theta_{cor})i\epsilon'-\frac{\delta b}{b'}((\eta'\theta_{\eta'}+\theta_{cor}')i\epsilon'+i\mathcal{L}_{Q}\epsilon'+i\tilde R_{\mathrm{L-L}}'+i\tilde R_{\mathrm{NL}}'+i\Psi')\nonumber \\
 & \quad-\tilde{\mathbf{Mod}}\cdot(\delta\mathbf{v})+(i\delta\tilde R_{\mathrm{L-L}}+i\delta\tilde R_{\mathrm{NL}}+i\delta\Psi).\nonumber 
\end{align}
As in the proof of Lemma \ref{lem:mod-est}, we take the inner product
of \eqref{eq:DiffModEst-tmp1} and $\mathcal{Z}_{k}$.

We first consider the inner product of $\mathcal{Z}_{k}$ and the
LHS of \eqref{eq:DiffModEst-tmp1}. By \eqref{eq:ortho-jacobian},
the inner product coming from $\mathbf{v}'$ becomes the main term.
The remaining contributions are treated as errors: (see the proof
of Lemma \ref{lem:mod-est}) 
\begin{align*}
(b')^{-1}|(\eta'\theta_{\eta'}+\theta_{cor}')(i\epsilon',\mathcal{Z}_{k})_{r}| & \lesssim b^{\frac{1}{2}}M^{C}\|\epsilon'\|_{\dot{H}_{m}^{3}},\\
(b')^{-1}|(i\mathcal{L}_{Q}\epsilon',\mathcal{Z}_{k})_{r}| & \lesssim b^{-1}\|\epsilon_{3}'\|_{L^{2}},\\
(b')^{-1}|(i\tilde R_{\mathrm{L-L}}'+i\tilde R_{\mathrm{NL}}'+i\Psi',\mathcal{Z}_{k})_{r}| & \lesssim b^{-1}M^{C}\|\tilde R_{\mathrm{L-L}}'+\tilde R_{\mathrm{NL}}'+\Psi'\|_{\dot{H}_{m}^{3}},\\
|(i\epsilon',\mathcal{Z}_{k})_{r}| & \lesssim M^{C}\|\epsilon'\|_{\dot{H}_{m}^{3}}.
\end{align*}
These are all bounded by $b^{2}$. Therefore, the matrix formed by
taking the inner product of 
\[
(\mathbf{v}'-\frac{e_{1}}{b'}((\eta'\theta_{\eta'}+\theta_{cor}')i\epsilon'+i\mathcal{L}_{Q}\epsilon'+i\tilde R_{\mathrm{L-L}}'+i\tilde R_{\mathrm{NL}}'+i\Psi')-e_{2}(i\epsilon'))
\]
and $\mathcal{Z}_{k}$ has uniformly bounded inverse.

Henceforth, it suffices to consider the inner product of $\mathcal{Z}_{k}$
and the RHS of \eqref{eq:DiffModEst-tmp1}. As in the proof of Lemma
\ref{lem:mod-est}, the leading term comes from $(\delta\epsilon,\mathcal{L}_{Q}i\mathcal{Z}_{k})_{r}$:
\[
|(\delta\epsilon,\mathcal{L}_{Q}i\mathcal{Z}_{k})_{r}|\lesssim o_{M\to\infty}(1)\|\delta\epsilon_{3}\|_{L^{2}}.
\]
All the remaining terms will be considered as errors. We have 
\begin{align*}
(\partial_{s}\delta\epsilon,\mathcal{Z}_{k})_{r} & =\partial_{s}(\delta\epsilon,\mathcal{Z}_{k})_{r}=0,\\
|\tfrac{\lambda_{s}}{\lambda}(\delta\epsilon,\Lambda\mathcal{Z}_{k})_{r}|+|\gamma_{s}(\delta\epsilon,i\mathcal{Z}_{k})_{r}| & \lesssim b\cdot M^{C}\|\delta\epsilon\|_{\dot{H}_{m}^{3}}\lesssim s^{-3}D.
\end{align*}
Next, we have 
\begin{align*}
|(\delta(\eta\theta_{\eta}+\theta_{cor})i\epsilon',\mathcal{Z}_{k})_{r}| & \lesssim|\delta(\eta\theta_{\eta}+\theta_{cor})|\cdot M^{C}\|\epsilon'\|_{\dot{H}_{m}^{3}}\lesssim s^{-3}D.
\end{align*}
Next, by the previous paragraph, we have 
\[
|\frac{\delta b}{b'}((\eta'\theta_{\eta'}+\theta_{cor}')i\epsilon'+i\mathcal{L}_{Q}\epsilon'+i\tilde R_{\mathrm{L-L}}'+i\tilde R_{\mathrm{NL}}'+i\Psi',\mathcal{Z}_{k})_{r}|\lesssim b^{2}|\delta b|\lesssim s^{-3}D.
\]
Next, by the modulation estimate (Lemma \ref{lem:mod-est}), estimates
of $\delta\mathbf{v}$ (Lemma \ref{lem:Diff-Est-ProfilesModRad}),
and pointwise bounds of $\mathcal{Z}_{k}$ \eqref{eq:Orthog-Pointwise},
we have 
\begin{align*}
|\tilde{\mathbf{Mod}}||(\delta\mathbf{v},\mathcal{Z}_{k})_{r}| & \lesssim b^{\frac{5}{2}}\cdot M^{C}(|\delta b|+|\delta\eta|)\lesssim s^{-3}D.
\end{align*}
Finally, by \eqref{eq:Diff-RLL-est}, \eqref{eq:Diff-RNL-est}, and
Lemma \ref{lem:Diff-Est-ProfilesModRad}, we have 
\begin{align*}
|(i\delta\tilde R_{\mathrm{L-L}}+i\delta\tilde R_{\mathrm{NL}}+i\delta\Psi,\mathcal{Z}_{k})_{r}| & \lesssim M^{C}\|\delta\tilde R_{\mathrm{L-L}}+\delta\tilde R_{\mathrm{NL}}+\delta\Psi\|_{\dot{H}_{m}^{3}}\\
 & \lesssim M^{C}b(\|\delta\epsilon\|_{\dot{H}_{m}^{3}}+s^{-\frac{1}{2}}(s^{-\frac{5}{2}}D))\lesssim s^{-3}D,
\end{align*}
which is absorbed into $s^{-3}D$. This completes the proof of \eqref{eq:DiffModEst}.
\end{proof}

\subsection{Energy estimate in $\dot{H}_{m}^{3}$}

In this subsection we propagate $\|\delta\epsilon_{3}\|_{L^{2}}$
forwards in time. The argument is similar to Section \ref{subsec:energy-identity}.
We again use local virial corrections to derive a modified energy
inequality.

We start with the equation of $\delta\epsilon_{2}$ obtained by taking
$A_{Q}L_{Q}$ to \eqref{eq:Diff-e-Eq}: 
\begin{align}
 & (\partial_{s}-\frac{\lambda_{s}}{\lambda}\Lambda_{-2}+\gamma_{s}i)\delta\epsilon_{2}+iA_{Q}A_{Q}^{\ast}\delta\epsilon_{2}\label{eq:Diff-e2-Eq}\\
 & =-(\delta\gamma)_{s}i\epsilon_{2}'+\Big(\frac{ds'}{ds}-1\Big)i\epsilon_{4}'+\frac{\lambda_{s}}{\lambda}\partial_{\lambda}(A_{Q_{\lambda}}L_{Q_{\lambda}})\delta\epsilon-\gamma_{s}A_{Q}[L_{Q},i]\delta\epsilon\nonumber \\
 & \quad+\tilde{\mathbf{Mod}}\cdot A_{Q}L_{Q}(\delta\mathbf{v})+(\tilde{\mathbf{Mod}}-\frac{ds'}{ds}\tilde{\mathbf{Mod}}')\cdot A_{Q}L_{Q}\mathbf{v}'\nonumber \\
 & \quad-A_{Q}L_{Q}(i\delta\tilde R{}_{\mathrm{L-L}}+i\delta\tilde R_{\mathrm{NL}}+i\delta\Psi)\nonumber \\
 & \quad+\Big(\frac{ds'}{ds}-1\Big)A_{Q}L_{Q}(i\tilde R_{\mathrm{L-L}}'+i\tilde R_{\mathrm{NL}}'+i\Psi').\nonumber 
\end{align}
Similarly to Lemma \ref{lem:E3-Identity}, we prove the following.
\begin{lem}[Energy identity for $\dot{H}_{m}^{3}$]
\label{lem:Diff-H3-Identity}We have 
\begin{align*}
 & \Big|\Big(\partial_{s}-6\frac{\lambda_{s}}{\lambda}\Big)\|\delta\epsilon_{3}\|_{L^{2}}^{2}\Big|\lesssim b\|\delta\epsilon_{3}\|_{L^{2}}\\
 & \qquad\times\Big(o_{M\to\infty}(1)\|\delta\epsilon_{3}\|_{L^{2}}+\|\delta\epsilon\|_{\dot{\mathcal{H}}_{m,\leq M_{2}}^{3}}+o_{M_{2}\to\infty}(1)\|\delta\epsilon\|_{\dot{H}_{m}^{3}}+s^{-\frac{1}{2}}(s^{-\frac{5}{2}}D)\Big).
\end{align*}
\end{lem}

\begin{rem}
Here the reader can observe why we need a priori $H^{5}$-control
of $\epsilon$. In \eqref{eq:Diff-e3-Eq}, the term $(\frac{ds'}{ds}-1)i\epsilon_{5}'$
appears. Thus we need to perform an energy estimate even for $\epsilon_{5}$
and hence obtain $H_{m}^{5}$-trapped solutions.
\end{rem}

\begin{proof}
Taking $A_{Q}^{\ast}$ to \eqref{eq:Diff-e2-Eq}, we get the equation
of $\delta\epsilon_{3}$: 
\begin{align}
 & (\partial_{s}-\frac{\lambda_{s}}{\lambda}\Lambda_{-3}+\gamma_{s}i)\delta\epsilon_{3}+iA_{Q}^{\ast}A_{Q}\delta\epsilon_{3}\nonumber \\
 & =-(\delta\gamma)_{s}i\epsilon_{3}'+\Big(\frac{ds'}{ds}-1\Big)i\epsilon_{5}'+\frac{\lambda_{s}}{\lambda}\partial_{\lambda}(A_{Q_{\lambda}}^{\ast}A_{Q_{\lambda}}L_{Q_{\lambda}})\delta\epsilon-\gamma_{s}A_{Q}^{\ast}A_{Q}[L_{Q},i]\delta\epsilon\label{eq:Diff-e3-Eq}\\
 & \quad+\tilde{\mathbf{Mod}}\cdot A_{Q}^{\ast}A_{Q}L_{Q}(\delta\mathbf{v})+(\tilde{\mathbf{Mod}}-\frac{ds'}{ds}\tilde{\mathbf{Mod}}')\cdot A_{Q}^{\ast}A_{Q}L_{Q}\mathbf{v}'\nonumber \\
 & \quad-A_{Q}^{\ast}A_{Q}L_{Q}(i\delta\tilde R{}_{\mathrm{L-L}}+i\delta\tilde R_{\mathrm{NL}}+i\delta\Psi)\nonumber \\
 & \quad+\Big(\frac{ds'}{ds}-1\Big)A_{Q}^{\ast}A_{Q}L_{Q}(i\tilde R_{\mathrm{L-L}}'+i\tilde R_{\mathrm{NL}}'+i\Psi').\nonumber 
\end{align}
We take the inner product of \eqref{eq:Diff-e3-Eq} with $\delta\epsilon_{3}$
to get 
\[
\Big|\frac{1}{2}\Big(\partial_{s}-6\frac{\lambda_{s}}{\lambda}\Big)\|\delta\epsilon_{3}\|_{L^{2}}^{2}\Big|=\Big|(\delta\epsilon_{3},\text{RHS of }\eqref{eq:Diff-e3-Eq})_{r}\Big|\lesssim\|\delta\epsilon_{3}\|_{L^{2}}\|\text{RHS of }\eqref{eq:Diff-e3-Eq}\|_{L^{2}}.
\]

It now suffices to estimate $\|\text{RHS of }\eqref{eq:Diff-e3-Eq}\|_{L^{2}}$.
By \eqref{eq:DiffGammaEstimate} and a priori $\epsilon_{3}$ bound,
we have 
\[
\|(\delta\gamma)_{s}i\epsilon_{3}'\|_{L^{2}}\lesssim bs^{-3}D.
\]
By \eqref{eq:Diff-dsds-Estimate} and \emph{a priori $\epsilon_{5}$
bound}, we have 
\[
\|\Big(\frac{ds'}{ds}-1\Big)i\epsilon_{5}'\|_{L^{2}}\lesssim bs^{-3}D.
\]
By Lemma \ref{lem:commutator}, we have 
\[
\|\frac{\lambda_{s}}{\lambda}\partial_{\lambda}(A_{Q_{\lambda}}^{\ast}A_{Q_{\lambda}}L_{Q_{\lambda}})\delta\epsilon-\gamma_{s}A_{Q}^{\ast}A_{Q}[L_{Q},i]\delta\epsilon\|_{L^{2}}\lesssim b(\|\delta\epsilon\|_{\dot{\mathcal{H}}_{m,\leq M_{2}}^{3}}+o_{M_{2}\to\infty}(1)\|\delta\epsilon\|_{\dot{H}_{m}^{3}}).
\]
By Lemma \ref{lem:mod-est} and Lemma \ref{lem:Diff-Est-ProfilesModRad},
we have 
\[
\|\tilde{\mathbf{Mod}}\cdot A_{Q}^{\ast}A_{Q}L_{Q}(\delta\mathbf{v})\|_{L^{2}}\lesssim bs^{-3}D.
\]
By Lemma \ref{lem:DiffModEst} and \eqref{eq:gen-null-rel}, we have
\[
\|(\tilde{\mathbf{Mod}}-\frac{ds'}{ds}\tilde{\mathbf{Mod}}')\cdot A_{Q}^{\ast}A_{Q}L_{Q}\mathbf{v}'\|_{L^{2}}\lesssim b(o_{M\to\infty}(1)\|\delta\epsilon_{3}\|_{L^{2}}+s^{-\frac{1}{2}}(s^{-\frac{5}{2}}D)).
\]
Next, by \eqref{eq:Diff-RLL-est}, \eqref{eq:Diff-RNL-est}, and Lemma
\ref{lem:Diff-Est-ProfilesModRad}, we have 
\begin{align*}
 & \|A_{Q}^{\ast}A_{Q}L_{Q}(i\delta\tilde R{}_{\mathrm{L-L}}+i\delta\tilde R_{\mathrm{NL}}+i\delta\Psi)\|_{L^{2}}\\
 & \lesssim b(\|\delta\epsilon\|_{\dot{\mathcal{H}}_{m,\leq M_{2}}^{3}}+o_{M_{2}\to\infty}(1)\|\delta\epsilon\|_{\dot{H}_{m}^{3}}+s^{-\frac{1}{2}}(s^{-\frac{5}{2}}D)).
\end{align*}
Finally, by \eqref{eq:Diff-dsds-Estimate}, \eqref{eq:LL-E3-est-m-geq-3},
\eqref{eq:NL-E3-est}, and \eqref{eq:Radiation-Hdot3-est}, we have
\[
\|\Big(\frac{ds'}{ds}-1\Big)A_{Q}^{\ast}A_{Q}L_{Q}(i\tilde R_{\mathrm{L-L}}'+i\tilde R_{\mathrm{NL}}'+i\Psi')\|_{L^{2}}\lesssim bs^{-3}D.
\]
This completes the proof.
\end{proof}
\begin{lem}[Local virial control]
\label{lem:DiffLocalVirialControlH3}We have 
\begin{align}
 & \Big|\frac{b}{\log M_{2}}\int_{M_{2}}^{M_{2}^{2}}(\delta\epsilon_{2},-i\Lambda_{M_{2}'}\delta\epsilon_{2})_{r}\frac{dM_{2}'}{M_{2}'}\Big|\lesssim M_{2}^{C}b\|\delta\epsilon_{3}\|_{L^{2}}^{2},\label{eq:DiffLocalVirialBdryH3}\\
 & \Big(\partial_{s}-6\frac{\lambda_{s}}{\lambda}\Big)\Big[\frac{b}{\log M_{2}}\int_{M_{2}}^{M_{2}^{2}}(\delta\epsilon_{2},-i\Lambda_{M_{2}'}\delta\epsilon_{2})_{r}\frac{dM_{2}'}{M_{2}'}\Big]\label{eq:DiffLocalVirialDerivH3}\\
 & \geq b\Big(c_{M}\|\delta\epsilon\|_{\dot{\mathcal{H}}_{m,\leq M_{2}}^{3}}^{2}-o_{M_{2}\to\infty}(1)\|\delta\epsilon\|_{\dot{\mathcal{H}}_{m}^{3}}^{2}-\|\delta\epsilon_{3}\|_{L^{2}}(M_{2}^{C}b(s^{-\frac{5}{2}}D))\Big).\nonumber 
\end{align}
\end{lem}

\begin{proof}
We note that \eqref{eq:DiffLocalVirialBdryH3} is same as \eqref{eq:LocalVirialBdryH3},
after replacing $\epsilon_{2}$ by $\delta\epsilon_{2}$. Henceforth,
we focus on \eqref{eq:DiffLocalVirialDerivH3}. By the averaging argument
(see the proof of Lemma \ref{lem:LocalVirialControl}), it suffices
to show the unaveraged version of \eqref{eq:DiffLocalVirialDerivH3}:
\begin{align*}
 & \Big(\partial_{s}-6\frac{\lambda_{s}}{\lambda}\Big)(\delta\epsilon_{2},-i\Lambda_{M_{2}}\delta\epsilon_{2})_{r}\\
 & \geq c_{M}\|\delta\epsilon\|_{\dot{\mathcal{H}}_{m,\leq M_{2}}^{3}}^{2}-O\Big(\int\mathbf{1}_{y\sim M_{2}}|\delta\epsilon|_{-3}^{2}\Big)\\
 & \quad-o_{M_{2}\to\infty}(1)\|\delta\epsilon\|_{\dot{H}_{m}^{3}}^{2}-\|\delta\epsilon_{3}\|_{L^{2}}\cdot O(M_{2}^{2}b\|\delta\epsilon\|_{\dot{H}_{m}^{3}}+M_{2}^{2}bs^{-\frac{1}{2}}(s^{-\frac{5}{2}}D)).
\end{align*}

To show this, we start with 
\begin{align*}
 & \frac{1}{2}\Big(\partial_{s}-6\frac{\lambda_{s}}{\lambda}\Big)(\delta\epsilon_{2},-i\Lambda_{M_{2}}\delta\epsilon_{2})_{r}=(\partial_{s}\delta\epsilon_{2},-i\Lambda_{M_{2}}\delta\epsilon_{2})_{r}\\
 & =(A_{Q}^{\ast}A_{Q}\delta\epsilon_{2},\Lambda_{M_{2}}\delta\epsilon_{2})_{r}+\frac{\lambda_{s}}{\lambda}(y\partial_{y}\delta\epsilon_{2},-i\Lambda_{M_{2}}\delta\epsilon_{2})_{r}+(\text{RHS of }\eqref{eq:Diff-e2-Eq},-i\Lambda_{M_{2}}\delta\epsilon_{2})_{r}
\end{align*}
The first term of the RHS is treated in Lemma \ref{lem:LocalizedRepulsivity},
after replacing $\epsilon_{2}$ by $\delta\epsilon_{2}$. The second
term of the RHS is bounded by 
\[
\Big|\frac{\lambda_{s}}{\lambda}(y\partial_{y}\delta\epsilon_{2},-i\Lambda_{M_{2}}\delta\epsilon_{2})_{r}\Big|\lesssim b(M_{2})^{2}\|\delta\epsilon_{3}\|_{L^{2}}^{2}.
\]
Next, proceeding as in the proof of Lemma \ref{lem:Diff-H3-Identity},
we have 
\[
\|y^{-1}(\text{RHS of }\eqref{eq:Diff-e2-Eq})\|_{L^{2}}\lesssim b\|\delta\epsilon\|_{\dot{H}_{m}^{3}}+bs^{-\frac{1}{2}}(s^{-\frac{5}{2}}D).
\]
Thus 
\begin{align*}
|(\text{RHS of }\eqref{eq:Diff-e2-Eq},-i\Lambda_{M_{2}}\delta\epsilon_{2})_{r}| & \lesssim\|y^{-1}(\text{RHS of }\eqref{eq:Diff-e2-Eq})\|_{L^{2}}\cdot M_{2}^{2}\|\delta\epsilon_{3}\|_{L^{2}}\\
 & \lesssim\|\delta\epsilon_{3}\|_{L^{2}}(M_{2}^{2}b\|\delta\epsilon\|_{\dot{H}_{m}^{3}}+M_{2}^{2}bs^{-\frac{1}{2}}(s^{-\frac{5}{2}}D)).
\end{align*}
Summing up the above estimates completes the proof.
\end{proof}
Define the modified energy by 
\[
\mathcal{F}_{3}[\delta\epsilon]\coloneqq\|\delta\epsilon_{3}\|_{L^{2}}^{2}-M_{1}\frac{b}{\log M_{2}}\int_{M_{2}}^{M_{2}^{2}}(\delta\epsilon_{2},-i\Lambda_{M_{2}'}\delta\epsilon_{2})_{r}\frac{dM_{2}'}{M_{2}'}.
\]

\begin{prop}[Modified energy inequality for $\dot{H}_{m}^{3}$]
We have 
\begin{align}
|\mathcal{F}_{3}[\delta\epsilon]-\|\delta\epsilon_{3}\|_{L^{2}}^{2}| & \leq\frac{1}{100}\|\delta\epsilon_{3}\|_{L^{2}}^{2},\label{eq:DiffModifiedEnergyBdryH3}\\
\Big(\partial_{s}-6\frac{\lambda_{s}}{\lambda}\Big)\mathcal{F}_{3}[\delta\epsilon] & \leq b\Big(\frac{1}{100}\|\delta\epsilon_{3}\|_{L^{2}}^{2}+Cs^{-1}(s^{-5}D^{2})\Big).\label{eq:DiffModifiedEnergyDerivH3}
\end{align}
\end{prop}

\begin{proof}
This follows as the same fashion as the proof of Proposition \ref{prop:ModifiedEnergyInequalityH3}.
\end{proof}

\subsection{Energy estimate in $L^{2}$}

In this subsection, we propagate $\|\delta\epsilon\|_{L^{2}}$. Similarly
to Section \ref{subsec:EnergyEstimateL2}, we replace the linear part
$i\mathcal{L}_{Q}\delta\epsilon$ by $-i\Delta_{m}\delta\epsilon$,
as the $i\mathcal{L}_{Q}$-flow does not conserve $L^{2}$-norm. Thus
we rewrite the equation of $\delta\epsilon$ as 
\begin{align}
(\partial_{s}-\frac{\lambda_{s}}{\lambda}\Lambda+\gamma_{s}i)\delta\epsilon-i\Delta_{m}\delta\epsilon & =-(\delta\gamma)_{s}i\epsilon'-\Big(\frac{ds'}{ds}-1\Big)i\Delta_{m}\epsilon'\label{eq:Diff-e-Eq-L2}\\
 & \quad+\mathbf{Mod}\cdot(\delta\mathbf{v})+(\mathbf{Mod}-\frac{ds'}{ds}\mathbf{Mod}')\cdot\mathbf{v}'\nonumber \\
 & \quad-i\delta[\mathcal{N}(P+\epsilon)-\mathcal{N}(P)]-i\delta\Psi\nonumber \\
 & \quad+\Big(\frac{ds'}{ds}-1\Big)(i[\mathcal{N}(P'+\epsilon')-\mathcal{N}(P')]+i\Psi').\nonumber 
\end{align}
By an analogue of Lemma \ref{lem:L2-EstimateRemainders}, we record
the following.
\begin{lem}[Some remainder estimates for $L^{2}$ difference]
\label{lem:L2Diff-Remainders}We have 
\begin{align}
\|(\mathbf{Mod}-\tfrac{ds'}{ds}\mathbf{Mod}')\cdot\mathbf{v}'\|_{L^{2}} & \lesssim s^{-2}D,\label{eq:L2Diff-Mod}\\
\|\delta[\mathcal{N}(P+\epsilon)-\mathcal{N}(P)]\|_{L^{2}} & \lesssim_{M}s^{-\frac{5}{3}}D.\label{eq:L2Diff-Nonlin}
\end{align}
\end{lem}

\begin{proof}
First, \eqref{eq:L2Diff-Mod} follows from writing 
\[
|\mathbf{Mod}-\tfrac{ds'}{ds}\mathbf{Mod}'|\lesssim|\tilde{\mathbf{Mod}}-\tfrac{ds'}{ds}\tilde{\mathbf{Mod}}'|+|\delta\theta_{\mathrm{L-L}}+\delta\theta_{\mathrm{NL}}|+|\tfrac{ds'}{ds}-1||\theta_{\mathrm{L-L}}'+\theta_{\mathrm{NL}}'|
\]
and applying \eqref{eq:DiffModEst}, \eqref{eq:Diff-theta-LL-est},
\eqref{eq:Diff-theta-NL-est}, and \eqref{eq:Diff-dsds-Estimate}.
and $\|\mathbf{v}'\|_{L^{2}}\lesssim1$ (as $m\geq3$).

It remains to show \eqref{eq:L2Diff-Nonlin}. Rewrite the bound of
\eqref{eq:L2-nonlinear-est} as 
\[
\|\mathcal{N}(P+\epsilon)-\mathcal{N}(P)\|_{L^{2}}\lesssim\|\epsilon\|_{\dot{H}_{m}^{1}}^{\frac{1}{2}}\|\epsilon\|_{\dot{H}_{m}^{3}}^{\frac{1}{2}}+\|\epsilon\|_{\dot{H}_{m}^{1}}^{2}(1+\|\epsilon\|_{L^{2}})^{3}.
\]
Here, the first term corresponds to the estimates for a multilinear
expression of $\mathcal{N}(P+\epsilon)-\mathcal{N}(P)$ having exactly
one $\epsilon$; the second term corresponds to those having two or
more $\epsilon$'s.

We proceed similarly to earlier proofs of difference estimates (Lemmas
\ref{lem:Diff-RLL-Est} and \ref{lem:Diff-RNL-Est}). If $\delta$
hits $P$, then we lose $b^{-1}(|\delta b|+|\delta\eta|)$, yielding
the bound 
\[
b^{-1}(|\delta b|+|\delta\eta|)(\|\epsilon\|_{\dot{H}_{m}^{1}}^{\frac{1}{2}}\|\epsilon\|_{\dot{H}_{m}^{3}}^{\frac{1}{2}}+\|\epsilon\|_{\dot{H}_{m}^{1}}^{2})\lesssim_{M}s^{-2}D.
\]
If $\delta$ hits $\epsilon$, then we replace one $\epsilon$ by
$\delta\epsilon$, yielding the bound (use $\|\delta\epsilon\|_{\dot{H}_{m}^{1}}\lesssim\|\delta\epsilon\|_{L^{2}}^{\frac{2}{3}}\|\delta\epsilon\|_{\dot{H}_{m}^{3}}^{\frac{1}{3}}\lesssim_{M}s^{-\frac{5}{6}}D$)
\[
\|\delta\epsilon\|_{\dot{H}_{m}^{1}}^{\frac{1}{2}}\|\delta\epsilon\|_{\dot{H}_{m}^{3}}^{\frac{1}{2}}+\|\delta\epsilon\|_{L^{2}}\|\epsilon\|_{\dot{H}_{m}^{1}}^{2}+\|\delta\epsilon\|_{\dot{H}_{m}^{1}}\|\epsilon\|_{\dot{H}_{m}^{1}}\lesssim_{M}s^{-\frac{5}{3}}D.
\]
This shows \eqref{eq:L2Diff-Nonlin}.
\end{proof}
\begin{lem}[Energy identity for $L^{2}$]
We have 
\begin{equation}
\partial_{s}\|\delta\epsilon\|_{L^{2}}^{2}\lesssim bs^{-\frac{1}{2}}D^{2}.\label{eq:DiffEnergyDerivL2}
\end{equation}
\end{lem}

\begin{proof}
We have 
\[
\tfrac{1}{2}\partial_{s}\|\delta\epsilon\|_{L^{2}}^{2}=(\delta\epsilon,\text{RHS of }\eqref{eq:Diff-e-Eq-L2})_{r}\lesssim\|\delta\epsilon\|_{L^{2}}\|\text{RHS of }\eqref{eq:Diff-e-Eq-L2}\|_{L^{2}}.
\]
Thus it suffices to show 
\[
\|\text{RHS of }\eqref{eq:Diff-e-Eq-L2}\|_{L^{2}}\lesssim s^{-\frac{3}{2}}D.
\]

First, by \eqref{eq:DiffGammaEstimate}, we have 
\[
\|(\delta\gamma)_{s}i\epsilon'\|_{L^{2}}\lesssim s^{-\frac{3}{2}}D.
\]
By \eqref{eq:Diff-dsds-Estimate} and a priori bound of $\epsilon$',
we have 
\[
\|(\tfrac{ds'}{ds}-1)i\Delta_{m}\epsilon'\|_{L^{2}}\lesssim D\cdot\|\epsilon'\|_{\dot{H}_{m}^{1}}^{\frac{1}{2}}\|\epsilon'\|_{\dot{H}_{m}^{3}}^{\frac{1}{2}}\lesssim_{M}s^{-2}D\lesssim s^{-\frac{3}{2}}D.
\]
By \eqref{eq:non-tilde-mod-est} and Lemma \ref{lem:Diff-Est-ProfilesModRad},
\[
\|\mathbf{Mod}\cdot(\delta\mathbf{v})\|_{L^{2}}\lesssim_{M}s^{-2}\cdot|\log s|^{\frac{1}{2}}(|\delta b|+|\delta\eta|)\lesssim s^{-\frac{3}{2}}D.
\]
By Lemma \ref{lem:L2Diff-Remainders}, we have 
\[
\|(\mathbf{Mod}-\frac{ds'}{ds}\mathbf{Mod}')\cdot\mathbf{v}'\|_{L^{2}}+\|i\delta[\mathcal{N}(P+\epsilon)-\mathcal{N}(P)]\|_{L^{2}}\lesssim s^{-\frac{3}{2}}D.
\]
By Lemma \ref{lem:Diff-Est-ProfilesModRad}, 
\[
\|i\delta\Psi\|_{L^{2}}\lesssim s^{-3}D.
\]
Finally, by \eqref{eq:Diff-dsds-Estimate}, \eqref{eq:L2-nonlinear-est},
and \eqref{eq:L2-Radiation-est}, we have 
\[
\|(\tfrac{ds'}{ds}-1)(i[\mathcal{N}(P'+\epsilon')-\mathcal{N}(P')]+i\Psi')\|_{L^{2}}\lesssim_{M}s^{-2}D\lesssim s^{-\frac{3}{2}}D.
\]
This completes the proof.
\end{proof}

\subsection{Proof of Proposition \ref{prop:ForwardBackwardControl}}
\begin{proof}[Proof of Proposition \ref{prop:ForwardBackwardControl}]
We first show the forward-in-time control \eqref{eq:ForwardControlStableModes}.
It suffices to show 
\begin{align}
\sup_{s\in[s_{0},\infty)}s^{5}\|\delta\epsilon_{3}(s)\|_{L^{2}}^{2} & \lesssim s_{0}^{5}\|\delta\epsilon_{3}(s_{0})\|_{L^{2}}^{2}+s_{0}^{-\frac{1}{2}}\sup_{s\in[s_{0},\infty)}D^{2}(s),\label{eq:ForwardBackwardClaim1}\\
\sup_{s\in[s_{0},\infty)}\|\delta\epsilon(s)\|_{L^{2}}^{2} & \lesssim\|\delta\epsilon(s_{0})\|_{L^{2}}^{2}+s_{0}^{-\frac{1}{2}}\sup_{s\in[s_{0},\infty)}D^{2}(s),\label{eq:ForwardBackwardClaim2}\\
\sup_{s\in[s_{0},\infty)}s|\delta b|(s) & \lesssim s_{0}|\delta b|(s_{0})+s_{0}^{-\frac{1}{2}}\sup_{s\in[s_{0},\infty)}D(s).\label{eq:ForwardBackwardClaim3}
\end{align}

To show \eqref{eq:ForwardBackwardClaim1}, we integrate \eqref{eq:DiffModifiedEnergyDerivH3}
on $[s_{0},s]$ to obtain 
\[
\frac{\mathcal{F}_{3}[\delta\epsilon](s)}{\lambda^{6}(s)}\leq\frac{\mathcal{F}_{3}[\delta\epsilon](s_{0})}{\lambda^{6}(s_{0})}+\int_{s_{0}}^{s}\frac{b}{\lambda^{6}}\Big(\frac{1}{100}\|\delta\epsilon_{3}\|_{L^{2}}^{2}+\sigma^{-\frac{1}{2}}(\sigma^{-5}D^{2})\Big)d\sigma.
\]
Applying \eqref{eq:DiffModifiedEnergyBdryH3} and $b(\sigma)\approx\lambda(\sigma)\approx\sigma^{-1}$
of Lemma \ref{lem:RoughAsymptotics}, we have 
\[
\frac{98}{100}s^{6}\|\delta\epsilon_{3}(s)\|_{L^{2}}^{2}\leq\frac{102}{100}s_{0}^{6}\|\delta\epsilon_{3}(s_{0})\|_{L^{2}}^{2}+\frac{2}{100}s\sup_{\sigma\in[s_{0},s]}(\sigma^{5}\|\delta\epsilon_{3}(\sigma)\|_{L^{2}}^{2})+3s^{\frac{1}{2}}\sup_{\sigma\in[s_{0},s]}D^{2}(\sigma).
\]
We divide both sides by $s$ and take the supremum over $s\in[s_{0},\infty)$
to get the claim \eqref{eq:ForwardBackwardClaim1}.

To show \eqref{eq:ForwardBackwardClaim2}, we integrate \eqref{eq:DiffEnergyDerivL2}
on $[s_{0},s]$ and use $b(\sigma)\approx\sigma^{-1}$ to obtain 
\[
\|\delta\epsilon(s)\|_{L^{2}}^{2}\lesssim\|\delta\epsilon(s_{0})\|_{L^{2}}^{2}+s_{0}^{-\frac{1}{2}}\sup_{\sigma\in[s_{0},s]}D^{2}(\sigma).
\]
Taking the supremum over $s\in[s_{0},\infty)$ yields the claim \eqref{eq:ForwardBackwardClaim2}.

To show \eqref{eq:ForwardBackwardClaim3}, we observe that 
\begin{align*}
(\delta b)_{s}+b(\delta b) & =(\delta b)_{s}+b^{2}-\tfrac{ds'}{ds}(b')^{2}-b'(b-\tfrac{ds'}{ds}b')\\
 & =(\tilde{\mathrm{Mod}}-\tfrac{ds'}{ds}\tilde{\mathrm{Mod}}')_{3}-b'(\tilde{\mathrm{Mod}}-\tfrac{ds'}{ds}\tilde{\mathrm{Mod}}')_{1}-(\eta^{2}-\tfrac{ds'}{ds}(\eta')^{2}).
\end{align*}
Applying the modulation estimate (Lemma \ref{lem:DiffModEst}) and
the definition of $D$, we have 
\[
|(\delta b)_{s}+b(\delta b)|\lesssim s^{-\frac{5}{2}}D.
\]
Notice that $b\approx-\frac{\lambda_{s}}{\lambda}$. Combining this
with the modulation estimate (Lemma \ref{lem:mod-est}), we obtain
\[
\lambda\Big|\Big(\frac{\delta b}{\lambda}\Big)_{s}\Big|\leq\Big|(\delta b)_{s}+b(\delta b)\Big|+\Big|\Big(\frac{\lambda_{s}}{\lambda}+b\Big)(\delta b)\Big|\lesssim s^{-\frac{5}{2}}D.
\]
Integrating this and applying $\lambda(s)\approx s^{-1}$ yield the
claim \eqref{eq:ForwardBackwardClaim3}.

We turn to \eqref{eq:BackwardControlUnstableModes}. By the modulation
estimate (Lemma \ref{lem:DiffModEst}), we have 
\[
|(\delta\eta)_{s}|\lesssim s^{-\frac{5}{2}}D^{s}+s^{-\frac{1}{2}}(s^{-\frac{5}{2}}D).
\]
Integrating this backwards in time yields \eqref{eq:BackwardControlUnstableModes}.
\end{proof}

\appendix

\section{\label{sec:HardyInequalities}Hardy inequalities and adapted function
spaces}

In this section, we assume $m\ge1$ as before. Here we collect results
related to Hardy's inequality for equivariant functions. We also present
proofs of the facts regarding to the adapted function spaces introduced
in Section \ref{subsec:Adapted-function-spaces}. Most importantly,
we prove (sub-)coercivity estimates of Lemmas \ref{lem:Mapping-L-Section2}-\ref{lem:Coercivity-AAL-Section2}.

\subsection{Weighted Hardy's inequality}

In Section \ref{subsec:Adapted-function-spaces}, we motivated the
adapted function space $\dot{\mathcal{H}}_{m}^{3}$, in spirit that
how much weighted Hardy's inequality we can obtain from the operators
$L_{Q}$, $A_{Q}$, and $A_{Q}^{\ast}$. We formulate and prove the
weighted Hardy's inequality, which is the main tool in Section \ref{subsec:Adapted-function-spaces}.
\begin{lem}[Weighted Hardy's inequality for $\partial_{r}$]
\label{lem:Monotonicity-dr}Let $0<r_{1}<r_{2}<\infty$; let $\varphi:[r_{1},r_{2}]\to\R_{+}$
be a $C^{1}$ weight function such that $\partial_{r}\varphi$ is
nonvanishing and $\varphi\lesssim|r\partial_{r}\varphi|$. Then, for
smooth $f:[r_{1},r_{2}]\to\C$, we have 
\[
\int_{r_{1}}^{r_{2}}\Big|\frac{f}{r}\Big|^{2}|r\partial_{r}\varphi|rdr\lesssim\int_{r_{1}}^{r_{2}}|\partial_{r}f|^{2}\varphi\,rdr+\begin{cases}
\varphi(r_{2})|f(r_{2})|^{2} & \text{if }\partial_{r}\varphi>0,\\
\varphi(r_{1})|f(r_{1})|^{2} & \text{if }\partial_{r}\varphi<0.
\end{cases}
\]
\end{lem}

\begin{proof}
We only consider the case with $\partial_{r}\varphi>0$ and leave
the case $\partial_{r}\varphi<0$ to the readers. We compute 
\[
\partial_{r}(\varphi|f|^{2})=(\partial_{r}\varphi)|f|^{2}+2\varphi\Re(\overline{f}\partial_{r}f).
\]
Integrating on the interval $[r_{1},r_{2}]$ and using the fundamental
theorem of calculus, we have 
\[
\varphi(r_{2})|f(r_{2})|^{2}\geq\int_{r_{1}}^{r_{2}}(r\partial_{r}\varphi)\Big|\frac{f}{r}\Big|^{2}rdr+O\Big(\int_{r_{1}}^{r_{2}}\varphi\Big|\frac{f}{r}\Big||\partial_{r}f|rdr\Big).
\]
Applying the Cauchy-Schwarz to the last term and using $\varphi\lesssim|r\partial_{r}\varphi|$,
we obtain 
\[
\int_{r_{1}}^{r_{2}}(r\partial_{r}\varphi)\Big|\frac{f}{r}\Big|^{2}rdr\lesssim\varphi(r_{2})|f(r_{2})|^{2}+\int_{r_{1}}^{r_{2}}|\partial_{r}f|^{2}\varphi\,rdr,
\]
which is the desired estimate.
\end{proof}
\begin{rem}[Monotonicity of $\partial_{r}$ in weighted space]
Lemma \ref{lem:Monotonicity-dr} can be understood as the monotonicity
of the operator $\partial_{r}$ in various weighted spaces. The LHS,
$r^{-1}f$, is the lower bound obtained from the monotonicity. The
first term of the RHS, $\partial_{r}f$, is the inhomogeneous term
of the equation $\partial_{r}f=F$. The last term $\varphi(r)|f(r)|^{2}$
at either $r=r_{1}$ or $r=r_{2}$ are the initial or final data.
Roughly speaking, the boundary value controls $r^{-1}f$ (with an
error of inhomogeneous term $\partial_{r}f$) \emph{to the left} when
$\varphi$ is increasing, and \emph{to the right} when $\varphi$
is decreasing.
\end{rem}

We then apply the techniques of conjugation by $r^{\ell}$ weights
and using logarithmically decreasing weight functions to Lemma \ref{lem:Monotonicity-dr}
to get the following estimates.
\begin{cor}[Weighted Hardy's inequality for $\partial_{r}-\frac{\ell}{r}$]
\label{cor:WeightedHardy}Let $\ell,k\in\R$; let $0<r_{1}<r_{2}<\infty$.
Let $f:[r_{1},r_{2}]\to\C$ be a smooth function.
\begin{itemize}
\item (Noncritical case) If $\ell\neq k$, then 
\[
\int_{r_{1}}^{r_{2}}\Big|\frac{f}{r^{k+1}}\Big|^{2}rdr\lesssim_{\ell-k}\int_{r_{1}}^{r_{2}}\Big|\frac{(\partial_{r}-\frac{\ell}{r})f}{r^{k}}\Big|^{2}rdr+\begin{cases}
|(r_{2})^{-k}f(r_{2})|^{2} & \text{if }\ell>k,\\
|(r_{1})^{-k}f(r_{1})|^{2} & \text{if }\ell<k.
\end{cases}
\]
\item (Critical case, or logarithmic Hardy's inequality) If $\ell=k$, then
\[
\int_{r_{1}}^{r_{2}}\Big|\frac{f}{r^{k+1}\langle\log r\rangle}\Big|^{2}rdr\lesssim\int_{r_{1}}^{r_{2}}\Big|\frac{(\partial_{r}-\frac{k}{r})f}{r^{k}}\Big|^{2}rdr+\begin{cases}
|f(1)|^{2} & \text{if }1\in[r_{1},r_{2}],\\
|(r_{2})^{-k}f(r_{2})|^{2} & \text{if }r_{2}\leq1,\\
|(r_{1})^{-k}f(r_{1})|^{2} & \text{if }r_{1}\geq1.
\end{cases}
\]
\end{itemize}
\end{cor}

\begin{rem}
In Corollary \ref{cor:WeightedHardy}, if $\ell>k$ and $f$ is assumed
to decay rapidly (this is in general easy to assume in practice, thanks
to the density argument), then we may ignore the boundary term. Indeed,
we send $r_{2}\to\infty$ and $r_{1}\to0$ to obtain 
\[
\int_{0}^{\infty}\Big|\frac{f}{r^{k+1}}\Big|^{2}rdr\lesssim_{\ell-k}\int_{0}^{\infty}\Big|\frac{1}{r^{k}}\Big(\partial_{r}-\frac{\ell}{r}\Big)f\Big|^{2}rdr.
\]
\end{rem}

\begin{proof}[Proof of Corollary \ref{cor:WeightedHardy}]
For the noncritical case, we simply apply the method of conjugation
as 
\[
\partial_{r}-\tfrac{\ell}{r}=r^{\ell}\partial_{r}r^{-\ell}.
\]
The noncritical case follows from an application of Lemma \ref{lem:Monotonicity-dr}
with the weight 
\[
\varphi=r^{2\ell-2k}
\]
and substituting $r^{-\ell}f$ in place of $f$.

For the critical case, by separating the integral $\int_{r_{1}}^{r_{2}}=\int_{r_{1}}^{1}+\int_{1}^{r_{2}}$
if $1\in[r_{1},r_{2}]$, we may assume either $r_{2}\leq1$ or $r_{1}\geq1$.
Let us only consider the case $r_{2}\leq1$ and leave the case $r_{1}\geq1$
for the readers. Again, by the method of conjugation $\partial_{r}-\frac{k}{r}=r^{k}\partial_{r}r^{-k}$,
it suffices to show 
\[
\int_{r_{1}}^{r_{2}}\Big|\frac{f}{r\langle\log r\rangle}\Big|^{2}rdr\lesssim\int_{r_{1}}^{r_{2}}|\partial_{r}f|^{2}rdr+|f(r_{2})|^{2}.
\]
Define the weight $\varphi:(0,1]\to\R_{+}$ solving the differential
equation 
\[
r\partial_{r}\varphi=\langle\log r\rangle^{-2}\quad\text{with}\quad\lim_{r\to0^{+}}\varphi(r)=0.
\]
Note that $\varphi(r)\sim\langle\log r\rangle^{-1}$, so $\varphi\lesssim r\partial_{r}\varphi$
\emph{does not hold}. However, one can follow the proof of Lemma \ref{lem:Monotonicity-dr}
and uses 
\[
\int_{r_{1}}^{r_{2}}\varphi\Big|\frac{f}{r}\Big||\partial_{r}f|rdr\lesssim\Big(\int_{r_{1}}^{r_{2}}|\partial_{r}f|^{2}rdr\Big)^{\frac{1}{2}}\Big(\int_{r_{1}}^{r_{2}}\Big|\frac{f}{\langle\log r\rangle}\Big|^{2}rdr\Big)^{\frac{1}{2}}
\]
instead to obtain the desired estimate.
\end{proof}

\subsection{Equivariant Sobolev spaces}

Here we recall some basic notations and facts mentioned in Section
\ref{subsec:Adapted-function-spaces}. For an $m$-equivariant function
$f$, its \emph{radial part} $g:\R_{+}\to\C$ is defined to satisfy
$f(x)=g(r)e^{im\theta}$, under the usual polar coordinates relation
$x_{1}+ix_{2}=re^{i\theta}$. We use an abuse of notation that $g$
is often considered as an $m$-equivariant function. For example,
we say that $g$ belongs to some $m$-equivariant function space if
its $m$-equivariant extension belongs to that. Associated function
norms are also used.

For $s\geq0$, denote by $H_{m}^{s}$ the set of $m$-equivariant
$H^{s}(\R^{2})$ functions. The set of $m$-equivariant Schwartz functions,
denoted by $\mathcal{S}_{m}$, is dense in $H_{m}^{s}$. The $H_{m}^{s}$-norm
and $\dot{H}_{m}^{s}$-norm mean the usual $H^{s}(\R^{2})$-norm and
$\dot{H}^{s}(\R^{2})$-norm, but we use the subscript $m$ to indicate
the equivariance index. When $0\leq k\leq m$, we have \emph{generalized
Hardy's inequality} \cite[Lemma A.7]{KimKwon2019arXiv}:
\begin{equation}
\||f|_{-k}\|_{L^{2}}\sim_{k,m}\|f\|_{\dot{H}_{m}^{k}},\qquad\forall f\in\mathcal{S}_{m}.\label{eq:GenHardyAppendix}
\end{equation}
As a special case, when $m\geq1$ and $k=1$, we have the \emph{Hardy-Sobolev
inequality} \cite[Lemma A.6]{KimKwon2019arXiv}:
\begin{equation}
\|r^{-1}f\|_{L^{2}}+\|f\|_{L^{\infty}}\lesssim\|f\|_{\dot{H}_{m}^{1}}.\label{eq:HardySobolevAppendix}
\end{equation}
The generalized Hardy's inequality \eqref{eq:GenHardyAppendix} allows
us define the space $\dot{H}_{m}^{k}$ when $0\leq k\leq m$ by taking
the completion of $\mathcal{S}_{m}$ under the $\dot{H}_{m}^{k}$-norm,
with the embedding properties 
\[
\mathcal{S}_{m}\hookrightarrow H_{m}^{k}\hookrightarrow\dot{H}_{m}^{k}\hookrightarrow L_{\mathrm{loc}}^{2}.
\]

\subsection{Adapted function space $\dot{H}_{m}^{1}$}
\begin{lem}[Boundedness and subcoercivity for $L_{Q}$ on $\dot{H}_{m}^{1}$;
Lemma \ref{lem:Mapping-L-Section2}]
\label{lem:mapping-L-Appendix}For $v\in\dot{H}_{m}^{1}$, we have
\begin{align*}
\|L_{Q}v\|_{L^{2}}+\|Qv\|_{L^{2}} & \sim\|v\|_{\dot{H}_{m}^{1}}.
\end{align*}
Moreover, the kernel of $L_{Q}:\dot{H}_{m}^{1}\to L^{2}$ is $\mathrm{span}_{\R}\{\Lambda Q,iQ\}$.
\end{lem}

\begin{proof}
By density, we may assume $v\in\mathcal{S}_{m}$. We first express
$L_{Q}=\D_{+}^{(Q)}+QB_{Q}$ and treat $QB_{Q}$-term as an error:
\[
\|QB_{Q}v\|_{L^{2}}\lesssim\|rQ\|_{L^{\infty}}\|r^{-1}B_{Q}v\|_{L^{2}}\lesssim\|Qv\|_{L^{2}}.
\]
Thus we have 
\begin{align*}
 & \|L_{Q}v\|_{L^{2}}\lesssim\|\D_{+}^{(Q)}v\|_{L^{2}}+\|Qv\|_{L^{2}}\lesssim\|\D_{+}^{(Q)}v\|_{L^{2}}+\|v\|_{\dot{H}_{m}^{1}},\\
 & \|L_{Q}v\|_{L^{2}}+\|Qv\|_{L^{2}}\gtrsim\|\D_{+}^{(Q)}v\|_{L^{2}}.
\end{align*}

We now focus on $\D_{+}^{(Q)}v$. Since $\D_{+}^{(Q)}\approx\partial_{r}-\frac{m}{r}$
for small $r$ and $\D_{+}^{(Q)}\approx\partial_{r}+\frac{m+2}{r}$
for large $r$, the Hardy lower bound $|v|_{-1}$ is expected in view
of Corollary \ref{cor:WeightedHardy} as explained in Section \ref{subsec:Adapted-function-spaces}.
But there is a simple but different route to achieve this; we merely
integrate by parts as 
\begin{align*}
\|\D_{+}^{(Q)}v\|_{L^{2}}^{2} & =\|\partial_{r}v\|_{L^{2}}^{2}-2(\partial_{r}v,\tfrac{1}{r}(m+A_{\theta}[Q])v)_{r}+\|\tfrac{1}{r}(m+A_{\theta}[Q])v\|_{L^{2}}^{2}\\
 & =\|\partial_{r}v\|_{L^{2}}^{2}+{\textstyle \int}((m+A_{\theta}[Q])^{2}-\tfrac{1}{2}r^{2}Q^{2})|\tfrac{1}{r}v|^{2}.
\end{align*}
Since both of the limits $m+A_{\theta}[Q]$ as $r\to0$ and $r\to\infty$
do not vanish, we get 
\begin{align*}
 & \|\D_{+}^{(Q)}v\|_{L^{2}}\lesssim\||v|_{-1}\|_{L^{2}}\lesssim\|v\|_{\dot{H}_{m}^{1}},\\
 & \|\D_{+}^{(Q)}v\|_{L^{2}}+\|r^{2}\langle r\rangle^{-2}Qv\|_{L^{2}}\gtrsim\|v\|_{\dot{H}_{m}^{1}}.
\end{align*}

It now suffices to characterize the kernel of $L_{Q}:\dot{H}_{m}^{1}\to L^{2}$.
Let $f\in\dot{H}_{m}^{1}$ be such that $L_{Q}f=0$. It is shown in
\cite[Lemma 3.3]{KimKwon2019arXiv} that $f\in\mathrm{span}_{\R}\{\Lambda Q,iQ\}$
when $f$ is smooth (on $\R^{2}$). To use this result, we want to
apply the elliptic regularity on $\R^{2}$, so that $f\in\dot{H}_{m}^{1}$
can be upgraded to a smooth function. We consider the cartesian representation
of $L_{Q}$; we obtain $L_{Q}f$ by taking the linear part of $\tilde{\D}_{+}^{(Q+f)}(Q+f)$
and recalling $\tilde{\D}_{+}=\D_{1}+i\D_{2}$ as 
\begin{align*}
L_{Q}f & =\tilde{\D}_{+}^{(Q)}f+2i(A_{1}[Q,f]+iA_{2}[Q,f])Q,\\
A_{j}[f,g] & \coloneqq\tfrac{1}{2}\epsilon_{jk}\Delta^{-1}\partial_{k}\Re(\overline{f}g),
\end{align*}
where the anti-symmetric tensor $\epsilon_{jk}$ is defined by $\epsilon_{12}=1$.
Thus $L_{Q}f=0$ means that 
\[
(\partial_{1}+i\partial_{2})f=\tfrac{A_{\theta}[Q](x)}{|x|^{2}}(x_{1}+ix_{2})f-2i(A_{1}[Q,f]+iA_{2}[Q,f])Q.
\]
Note that the radial part of the above display corresponds to $(\partial_{r}-\frac{m}{r})f=\frac{A_{\theta}[Q]}{r}f-QB_{Q}f$.
Starting from $f\in\dot{H}_{m}^{1}$ (thus a priori having $\nabla f\in L^{2}$,
$|x|^{-1}f\in L^{2}$, and $f\in L_{\mathrm{loc}}^{2}$), iterating
the above display concludes that $f\in H_{\mathrm{loc}}^{\infty}$.
Therefore, $f$ is indeed a smooth solution to $L_{Q}f=0$. By \cite[Lemma 3.3]{KimKwon2019arXiv},
we finish the proof.
\end{proof}
\begin{lem}[Coercivity for $L_{Q}$ on $\dot{H}_{m}^{1}$; Lemma \ref{lem:Coercivity-L-Section2}]
\label{lem:Coercivity-L-Appendix}Let $\psi_{1}$ and $\psi_{2}$
be elements of $(\dot{H}_{m}^{1})^{\ast}$, which is the dual space
of $\dot{H}_{m}^{1}$. If the $2\times2$ matrix $(a_{ij})$ defined
by $a_{i1}=(\psi_{i},\Lambda Q)_{r}$ and $a_{i2}=(\psi_{i},iQ)_{r}$
has nonzero determinant, then we have a coercivity estimate 
\[
\|v\|_{\dot{H}_{m}^{1}}\gtrsim\|L_{Q}v\|_{L^{2}}\gtrsim_{\psi_{1},\psi_{2}}\|v\|_{\dot{H}_{m}^{1}},\qquad\forall v\in\dot{H}_{m}^{1}\cap\{\psi_{1},\psi_{2}\}^{\perp}.
\]
\end{lem}

\begin{proof}
This is indeed shown in \cite[Lemma 3.9]{KimKwon2019arXiv}, but let
us give a sketch of the proof for convenience.

Suppose not. We can choose a sequence $\{v_{n}\}_{n\in\N}\subseteq\dot{H}_{m}^{1}$
such that $\|L_{Q}v_{n}\|_{L^{2}}=\frac{1}{n}$, $\|v_{n}\|_{\dot{H}_{m}^{1}}=1$,
and $(\psi_{1},v_{n})_{r}=(\psi_{2},v_{n})_{r}=0$. In particular,
$\{v_{n}\}_{n\in\N}$ is bounded in $\dot{H}_{m}^{1}$. After passing
to a subsequence, there exists $v\in\dot{H}_{m}^{1}$ such that $v_{n}$
converges to $v_{\infty}$ weakly in $\dot{H}_{m}^{1}$ and strongly
in $L_{\mathrm{loc}}^{2}$. By weak convergence, we have $L_{Q}v_{\infty}=0$
and $(\psi_{1},v_{\infty})_{r}=(\psi_{2},v_{\infty})_{r}=0$.

On one hand, $v_{\infty}=0$ by the kernel characterization of Lemma
\ref{lem:mapping-L-Appendix} and the orthogonality conditions. On
the other hand, the subcoercivity estimate says that $\|Qv_{n}\|_{L^{2}}\gtrsim1$
uniformly for all large $n$, and the strong $L_{\mathrm{loc}}^{2}$-convergence
(and $\dot{H}_{m}^{1}$-boundedness) says that $\|Qv_{n}\|_{L^{2}}\to\|Qv_{\infty}\|_{L^{2}}\gtrsim1$.
This yields $v_{\infty}\neq0$, a contradiction.
\end{proof}

\subsection{Adapted function space $\dot{\mathcal{H}}_{m}^{3}$}

Recall from Section \ref{subsec:Adapted-function-spaces} that the
$\dot{\mathcal{H}}_{m}^{3}$-norm is defined by 
\begin{equation}
\|f\|_{\dot{\mathcal{H}}_{m}^{3}}\coloneqq\|\partial_{+}f\|_{\dot{H}_{m+1}^{2}}+\begin{cases}
\||f|_{-3}\|_{L^{2}} & \text{if }m\geq3,\\
\|\partial_{rrr}f\|_{L^{2}}+\|r^{-1}\langle\log_{-}r\rangle^{-1}|f|_{-2}\|_{L^{2}} & \text{if }m=2,\\
\||\partial_{rr}f|_{-1}\|_{L^{2}}+\|r^{-1}\langle r\rangle^{-1}\langle\log_{-}r\rangle^{-1}|f|_{-1}\|_{L^{2}} & \text{if }m=1,
\end{cases}\label{eq:Def-Hdot3-Appendix}
\end{equation}
initially for $m$-equivariant Schwartz function $f$. The space $\dot{\mathcal{H}}_{m}^{3}$
is obtained by taking the completion of $\mathcal{S}_{m}$ under the
$\dot{\mathcal{H}}_{m}^{3}$ norm. As all $L^{2}$-norms are involved,
one can equip $\dot{\mathcal{H}}_{m}^{3}$ with a Hilbert space structure
(by modifying the $\dot{\mathcal{H}}_{m}^{3}$-norm by some comparable
one). The choice of $\dot{\mathcal{H}}_{m}^{3}$ is motivated to have
boundedness/subcoercivity estimates of the operator $A_{Q}^{\ast}A_{Q}L_{Q}$.

\subsubsection*{Comparison of $\dot{H}_{m}^{3}$ and $\dot{\mathcal{H}}_{m}^{3}$}

$ $\\ One may naturally ask how much the $\dot{H}_{m}^{3}$-norm
and $\dot{\mathcal{H}}_{m}^{3}$-norm differ. It turns out that they
are equivalent when $m\geq3$, but the $\dot{\mathcal{H}}_{m}^{3}$-norm
is\emph{ stronger }than the $\dot{H}_{m}^{3}$-norm when $m\in\{1,2\}$.
\begin{lem}[Comparison of $\dot{H}_{m}^{3}$ and $\dot{\mathcal{H}}_{m}^{3}$]
\label{lem:Comparison-H3-Hdot3}For $f\in\mathcal{S}_{m}$, we have
\[
\|f\|_{\dot{H}_{m}^{3}}\sim\|\partial_{+}f\|_{\dot{H}_{m+1}^{2}}
\]
and 
\[
\|f\|_{\dot{\mathcal{H}}_{m}^{3}}\sim\begin{cases}
\|f\|_{\dot{H}_{m}^{3}} & \textrm{if }m\geq3,\\
\|f\|_{\dot{H}_{m}^{3}}+\|\mathbf{1}_{r\geq1}\tfrac{1}{r^{3}}f\|_{L^{2}} & \textrm{if }m=2,\\
\|f\|_{\dot{H}_{m}^{3}}+\|\mathbf{1}_{r\sim1}f\|_{L^{2}} & \textrm{if }m=1.
\end{cases}
\]
Due to $\|\mathbf{1}_{r\gtrsim1}\frac{1}{r^{3}}f\|_{L^{2}}\lesssim\|f\|_{L^{2}}$,
we have 
\[
L^{2}\cap\dot{\mathcal{H}}_{m}^{3}=H_{m}^{3}.
\]
\end{lem}

\begin{rem}
\label{rem:OptimalityComparisonH3H3}One cannot replace $\|\mathbf{1}_{r\geq1}\tfrac{1}{r^{3}}f\|_{L^{2}}$
by $\|\mathbf{1}_{r\sim1}f\|_{L^{2}}$ when $m=2$; and one cannot
eliminate $\|\mathbf{1}_{r\sim1}f\|_{L^{2}}$ when $m=1$. When $m=2$,
the estimate $\|f\|_{\dot{H}_{2}^{3}}\gtrsim\|\mathbf{1}_{r\geq r_{0}}r^{-3}f\|_{L^{2}}$
is \emph{false} for any $r_{0}>0$, which can be seen by considering
$f(x)=(x_{1}+ix_{2})^{2}\chi_{\leq R}(|x|)$ and take $R\to\infty$.
When $m=1$, the estimate $\|f\|_{\dot{H}_{1}^{3}}\gtrsim\|\mathbf{1}_{r\sim1}f\|_{L^{2}}$
is \emph{false}, which can be seen by the example $f(x)=R(x_{1}+ix_{2})\chi_{\leq R}(|x|)$.
\end{rem}

\begin{proof}
We first show for $f\in\mathcal{S}_{m}$: 
\[
\|f\|_{\dot{H}_{m}^{3}}\sim\|\partial_{+}f\|_{\dot{H}_{m+1}^{2}}\lesssim\|f\|_{\dot{\mathcal{H}}_{m}^{3}}.
\]
This follows from 
\[
\|f\|_{\dot{H}_{m}^{3}}\sim\|\nabla f\|_{\dot{H}^{2}}\sim\|\partial_{+}f\|_{\dot{H}_{m+1}^{2}}+\|\partial_{-}f\|_{\dot{H}_{m-1}^{2}}\sim\|\partial_{+}f\|_{\dot{H}_{m+1}^{2}}+\|\Delta\partial_{-}f\|_{L^{2}}.
\]
and the observation $\Delta\partial_{-}=\partial_{-}\Delta=\partial_{-}\partial_{-}\partial_{+}$.

We now consider the reverse inequality for $f\in\mathcal{S}_{m}$.
When $m\geq3$, we have $\|f\|_{\dot{H}_{m}^{3}}\sim\||f|_{-3}\|_{L^{2}}$
by \eqref{eq:GenHardyAppendix}. Thus $\|f\|_{\dot{H}_{m}^{3}}\gtrsim\|f\|_{\dot{\mathcal{H}}_{m}^{3}}$
easily follows.

When $m=2$, we apply \eqref{eq:GenHardyAppendix} to $\|\partial_{+}f\|_{\dot{H}_{3}^{2}}$
and use the pointwise estimate $|\partial_{+}f|_{-2}+|\frac{1}{r^{3}}f|\gtrsim|f|_{-3}$
to get $\|\partial_{+}f\|_{\dot{H}_{3}^{2}}+\|\mathbf{1}_{r\geq1}\tfrac{1}{r^{3}}f\|_{L^{2}}\gtrsim\|\mathbf{1}_{r\geq1}|f|_{-3}\|_{L^{2}}$.
This treats the $r\geq1$ part. Thus it suffices to show the $r\leq1$
contribution: 
\begin{align*}
\|r^{-1}|\partial_{+}f|_{-1}\|_{L^{2}}+\|\mathbf{1}_{r\sim1}f\|_{L^{2}} & \gtrsim\|\mathbf{1}_{r\leq1}r^{-1}\langle\log_{-}r\rangle^{-1}|f|_{-2}\|_{L^{2}},\\
\|\partial_{+}f\|_{\dot{H}_{3}^{2}} & \gtrsim\|\mathbf{1}_{r\leq1}\partial_{rrr}f\|_{L^{2}}.
\end{align*}
To show the first assertion, an application of Corollary \ref{cor:WeightedHardy}
with $\ell=k=2$ and averaging the boundary term yield 
\[
\|r^{-2}\partial_{+}f\|_{L^{2}}+\|\mathbf{1}_{r\sim1}f\|_{L^{2}}\gtrsim\|\mathbf{1}_{r\leq1}r^{-3}\langle\log_{-}r\rangle^{-1}f\|_{L^{2}}.
\]
Combining this with the pointwise estimates 
\begin{align*}
\mathbf{1}_{r\leq1}r^{-2}\langle\log_{-}r\rangle^{-1}|\partial_{r}f| & \leq\mathbf{1}_{r\leq1}r^{-2}\langle\log_{-}r\rangle^{-1}|\partial_{+}f|+O(\mathbf{1}_{r\leq1}r^{-3}\langle\log_{-}r\rangle^{-1}|f|),\\
\mathbf{1}_{r\leq1}r^{-1}\langle\log_{-}r\rangle^{-1}|\partial_{rr}f| & \leq\mathbf{1}_{r\leq1}r^{-1}\langle\log_{-}r\rangle^{-1}|\partial_{+}f|_{-1}+O(\mathbf{1}_{r\leq1}r^{-2}\langle\log_{-}r\rangle^{-1}|f|_{-1}),
\end{align*}
we get 
\[
\|r^{-1}|\partial_{+}f|_{-1}\|_{L^{2}}+\|\mathbf{1}_{r\sim1}f\|_{L^{2}}\gtrsim\|\mathbf{1}_{r\leq1}r^{-1}\langle\log_{-}r\rangle^{-1}|f|_{-2}\|_{L^{2}},
\]
showing the first assertion. The second assertion follows by observing
the special algebra 
\[
\partial_{rrr}=(\partial_{r}+\tfrac{1}{r})^{2}(\partial_{r}-\tfrac{2}{r})
\]
and noticing that $\partial_{r}-\tfrac{2}{r}$ is the radial part
of $\partial_{+}$ acting on $2$-equivariant functions.

When $m=1$, it suffices to show the controls 
\begin{align*}
\|\mathbf{1}_{r\leq2}r^{-1}\partial_{+}f\|_{L^{2}}+\|\mathbf{1}_{r\sim1}f\|_{L^{2}} & \gtrsim\|\mathbf{1}_{r\leq1}r^{-1}\langle\log_{-}r\rangle^{-1}|f|_{-1}\|_{L^{2}},\\
\|\mathbf{1}_{r\geq\frac{1}{2}}r^{-2}\partial_{+}f\|_{L^{2}}+\|\mathbf{1}_{r\sim1}f\|_{L^{2}} & \gtrsim\|\mathbf{1}_{r\geq1}r^{-2}|f|_{-1}\|_{L^{2}},\\
\|\partial_{+}f\|_{\dot{H}_{2}^{2}} & \gtrsim\||\partial_{rr}f|_{-1}\|_{L^{2}}.
\end{align*}
To show the first assertion, we apply Corollary \ref{cor:WeightedHardy}
with $\ell=k=1$ to get 
\[
\|\mathbf{1}_{r\leq2}r^{-1}\partial_{+}f\|_{L^{2}}+\|\mathbf{1}_{r\sim1}f\|_{L^{2}}\gtrsim\|\mathbf{1}_{r\leq1}r^{-2}\langle\log_{-}r\rangle^{-1}f\|_{L^{2}}.
\]
Combining this with the pointwise estimate 
\[
\mathbf{1}_{r\leq1}r^{-1}\langle\log_{-}r\rangle^{-1}|\partial_{r}f|\leq\mathbf{1}_{r\leq1}r^{-1}\langle\log_{-}r\rangle^{-1}|\partial_{+}f|+\mathbf{1}_{r\leq1}r^{-2}\langle\log_{-}r\rangle^{-1}|f|
\]
yields 
\[
\|\mathbf{1}_{r\leq2}r^{-1}\partial_{+}f\|_{L^{2}}+\|\mathbf{1}_{r\sim1}f\|_{L^{2}}\gtrsim\|\mathbf{1}_{r\leq1}r^{-1}\langle\log_{-}r\rangle^{-1}|f|_{-1}\|_{L^{2}},
\]
showing the first assertion. To show the second assertion, we apply
Corollary \ref{cor:WeightedHardy} with $\ell=1$ and $k=2$. Here
notice that we can control the region $r\geq1$ by the values of $f$
on $r\sim1$. We left this case to the interested readers. Finally,
the third assertion follows from observing the special algebra 
\[
\partial_{rr}=(\partial_{r}+\tfrac{1}{r})(\partial_{r}-\tfrac{1}{r})
\]
and noticing that $\partial_{r}-\tfrac{1}{r}$ is the radial part
of $\partial_{+}$ acting on $1$-equivariant functions.
\end{proof}

\subsubsection*{Subcoercivity estimates for $A_{Q}^{\ast}A_{Q}L_{Q}$}

$ $\\ From now on, we turn to prove (sub-)coercivity estimates of
$A_{Q}^{\ast}A_{Q}L_{Q}$ in $\dot{\mathcal{H}}_{m}^{3}$: Lemmas
\ref{lem:Mapping-AAL-Section2} and \ref{lem:Coercivity-AAL-Section2}.
We start by proving subcoercivity estimates of each differential operator
$A_{Q}^{\ast}$, $A_{Q}$, and $L_{Q}$.
\begin{lem}[Boundedness and positivity for $A_{Q}^{\ast}$]
\label{lem:mapping-AastQ}For $v\in\dot{H}_{m+2}^{1}$, we have 
\[
\|A_{Q}^{\ast}v\|_{L^{2}}\sim\|v\|_{\dot{H}_{m+2}^{1}}.
\]
\end{lem}

\begin{proof}
By density, we may assume $v\in\mathcal{S}_{m+2}$. Thus the integration
by parts $\|A_{Q}^{\ast}v\|_{L^{2}}^{2}=(v,A_{Q}A_{Q}^{\ast}v)_{r}$
is justified. As $A_{Q}A_{Q}^{\ast}=-\partial_{rr}-\frac{1}{r}\partial_{r}+\frac{\tilde V}{r^{2}}$
with $\tilde V\sim1$ (this property holds \emph{only }when $m\geq1$),
we conclude 
\[
\|A_{Q}^{\ast}v\|_{L^{2}}^{2}=(v,A_{Q}A_{Q}^{\ast}v)_{r}\sim\||v|_{-1}\|_{L^{2}}^{2}.
\]
This completes the proof.
\end{proof}
\begin{lem}[Boundedness and subcoercivity for $A_{Q}$]
\label{lem:mapping-AQ}For $v\in\dot{H}_{m+1}^{2}$, we have 
\begin{align*}
 & \|A_{Q}v\|_{\dot{H}_{m+2}^{1}}+\|r\langle r\rangle^{-2}Qv\|_{L^{2}}\sim\|v\|_{\dot{H}_{m+1}^{2}}.
\end{align*}
Moreover, the kernel of $A_{Q}:\dot{H}_{m+1}^{2}\to\dot{H}_{m+2}^{1}$
is $\mathrm{span}_{\C}\{rQ\}$.
\end{lem}

\begin{proof}
By density, we may assume $v\in\mathcal{S}_{m+1}$. Note that $r^{-1}A_{Q}v=\D_{+}^{(Q)}(r^{-1}v)$.
Thus we can proceed as in the proof of Lemma \ref{lem:mapping-L-Appendix}
to have 
\begin{align*}
 & \|r^{-1}A_{Q}v\|_{L^{2}}\lesssim\||r^{-1}v|_{-1}\|_{L^{2}},\\
 & \|r^{-1}A_{Q}v\|_{L^{2}}+\|r\langle r\rangle^{-2}Qv\|_{L^{2}}\gtrsim\||r^{-1}v|_{-1}\|_{L^{2}}.
\end{align*}
From 
\[
\partial_{r}A_{Q}v=\partial_{rr}v-\tfrac{1}{r}(m+A_{\theta}[Q])\partial_{r}v+\tfrac{1}{2}Q^{2}v=\partial_{rr}v+O(r^{-1}|v|_{-1}),
\]
we have 
\begin{align*}
 & \|\partial_{r}A_{Q}v\|_{L^{2}}\lesssim\||v|_{-2}\|_{L^{2}},\\
 & \|\partial_{r}A_{Q}v\|_{L^{2}}+\||r^{-1}v|_{-1}\|_{L^{2}}\gtrsim\|\partial_{rr}v\|_{L^{2}}.
\end{align*}
Summing up the above results, we have 
\begin{align*}
 & \|A_{Q}v\|_{\dot{H}_{m+2}^{1}}\lesssim\|v\|_{\dot{H}_{m+1}^{2}},\\
 & \|A_{Q}v\|_{\dot{H}_{m+2}^{1}}+\|r\langle r\rangle^{-2}Qv\|_{L^{2}}\gtrsim\|v\|_{\dot{H}_{m+1}^{2}}.
\end{align*}

Finally, we characterize the kernel of $A_{Q}:\dot{H}_{m+1}^{2}\to\dot{H}_{m+2}^{1}$.
If $A_{Q}f=0$ for some $f\in\dot{H}_{m+1}^{2}$, then $(\partial_{r}-\frac{m+1+A_{\theta}[Q]}{r})f(r)=0$
on $(0,\infty)$. Since $A_{Q}$ is a first-order differential operator
and $A_{Q}(rQ)=0$, $f\in\mathrm{span}_{\C}\{rQ\}$ follows by the
ODE uniqueness result.
\end{proof}
It remains to obtain the boundedness/subcoercivity properties of $L_{Q}$
at the $\dot{\mathcal{H}}_{m}^{3}$-level. We first show that we can
always ignore $QB_{Q}$ part of $L_{Q}$.
\begin{lem}[Contribution of $QB_{Q}$]
\label{lem:Contribution-QBQ}Let $v\in\dot{\mathcal{H}}_{m}^{3}$.
We have 
\[
\|QB_{Q}v\|_{\dot{H}_{m+1}^{2}}\lesssim\|r\langle r\rangle^{-2}Q|v|_{-1}\|_{L^{2}}.
\]
\end{lem}

\begin{rem}
The proof only uses the property $|Q|_{2}\lesssim Q$. Thus the above
boundedness property holds if we replace some $Q$ in $QB_{Q}$ by
$\Lambda Q$, as $|\Lambda Q|_{2}\lesssim Q$ also holds. This fact
is used in the proof of Lemma \ref{lem:commutator}.
\end{rem}

\begin{proof}
Note the pointwise inequality 
\begin{align*}
|QB_{Q}v|_{-2} & \lesssim r^{-2}Q|B_{Q}v|+|Q^{2}v|_{-1}\\
 & \lesssim r^{-1}\langle r\rangle^{2}Q\cdot r^{-1}B_{Q}|\langle r\rangle^{-2}v|+Q^{2}|v|_{-1}.
\end{align*}
We estimate its $L^{2}$-norm by 
\[
\||QB_{Q}v|_{-2}\|_{L^{2}}\lesssim\|r^{-1}\langle r\rangle^{2}Q\|_{L^{\infty}}\|r\langle r\rangle^{-2}Q|v|_{-1}\|_{L^{2}}\lesssim\|r\langle r\rangle^{-2}Q|v|_{-1}\|_{L^{2}},
\]
which is the desired estimate.
\end{proof}
\begin{lem}[Boundedness and subcoercivity for $L_{Q}$ on $\dot{\mathcal{H}}_{m}^{3}$]
\label{lem:mapping-LQ-H3}For $v\in\dot{\mathcal{H}}_{m}^{3}$, we
have 
\[
\|L_{Q}v\|_{\dot{H}_{m+1}^{2}}+\|Q|v|_{-2}\|_{L^{2}}\sim\|v\|_{\dot{\mathcal{H}}_{m}^{3}}.
\]
Moreover, the kernel of $L_{Q}:\dot{\mathcal{H}}_{m}^{3}\to\dot{H}_{m+1}^{2}$
is $\mathrm{span}_{\R}\{\Lambda Q,iQ\}$.
\end{lem}

\begin{proof}
By density, we may assume $v\in\mathcal{S}_{m}$. By Lemma \ref{lem:Contribution-QBQ},
we can remove the $QB_{Q}$-part of $L_{Q}$ with an error of the
size $\|r\langle r\rangle^{-2}Q|v|_{-1}\|_{L^{2}}$, so it suffices
to show the boundedness/subcoercivity estimates for $\D_{+}^{(Q)}$.

From now on, we focus on $\|\D_{+}^{(Q)}v\|_{\dot{H}_{m+1}^{2}}$.
We integrate by parts as 
\begin{align*}
\|\tfrac{1}{r^{2}}\D_{+}^{(Q)}v\|_{L^{2}}^{2} & =\|\tfrac{1}{r^{2}}\partial_{+}v\|_{L^{2}}^{2}-2(\tfrac{1}{r^{2}}\partial_{+}v,\tfrac{1}{r^{3}}A_{\theta}[Q]v)_{r}+\|\tfrac{1}{r^{3}}A_{\theta}[Q]v\|_{L^{2}}^{2}\\
 & =\|\tfrac{1}{r^{2}}\partial_{+}v\|_{L^{2}}^{2}+{\textstyle \int}((2m-4+A_{\theta}[Q])A_{\theta}[Q]-\tfrac{1}{2}r^{2}Q^{2})|\tfrac{1}{r^{3}}v|^{2}.
\end{align*}
Using the asymptotics of $A_{\theta}[Q]$ as $r\to0$ or $r\to\infty$,
we are led to 
\[
\|\tfrac{1}{r^{2}}\D_{+}^{(Q)}v\|_{L^{2}}+\|\tfrac{1}{r^{2}}Qv\|_{L^{2}}\sim\|\tfrac{1}{r^{2}}\partial_{+}v\|_{L^{2}}+\|\mathbf{1}_{r\geq1}\tfrac{1}{r^{3}}v\|_{L^{2}}.
\]
Notice the $\|\mathbf{1}_{r\geq1}\frac{1}{r^{3}}v\|_{L^{2}}$ term,
which cannot be obtained for the $\dot{H}_{m}^{3}$-norm when $m=2$;
see Lemma \ref{lem:Comparison-H3-Hdot3}. Noting brief computations
(where we used $|A_{\theta}[Q]|\gtrsim Q^{2}$) 
\begin{align*}
\tfrac{1}{r}\partial_{r}\D_{+}^{(Q)}v & =\tfrac{1}{r}\partial_{r}\partial_{+}v+O(\tfrac{|A_{\theta}[Q]|}{r^{2}}|v|_{-1}),\\
\partial_{rr}\D_{+}^{(Q)}v & =\partial_{rr}\partial_{+}v+O(\tfrac{|A_{\theta}[Q]|}{r}|v|_{-2}),
\end{align*}
and (by \eqref{eq:Def-Hdot3-Appendix})
\[
\|Q|v|_{-2}\|_{L^{2}}+\|\tfrac{A_{\theta}[Q]}{r}|v|_{-2}\|_{L^{2}}\lesssim\|v\|_{\dot{\mathcal{H}}_{m}^{3}},
\]
we obtain 
\begin{align*}
 & \|\D_{+}^{(Q)}v\|_{\dot{H}_{m+1}^{2}}+\|Q|v|_{-2}\|_{L^{2}}\lesssim\|v\|_{\dot{\mathcal{H}}_{m}^{3}},\\
 & \|\D_{+}^{(Q)}v\|_{\dot{H}_{m+1}^{2}}+\|Q|v|_{-2}\|_{L^{2}}\gtrsim\|\partial_{+}v\|_{\dot{H}_{m+1}^{2}}+\|\mathbf{1}_{r\geq1}|v|_{-3}\|_{L^{2}}\gtrsim\|v\|_{\dot{\mathcal{H}}_{m}^{3}}.
\end{align*}
The first estimate establishes the boundedness property of $\D_{+}^{(Q)}$.
The second estimate establishes the subcoercivity property of $\D_{+}^{(Q)}$
by Lemma \ref{lem:Comparison-H3-Hdot3}.

Finally, we characterize the kernel of $L_{Q}:\dot{\mathcal{H}}_{m}^{3}\to\dot{H}_{m+1}^{2}$.
Note the embedding $\dot{\mathcal{H}}_{m}^{3}\hookrightarrow\dot{H}_{m,\mathrm{loc}}^{1}$.
We have seen in Lemma \ref{lem:mapping-L-Appendix} that the kernel
of $L_{Q}:\dot{H}_{m}^{1}\to L^{2}$ is $\mathrm{span}_{\R}\{\Lambda Q,iQ\}$.
Here we combine a cutoff argument to extend this kernel characterization
to $L_{Q}:\dot{H}_{m,\mathrm{loc}}^{1}\to L_{\mathrm{loc}}^{2}$.
This is possible as the nonlocal term $QB_{Q}$ of $L_{Q}$ contains
only the $\int_{0}^{r}$-integral. Let $f\in\dot{H}_{m,\mathrm{loc}}^{1}$
be such that $L_{Q}f=0$. Then, for any $R>0$, $L_{Q}(\chi_{\leq R}f)(x)=0$
for all $|x|\leq R$, thanks to the expression $L_{Q}=\D_{+}^{(Q)}+QB_{Q}$.
Since $\chi_{\leq R}f\in\dot{H}_{m}^{1}$, as in the proof of Lemma
\ref{lem:mapping-L-Appendix}, $f$ is smooth in $|x|<R$. As $R$
was arbitrary, $f$ is a smooth solution to $L_{Q}f=0$ and hence
$f\in\mathrm{span}_{\R}\{\Lambda Q,iQ\}$.
\end{proof}
We now collect the above lemmas to have our main subcoercivity property
at the $\dot{\mathcal{H}}_{m}^{3}$ level.
\begin{lem}[Boundedness and subcoercivity for $A_{Q}^{\ast}A_{Q}L_{Q}$ on $\dot{\mathcal{H}}_{m}^{3}$;
Lemma \ref{lem:Mapping-AAL-Section2}]
\label{lem:subcoer-H3}For $v\in\dot{\mathcal{H}}_{m}^{3}$, we have
\begin{align*}
 & \|A_{Q}^{\ast}A_{Q}L_{Q}v\|_{L^{2}}+\|Q|v|_{-2}\|_{L^{2}}\sim\|v\|_{\dot{\mathcal{H}}_{m}^{3}}.
\end{align*}
Moreover, the kernel of $A_{Q}^{\ast}A_{Q}L_{Q}:\dot{\mathcal{H}}_{m}^{3}\to L^{2}$
is $\mathrm{span}_{\R}\{\Lambda Q,iQ,\rho,ir^{2}Q\}$.
\end{lem}

\begin{proof}
By collecting the boundedness estimates of the previous lemmas, we
see that $A_{Q}^{\ast}A_{Q}L_{Q}:\dot{\mathcal{H}}_{m}^{3}\to L^{2}$
bounded. For the subcoercivity estimate, applying the previous lemmas
yields 
\[
\|A_{Q}^{\ast}A_{Q}L_{Q}v\|_{L^{2}}+\|r\langle r\rangle^{-2}Q(L_{Q}v)\|_{L^{2}}+\|Q|v|_{-2}\|_{L^{2}}\gtrsim\|v\|_{\dot{\mathcal{H}}_{m}^{3}}.
\]
It now suffices to show that 
\[
\|r\langle r\rangle^{-2}Q(L_{Q}v)\|_{L^{2}}\lesssim\|r^{2}\langle r\rangle^{-2}Q|v|_{-2}\|_{L^{2}}\lesssim\|Q|v|_{-2}\|_{L^{2}}.
\]
This is an easy consequence of the following estimates: 
\begin{align*}
\|r\langle r\rangle^{-2}Q(QB_{Q}v)\|_{L^{2}} & \lesssim\|rQ\|_{L^{\infty}}\|r^{-2}QB_{Q}v\|_{L^{2}}\lesssim\|r\langle r\rangle^{-2}Q|v|_{-1}\|_{L^{2}},\\
\|r\langle r\rangle^{-2}Q\D_{+}^{(Q)}v\|_{L^{2}} & \lesssim\|r\langle r\rangle^{-2}Q|v|_{-1}\|_{L^{2}}.
\end{align*}
Therefore, the subcoercivity estimate is proved.

We turn to the characterization of the kernel of $A_{Q}^{\ast}A_{Q}L_{Q}$
in $\dot{\mathcal{H}}_{m}^{3}$. Let $v$ be an element of the kernel.
Since $A_{Q}L_{Q}v\in\dot{H}_{m+2}^{1}$, the positivity in Lemma
\ref{lem:mapping-AastQ} says that $A_{Q}L_{Q}v=0$. Since $L_{Q}v\in\dot{H}_{m+1}^{2}$,
the kernel characterization of $A_{Q}:\dot{H}_{m+1}^{2}\to\dot{H}_{m+2}^{1}$
(Lemma \ref{lem:mapping-AQ}) says that $L_{Q}v\in\mathrm{span}_{\C}\{rQ\}$.
Now, combining the facts $L_{Q}\rho=\frac{1}{2(m+1)}rQ$, $L_{Q}ir^{2}Q=2irQ$,
and the kernel characterization of $L_{Q}$ in $\dot{\mathcal{H}}_{m}^{3}$
completes the proof.
\end{proof}
\begin{lem}[Coercivity for $A_{Q}^{\ast}A_{Q}L_{Q}$ on $\dot{\mathcal{H}}_{m}^{3}$;
Lemma \ref{lem:Coercivity-AAL-Section2}]
\label{lem:coer-H3}Let $\psi_{1},\psi_{2},\psi_{3},\psi_{4}$ be
elements of $(\dot{\mathcal{H}}_{m}^{3})^{\ast}$, which is the dual
space of $\dot{\mathcal{H}}_{m}^{3}$. If the $4\times4$ matrix $(a_{ij})$
defined by $a_{i1}=(\psi_{i},\Lambda Q)_{r}$, $a_{i2}=(\psi_{i},iQ)_{r}$,
$a_{i3}=(\psi_{i},ir^{2}Q)_{r}$, and $a_{i4}=(\psi_{i},\rho)_{r}$
has nonzero determinant, then we have a coercivity estimate 
\[
\|v\|_{\dot{\mathcal{H}}_{m}^{3}}\gtrsim\|A_{Q}^{\ast}A_{Q}L_{Q}v\|_{L^{2}}\gtrsim_{\psi_{1},\psi_{2},\psi_{3},\psi_{4}}\|v\|_{\dot{\mathcal{H}}_{m}^{3}},\qquad\forall v\in\dot{\mathcal{H}}_{m}^{3}\cap\{\psi_{1},\psi_{2},\psi_{3},\psi_{4}\}^{\perp}.
\]
\end{lem}

\begin{proof}
Suppose not. Choose a sequence $\{v_{n}\}_{n\in\N}\subseteq\dot{\mathcal{H}}_{m}^{3}$
such that $\|A_{Q}^{\ast}A_{Q}L_{Q}v_{n}\|_{L^{2}}=\frac{1}{n}$,
$\|v_{n}\|_{\dot{\mathcal{H}}_{m}^{3}}=1$, and $(\psi_{i},v_{n})_{r}=0$
for $i\in\{1,2,3,4\}$. In particular, $\{v_{n}\}_{n\in\N}$ is bounded
in $\dot{\mathcal{H}}_{m}^{3}$. After passing to a subsquence, there
exists $v\in\dot{\mathcal{H}}_{m}^{3}$ such that $v_{n}$ converges
to $v_{\infty}$ weakly in $\dot{\mathcal{H}}_{m}^{3}$ and strongly
in $H_{\mathrm{loc}}^{2}$. By the weak convergence, we have $A_{Q}^{\ast}A_{Q}L_{Q}v_{\infty}=0$
and $(\psi_{i},v_{\infty})_{r}=0$ for all $i$.

On one hand, $v_{\infty}=0$ by the kernel characterization of Lemma
\ref{lem:subcoer-H3} and orthogonality conditions. On the other hand,
the subcoercivity estimate of Lemma \ref{lem:subcoer-H3} says that
$\|Q|v_{n}|_{-2}\|_{L^{2}}\gtrsim1$ uniformly for all large $n$,
and the strong $H_{\mathrm{loc}}^{2}$-convergence says that $\|Q|v_{n}|_{-2}\|_{L^{2}}\to\|Q|v_{\infty}|_{-2}\|_{L^{2}}\gtrsim1$.
This yields $v_{\infty}\neq0$, a contradiction.
\end{proof}

\subsubsection*{Interpolation and $L^{\infty}$ estimates}
\begin{lem}
\label{lem:interpolation-L-infty}For $v\in H_{m}^{3}$, the following
estimates hold. When $m\geq1$, 
\begin{align}
\||v|_{-1}\|_{L^{\infty}}+\|\partial_{rr}v\|_{L^{2}} & \lesssim\|v\|_{\dot{H}_{m}^{1}}^{\frac{1}{2}}\|v\|_{\dot{\mathcal{H}}_{m}^{3}}^{\frac{1}{2}},\label{eq:interpolation}\\
\|r\langle r\rangle^{-1}|v|_{-2}\|_{L^{\infty}} & \lesssim\|v\|_{\dot{\mathcal{H}}_{m}^{3}}.\label{eq:L-infty-weighted-hardy}
\end{align}
In fact, we have stronger estimates when $m\geq2$ 
\begin{align*}
\||v|_{-1}\|_{L^{\infty}}+\||v|_{-2}\|_{L^{2}} & \lesssim\|v\|_{\dot{H}_{m}^{2}}\lesssim\|v\|_{\dot{H}_{m}^{1}}^{\frac{1}{2}}\|v\|_{\dot{\mathcal{H}}_{m}^{3}}^{\frac{1}{2}},\\
\|\langle\log_{-}r\rangle^{-1}|v|_{-2}\|_{L^{\infty}} & \lesssim\|v\|_{\dot{\mathcal{H}}_{m}^{3}},
\end{align*}
and when $m\geq3$ 
\[
\||v|_{-2}\|_{L^{\infty}}\lesssim\|v\|_{\dot{\mathcal{H}}_{m}^{3}}.
\]
\end{lem}

\begin{proof}
By density, we may assume $v\in\mathcal{S}_{m}$.

When $m=1$, $\partial_{+}v=(\partial_{r}-\tfrac{1}{r})v$ so 
\begin{align*}
\|\tfrac{1}{r}v\|_{L^{\infty}}^{2} & \lesssim{\textstyle \int_{0}^{\infty}}|(\tfrac{1}{r}(\partial_{r}-\tfrac{1}{r})v)(\tfrac{1}{r}v)|dr'\lesssim\|\tfrac{1}{r^{2}}\partial_{+}v\|_{L^{2}}\|\tfrac{1}{r}v\|_{L^{2}},\\
\|\partial_{+}v\|_{L^{\infty}}^{2} & \lesssim{\textstyle \int_{0}^{\infty}}|\partial_{r}\partial_{+}v||\partial_{+}v|dr'\lesssim\|\tfrac{1}{r}\partial_{r}\partial_{+}v\|_{L^{2}}\|\partial_{+}v\|_{L^{2}}.
\end{align*}
Thus we have 
\[
\||v|_{-1}\|_{L^{\infty}}^{2}\lesssim\||\partial_{+}v|_{-2}\|_{L^{2}}\||v|_{-1}\|_{L^{2}}\lesssim\|v\|_{\dot{H}_{1}^{1}}\|v\|_{\dot{\mathcal{H}}_{1}^{3}}.
\]
Next, we use $\partial_{rr}v=(\partial_{r}+\tfrac{1}{r})\partial_{+}v$,
the interpolation $\dot{H}_{2}^{1}$ by $L^{2}$ and $\dot{H}_{2}^{2}$,
and the embedding $\dot{\mathcal{H}}_{1}^{3}\hookrightarrow\dot{H}_{1}^{3}$
(Lemma \ref{lem:Comparison-H3-Hdot3})
\[
\|\partial_{rr}v\|_{L^{2}}^{2}\lesssim\|\partial_{+}v\|_{\dot{H}_{2}^{1}}^{2}\lesssim\|\partial_{+}v\|_{L^{2}}\|\partial_{+}v\|_{\dot{H}_{2}^{2}}\lesssim\|v\|_{\dot{H}_{1}^{1}}\|v\|_{\dot{\mathcal{H}}_{1}^{3}}.
\]
For $r\langle r\rangle^{-1}|v|_{-2}$, we again perform the FTC argument
as in $\||v|_{-1}\|_{L^{\infty}}^{2}$ to get 
\begin{align*}
\|r^{-1}\langle r\rangle^{-1}v\|_{L^{\infty}}^{2} & \lesssim\|\min\{\tfrac{1}{r},\tfrac{1}{r^{2}}\}v\|_{L^{\infty}}^{2}\\
 & \lesssim\|\mathbf{1}_{r\leq1}\tfrac{1}{r^{2}}\partial_{+}v\|_{L^{2}}\|\mathbf{1}_{r\leq1}\tfrac{1}{r}v\|_{L^{2}}+\|\mathbf{1}_{r\geq1}\tfrac{1}{r^{2}}|v|_{-1}\|_{L^{2}}^{2}\lesssim\|v\|_{\dot{\mathcal{H}}_{1}^{3}}^{2},
\end{align*}
and 
\[
\|r\langle r\rangle^{-1}|\partial_{+}v|_{-1}\|_{L^{\infty}}^{2}\lesssim\||\partial_{+}v|_{-1}\|_{L^{\infty}}^{2}\lesssim\||\partial_{+}v|_{-2}\|_{L^{2}}\|r^{-1}|\partial_{+}v|_{-1}\|_{L^{2}}\lesssim\|v\|_{\dot{\mathcal{H}}_{1}^{3}}^{2}.
\]
Thus $\|r\langle r\rangle^{-1}|v|_{-2}\|_{L^{\infty}}\lesssim\|v\|_{\dot{\mathcal{H}}_{m}^{3}}$
follows.

When $m\geq2$, we can use \eqref{eq:HardySobolevAppendix} using
$m-1\geq1$ to have 
\[
\||v|_{-1}\|_{L^{\infty}}\lesssim\|\partial_{+}v\|_{L^{\infty}}+\|\partial_{-}v\|_{L^{\infty}}\lesssim\|\partial_{+}v\|_{\dot{H}_{m+1}^{1}}+\|\partial_{-}v\|_{\dot{H}_{m-1}^{1}}\lesssim\|v\|_{\dot{H}_{m}^{2}}
\]
and also by \eqref{eq:GenHardyAppendix} 
\[
\||v|_{-2}\|_{L^{2}}\lesssim\|v\|_{\dot{H}_{m}^{2}}.
\]
We again perform the FTC argument to get 
\[
\|\langle\log_{-}r\rangle^{-1}|v|_{-2}\|_{L^{\infty}}^{2}\lesssim\|\langle\log_{-}r\rangle^{-1}|v|_{-3}\|_{L^{2}}\|r^{-1}\langle\log_{-}r\rangle^{-1}|v|_{-2}\|_{L^{2}}\lesssim\|v\|_{\dot{\mathcal{H}}_{m}^{3}}^{2}.
\]

When $m\geq3$, we can perform the FTC argument to get 
\[
\||v|_{-2}\|_{L^{\infty}}^{2}\lesssim\||v|_{-3}\|_{L^{2}}\|r^{-1}|v|_{-2}\|_{L^{2}}\lesssim\|v\|_{\dot{\mathcal{H}}_{m}^{3}}^{2}.
\]
This completes the proof.
\end{proof}

\subsection{Adapted function space $\dot{\mathcal{H}}_{m}^{5}$}

In this subsection, we assume $m\geq3$.

Recall from Section \ref{subsec:Adapted-function-spaces} that the
$\dot{\mathcal{H}}_{m}^{5}$-norm is defined by

\begin{equation}
\|f\|_{\dot{\mathcal{H}}_{m}^{5}}\coloneqq\|\partial_{+}f\|_{\dot{H}_{m+1}^{4}}+\begin{cases}
\||f|_{-5}\|_{L^{2}} & \text{if }m\geq5,\\
\|\partial_{rrrrr}f\|_{L^{2}}+\|r^{-1}\langle\log_{-}r\rangle^{-1}|f|_{-4}\|_{L^{2}} & \text{if }m=4,\\
\||\partial_{rrrr}f|_{-1}\|_{L^{2}}+\|r^{-1}\langle r\rangle^{-1}\langle\log_{-}r\rangle^{-1}|f|_{-3}\|_{L^{2}} & \text{if }m=3,
\end{cases}\label{eq:Def-Hdot5-Appendix}
\end{equation}
initially for $m$-equivariant Schwartz function $f$. The space $\dot{\mathcal{H}}_{m}^{5}$
is obtained by taking the completion of $\mathcal{S}_{m}$ under the
$\dot{\mathcal{H}}_{m}^{5}$ norm. The choice of $\dot{\mathcal{H}}_{m}^{5}$
is motivated to have boundedness/subcoercivity estimates of the operator
$A_{Q}^{\ast}A_{Q}A_{Q}^{\ast}A_{Q}L_{Q}$.

Most of the results and their proofs in this subsection follow by
shifting the equivariance index by $2$ (i.e. $m\geq1$ to $m\geq3$)
and the regularity index by $2$ (i.e. from $3$ to $5$). Because
the equivariance index and the regularity index are shifted by the
same amount $2$, weighted Hardy's inequality (Corollary \ref{cor:WeightedHardy})
applies in the same way. Henceforth, we shall record the facts and
sketch (or omit) their proofs.

\subsubsection*{Comparison of $\dot{H}_{m}^{5}$ and $\dot{\mathcal{H}}_{m}^{5}$}

$ $\\ We start by comparing the $\dot{H}_{m}^{5}$-norm and $\dot{\mathcal{H}}_{m}^{5}$-norm.
It turns out that they are equivalent when $m\geq5$, but the $\dot{\mathcal{H}}_{m}^{5}$-norm
is\emph{ stronger }than the $\dot{H}_{m}^{5}$-norm when $m\in\{3,4\}$.
\begin{lem}[Comparison of $\dot{H}_{m}^{5}$ and $\dot{\mathcal{H}}_{m}^{5}$]
\label{lem:ComparisonH5H5Appendix}For $f\in\mathcal{S}_{m}$, we
have 
\[
\|f\|_{\dot{H}_{m}^{5}}\sim\|\partial_{+}f\|_{\dot{H}_{m+1}^{4}}
\]
and 
\[
\|f\|_{\dot{\mathcal{H}}_{m}^{5}}\sim\begin{cases}
\|f\|_{\dot{H}_{m}^{5}} & \textrm{if }m\geq5,\\
\|f\|_{\dot{H}_{m}^{5}}+\|\mathbf{1}_{r\geq1}\tfrac{1}{r^{5}}f\|_{L^{2}} & \textrm{if }m=4,\\
\|f\|_{\dot{H}_{m}^{5}}+\|\mathbf{1}_{r\sim1}f\|_{L^{2}} & \textrm{if }m=3.
\end{cases}
\]
Due to $\|\mathbf{1}_{r\gtrsim1}\frac{1}{r^{5}}f\|_{L^{2}}\lesssim\|f\|_{L^{2}}$,
we have 
\[
L^{2}\cap\dot{\mathcal{H}}_{m}^{5}=H_{m}^{5}.
\]
\end{lem}

\begin{rem}
Similarly as in Remark \ref{rem:OptimalityComparisonH3H3}, one cannot
replace $\|\mathbf{1}_{r\geq1}\tfrac{1}{r^{5}}f\|_{L^{2}}$ by $\|\mathbf{1}_{r\sim1}f\|_{L^{2}}$
when $m=4$; and one cannot eliminate $\|\mathbf{1}_{r\sim1}f\|_{L^{2}}$
when $m=3$.
\end{rem}

\begin{proof}
Similarly proceeding as in the proof of Lemma \ref{lem:Comparison-H3-Hdot3},
we have 
\[
\|f\|_{\dot{H}_{m}^{5}}\sim\|\partial_{+}f\|_{\dot{H}_{m+1}^{4}}\lesssim\|f\|_{\dot{\mathcal{H}}_{m}^{5}}.
\]

We now consider the reverse inequality for $f\in\mathcal{S}_{m}$.
When $m\geq5$, we have $\|f\|_{\dot{H}_{m}^{5}}\sim\||f|_{-5}\|_{L^{2}}$
by \eqref{eq:GenHardyAppendix}. Thus $\|f\|_{\dot{H}_{m}^{5}}\gtrsim\|f\|_{\dot{\mathcal{H}}_{m}^{5}}$
easily follows.

When $m=4$, we apply \eqref{eq:GenHardyAppendix} to $\|\partial_{+}f\|_{\dot{H}_{5}^{4}}$
to get $\|\partial_{+}f\|_{\dot{H}_{5}^{4}}+\|\mathbf{1}_{r\geq1}\tfrac{1}{r^{5}}f\|_{L^{2}}\gtrsim\|\mathbf{1}_{r\geq1}|f|_{-5}\|_{L^{2}}$.
This treats the $r\geq1$ part. Thus it suffices to show the $r\leq1$
contribution: 
\begin{align*}
\|r^{-1}|\partial_{+}f|_{-3}\|_{L^{2}}+\|\mathbf{1}_{r\sim1}f\|_{L^{2}} & \gtrsim\|\mathbf{1}_{r\leq1}r^{-1}\langle\log_{-}r\rangle^{-1}|f|_{-4}\|_{L^{2}},\\
\|\partial_{+}f\|_{\dot{H}_{5}^{4}} & \gtrsim\|\mathbf{1}_{r\leq1}\partial_{rrrrr}f\|_{L^{2}}.
\end{align*}
To show the first assertion, an application of Corollary \ref{cor:WeightedHardy}
with $\ell=k=4$ and averaging the boundary term yield 
\[
\|r^{-4}\partial_{+}f\|_{L^{2}}+\|\mathbf{1}_{r\sim1}f\|_{L^{2}}\gtrsim\|\mathbf{1}_{r\leq1}r^{-5}\langle\log_{-}r\rangle^{-1}f\|_{L^{2}}.
\]
Starting from this, one can proceed as in the proof of Lemma \ref{lem:Comparison-H3-Hdot3}
to control higher derivatives, hence getting the first assertion.
The second assertion follows by observing the special algebra 
\[
\partial_{rrrrr}=(\partial_{r}+\tfrac{1}{r})^{4}(\partial_{r}-\tfrac{4}{r}),
\]
and noticing that $\partial_{r}-\tfrac{4}{r}$ is the radial part
of $\partial_{+}$ acting on $4$-equivariant functions.

When $m=3$, it suffices to show the controls 
\begin{align*}
\|\mathbf{1}_{r\leq2}r^{-1}|\partial_{+}f|_{-2}\|_{L^{2}}+\|\mathbf{1}_{r\sim1}f\|_{L^{2}} & \gtrsim\|\mathbf{1}_{r\leq1}r^{-1}\langle\log_{-}r\rangle^{-1}|f|_{-3}\|_{L^{2}},\\
\|\mathbf{1}_{r\geq\frac{1}{2}}r^{-2}|\partial_{+}f|_{-2}\|_{L^{2}}+\|\mathbf{1}_{r\sim1}f\|_{L^{2}} & \gtrsim\|\mathbf{1}_{r\geq1}r^{-2}|f|_{-3}\|_{L^{2}},\\
\|\partial_{+}f\|_{\dot{H}_{4}^{4}} & \gtrsim\||\partial_{rrrr}f|_{-1}\|_{L^{2}}.
\end{align*}
To show the first assertion, apply Corollary \ref{cor:WeightedHardy}
with $\ell=k=3$ and proceed as in the proof of Lemma \ref{lem:Comparison-H3-Hdot3}
to control higher derivatives. To show the second assertion, apply
Corollary \ref{cor:WeightedHardy} with $\ell=3$ and $k=4$. To show
the third assertion, use the special algebra 
\[
\partial_{rrrr}=(\partial_{r}+\tfrac{1}{r})^{3}(\partial_{r}-\tfrac{3}{r})
\]
and notice that $\partial_{r}-\tfrac{3}{r}$ is the radial part of
$\partial_{+}$ acting on $3$-equivariant functions.
\end{proof}

\subsubsection*{Subcoercivity estimates for $A_{Q}^{\ast}A_{Q}A_{Q}^{\ast}A_{Q}L_{Q}$}

$ $\\ We turn to prove (sub-)coercivity estimates of $A_{Q}^{\ast}A_{Q}A_{Q}^{\ast}A_{Q}L_{Q}$
in $\dot{\mathcal{H}}_{m}^{5}$. As previously, we prove (sub-)coercivity
estimates for $A_{Q}^{\ast}A_{Q}A_{Q}^{\ast}$, $A_{Q}$, and $L_{Q}$.
The latter two subcoercivity estimates $A_{Q}:\dot{H}_{m+1}^{4}\to\dot{H}_{m+2}^{3}$
and $L_{Q}:\dot{\mathcal{H}}_{m}^{5}\to\dot{H}_{m+1}^{4}$ are very
similar to $A_{Q}:\dot{H}_{m+1}^{2}\to\dot{H}_{m+2}^{1}$ and $L_{Q}:\dot{\mathcal{H}}_{m}^{3}\to\dot{H}_{m+1}^{2}$
in the previous subsection and their proofs are omitted. As like $A_{Q}^{\ast}$,
the operator $A_{Q}^{\ast}A_{Q}A_{Q}^{\ast}$ turns out to be positive
for $m\geq3$. (See also Remark \ref{rem:Positivity-AAA-small-m}.)
\begin{lem}[Boundedness and positivity for $A_{Q}^{\ast}A_{Q}A_{Q}^{\ast}$]
\label{lem:Positivity-AAA-Appendix}For $v\in\dot{H}_{m+2}^{3}$,
we have 
\begin{equation}
\|A_{Q}^{\ast}A_{Q}A_{Q}^{\ast}v\|_{L^{2}}\sim\|v\|_{\dot{H}_{m+2}^{3}}.\label{eq:Positivity-AAA}
\end{equation}
\end{lem}

\begin{proof}
We first claim the subcoercivity estimate for $A_{Q}^{\ast}A_{Q}A_{Q}^{\ast}$:
\begin{equation}
\|A_{Q}^{\ast}A_{Q}A_{Q}^{\ast}v\|_{L^{2}}+\|r^{2}\langle r\rangle^{-2}Q|v|_{-2}\|_{L^{2}}\sim\|v\|_{\dot{H}_{m+2}^{3}}.\label{eq:Positivity-AAA-claim}
\end{equation}
The perturbative term $\|r^{2}\langle r\rangle^{-2}Q|v|_{-2}\|_{L^{2}}$
will be deleted at the end of the proof. By density, we may assume
$v\in\mathcal{S}_{m+2}$. By Lemmas \ref{lem:mapping-AastQ} and \ref{lem:mapping-AQ},
we have 
\[
\|A_{Q}^{\ast}A_{Q}A_{Q}^{\ast}v\|_{L^{2}}+\|r\langle r\rangle^{-2}Q|v|_{-1}\|_{L^{2}}\sim\|A_{Q}^{\ast}v\|_{\dot{H}_{m+1}^{2}}.
\]
It now suffices to invert $A_{Q}^{\ast}$. Note that $-A_{Q}^{\ast}\approx\partial_{r}-\tfrac{m}{r}$
near the infinity. As $m\geq3$, we can apply Corollary \ref{cor:WeightedHardy}
with $\ell=m$ and $k=2$ (noncritical case) near the infinity. Near
the origin, $-A_{Q}^{\ast}\approx\partial_{r}+\tfrac{m+2}{r}$ so
we can apply Corollary \ref{cor:WeightedHardy} with $\ell=-m-2$
and $k=2$ (noncritical case). Thus one has 
\[
\|\tfrac{1}{r^{2}}A_{Q}^{\ast}v\|_{L^{2}}+\|\langle r\rangle^{-2}Qv\|_{L^{2}}\sim\|\tfrac{1}{r^{2}}|v|_{-1}\|_{L^{2}}.
\]
Starting from this, one can control higher derivatives as before to
get \eqref{eq:Positivity-AAA-claim}. We omit the proof.

To delete the perturbative term of \eqref{eq:Positivity-AAA-claim},
one can proceed as in the proof of Lemma \ref{lem:coer-H3}, but in
this case $A_{Q}^{\ast}A_{Q}A_{Q}^{\ast}:\dot{H}_{m+2}^{3}\to L^{2}$
does not have nontrivial kernel; see Remark \ref{rem:Positivity-AAA-small-m}.
Thus the coercivity estimate follows without any orthogonality conditions.
We omit the proof.
\end{proof}
Having settled the coercivity estiamtes of $A_{Q}^{\ast}A_{Q}A_{Q}^{\ast}$,
it suffices to obtain subcoercivity estimates of $A_{Q}$ and $L_{Q}$
at $\dot{H}^{4}$ and $\dot{H}^{5}$ levels, respectively. Compared
to the previous subsection, both the equivariance index and regularity
index are shifted by the same amount $2$. Thus the following lemmas
are obtained in the same way as for Lemmas \ref{lem:mapping-AQ}-\ref{lem:coer-H3},
after adding two more derivatives to the perturbative terms. We omit
the proofs.
\begin{lem}[Boundedness and subcoercivity for $A_{Q}$]
\label{lem:mapping-AQ-1}For $v\in\dot{H}_{m+1}^{4}$, we have 
\begin{align*}
 & \|A_{Q}v\|_{\dot{H}_{m+2}^{3}}+\|r\langle r\rangle^{-2}Q|v|_{-2}\|_{L^{2}}\sim\|v\|_{\dot{H}_{m+1}^{4}}.
\end{align*}
Moreover, the kernel of $A_{Q}:\dot{H}_{m+1}^{4}\to\dot{H}_{m+2}^{3}$
is $\mathrm{span}_{\C}\{rQ\}$.
\end{lem}

\begin{lem}[Contribution of $QB_{Q}$]
\label{lem:Contribution-QBQ-H5}Let $v\in\dot{\mathcal{H}}_{m}^{5}$.
We have 
\[
\|QB_{Q}v\|_{\dot{H}_{m+1}^{4}}\lesssim\|r\langle r\rangle^{-2}Q|v|_{-3}\|_{L^{2}}.
\]
\end{lem}

\begin{lem}[Boundedness and subcoercivity for $L_{Q}$ on $\dot{\mathcal{H}}_{m}^{5}$]
\label{lem:mapping-LQ-H5}For $v\in\dot{\mathcal{H}}_{m}^{5}$, we
have 
\[
\|L_{Q}v\|_{\dot{H}_{m+1}^{4}}+\|Q|v|_{-4}\|_{L^{2}}\sim\|v\|_{\dot{\mathcal{H}}_{m}^{5}}.
\]
Moreover, the kernel of $L_{Q}:\dot{\mathcal{H}}_{m}^{5}\to\dot{H}_{m+1}^{4}$
is $\mathrm{span}_{\R}\{\Lambda Q,iQ\}$.
\end{lem}

Combining the above subcoercivity lemmas for $A_{Q}$ and $L_{Q}$
with the positivity estimates of $A_{Q}^{\ast}A_{Q}A_{Q}^{\ast}$,
we finally obtain the (sub-)coercivity estimates for $A_{Q}^{\ast}A_{Q}A_{Q}^{\ast}A_{Q}L_{Q}$.
\begin{lem}[Boundedness and subcoercivity for $A_{Q}^{\ast}A_{Q}A_{Q}^{\ast}A_{Q}L_{Q}$
on $\dot{\mathcal{H}}_{m}^{5}$; Lemma \ref{lem:Mapping-AAAAL-Section2}]
\label{lem:subcoer-H5}For $v\in\dot{\mathcal{H}}_{m}^{5}$, we have
\[
\|A_{Q}^{\ast}A_{Q}A_{Q}^{\ast}A_{Q}L_{Q}v\|_{L^{2}}+\|Q|v|_{-4}\|_{L^{2}}\sim\|v\|_{\dot{\mathcal{H}}_{m}^{5}}.
\]
Moreover, the kernel of $A_{Q}^{\ast}A_{Q}A_{Q}^{\ast}A_{Q}L_{Q}:\dot{\mathcal{H}}_{m}^{5}\to L^{2}$
is $\mathrm{span}_{\R}\{\Lambda Q,iQ,\rho,ir^{2}Q\}$.
\end{lem}

\begin{lem}[Coercivity for $A_{Q}^{\ast}A_{Q}A_{Q}^{\ast}A_{Q}L_{Q}$ on $\dot{\mathcal{H}}_{m}^{5}$;
Lemma \ref{lem:Coercivity-AAAAL-Section2}]
\label{lem:coer-H5}Let $\psi_{1},\psi_{2},\psi_{3},\psi_{4}$ be
elements of $(\dot{\mathcal{H}}_{m}^{5})^{\ast}$, which is the dual
space of $\dot{\mathcal{H}}_{m}^{5}$. If the $4\times4$ matrix $(a_{ij})$
defined by $a_{i1}=(\psi_{i},\Lambda Q)_{r}$, $a_{i2}=(\psi_{i},iQ)_{r}$,
$a_{i3}=(\psi_{i},ir^{2}Q)_{r}$, and $a_{i4}=(\psi_{i},\rho)_{r}$
has nonzero determinant, then we have a coercivity estimate 
\[
\|v\|_{\dot{\mathcal{H}}_{m}^{5}}\gtrsim\|A_{Q}^{\ast}A_{Q}A_{Q}^{\ast}A_{Q}L_{Q}v\|_{L^{2}}\gtrsim_{\psi_{1},\psi_{2},\psi_{3},\psi_{4}}\|v\|_{\dot{\mathcal{H}}_{m}^{5}},\qquad\forall v\in\dot{\mathcal{H}}_{m}^{5}\cap\{\psi_{1},\psi_{2},\psi_{3},\psi_{4}\}^{\perp}.
\]
\end{lem}

\subsubsection*{Interpolation and $L^{\infty}$ estimates}
\begin{lem}
\label{lem:interpolation-L-infty-m-geq-3}For $v\in H_{m}^{3}$, the
following estimates hold. When $m\geq3$, 
\begin{align}
\||v|_{-3}\|_{L^{\infty}}+\|\partial_{rrrr}v\|_{L^{2}} & \lesssim\|v\|_{\dot{H}_{m}^{3}}^{\frac{1}{2}}\|v\|_{\dot{\mathcal{H}}_{m}^{5}}^{\frac{1}{2}},\label{eq:interpolation-H5}\\
\|r\langle r\rangle^{-1}|v|_{-4}\|_{L^{\infty}} & \lesssim\|v\|_{\dot{\mathcal{H}}_{m}^{5}}.\label{eq:Weighted-L-infty-H5}
\end{align}
In fact, we have stronger estimates when $m\geq4$
\begin{align*}
\||v|_{-3}\|_{L^{\infty}}+\||v|_{-4}\|_{L^{2}} & \lesssim\|v\|_{\dot{H}_{m}^{4}}\lesssim\|v\|_{\dot{H}_{m}^{3}}^{\frac{1}{2}}\|v\|_{\dot{\mathcal{H}}_{m}^{5}}^{\frac{1}{2}},\\
\|\langle\log_{-}r\rangle^{-1}|v|_{-4}\|_{L^{\infty}} & \lesssim\|v\|_{\dot{\mathcal{H}}_{m}^{5}},
\end{align*}
and when $m\geq5$
\[
\||v|_{-4}\|_{L^{\infty}}\lesssim\|v\|_{\dot{\mathcal{H}}_{m}^{5}}.
\]
\end{lem}

\begin{proof}
By density, we may assume $v\in\mathcal{S}_{m}$. One can proceed
as in the proof of Lemma \ref{lem:interpolation-L-infty} with suitable
modifications. For example, one uses Lemma \ref{lem:ComparisonH5H5Appendix},
\eqref{eq:GenHardyAppendix}, and the algebras $\partial_{+}=\partial_{r}-\frac{3}{r}$
and $\partial_{rrrr}=(\partial_{r}+\frac{1}{r})^{3}\partial_{+}$
when $m=3$.
\end{proof}
\bibliographystyle{abbrv}
\bibliography{References}

\end{document}